\documentclass[reqno,11pt]{amsart}
\usepackage{amsthm,amsfonts,amssymb,euscript,mathrsfs,graphics,color,amsmath,latexsym,marginnote,hyperref}
\usepackage{dsfont}   
\usepackage[latin1]{inputenc}
\usepackage{quoting}

\theoremstyle{plain}

\usepackage{fourier}

\numberwithin{equation}{section}

\newcommand{\ZZZ}{\mathbb{Z}}
\newcommand{\CCC}{\mathbb{C}}
\newcommand{\NNN}{\mathbb{N}}

\newcommand{\RRR}{\mathbb{R}}
\newcommand{\TTT}{\mathbb{T}}
\newcommand{\RR}{{\mathcal R}}

\newcommand{\Op}{\mathrm{Op}\,}

\newcommand{\h}{\eta}  % depth

\newcommand{\set}[1]{\left\{{#1}\right\}}

\newtheorem{theorem}{Theorem}[section]
\newtheorem{proposition}[theorem]{Proposition}
\newtheorem{lemma}[theorem]{Lemma}
\newtheorem{corollary}[theorem]{Corollary}

\newtheorem{remark}[theorem]{Remark}
\newtheorem{remarks}[theorem]{Remark}
\newtheorem{definition}[theorem]{Definition}
\newtheorem{notation}[theorem]{Notation}
\newcommand{\be}{\begin{equation}}
\newcommand{\ee}{\end{equation}}
\newcommand{\teta}{\theta}

\newcommand{\e}{\varepsilon}

\newcommand{\al}{\alpha}

\newcommand{\ov}{\overline}

\newcommand{\R}{\mathbb R}
\newcommand{\C}{\mathbb C}

\newcommand{\Z}{\mathbb Z}

\newcommand{\N}{\mathbb N}
\newcommand{\T}{\mathbb T}
\newcommand{\sign}{{\rm sign}}

\newcommand{\s }{\sigma }
\newcommand{\ii }{{\rm i} }

\newcommand{\bral}{[ \! [} 
\newcommand{\brar}{] \! ]}

\newcommand{\z }{\zeta }

\newenvironment{defn}{\begin{definition} \rm}{\end{definition}} %%

\newcommand{\x }{\xi }
\newcommand{\uno}{\mathds{1}}

\newcommand{\pa}{\partial}

\newcommand{\opw}{{Op^{\mathrm{W}}}}
\newcommand{\opbw}{{Op^{\mathrm{BW}}}}
\newcommand{\opb}{{Op^{\mathrm{B}}}}

\newcommand{\jap}[1]{\left\langle #1 \right\rangle}
\newcommand{\norm}[1]{\big\| #1 \big\|}

\def\hat{\widehat}

\def\bar{\overline}

\def\cal{\mathcal}

\renewcommand{\Re}{\mathrm{Re}\,}

\newcommand{\calF}{{\mathcal F}}

\def\ba{\begin{aligned}}
\def\ea{\end{aligned}}
\def\beginm{\begin{multline}}
\def\endm{\end{multline}}

\newcommand{\MM}{{\mathcal M}}
\providecommand{\vect}[2]{{\bigl[\begin{smallmatrix}#1\\#2\end{smallmatrix}\bigr]}} \providecommand{\sign}{\mathrm{sgn}\,}  
\providecommand{\sm}[4]{{\bigl[\begin{smallmatrix}#1&#2\\#3&#4\end{smallmatrix}\bigr]}}

\makeatletter

\def\l@subsection{\@tocline{2}{0pt}{2.5pc}{5pc}{}}
\def\l@subsubsection{\@tocline{3}{0pt}{4.5pc}{5pc}{}}
\renewcommand\tocchapter[3]{%
  \indentlabel{\@ifnotempty{#2}{\ignorespaces#2.\quad}}#3%
}
\newcommand\@dotsep{4.5}
\def\@tocline#1#2#3#4#5#6#7{\relax
  \ifnum #1>\c@tocdepth % then omit
  \else
    \par \addpenalty\@secpenalty\addvspace{#2}%
    \begingroup \hyphenpenalty\@M
    \@ifempty{#4}{%
      \@tempdima\csname r@tocindent\number#1\endcsname\relax
    }{%
      \@tempdima#4\relax
    }%
    \parindent\z@ \leftskip#3\relax \advance\leftskip\@tempdima\relax
    \rightskip\@pnumwidth plus1em \parfillskip-\@pnumwidth
    #5\leavevmode\hskip-\@tempdima{#6}\nobreak
    \leaders\hbox{$\m@th\mkern \@dotsep mu\hbox{.}\mkern \@dotsep mu$}\hfill
    \nobreak
    \hbox to\@pnumwidth{\@tocpagenum{#7}}\par
    \nobreak
    \endgroup
  \fi}
\def\l@subsection{\@tocline{2}{0pt}{2.5pc}{5pc}{}}

\begin{document}
\bibliographystyle{plain}
\title[A non-linear Egorov Theorem]{A non-linear Egorov Theorem and Poincar\'e-Birkhoff Normal forms for quasi-linear PDEs on the circle} 

\date{}

\author{Roberto Feola}
\address{University of Nantes}
\email{roberto.feola@univ-nantes.fr}

\author{Felice Iandoli}
\address{LJLL (Sorbonne Universit\'e)}
\email{felice.iandoli@sorbonne-universite.fr}

\thanks{  %We would like to thank Michela Procesi for the  
%inspiring discussions and Massimiliano Berti 
%for having introduced us to these interesting problems. 
Felice Iandoli has been supported  by ERC grant ANADEL 757996.
Roberto Feola has been supported by 
 the Centre Henri Lebesgue ANR-11-LABX- 0020-01 
 and by ANR-15-CE40-0001-02 ``BEKAM'' of the Agence Nationale de la Recherche. 
%This research was supported by PRIN 2015 ``Variational methods, with applications to problems in 
%mathematical physics and geometry''.
%The third author was supported in part by Start-up grants from Princeton University and the University of Toronto, 
%and NSERC grant RGPIN-201
}

\begin{abstract} 
 In this paper we consider an abstract class of quasi-linear para-differential equations on the circle. For each equation in the class we prove the existence of a 
 change of coordinates which conjugates the equation to a 
 diagonal and constant coefficient para-differential one.  
 In the case of  Hamiltonian equations we also put them 
 in Poincar\'e-Birkhoff normal form.  By means of the Bony para-linearization formula we prove that it is possible
 to apply this transformation to quasi-linear Hamiltonian 
 perturbations of the Schr\"odinger and beam equations.
 In this way we obtain the first long time existence result without 
 requiring any symmetry on the initial data in the case of quasi-linear PDEs with super-linear dispersion law. 
 We also prove the local in time existence of solutions
 for quasi-linear (not necessarily Hamiltonian) perturbations of the Benjamin-Ono equation.
\end{abstract}

\maketitle

\tableofcontents

\section{Introduction}
In the past years numerous progresses have been done concerning the study of quasi-linear, dispersive, evolution equations on the circle. The local in time solvability has been proven for several models, as well as the existence of large sets of periodic and quasi-periodic in time  solutions. More recently it has been established the long time  existence and stability of small amplitude solutions for the gravity-capillary water waves equation in \cite{BD} and for quasi-linear perturbations of the Schr\"odinger equation in \cite{Feola-Iandoli-Long}. 

Due to the compactness of the circle the dispersive character of the equation is absent in the following sense: the solutions of the linear equation do not decay when the time goes to infinity. On the other hand, thanks, among several other things,  to such a compactness, in the last decade it has been developed  a very fruitful and systematic way to study the linearization of such equations at certain approximate solutions. The linearized equation coming from a quasi-linear problem is much more complicated compared with the one coming from semi-linear equations because it has \emph{variable coefficients}. In this direction the first breakthrough result is due to Plotnikov-Toland \cite{P-T} and Ioos-Plotnikov-Toland \cite{IPT}, the new idea is to apply a suitable diffeomorphism of the circle and \emph{pseudo-differential} changes of coordinates in order to invert the linearized operator at an approximate solution. This is done in the context of a Nash-Moser iterative scheme in order to prove the existence of periodic in time solutions. The strategy in \cite{P-T,IPT} has been improved by Baldi-Berti-Montalto \cite{BBM,BBM1} in order to show the existence of \emph{quasi-periodic} solutions for quasi-linear perturbations of the KdV equation. These ideas  has been used also for several other equations, for instance one can look at \cite{BM1,FP,BBHM,Filippo,FGP1}.
A similar method has been successfully applied also the context 
of reducibility of linear operators. We quote for instance 
\cite{FGMP,BLM,FGP}.

This procedure, which has allowed to obtain so many results in the study of linearized equation, has been, very recently, transported  to a non-linear level by Berti-Delort in \cite{BD} obtaining the aforementioned long-time existence result. 
Inspired by this work, we have proved first the local in time well-posedness  in \cite{FIloc} and then the long time stability in \cite{Feola-Iandoli-Long} for quasi-linear perturbations of the Schr\"odinger equation. In order to obtain long-time stability of the solutions  on compact manifolds the only helpful approach, due to the lack of dispersion, is the  \emph{normal forms} one. For semi-linear equations we have, nowadays, a very good knowledge of the topic, see for instance \cite{DelortSzeft1,DelortSzeft2,BDGS,BG,grefaou}. This approach does not apply directly in the quasi-linear case because of the loss of derivatives introduced by the nonlinearity. In this  direction it is fundamental  to \emph{reduce} the equation to a constant coefficient one before performing a \emph{normal form} procedure. This problem has been extensively  explained in the introductions of \cite{BD} and \cite{Feola-Iandoli-Long}, to which  we refer  for  more details.
 
The non-linear adaptation of such techniques is a difficult problem.
The main  tool  used is the   \emph{para-differential} calculus. 
Roughly speaking one wants  to 
 try to mimic the changes of coordinates previously operated 
 on the linearized equations, on the \emph{para-linearized} 
 equations \`a la Bony \cite{bony}, (see also \cite{Met}). 
 In \cite{BD,Feola-Iandoli-Long} the authors prove that the 
 original equation may be \emph{reduced} to a new one 
 which has constant coefficients up to very regularizing terms. 
 They do not find \emph{changes of coordinates} of the phase space, 
 but they provide some \emph{modified energies} which, in any case, 
 are sufficient to establish the wanted results. In the following 
 we shall fix some notation and we shall provide an example 
 in order  to explain precisely all the  concepts that 
 have been previously mentioned and not been introduced yet.

The main purpose of this paper is to prove a  theorem 
which provides the existence of a  \emph{change of coordinates} 
of the phase space which transforms  a quite general \emph{para-differential} 
equation on the circle to another one which has 
constant coefficients up to very regularizing terms 
(see Theorem \ref{thm:main}). 
For a comprehensive introduction to the importance of 
Hamiltonian PDEs in mathematics and physics we refer to the books
by Berti
\cite{Bertibook}, Kuksin \cite{kuksinbook}, Zhakarov  \cite{zakhbook}.
As a consequence we are able to put in Poincar\'e-Birkhoff normal form a class of abstract \emph{Hamiltonian} para-differential equations satisfying some non-resonance conditions. Consequently we shall obtain a \emph{long time} existence theorem (in the same sense of \cite{BD,Feola-Iandoli-Long}) for the latter class of para-differential equations posed on $H^{s}(\TTT)$. 
We will apply this last theorem to some explicit examples of quasi-linear dispersive PDEs (with super-linear dispersion law) on the circle, obtaining, to the best of our knowledge, the first long time existence result for an equation of this kind without assuming any symmetry on the initial data (in \cite{BD,Feola-Iandoli-Long} the parity of the initial condition is necessary). 
Let us mention that, recently, Berti-Feola-Pusateri in \cite{BFP}, and Berti-Feola-Franzoi
in \cite{BFF1}, prove some long time existence results for 
gravity and gravity-capillary water waves equation without symmetry 
assumptions on the initial data.
 In the quoted papers the authors 
exploit an \emph{a posteriori} identification argument of normal form
which allows them to compute the Birkhoff normal form at the fourth order.
With our method we will be able to compute the Birkhoff normal form, preserving the Hamiltonian structure, at \emph{any} order.

In Section \ref{intro-ego} we introduce the main Theorem of this paper, in \ref{intro-idea} we give some ideas on the proof and we make a comparison between the strategy of this paper and the one in \cite{BD,Feola-Iandoli-Long}. In Section \ref{intro-ham} we shall define a class of \emph{Hamiltonian} para-differential equation which we are able to put in Poincar\'e-Birkhoff normal form. We shall also give some examples of PDEs which satisfy the hypotheses we shall require. Finally in \ref{intro-plan} of this introduction we explain how the paper is organized.

\subsection{Introduction to the Egorov theorem}\label{intro-ego}
Let us introduce some notation and a model problem in order to give an informal statement of our main theorem.  An accurate  definition, with an extensive analysis, of the  para-differential operators is given in Section \ref{paraparapara}.
 We shall  deal with symbols $\TTT\times \RRR\ni (x,\x)\to a(x,\x)$ with limited smoothness in $x$ satisfying, for some $m\in \RRR$, the following estimate
\begin{equation*}
|\partial_{\x}^{\beta}a(x,\xi)|\leq C_{\beta}\langle \x\rangle^{m-\beta}, \;\; \forall \; \beta\in \NNN,
\end{equation*}
where $\langle \x\rangle :=\sqrt{1+|\xi|^{2}}$. These functions will have limited smoothness in $x$ because they will depend on $x$ through the dynamical variable $U$ of a para-differential equation which is in $H^{s}(\TTT)$ for some $s$. From the symbol $a(x,\x)$ one can define the \emph{para-differential} operator $\opb(a(x,\x))[\cdot]$,
acting on periodic functions of the form $u(x)=\sum_{j\in\ZZZ}\hat{u}(j)\frac{e^{\ii j x}}{\sqrt{2\pi}}$, 
 in the following way:
\begin{equation*}
\opb(a(x,\x))[u]:=\frac{1}{2\pi}\sum_{k\in \ZZZ}e^{\ii k x}\left(
\sum_{j\in \ZZZ}\chi\Big(\frac{k-j}{\langle j\rangle}\Big)\hat{a}\big(k-j,j\big)\hat{u}(j)
\right),
\end{equation*}
where $\hat{a}(k,j)$ is the $k^{th}$-Fourier coefficient of the $2\pi$-periodic in $x$ function $a(x,\x)$, and where
$\chi(\h)$ is a $C^{\infty}_0$ function  supported in a sufficiently small neighborhood of the origin. This is the standard para-differential quantization of a symbol, in the paper we shall use the Weyl one, see formula \eqref{bambola202}, in this introduction we preferred to use it in order to simplify the  presentation.

\noindent We consider a para-differential equation of the form
\begin{equation}\label{alienazione}
\partial_tU=\ii E\mathcal{A}(U;x)U+\mathcal{R}(U)U\,,\qquad E=\sm{1}{0}{0}{-1},
\end{equation} 
where $U=(u,\bar{u})$, $R(U)$ is a $2\times 2$-matrix of smoothing operators and $\mathcal{A}$ is a $2\times 2$-matrix of para-differential operators of order $m>1$. We assume that the initial datum is of the form $U(0)=U_0=(u_0,\bar{u}_0)$, where $u_0$ belongs to a classical Sobolev space $H^s(\TTT;\CCC)$ for $s$ big enough, the smoothing remainder maps $H^s$ in $H^{s+\rho}$ for a large number $\rho>0$. We shall make in Section \ref{regularization} some precise  hypotheses on the matrix $\mathcal{A}(U;x)$, we do not recall them here in the complete generality. For simplicity one can think that $\mathcal{A}(U;x)$ is an elliptic, self-adjoint matrix of para-differential operators. We assume moreover that the matrices of operators $\mathcal{A}(U;x)[\cdot]$ and $\mathcal{R}(U)(\cdot)$ are \emph{reality preserving}, i.e. they leave invariant the \emph{real subspace} of $(C^{\infty}(\mathbb{T};\CCC))^2$ made of those couples having the form $(u,\bar{u})$. 

\noindent We shall prove that there exists a non-linear map $\Psi$, which is \emph{bounded} and \emph{invertible} in a small neighborhood of the origin of $H^s$ such that $Z:=\Psi(U)$ solves the following problem
\begin{equation}\label{modellino}
\partial_tZ=\ii E\mathcal{D}(Z)Z+\mathcal{\tilde{R}}(Z)Z\,,
\end{equation} 
where $\mathcal{D}(Z)$ is a reality preserving and \emph{diagonal matrix}, whose entries  are real and \emph{constant in $x$} non-linear functions in $Z$;  $\mathcal{\tilde{R}}(Z)$ is a reality preserving matrix of smoothing operators mapping $H^s$ in $H^{s+\tilde{\rho}}$ with $\tilde{\rho}>0$ possibly smaller than $\rho$. 
This is essentially a rough statement of our main result which is Theorem \ref{thm:main}. The hypotheses of Theorem \ref{thm:main} are slightly milder, it is not strictly necessary to require the self-adjointness of the matrix $\mathcal{A}(U;x)$, see \eqref{Forma-di-Ainizio} and \eqref{operatoreEgor1inizio}. 

To clarify things  one can think about the following concrete example. Consider the quasi-linear perturbation of the Schr\"odinger equation
\begin{equation}\label{esempietto-alienato}
\partial_t u=\ii(1+|u|^2)u_{xx}\,.
\end{equation}
By applying the celebrated para-linearization formula of Bony one finds that such an equation is equivalent to the following system
\begin{equation*}
\partial_tU=\ii E\opb(A(U;x,\xi))U+R(U),
\end{equation*}
where the matrix $A(U;x,\xi)$ has the following form
\begin{equation*}
A(U;x,\xi)=\left(\begin{matrix}
a(U;x,\xi) & b(U;x)\\
\ov{b(U;x)} & a(U;x,\xi)
\end{matrix}\right)\,, \quad \begin{cases}
&a(U;x,\xi)=(1+|u|^2)(\ii\xi)^2+\bar{u}u_{xx}\\
&b(U;x,\xi)=u u_{xx}\,,
\end{cases}
\end{equation*}
and the matrix $R$ is a matrix of smoothing operators. In this case our theorem applies and provides a change of variable $Z=\Psi(U)$ such that the system in the new coordinates reads as
\begin{equation}\label{coniugato-alienato}
\partial_tZ=\ii E\opb\Big(A^+(Z)(\ii\xi^2)\Big)Z+\ell.o.t.\,,
\quad A^+(Z):=\left(\begin{matrix}
a^+(Z)&0\\
0&a^+(Z)
\end{matrix}\right)\,,
\end{equation}
and the function on the diagonal has the form 
\begin{equation*}
a^+(Z):=\int_{\TTT}f(Z)dx \in\RRR
\end{equation*}
for some regular and \emph{real} function $f$. Essentially we have reduced the equation to a \emph{Kirchhoff-type} one in the following sense.
In the new coordinates system the equation is still a \emph{quasi-linear} one, but the coefficients of the para-differential operators are constants.  In this way the system enjoys a property which is peculiar of \emph{semi-linear} equations: the linearized equation at any solution has constant coefficients at highest order. We are confident that it would be convenient  to start a KAM procedure from the regularized system \eqref{coniugato-alienato}, instead of the original system \eqref{alienazione}, because it would be easier to study the invertibility of the linearized operator at periodic/quasi-periodic solutions. This is an advantage of having produced changes of coordinates of the phase space instead of modified energies. 

 A  deeper advantage of  our method is the following. In the case that \eqref{alienazione} is an Hamiltonian equation (in the sense specified in section \ref{intro-ham})  it is possible to find a symplectic version of the aforementioned change of coordinates. 
 In such a way the new system \eqref{coniugato-alienato} is still Hamiltonian. 
 Moreover, since  it is reduced to constant coefficients up to smoothing reminders, it is possible to perform a Poincar\'e-Birkhoff normal argument by means, again, of a symplectic change of coordinates of the phase space. In this way standard arguments on \emph{non-resonant} systems provide control on the growth of Sobolev norms for long time (for a detailed introduction on this topic see \ref{intro-ham}). 
 In this way we obtain, to the best of our knowledge, the first results (in the case of super-linear dispersion law, for the linear  see \cite{Delort-Sphere, Delort-2009}) of long time existence for quasi-linear equations without assuming any symmetry on the initial condition. This is  a consequence of the fact that we are able to fully exploit the \emph{Hamiltonian  structure}. We shall consider quasi-linear Hamiltonian perturbation of the Schr\"odinger and beam equation and prove a long time existence in the spirit of \cite{BD, Feola-Iandoli-Long} but without the parity assumption therein. Besides these applications to long time solutions of Hamiltonian equations, we think that the Egorov theorem \ref{thm:main}  is important \emph{per se}, in order to stress this fact we prove a local well posedness theorem for quasi-linear, not necessarily Hamiltonian, perturbations of the Benjamin-Ono equation. This is deduced from the fact that a system like \eqref{coniugato-alienato} admits energy inequality on $H^s(\TTT)$. In a forthcoming paper we want to apply this theorem also to the gravity-capillary water waves system in order to obtain a long time result without any symmetry on the initial datum, the proof of this fact is much more involved compared with the examples we provide here.

 %We assume moreover that the matrix of operators $\mathcal{A}(U;x)$ has the following form
%\begin{equation}
%\mathcal{A}(U;x):=\opbw\left(\begin{matrix}
%a(U;x,\x) & b(U;x,\x) \\
%\ov{b(U;x,-\x)}&\ov{a(U;x,-\x)}
%\end{matrix}\right)
%\end{equation}
\subsection{Ideas of the proof of the Egorov theorem}\label{intro-idea}
The previously  mentioned non-linear change of coordinates $Z=\Psi(U)$ will be obtained as the composition of numerous changes of coordinates. Each change of coordinates, let us call it $\psi$, will be generated as the \emph{flow} of a para-differential operator as follows
\begin{equation}\label{trasformazione-alienata}
\left\{
\begin{aligned}
&\pa_{\tau}\psi^{\tau}=G^{\tau}(\psi^{\tau})\psi^\tau\\
&\psi^{0}={\rm Id},
\end{aligned}\right.
\end{equation}
where $G^{\tau}$ is some non-linear vector field
possibly depending explicitly on $\tau$. We shall define $\psi:=\psi^{\tau}|_{\tau=1}$. Note that the equation above is a non-linear para-differential equation, whose well-posedness has to be analyzed, this is done in Section \ref{FLOWS}.  Then defining $Z:=\psi(U)$ one has that the new unknown $Z$ solves a system of the form
\begin{equation*}
\partial_tZ=P^{1}(Z)\,,
\end{equation*}
where $P^{\tau}(Z)$, for $\tau\in [0,1]$, is the solution of the non-linear Heisenberg equation
\begin{equation}\label{heisen-alieno}
\left\{
\begin{aligned}
&\pa_{\tau}P^{\tau}(Z)=\big[ G^{\tau}(Z) , P^{\tau}(Z)\big]\\
&P^{0}(Z)=\ii E\mathcal{A}(Z;x)Z+\mathcal{R}(Z)Z\,,
%\opbw(a(z;x,\x))[z]\,.
\end{aligned}\right.
\end{equation}
where we have denoted by $[\cdot,\cdot]$ the non-linear commutator. More precisely, given two vector fields $X(u)$ and $Y(u)$,  we define the non-linear commutator between the two as
 \begin{equation}\label{nonlinCommu}
[X,Y](u):=dX(u)\big[Y(u)\big]-dY(u)\big[X(u)\big]\,.
\end{equation}
We have chosen to generate changes of coordinates in this way because there is a systematic method to prove that the equation for the new variable $Z$ is still a para-differential one. For more details one should look at Section \ref{secconjconj}.

In Section \ref{secconjconj} we perform all the changes of coordinates. First of all one has to diagonalize the matrix of symbols $\mathcal{A}(U)$, this is done by applying several changes of coordinates which diagonalize  $\mathcal{A}(U)$ at each order, this is the content of Section \ref{diago-lineare}. 

Once achieved the diagonalization of $\mathcal{A}(U)$ we shall perform further changes of variables in order to put to constant coefficient the symbol at each order appearing on the diagonal of the new system, this is the content of Section \ref{super-costanti}. Among all the changes of coordinates the most difficult is the one which is needed to remove the $x$-dependence from the symbol at the highest order. We explain here the strategy we adopted for this step, this may be considered as a guideline also for all the other changes of coordinates which are performed in Section \ref{secconjconj}.
To clarify things we explain the idea, having in mind equation \eqref{esempietto-alienato}, on the following toy model
\begin{equation}\label{toy}
\partial_t u=\ii\opb((1+|u|^2)(\ii\xi)^2)u\,.
\end{equation}
We define  $z=\psi^{\tau}(u)_{\tau=1}$, where $\psi^{\tau}$ solves \eqref{trasformazione-alienata} with $G^{\tau}(u):=\ii\opb(b(\tau, u;x)\xi)u$ and $b(\tau,u;x)$ is a path of symbols with $\tau\in[0,1]$. Then the equation in the new variable is $\partial_tz=P^1(z)$, where $P^{\tau}$ solves the Heisenberg equation
\begin{equation*}
\begin{cases}
\partial_{\tau}P^{\tau}(z)=\big[\ii\opb(b(\tau,z;x)\xi)z,P^{\tau}(z)\big]\\
P^0(z)=\ii\opb(1+|z|^2(\ii\xi)^2)z\,.
\end{cases}
\end{equation*}
At this point one can make the ansatz that $P^{\tau}$ is of the form $\opb(a^{+}(\tau;z,x,\xi))z+\ell.o.t.$ for a symbol $a^+$ of order $2$. By developing the commutator, by means of symbolic calculus (see Prop. \ref{teoremadicomposizione}), one obtains that the equation for the highest order  symbol $a^{+}(\tau;z,x,\xi)$ is the following
\begin{equation}\label{alto-alienato}
\begin{cases}
\partial_{\tau}a^+(\tau,Z;x,\xi)=\big\{b(\tau,Z;x)\xi, a^+(\tau,Z;x,\xi)\big\}-d_za^{+}(\tau,z;x,\x)[\opb(\ii  b(\tau,z;x,\x))[z]]\\
a^{+}(0;z,x,\xi)=(1+|z|^2)(\ii\xi)^2\,,
\end{cases}
\end{equation}
where $\{f,g\}:=\pa_{\x}f\pa_{x}g-\pa_{x}f\pa_{\x}g$ 
denotes the Poisson bracket between function. 
Such an equation is a non-linear transport one. In order to solve it  we proceed as follows. Define the function
\begin{equation*}
g(\tau)=a^{+}(\tau,z(\tau);x(\tau),\x(\tau))
\end{equation*}
and note that this function is constant along the flow of the following non-linear system
\begin{equation}\label{sistema-alienato}
\left\{
\begin{aligned}
&\pa_{\tau}x(\tau)=-b(\tau,z(\tau);x(\tau)) \\
& \pa_{\tau}\x(\tau)=b_x(\tau,z(\tau);x(\tau))\x(\tau)\\
&\pa_{\tau}z(\tau)=\opb\big(\ii b(\tau,z(\tau);x(\tau))\x(\tau)\big)[z(\tau)]\,.
\end{aligned}\right.
\end{equation}
Note that the system above is a  system of non-linear coupled equations, its well-posedness is not trivial and it is discussed in Section \ref{sec:egohigh}. Denote by
\begin{equation*}
\widetilde{\Upsilon}^{\tau}_b(z,x,\x):=\Big(\widetilde{\Upsilon}^{(z)}_b,
\widetilde{\Upsilon}^{(x)}_b, \widetilde{\Upsilon}^{(\x)}_b \Big)(\tau,z,x,\x)\,
\end{equation*}
the inverse of the flow of system \eqref{sistema-alienato}, then the solution of equation \eqref{alto-alienato} is
\begin{equation}\label{1-cambio}
a^+(z,x,\xi)=\left(1+\left|\widetilde{\Upsilon}_b^{(z)}
\left(\widetilde{\Upsilon}_b^{(x)}\right)\right|^2\right)
\left(\ii \widetilde{\Upsilon}_b^{(\x)}\right)^2\,.
\end{equation}
We want to find a function $b(\tau,z;x)$ and a constant (w.r.t. $x$) $m_b$ such that one solves the equation $a^+(z,x,\xi)=m_b(\ii\xi)^2$. This is another non-linear and implicit equation which is solved in Theorem \ref{constEgo}. Choosing such a $b$ we have removed the $x$-dependence on the highest order term.

We are now in position to  make a  short comparison between this way of generating changes of coordinates and the ones used in \cite{BD,Feola-Iandoli-Long}. In those papers the authors look for \emph{modified energies}, more precisely, starting from a solution $U(t,x)$ of the equation \eqref{alienazione}, they want to define a new unknown  $W=\Phi(U)U$ in the following way. One requires that $\norm{W}_{L^{\infty}H^s}\sim\norm{U}_{L^{\infty}H^s}$ and that the time-dependent map $\Phi(U)[\cdot]$ is linearly invertible in a small neighborhood of the origin of the space of continuous function in time  with values in $H^s$ for any \emph{fixed} function $U(t,x)$. The map $\Phi(U)[\cdot]$ is chosen in such a way that the equation for the new unknown $W$ has constant coefficient in $x$. By using the existing results on the local well-posedness of the equations, they may recover information on the original variable $U$. Note that  these modified energies make sense only in the case that one already owns a local Cauchy theory for the equation.
An advantage of performing a complete change of coordinates is that we 
do not need any \emph{a priori} local well-posedness result. As we shall see, these changes of coordinates can actually be used in order to prove local well-posedness results for several quasi-linear equations.
 Let us suppose that $U$ is a solution of \eqref{alienazione}, then equation solved by the new unknown (\emph{\`a la} Berti-Delort) $W=\Phi(U)U$ takes the following form
 \begin{equation}\label{energia-aliena}
 \begin{aligned}
 \partial_t W&=\Phi(U)U_t+[\partial_t\Phi(U)]U+\ell.o.t.\\
& =\Phi(U)\ii E\mathcal{A}(U;x)(\Phi(U))^{-1}W+(\partial_t\Phi(U))[(\Phi(U))^{-1}W]+\ell.o.t.\,.
\end{aligned} \end{equation}
If, for instance, the operator $\Phi(U)[\cdot]$ is a para-differential one, then a time derivative falls on the symbol, this fact has to be taken into account in the definition of symbols in \cite{BD,Feola-Iandoli-Long}. In analogy we shall take into account that a differential with respect to the variable $U$ falls on the symbols, as one can see from equations \eqref{heisen-alieno} and \eqref{nonlinCommu}. Further comments on these differences are given in Section \ref{paraparapara}. In the case of the modified energies, since one looks only for a \emph{linear invertibility} of the map $\Phi(U)[\cdot]$, one can realize it as the flow of a linear operator, for instance as follows
\begin{equation}\label{linearissimo}
\begin{cases}
\partial_{\tau}\Phi(\tau)=\ii\opb(F(U;x,\xi))\Phi(\tau)\\
\Phi(\tau)=\rm{Id}\,,
\end{cases}
\end{equation}
whose well-posedness is easier to discuss compared to the non-linear equation \eqref{trasformazione-alienata}. (Note that here one has to reason at \emph{fixed} $U$, instead by following our strategy one would have found the same equation with $U\rightsquigarrow \Phi(\tau)$).
It turns out that, being the equations super-linear, the highest order term in the right hand side of \eqref{energia-aliena} is the first one $\,\,Q:=\Phi(U)\ii E\mathcal{A}(U;x)(\Phi(U))^{-1}$. Moreover one has that $Q=Q^{\tau}|_{\tau=1}$, where $Q^{\tau}$ solves the \emph{linear} Heisenberg equation
\begin{equation}\label{commuta-alieni}
\begin{cases}
\partial_{\tau}Q^{\tau}(U)[\cdot]=\Big[\ii\opb(F(U;x,\xi))[\cdot],Q^{\tau}(U)[\cdot]\Big]_{-}\\
Q^{0}(U)[\cdot]=\ii E\mathcal{A}(U;x)[\cdot]\,,
\end{cases}
\end{equation}
where we denoted by $[\cdot,\cdot]_{-}$ the commutator between two linear operators. Consider again the toy model in \eqref{toy}, let us see how the procedure changes if one looks only for a modified energy in order to put the system to constant coefficients. One should consider the flow in \eqref{linearissimo} with $F(U;x,\xi):=b(u;x)\xi$ with $b\in\RRR$.
Then the highest order operator in the r.h.s. of \eqref{energia-aliena} is of the form $Q^{\tau}(u)w=\opb(a^+(u;x)(\ii\xi)^2)w+\ell.o.t.$. By developing the linear commutators in \eqref{commuta-alieni} on find that the equation solved by the new symbol $a^+(u;x)$ is
\begin{equation}\label{alto-lineare-alienato}
\begin{cases}
\partial_{\tau}a^+(\tau,u;x,\xi)=\big\{b(\tau,u;x)\xi, a^+(\tau,u;x,\xi)\big\}\\
a^{+}(0;u,x,\xi)=(1+|u|^2)(\ii\xi)^2\,.
\end{cases}
\end{equation}
For any \emph{fixed}  $u$ one notes that the function
\begin{equation*}
g(\tau,u;x,\xi)=a^+(\tau,u;x(\tau),\xi(\tau))
\end{equation*}
is constant along the solutions of the system
\begin{equation}\label{sistema-alienato-lineare}
\left\{
\begin{aligned}
&\pa_{\tau}x(\tau)=-b(\tau,u;x(\tau)) \\
& \pa_{\tau}\x(\tau)=b_x(\tau,u;x(\tau))\x(\tau)\,,
\end{aligned}\right.
\end{equation}
which is different from \eqref{sistema-alienato} because here the function $u$ is fixed and it does not depend on $\tau$. At this point one denotes by $(\tilde{G}_b^{(x)},\tilde{G}_b^{\xi})$ the inverse flow of \eqref{sistema-alienato-lineare} and finds 
 the solutions of \eqref{alto-lineare-alienato} as 
 \begin{equation}\label{1-energia-lineare}
 a^+(\tau,u;x,\xi)=\left(1+\left|u\left(\tilde{G}_b^{(x)}(\tau)\right)
 \right|^2\right)\left(\ii \tilde{G}_b^{(\xi)}(\tau)\right)^2\,.
 \end{equation}
 One could show that at the first order of homogeneity the symbol in \eqref{1-energia-lineare} and the one in \eqref{1-cambio} coincide by taking the Taylor expansion at $\tau=0$. In other words the symbol  \emph{\`a la Berti-Delort} in \eqref{1-energia-lineare} is an approximation of the symbol we find through changes of coordinates in \eqref{1-cambio} up to terms of higher homogeneity. 
\subsection{Poincar\'e-Birkhoff normal forms and applications to quasi-linear PDEs}\label{intro-ham}

The main consequence of  the Egorov regularization theorem \ref{thm:main} we are interested in putting in Poincar\'e-Birkhoff normal forms a class of Hamiltonian equations in the following sense. We consider system of the form 
\begin{equation}\label{X_H}
\partial_t U=X_H(U),\quad X_{H}(U)=\left(
\begin{matrix}\ii \pa_{\bar{u}}H\\
-\ii \pa_{u}H
\end{matrix}
\right)=\ii J\nabla H\,, \quad J:=\left(\begin{matrix} 0 & 1 \\ -1 & 0\end{matrix}\right)
\end{equation}
where $ H$ is an Hamiltonian function of the form 
\begin{equation*}
H(U)=\int_{\TTT}\Omega u\cdot\bar{u}dx+\int_{\TTT}F(U)dx\,,
\end{equation*}
where $\Omega$ is a linear \emph{pseudo-differential} operator satisfying some \emph{non-resonance} conditions at order $N$ (see Definition \ref{nonresOmegaCOND}) and $F$ is a non-linear, possibly unbounded,  operator on $H^s$.
The symplectic structure is the one induced by the non-degenerate symplectic form
\begin{equation}\label{symform}
\lambda(U,V):=\int_{\mathbb{T}} U\cdot \ii JV dx=\int_{\mathbb{T}} \ii (u\bar{v}-\bar{u} v)dx\,
\end{equation}
on the space made by the couples $U=\vect{u}{\bar{u}}$, $V=\vect{v}{\bar{v}}$. The Poisson brackets between two Hamiltonian $H,G$ are defined as
\begin{equation}\label{Poisson}
\{G,H\}:=\lambda(X_{G},X_{H})
\stackrel{\eqref{symform}}{=}
\int \ii J\nabla G\cdot\nabla H dx=
%\frac{1}{\ii} 
\ii \int \pa_{u}H\pa_{\bar{u}}G-  \pa_{\bar{u}}H\pa_{{u}}G dx\,.
\end{equation}
We have that
\begin{equation}\label{ham-field}
[X_{G},X_{H}]=-X_{\{G,H\}}\,,
\end{equation}
where $[\cdot,\cdot]$ is defined in \eqref{nonlinCommu}.

We introduce informally our result. Assume that the system \eqref{X_H} satisfy the hypotheses of Theorem \ref{thm:main}, then there exists a change of coordinates $\psi(U)=W$ such that the equation in the new variables reads
\begin{equation}\label{PBN}
\partial_tW=\ii E\Omega W+X_{G_{res}}(W)+X_{\geq N}(W),
\end{equation}
where $X_{G_{res}}(W)$ is the Hamiltonian vector-field of a \emph{resonant} Hamiltonian $G_{res}(W)$, i.e. 
\begin{equation*}
\left\{G_{res}(W),\int_{\TTT}\Omega W\cdot \bar{W}dx\right\}=0\,,
\end{equation*}
where the Poisson parenthesis are defined in \eqref{Poisson} and $X_{\geq N}(W)$ is an Hamiltonian vector-field satisfying 
\begin{equation*}
\Re\left(\langle D\rangle^sX_{\geq N}(W),\langle D\rangle^sW\right)_{L^2}
\leq\norm{W}_{H^s}^{N+2}\,,
\end{equation*}
where $\langle D\rangle$ is the Fourier multiplier defined 
by linearity  in \eqref{Sobnorm6000}.
We say that the system \eqref{PBN} is in Poincar\'e-Birkhoff normal form up to order $N$.
Thanks to the \emph{non-resonance} conditions in Definition \ref{nonresOmegaCOND}, the resonant Hamiltonian is the sum of monomials of the form
\begin{equation*}
w_{n_0}\cdots w_{n_{p}}\ov{w}_{n_{p+1}} \cdots \ov{w}_{n_{j+1}}\,, 
\quad \{|n_0|,\ldots,|n_{p}|\}=\{|n_{p+1}|,\ldots,|n_{j+1}|\}\,,  
\quad p=j/2\,.
\end{equation*}
Thanks to this structure one can prove that the \emph{super-actions} 
$|w_j|^2+|w_{-j}|^2$ are prime integrals of the Hamiltonian. 
This guarantees the long time stability for the 
whole system in the spirit of \cite{Faouplane}.
For a more detailed  statement we refer to Theorem \ref{thm:mainBNF}. 
In fact the classical dynamical consequence of having 
a system in such a form is a \emph{long time} 
existence and stability theorem. More precisely, if we 
consider a solution whose initial datum has size $\varepsilon$,  
then it exists for a time of size, at least, $\varepsilon^{-N}$ and 
moreover  the solution remains of size $\varepsilon$ for this  long time. 
This is the consequence of Corollary \ref{thm:energy}
and a standard bootstrap argument.

The first step in the proof of this theorem is to find a symplectic correction to the change of coordinates given in Theorem \ref{thm:main}. It turns out that this is possible 
up to correcting the generator of the transformation by a
 a smoothing operator, this is the content of Theorem \ref{thm:main2}
 whose proof 
 is done in Section \ref{simplettomorfo}. 
 At this point we obtain an Hamiltonian system having the following structure
\begin{equation}
\partial_tZ=\ii E\big(\Omega Z+\mathcal{D}(Z)Z\big)+\tilde{\mathcal{R}}(Z)Z\,,
\end{equation}
where $\mathcal{D}(Z)Z$ and  $\tilde{\mathcal{R}}(Z)Z$ are  vector-fields satisfying the properties listed below equation \eqref{modellino}. We shall perform two different Poincar\'e-Birkhoff normal forms, one on the matrix of para-differential (with constant coefficients) operators $\mathcal{D}(Z)Z$ and a second one on the matrix of the smoothing remainders  $\tilde{\mathcal{R}}(Z)Z$. This is the content of Section \ref{sec:BNF}.

We shall apply the abstract results we have introduced above to some quasi-linear PDEs on the circle. We conclude this introduction by stating the theorems we obtain on quasi-linear perturbations of the Schr\"odinger, beam and Benjamin-Ono equations.

%\noindent {\bf Hamiltonian structure.} On complex couple $U=\vect{u}{\bar{u}}$, $V=\vect{v}{\bar{v}}$ we define  the non-degenerate symplectic form
%\begin{equation}\label{symform}
%\lambda(U,V):=\int_{\mathbb{T}} U\cdot \ii JV dx=\int_{\mathbb{T}} \ii (u\bar{v}-\bar{u} v)dx\,,
%\qquad 
%J:=\left(\begin{matrix} 0 & 1 \\ -1 & 0\end{matrix}\right)\,, \qquad
%J^{T}=J^{-1}=\left(\begin{matrix} 0 & -1 \\ 1 & 0\end{matrix}\right)\,.
%\end{equation}
%Given a Hamiltonian function $H(U)$  
%we define its Hamiltonian vector field $X_{H}$ as
%the only vector field such that
%\begin{equation*}
%dH(U)[\hat{h}]=\lambda(\hat{h},X_{H}(U))\,, \quad \hat{h}=\vect{{h}}{\bar{{h}}}
%\,\;\;\; U=\vect{u}{\bar{u}}\,.
%\end{equation*}
%A simple computation shows that
%\begin{equation}\label{HamField2}
%X_{H}(U)=\left(
%\begin{matrix}\ii \pa_{\bar{u}}H\\
%-\ii \pa_{u}H
%\end{matrix}
%\right)=\ii J\nabla H\,.
%\end{equation}
%The Poisson brackets between two Hamiltonian $H,G$ are defined as
%\begin{equation}\label{Poisson}
%\{G,H\}:=\lambda(X_{G},X_{H})
%\stackrel{\eqref{symform},\eqref{HamField2}}{=}
%\int \ii J\nabla G\cdot\nabla H dx=
%%\frac{1}{\ii} 
%\ii \int \pa_{u}H\pa_{\bar{u}}G-  \pa_{\bar{u}}H\pa_{{u}}G dx\,.
%\end{equation}
%We have that
%\begin{equation}\label{ham-field}
%[X_{G},X_{H}]=-X_{\{G,H\}}\,,
%\end{equation}
%where $[\cdot,\cdot]$ is defined in \eqref{nonlinCommu}.

\noindent {\bf The quasi-linear Schr\"odinger equation}. Consider the following   equation
\begin{equation}\label{NLS}
\ii \partial_t u-\pa_{xx} u+P_{\vec{m}}*u+f(u,u_x,u_{xx})=0\,,
\quad u=u(x,t), \quad x\in \TTT\,,\quad t\in\mathbb{R}\\
\end{equation}
%where $\TTT:=\RRR/2\pi\ZZZ$. 
where the $f(z_0,z_1,z_2)$ is  a polynomial in the variables 
$(z_0,z_1,z_2)\in\CCC^{3}$
with a zero of order at least $2$ at the origin.
The potential $P_{\vec{m}}(x)=(\sqrt{2\pi})^{-1}\sum_{j\in\ZZZ} \hat{p}(j)e^{\ii jx}$ is a  
real function with 
{real}  Fourier coefficients  and the term $P_{\vec{m}}* u$ denotes the convolution between the potential $P_{\vec{m}}(x)$ and $u(x)=(\sqrt{2\pi})^{-1}\sum_{j\in\ZZZ}\hat{u}(j)e^{\ii jx}$
\begin{equation*}
P_{\vec{m}}*u(x)=\int_{\TTT}P_{\vec{m}}(x-y)u(y)dy=\sum_{j\in\ZZZ}
\hat{p}(j)\hat{u}(j)e^{\ii j x}\,.
\end{equation*}
Concerning the convolution potential $P_{\vec{m}}(x)$ we define its $j$-th Fourier coefficient as follows. Fix $M>0$ and set
\begin{equation}\label{potenziale1}
\hat{p}(j):=\hat{p}_{\vec{m}}({j})=
\sum_{k=1}^{M}\frac{m_{k}}{\langle j\rangle^{2k+1}}\,,
\end{equation}
where $\vec{m}=(m_1,\ldots,m_M)$ is a vector in $\mathcal{O}:=[-1/2,1/2]^{M}$ and $\langle j\rangle=\sqrt{1+|j|^2}$. 
We assume that the nonlinearity in \eqref{NLS} has the form
\begin{equation}\label{HamhypNLS}
f(u,u_x,u_{xx})=(\pa_{\bar{z}_0}F)(u,u_{x})-
\frac{d}{dx}[(\pa_{\bar{z}_{1}}F)(u,u_{x})]\,,
\end{equation}
where 
$\pa_{z_i}:=(\pa_{{\rm Re}({z}_i)}-\ii \pa_{{\rm Im}(z_i)})/\sqrt{2}$ and 
$\pa_{\bar{z}_i}:=(\pa_{{\rm Re}({z}_i)}+\ii \pa_{{\rm Im}({z}_i)})/\sqrt{2}$ for $i=0,1$,
and $F(z_0,z_1)$
is a \emph{real valued} polynomial in $(z_0,z_1)\in\CCC^{2}$
vanishing at order $3$ near the origin.
Thanks to \eqref{HamhypNLS} the  equation \eqref{NLS}
is Hamiltonian with respect to the symplectic form \eqref{symform}, namely
ita can be written in the complex Hamiltonian form
\begin{equation}\label{NLShamver}
\pa_{t}u=\ii \nabla_{\bar{u}}\mathcal{H}(u)=\ii \Omega u+\ii f(u,u_x,u_{xx})
\end{equation}
with Hamiltonian function
\begin{equation}\label{NLShamilton}
\mathcal{H}(u):=\int_{\mathbb{T}}\Omega u\cdot\bar{u}+F(u,u_x) dx\,,
%\qquad 
%\Omega:=-\pa_{xx}+P_{\vec{m}}*
\end{equation}
where the operator $\Omega:=-\pa_{xx}+P_{\vec{m}}*$ is the Fourier multiplier
(recall \eqref{potenziale1})
\begin{equation}\label{OmegoneNLS}
\Omega e^{\ii jx}=\omega_{j} e^{\ii jx}\,,\quad \omega_{j}=\omega_{j}(\vec{m}):=
j^{2}+\hat{p}(j)\,,\quad j\in\mathbb{Z}\,.
\end{equation}
We have the following.

\begin{theorem}{\bf (Long time existence for quasi-linear Schr\"odinger equations).}\label{thm:mainNLS}
There is a zero Lebesgue measure set $\mathcal{N} \subset \mathcal{O}$ such that for any  integer $0\leq N\leq M$ 
and any $\vec{m}\in \mathcal{O}\setminus \mathcal{N}$
there exists $s_0\in \RRR$ such that for any $s\geq s_0 $ there are constants $r_0\in (0,1)$,  $c_N>0$ and $C_N>0$
such that the following holds true.
For any $0<r\leq r_0$ and any function $u_0$ in the ball of radius $r$ 
of $H^{s}(\TTT;\CCC)$,  the equation \eqref{NLS} with initial datum $u_0$ has a unique solution
%,which is even in $x\in \TTT$, and 
\begin{equation}\label{spazioNLS}
u(t,x) \in C^{0}\Big([-T_r,T_r]; H^{s}(\TTT;\CCC)\Big)
\bigcap C^{1}\Big([-T_r,T_r]; H^{s-2}(\TTT;\CCC)\Big)\,,\qquad 
T_{r}\geq c_Nr^{-N}\,.
\end{equation}
Moreover one has that 
\begin{equation}\label{stimaNLS}
\sup_{t\in (-T_{r},T_{r})}\|u(t,\cdot)\|_{H^{s}}\leq C_Nr\,.
\end{equation}
\end{theorem}
Notice that a similar Theorem is given in \cite{Feola-Iandoli-Long}
in the case of a 
\emph{parity-preserving} and 
\emph{reversible} perturbation $f$. 
In that paper it is fundamental to assume 
that the initial condition is an even function of
$x\in \mathbb{T}$.  Here we do not 
need such an assumption since we fully 
exploit the Hamiltonian structure of the equation. 
Notice also that, by using Theorem \ref{thm:main}, 
one could recover the local existence result given in 
\cite{FIloc} by following the strategy we 
suggest for the Benjamin-Ono equation 
in the following.

\noindent{\bf The quasi-linear beam equation.}
We consider the following \emph{quasi-linear} beam equation
\begin{equation}\label{beam1}
\psi_{tt}+\pa_{xx}^{2}\psi+m\psi+p(\psi)=0\,,
\quad \psi=\psi(t,x)\,,\quad x\in \mathbb{T}\,,
\;\;\;t\in\mathbb{R}\,,
\end{equation}
where the \emph{mass parameter} $m\in [1,2]$ and the non linearity has the form
\begin{equation}\label{beam2}
p(\psi)=g(\psi,\psi_x,\psi_{xx},\psi_{xxx},\psi_{xxxx})\,,
\end{equation}
with $g$ a polynomial function with a zero of order at least $2$ at the origin.
We shall also assume that
\begin{equation}\label{beam3}
\begin{aligned}
g(\psi,\psi_x,\psi_{xx},\psi_{xxx},\psi_{xxxx})&=(\pa_{\psi}G)(\psi,\psi_{x},\psi_{xx})
\\
&-\frac{d}{dx}\big[(\pa_{\psi_x}G)(\psi,\psi_{x},\psi_{xx})\big]
+\frac{d^{2}}{dx^2}\big[(\pa_{\psi_{xx}}G)(\psi,\psi_{x},\psi_{xx})\big]\,,
\end{aligned}
\end{equation}
for some polynomial $G(\psi,\psi_x,\psi_{xx})$. For some information
about the model we refer the reader to \cite{EGK}, \cite{Feireisl}, \cite{Ruda}.
 We have the following result.

\begin{theorem}{\bf (Long time existence 
for quasi-linear beam equation).}\label{thm:mainBeam}
There is a zero Lebesgue measure set $\mathcal{N} \subset [1,2]$ 
such that for any  $m\in[1,2]\setminus\mathcal{N}$ any integer 
 $1\leq N$
there exists $s_0\in \RRR$ such that for any $s\geq s_0 $ there are constants $r_0\in (0,1)$,  $c_N>0$ and $C_N>0$
such that the following holds true.
For any $0<r\leq r_0$ and any function $\psi_0$ in the ball of radius $r$ 
of $H^{s}(\TTT;\RRR)$,  the equation \eqref{beam1} %(see also \eqref{beam3})
with initial datum $\psi_0$ has a unique solution and 
\begin{equation}\label{spazioBEAM}
\psi(t,x) \in C^{0}\Big([-T_r,T_r]; H^{s}(\TTT;\RRR)\Big)
\bigcap C^{1}\Big([-T_r,T_r]; H^{s-2}(\TTT;\RRR)\Big)\,,\qquad 
T_{r}\geq c_Nr^{-N}\,.
\end{equation}
Moreover one has that 
\begin{equation}\label{stimaBEAM}
\sup_{t\in (-T_{r},T_{r})}\|\psi(t,\cdot)\|_{H^{s}}\leq C_Nr\,.
\end{equation}
\end{theorem}

\noindent {\bf The quasi-linear Benjamin-Ono equation.}
We consider the following model
\begin{equation}\label{benono}
u_{t}+\mathcal{H}u_{xx}+uu_x+\mathcal{N}(u)=0\,,
\quad u=u(t,x)\,,\quad x\in \mathbb{T}\,,\;\;\;t\in\mathbb{R}\,,
\end{equation}
where the unknown $u(t,x)$ is real valued, $\mathcal{H}$ is the periodic Hilbert transform,
namely the Fourier multiplier
\begin{equation}\label{benono2}
\mathcal{H}e^{\ii jx}=-\ii \sign(j)e^{\ii jx}\,,\quad j\in\mathbb{Z}\,.
\end{equation}
The non linearity $\mathcal{N}(u)$ has the form
\begin{equation}\label{nonlineBenono}
\mathcal{N}(u)=g(u,\mathcal{H}u, u_x,\mathcal{H}u_x, \mathcal{H}u_{xx})
\end{equation}
where $g(z_0,z_1,z_2,z_3,z_4)$ is a real valued polynomial in the variables
$(z_0,z_1,z_2,z_3,z_4)\in \mathbb{R}^{5}$ with a zero
of order at least $3$ in the origin. We assume that 
\begin{equation}\label{benonoassump}
\big(\pa_{z_3}g\big)(u,\mathcal{H}u, u_x,\mathcal{H}u_x, \mathcal{H}u_{xx})=
\frac{d}{dx}\Big[\big(\pa_{z_{4}}g\big)
(u,\mathcal{H}u, u_x,\mathcal{H}u_x, \mathcal{H}u_{xx})\Big]\,.
\end{equation}
Examples of admissible non linearities are the following:
\[
\begin{aligned}
 {\rm (i)}\quad &g= u^{2}\mathcal{H}u_{xx}+2uu_x\mathcal{H}u_x\\
 {\rm (ii)}\quad &g=g(u,\mathcal{H}u,u_x)\,.
\end{aligned}
\]
We remark that the non linearity $\mathcal{N}(u)$ is not
necessarily Hamiltonian, i.e. equation \eqref{benono} does not have necessarily the form
\begin{equation}\label{Hambenono}
u_{t}=J\nabla H(u)+J\nabla K(u)\,, \quad J=-\pa_{x}\,,\;\;\; H(u)=
\int_{\mathbb{T}}\left(\frac{u\mathcal{H}u_{x}}{2}+\frac{u^{3}}{6} \right)dx\,,
\quad
\nabla K(u)={\rm bounded \; operator} \,,
\end{equation}
with $\nabla$ the $L^{2}$-gradient.
Here $H(u)$ is the hamiltonian of the ``unperturbed''  
Benjamin-Ono equation, namely equation \eqref{benono} with $\mathcal{N}\equiv0$.
For further details regarding the admissible perturbations $\mathcal{N}(u)$
we refer the reader to \cite{Baldi1}.
In this paper we assume the \eqref{benonoassump} as an example, 
but of course other choices are possible. We have the following result.
\begin{theorem}[{\bf Local well-posedness for quasi linear Benjamin-Ono-type equations}]\label{teototale}
Consider equation \eqref{benono} with \eqref{nonlineBenono}, \eqref{benonoassump}. 
Then there exists $s_0>0$ such that for any $s\geq s_0$
there exists $r_0>0$ such that, for any $0\leq r\leq r_0$,
and for any $u_0$ in the ball of radius $r$ of 
 $H^{s}(\TTT;\RRR)$ the following holds.
 %there exists $T>0$, depending only on $\|u_0\|_{H^{s}}$, such that  
 The equation \eqref{benono} 
 with initial datum $u_0$ has a unique classical solution $u(t,x)$ 
 such that
\begin{equation}\label{spazioBenjOno}
u(t,x) \in C^{0}\Big([0,T); H^{s}(\TTT;\RRR)\Big)
\bigcap C^{1}\Big([0,T); H^{s-2}(\TTT;\RRR)\Big)\,,\qquad
T\gtrsim r^{-1}\,.
\end{equation}
Moreover there is a constant $C>0$ 
\begin{equation}\label{stimaBenono}
\sup_{t\in [0,T)}\|u(t,\cdot)\|_{H^{s}}\leq C\|u_0\|_{H^{s}}\,.
\end{equation}
\end{theorem}

\subsection{Plan of the paper}\label{intro-plan}
The paper is organized as follows. In Section \ref{paraparapara} we develop a para-differential calculus for symbols which depend non-linearly on a function $U$ in a certain Sobolev space. In Section \ref{regularization} we state our main Theorem \ref{thm:main} and its principal application to Poincar\'e-Birkhoff normal form (Theorem \ref{thm:mainBNF}). Moreover we apply this theorem to some quasi-linear equations, obtaining a result of \emph{long time} existence and stability. In Section \ref{sec:PreEgorov} we prove the well-posedness of several non-linear equation which we need to solve in order to diagonalize and put to constant coefficient the original para-differential equation. In Section \ref{secconjconj} we produce the changes of coordinates whose composition gives the proof of Theorem \ref{thm:main}. In Section \ref{simplettomorfo} we provide symplectic corrections to the previously found changes of coordinates, finally in Section \ref{sec:BNF} we prove Theorem \ref{thm:mainBNF}.

\vspace{0.6em}
\noindent
{\bf Acknowledgments.} We 
 would like to thank Michela Procesi for the  
inspiring discussions and Massimiliano Berti 
for having introduced us to these interesting problems.

 \section{Para-differential calculus}\label{paraparapara}
 In this section we develop a para-differential calculus following the ideas (and notation)
in \cite{BD}. We shall introduce some classes of symbols and operators which slightly differ from the ones in \cite{BD}. As in \cite{BD} we shall define classes of multilinear and non-homogeneous symbols (and their relative para-differential quantization) and smoothing operators.
The main difference between our classes and those in \cite{BD} is the following. The  non-homogeneous symbols in \cite{BD} depend on some extra function $U$ which depends on space-time variables, in the application the function $U$ is indeed the solution of certain evolution PDEs. In our case, see item $(ii)$ in Def. \ref{pomosimb},  the extra function $U$ depends only on the space variable, however, in the case of space-time dependence of the function $U$, we recover the definition given in \cite{BD} as shown in Remark \ref{super-simboli}. Roughly speaking the conditions on the time derivative in \cite{BD} are replaced by  some conditions on the differentials of the symbols with respect to $U$ in \eqref{maremma2} which makes our classes a generalization of the ones in \cite{BD}. 
 
In the following we fix some notation that will be kept until the end of the paper. 
For $s\in \mathbb{R}$ we denote by $H^{s}(\mathbb{T}; \mathbb{C})$ (respectively
$H^{s}(\mathbb{T}; \mathbb{C}^{2})$), with $\mathbb{T}:=\mathbb{R}/2\pi\mathbb{Z}$,
the Sobolev space of $2\pi$-periodic functions with values in $\mathbb{C}$ 
(respectively $\mathbb{C}^{2}$).
Sometimes we shall simply write $H^{s}$
when this does not create confusion.
Moreover if $r>0$ we define the ball of radius $r$ 
\begin{equation*}
B_{r}(H^{s}):=\{U\in H^{s}(\TTT;\CCC):\, \|u\|_{H^{s}}<r\}.
\end{equation*}
 We expand a $ 2 \pi $-periodic function $ u(x) $  in Fourier series as 
\be\label{complex-uU}
u(x) = \sum_{n \in \Z } \hat{u}(n)\frac{e^{\ii n x }}{\sqrt{2\pi}} \, , \qquad 
\hat{u}(n) := \frac{1}{\sqrt{2\pi}} \int_\T u(x) e^{-\ii n x } \, dx \, .
\ee
We also use the notation
\begin{equation}\label{notaFou}
u_n^+ := u_n := \hat{u}(n) \qquad  {\rm and} \qquad  u_n^- := \ov{u_n}  := \ov{\hat{u}(n)} \, . 
\end{equation}
For $n\in \N^*:= \N \! \setminus \! \{0\}$ we denote by $\Pi_{n}$ the orthogonal projector from $L^{2}(\T;\C)$ 
to the subspace spanned by $\{e^{\ii n x}, e^{-\ii nx}\}$,  i.e.
\begin{equation*}
(\Pi_{n}u)(x) := \hat{u}({n}) \frac{e^{\ii nx}}{\sqrt{2\pi}}+\hat{u}({-n})\frac{e^{-\ii nx}}{\sqrt{2\pi}} \, , 
\end{equation*}
while in the case $n=0$ we define the mean
$\Pi_0u=\frac{1}{\sqrt{2\pi}}\hat{u}(0)=\frac{1}{2\pi}\int_{\TTT}u(x)dx$.
We  denote by $ \Pi_n $ also the corresponding projector in  $L^{2}(\T,\C^{2})$.
We shall identify the Sobolev norm $\|\cdot\|_{H^{s}(\mathbb{T},\mathbb{C})}=:\|\cdot\|_{H^{s}}$
with the norms 
\begin{equation}\label{normsequence}
\|u\|_{H^{s}}^{2}:=\|u\|^{2}_{s}:=\sum_{j\in \mathbb{Z}}\langle j\rangle^{2s}|u_{j}|^{2}\,,\qquad
\langle j\rangle:=\sqrt{1+|j|^{2}}\,,\; j\in \mathbb{N}\,.
\end{equation}
We introduce the operator $\langle D\rangle$ defined by linearity as
\begin{equation}\label{Sobnorm6000}
\langle D\rangle e^{\ii j\cdot x}=\langle j\rangle  e^{\ii j\cdot x}\,.
\end{equation}
With this notation we have that the norm in \eqref{normsequence} reads
\begin{equation}\label{Sobnorm2}
\|u\|^{2}_{H^{s}}:=(\langle D\rangle^{s}u,\langle D\rangle^{s} u)_{L^{2}}
\end{equation}
where $(\cdot,\cdot)_{L^{2}}$ denotes the standard complex $L^{2}$-scalar product
 \begin{equation}\label{prodottoorologio}
 (u,v)_{L^{2}}:=\int_{\mathbb{T}}u\cdot\bar{v}dx\,, 
 \qquad \forall\, u,v\in L^{2}(\mathbb{T}^{d},\mathbb{C})\,.
 \end{equation}
If $\mathcal{U}=(U_1,\ldots,U_{p})$
is a $p$-tuple
 of functions, $\vec{n}=(n_1,\ldots,n_p)\in \NNN^{p}$, we set
 \begin{equation}\label{ptupla}
 \Pi_{\vec{n}}\mathcal{U}:=(\Pi_{n_1}U_1,\ldots,\Pi_{n_p}U_p).
 \end{equation}
 For a family $(n_1,\ldots,n_{p+1})\in \NNN^{p+1}$ we denote by
 $\max_{2}(\langle n_1\rangle,\ldots,\langle n_{p+1}\rangle)$,
 the second largest among the numbers 
 $\langle n_1\rangle,\ldots,\langle n_{p+1}\rangle$.
% We use the notation $A\lesssim B$ to denote $A\le C B$ where $C$ 
%is a positive constant possibly depending on fixed parameters 
%given by the problem. We use the notation 
%$A\lesssim_y B$ to denote $A\le C(y) B$ 
%if we wish to highlight the dependence 
%on the variable $y$ of the constant $C(y)>0$.
 
 \vspace{0.5em}
\noindent
{\bf Notation.}
 $A\lesssim_{s} B$
means $A \leq C(s) B$
 where $C(s) > 0 $ 
 is a  constant depending on $s \in \R $.

%\subsection{para-differential calculus}

\subsection{Classes of operators}

\subsubsection{Classes of symbols}

We give the definition of a class of  symbols we shall use along the paper.
\begin{definition}{\bf (Classes of  symbols).}\label{pomosimb}
Let $m\in\R$, $p, N \in \N$, $r>0 $.

\noindent
$(i)$ {\bf  $p$-homogeneous symbols.} We denote by $\widetilde{\Gamma}_{p}^{m}\!$ the space of symmetric $p$-linear maps
from $({H}^{\infty}(\T;\C^{2}))^{p}$ to the space of $C^{\infty}$ functions of $(x,\x)\in \T\times\R$, 
$ \mathcal{U}\to ((x,\x)\to a(\mathcal{U};x,\x)) $,  
satisfying the following. There is $\mu>0$ and,  
for any $\alpha,\beta\in \N $,  there is $C>0$ such that
\begin{equation}\label{pomosimbo1}
|\pa_{x}^{\alpha}\pa_{\x}^{\beta}a(\Pi_{\vec{n}}\mathcal{U};x,\x)|\leq C | \vec{n} |^{\mu+\alpha}\langle\x\rangle^{m-\beta}
\prod_{j=1}^{p}\|\Pi_{n_j}U_{j}\|_{L^{2}} 
\end{equation}
for any $\mathcal{U}=(U_1,\ldots, U_p)$ in $({H}^{\infty}(\T;\C^{2}))^{p}$,
and $\vec{n}=(n_1,\ldots,n_p)\in (\N^*)^{p}$. 
Moreover we assume that, if for some $(n_0,\ldots,n_{p})\in \N\times(\N^*)^{p}$,
\begin{equation}\label{pomosimbo2}
\Pi_{n_0}a(\Pi_{n_1}U_1,\ldots, \Pi_{n_p}U_{p};\cdot)\neq0 \, ,
\end{equation}
then there exists a choice of signs $\s_0,\ldots,\s_p\in\{-1,1\}$ such that $\sum_{j=0}^{p}\s_j n_j=0$.
For $p=0$ we denote by $\widetilde{\Gamma}_{0}^{m}$ the space of constant coefficients symbols
$\x\mapsto a(\x)$ which satisfy \eqref{pomosimbo1} with $\al=0$
and the right hand side replaced by $C\langle \x\rangle^{m-\beta}$. 
In addition we require the translation invariance property
\begin{equation}\label{def:tr-in}
a( \tau_\teta {\cal U}; x, \xi) =  a( {\cal U}; x + \theta, \xi) \, , \quad \forall \theta \in \R \, . 
\end{equation}

\noindent
$(ii)$ {\bf Non-homogeneous symbols.} Let $ p \geq 1 $, $d\geq0$. We denote by $\Gamma^{m,d}_{p}[r]:=\Gamma^m_{p}[r]$ the space 
of functions $(U;\! x,\!\xi)\!\!\mapsto a(U;x,\xi)$, defined for 
$U\in B_{r}(H^{s_0})$,
for some large enough $ s_0$, with complex values such that for any $s\geq s_0$, there are $C>0$, $0<r_0<r$ and for any $U\in B_{r_0}(H^{s_0})\cap H^{s}$
 and any $\alpha, \beta \in\N$, with $\alpha+d\cdot k\leq s-s_0$, the following holds 
\begin{equation}\label{maremma2}
\begin{aligned}
|\pa_{x}^{\alpha}\pa_{\x}^{\beta}&(d_{u}^{k}a(U;x,\x)[h_{1},\ldots,h_{k}])|\leq 
C \langle \x\rangle^{m-\beta} \Big[\max\{0,p-k\}\|{U}\|_{H^{s_0}}^{\max\{0,p-k-1\}}
\|{U}\|_{H^{s_0+\alpha}}
\prod_{j=1}^{k}\|{h_{j}}\|_{H^{s_0}}+\\
&+\|{U}\|_{H^{s_0}}^{\max\{0,p-k\}}
\sum_{i=1}^{k}\prod_{j=1,j\neq i}^{k}\|{h_{j}}\|_{H^{s_0}}^{\nu}
\|{h_{i}}\|_{H^{s_0+\alpha}}\Big]
\end{aligned}
\end{equation}
for any $h_{j}\in H^{s}$, where $\nu=1$ if $\,k\geq2\,$ and $\nu=0$ otherwise. 

%$(ii)$ {\bf Non-homogeneous symbols.} Let $ p \geq 1 $. We denote by $\Gamma^m_{p}[r]$ the space 
%of functions $(U;x,\xi)\!\mapsto \!a(U;x,\xi)$, defined for 
%$U\in B_{r}(H^{s_0})$,
%for some large enough $ s_0$, with complex values such that for any 
%$0\leq k\leq K-K'$, any $s\geq s_0$, there are $C>0$, $0<r(s)<r$ and for any $U\in B_{r(s)}(H^{s_0})\cap H^{s}$
% and any $\alpha, \beta \in\N$, with $\alpha\leq s-s_0$
%\begin{equation}\label{simbo}
%|{ \partial_x^{\alpha}\partial_{\xi}^{\beta}a(U;x,\xi)}|\leq 
%C\langle\xi\rangle^{m-\beta} 
%\|U\|^{p-1}_{H^{s_0}}
%\|{U}\|_{H^{s}} \, ;
%\end{equation}
%for any $0\leq k\leq K-K'$ with $s_0+\alpha+k\leq s$,
% if $k\leq p$ then
%\begin{equation}\label{maremma2}
%\begin{aligned}
%|\pa_{x}^{\alpha}\pa_{\x}^{\beta}&(d_{u}^{k}a(U;x,\x))[h_{1},\ldots,h_{\s}]|\leq 
%C \langle \x\rangle^{m-\beta}\|{U}\|_{H^{s_0}(\mathbb{T}))}^{p-k-1}
%\|{U}\|_{H^{s_0+\alpha}(\mathbb{T})}
%\prod_{j=1}^{k}\|{h_{j}}\|_{H^{s_0}(\mathbb{T})}+\\
%&+\|{U}\|_{H^{s_0}(\mathbb{T})}^{p-\s}
%\sum_{i=1}^{k}\prod_{j=1,j\neq i}^{k}\|{h_{j}}\|_{H^{s_0}(\mathbb{T})}
%\|{h_{i}}\|_{H^{s_0+\alpha}(\mathbb{T})}
%\end{aligned}
%\end{equation}
%for any $h_{j},u\in H^{s}$ with $j=1,\ldots,p$ 
% for some constant $C>0$; if $k\geq p $
%  \begin{equation}\label{maremma3}
% \begin{aligned}
%|\pa_{x}^{\alpha}\pa_{\x}^{\beta}
%&(d_{u}^{k}a(U;x\x))[h_{1},\ldots,h_{\s}]|\leq 
%C \langle \x\rangle^{m-\beta}
%\sum_{i=1}^{k}\prod_{j=1,j\neq i}^{k}\|{h_{j}}\|_{H^{s_0}(\mathbb{T})}
%\|{h_{i}}\|_{H^{s_0+\alpha}(\mathbb{T})}
% \end{aligned}
% \end{equation}

\noindent
$(iii)$ {\bf Symbols.} We denote by $\Sigma\Gamma^{m}_{p}[r,N]$
the space of functions 
$(U,x,\x)\to a(U;x,\x)$ such that there are homogeneous symbols 
$a_{q}\in \widetilde{\Gamma}_{q}^{m}$,
 $q=p,\ldots, N-1$, and a non-homogeneous symbol $a_{N}\in \Gamma^{m}_{N}[r]$
such that
\begin{equation}\label{simbotot1}
a(U;t,x,\x)=\sum_{q=p}^{N-1}a_{q}(U,\ldots,U;x,\x)+a_{N}(U;x,\x) \, .
\end{equation}
We denote by $\Sigma\Gamma^{m}_{p}[r,N]\otimes \mathcal{M}_{2}(\C)$ the space $2\times 2$
matrices with entries  in  $\Sigma\Gamma^{m}_{p}[r,N]$.
\end{definition}
\begin{remark}\label{omissione}
We omit the dependence on the number $d$ in the definition of $\Gamma^{m}_p[r,N]$ because this is a number that will be fixed once for all in the procedure that we shall implement in Sections  \ref{secconjconj} and \ref{sec:BNF} of our paper. In the classical definition of symbols one does not have the loss $d\cdot k$ appearing in the definition above. We have to include this loss in our definition because of the following reason. In Sections \ref{secconjconj} and \ref{sec:BNF} we shall perform some changes of variable, in one of them we shall redefine the variable of the ambient space $x\in\TTT$ in function of a solution of a para-differential equation whose principal symbol, in these notation, has order $d$. For this reason the smoothness in $x$ and in $U$ of the symbols is linked by the relation $\alpha+d\cdot k\leq s-s_0$. In such a way our classes will be closed for the changes of coordinates that we shall define in Sections \ref{secconjconj} and \ref{sec:BNF}.
\end{remark}

\begin{remark}\label{super-simboli}
Let us make a comment about the \textbf{non-homogeneous symbols} defined in the item $(ii)$ of Def. \ref{pomosimb}. We claim that if we plug  a function $U(t,x)$ (depending on the space-time couple $(x,t)$) in a symbol $a(U;x,\xi)$ we recover the definition given in the previous papers \cite{Feola-Iandoli-Long} and \cite{BD}, more precisely the following holds true. Let $K-K'>0$,  suppose for simplicity that $K-K'< p$, and consider a function $U\in C^{K-K'}_*(I,H^s)$ defined in Section 2 in \cite{Feola-Iandoli-Long}. Let $a$ be a symbol satisfying \eqref{maremma2} with $s_0:=\s_0-2(K-K')\gg1$. For any $0\leq k\leq K-K'$ we have
\begin{equation*}
\begin{aligned}
|\partial_t^k\partial_x^{\alpha}\partial_{\xi}^{\beta}a(U;t,x,\xi)|\leq& C\sum_{\gamma=1}^k\sum_{k_1+\ldots+k_{\gamma}=k}\langle\xi\rangle^{m-\beta}\\
&\Big[(p-\gamma)\|U\|_{\sigma_0-2(K-K')}^{p-\gamma-1}\|U\|_{\sigma_0-2(K-K')+\alpha}\prod_{j=1}^{\gamma}\|\partial_t^{k_j}U\|_{\sigma_0-2(K-K')}\\
&+\|U\|^{p-\gamma}_{\s_0}\sum_{i=1}^{\gamma}\prod_{i\neq j=1}^{\gamma}\|\partial_t^{k_j}U\|_{\s_0-2(K-K')}\|\partial_t^{k_i}U\|_{\s_o-2(K-K')+\alpha}\Big],
\end{aligned}
\end{equation*}
therefore, by recalling that $0\leq k\leq K-K'$ and $k_j\leq k$, one obtains the thesis. In order to recover the definition of non-homogeneous smoothing operator given in \cite{BD} the computation is the same.
\end{remark}

\begin{remark}\label{simbolosulladiagonalestima}
Let $s_0>\mu+1/2$ and $s\geq s_0$. Consider a function $U$ in $B_r(H^{s_0})\cap H^s$, then for any  $0\leq\alpha \leq s-s_0$ and $\beta\in\N$ the function $\tilde{a}(U;\xi):=a(U,\ldots,U;\xi)$ defines a non-homogeneous symbol as in item $(ii)$ of Definition \ref{pomosimb}.
%\begin{equation}\label{simbotame}
%\asso{\partial_t^k\partial_x^{\alpha}\partial_x^{\beta}a(U,\ldots,U;x,\xi)}\leq C\jap{\xi}^{m-\beta}\norm{U}{k,\sigma_0+\alpha}\norm{U}{k,\sigma_0}^{p-1}.
%\end{equation}
To see this let us develop  in Fourier series
\begin{equation*}
a(U,\ldots,U;x,\xi)=\sum_{{n_1,\ldots,n_p\in\N}}a(\Pi_{n_1}U,\ldots,\Pi_{n_p}U;\xi),
\end{equation*}
therefore by using condition \eqref{pomosimbo1} and supposing, for simplicity, that $n_1\geq\ldots\geq n_p$ we deduce that
\begin{equation*}
\begin{aligned}
&|\partial_x^{\alpha}\partial_{\x}^{\beta}a(U,\ldots,U;x,\xi)|\leq\\
&C\sum_{{n_1,\ldots,n_p\in\N}}\jap{\xi}^{m-\beta}\jap{n_1}^{\mu+\alpha}\prod_{j=1}^p\norm{\Pi_{n_j}U}_{L^2}=\\
&C\sum_{{n_1,\ldots,n_p\in\N}}\jap{\xi}^{m-\beta}\jap{n_1}^{\mu+\alpha}\prod_{j=1}^p \jap{n_j}^{-\sigma_0}\prod_{j=1}^p\jap{n_j}^{\sigma_0}\norm{\Pi_{n_j}U}_{L^2}.
\end{aligned}
\end{equation*}
From the latter inequality it is easy to obtain the condition \eqref{maremma2} by using the Cauchy-Schwartz inequality and the fact that $\mu-s_0<-1/2$ in the case that no differentials, with respect to $U$, on the symbols are taken. The estimate for the differentials of the symbol $\tilde{a}(U;\xi)$, i.e. the case $k>0$ in the notation of item $(ii)$ of Definition \ref{pomosimb},  is trivial thanks to multi-linearity of the symbol $a$ with respect its arguments.
\end{remark}
\noindent We also introduce the following class of ``functions'', i.e. those  symbols which are independent of the variable $\xi$.
\begin{definition}[{\bf Functions}]\label{apeape} Fix $N\in \NNN$, 
$p\in \NNN$ with $p\leq N$,
 %'\in \NNN$ 
 with  $r>0$.
We denote by $\widetilde{\calF}_{p}$ (resp. $\calF_{p}[r]$,
resp. $\Sigma\calF_{p}[r,N]$)
the subspace of $\widetilde{\Gamma}^{0}_{p}$ (resp. $\Gamma^{0}_{p}[r]$, 
resp. $\Sigma\Gamma^{0}_{p}[r,N]$)
made of those symbols which are independent of $\x$.
We shall write $\widetilde{\calF}^{\mathbb{R}}_{p}$ (resp. $\calF^{\mathbb{R}}_{p}[r]$,
resp. $\Sigma\calF^{\mathbb{R}}_{p}[r,N]$ 
to denote those functions in the class 
$\widetilde{\calF}_{p}$ (resp. $\calF_{p}[r]$,
resp. $\Sigma\calF_{p}[r,N]$) which are real valued.
\end{definition}

\begin{remark}\label{frechet}
The class of symbols $\Gamma^{m}_{p}[r]$, defined in item $(ii)$ 
of Definition \ref{pomosimb}, restricted to $B_r(H^s)$, is a Fr\'ech\'et 
space equipped with the family of semi-norms
\begin{equation}\label{seminormagamma}
|a|^{\Gamma^m,s}_{\alpha,\beta,k}:=\inf\{C>0:\,\,\eqref{maremma2} \; \mbox{holds true}\,\}
\end{equation}
%\begin{equation}\label{seminormagamma}
%|a|^{\Gamma^m,s}_{\alpha,\beta,k}:=\sup_{x\in\T}\sup_{\xi\in\R}\sup_{U\in B_{r}(H^{s})}\sup_{h_1,\ldots,h_k\in B_1(H^{s})}\jap{\xi}^{\beta-m}|\partial_x^{\alpha}\partial_{\xi}^{\beta}D_U^k a(U;x,\xi)[h_1,\ldots,h_k]|,
%\end{equation}
with $\alpha+d\cdot k\leq s-s_0$.
%Note that by \eqref{maremma2} the semi-norms above are  well defined and moreover  we have
%\begin{equation}\label{semi-piccola}
%|a|^{\Gamma^m}_{\alpha,\beta,k}\leq C r^{\max\{0,p-k\}}.
%\end{equation}
 The semi-norms on the space of functions $\mathcal{F}_{p}[r]$ are analogously defined and we shall denote them by $|a|^{\mathcal{F},s}_{\alpha,k}$.
%\begin{equation}\label{seminormaF}
%|a|^{\mathcal{F}}_{\alpha,k}:=\sup_{x\in\T}\sup_{U\in B_{r}(H^{s})}\sup_{h_1,\ldots,h_k\in B_1(H^{s})}|\partial_x^{\alpha}D_U^k a(U;x,\xi)[h_1,\ldots,h_k]|,
%\end{equation}
%which satisfy an inequality similar to \eqref{semi-piccola}. 
Moreover we shall use the following notation
\begin{equation}\label{semi-norma-totale}
|a|_s^{\mathcal{F}}:=\sum_{\alpha+d\cdot k\leq s-s_0} |a|^{\mathcal{F},s}_{\alpha,k}.
\end{equation}
The quantity  defined in \eqref{semi-norma-totale}  is a norm on the space $\mathcal{F}_{p}[r]$. The couple $(\mathcal{F}_{p}[r],|a|_s^{\mathcal{F}})$ defines a Banach space.
We shall also write
\begin{equation}\label{semi-norma-totale2}
|a|_{s,0}^{\mathcal{F}}:=\sum_{\alpha\leq s-s_0} |a|^{\mathcal{F},s}_{\alpha,0}.
\end{equation}
\end{remark}

\begin{remark}\label{smallnessSemi}
Let $a\in \mathcal{F}_1[r]$ restricted to $B_{r}(H^{s})$. Then, recalling the definition 
\eqref{seminormagamma}-\eqref{semi-norma-totale}, we have for any $\alpha\leq s-s_0$
\[
\|\pa_{x}^{\alpha}a(U;x)\|_{L^{\infty}_{x}}\lesssim_{s}  |a|^{\mathcal{F},s}_{\alpha,0}r\lesssim_{s}
|a|^{\mathcal{F}}_{s} r\,.
\]
More in general, for $a\in \Gamma_1^{m}[r]$, we have
\[
\langle \x\rangle^{-m}\|\pa_{x}^{\alpha}a(U;x)\|_{L^{\infty}_{x}}\lesssim_{s}  
|a|^{\Gamma^{m},s}_{\alpha,0,0}r\,.
\]
\end{remark}

\begin{remark}\label{prodottoSimboli}
Let $a\in \Gamma^{m}_{p}[r]$, $b\in \Gamma^{m'}_{p}[r]$. One can check, by using the Leibniz rule,
that
\begin{equation*}
|ab|_{\alpha,\beta,k}^{\Gamma^{m+m'},s}\lesssim_s\sum_{\substack{\alpha_1+\alpha_2=\alpha \\
 \beta_1+\beta_2=\beta\\
 k_1+k_2=k}}
 |a|_{\alpha_1,\beta_1,k_1}^{\Gamma^{m},s}|b|_{\alpha_2,\beta_2,k_2}^{\Gamma^{m'},s}
 \lesssim_{s}|a|_{\alpha,\beta,k}^{\Gamma^{m},s}|b|_{\alpha,\beta,k}^{\Gamma^{m'},s}\,.
\end{equation*}
Consider also a function $f(x)$ which is analytic in some neighborhood of the origin of $\mathbb{C}$
and let $c\in \mathcal{F}_{1}[r]$.
Then, using the formula of Faa di Bruno (i.e. the formula for the derivatives of the the composition of functions), one has that $d:=f(c)$ is a function in $ \mathcal{F}_{1}[r]$. In particular
\[
|d|_{s}^{\mathcal{F}}\lesssim_{s}C\,,
\]
where $C>0$ is a constant depending only on $s$ and the semi-norm $|c|_{s}^{\mathcal{F}}$.
If $f(x)=1/(1+x)-1$ and $c\in \Gamma^{m}_{1}[r]$, then $d=f(c)\in \Gamma^{m}_{1}[r]$
and 
\[
|d|_{\alpha,\beta,k}^{\Gamma^{m},s}\lesssim_{s}C\,,
\]
for some  constant $C>0$ depending only on $s$ and the 
seminorm $|c|_{\alpha,\beta,k}^{\Gamma^{m},s}$.
\end{remark}

\begin{remark}{\bf (Expansion of Homogeneous symbols).}\label{espansioneSSS}
Let $a_{p}\in \widetilde{\Gamma}_{p}^{m}$ (see Def. \ref{pomosimb}).
Notice that, by the \emph{autonomous} condition \eqref{pomosimbo2}
and the $x$-\emph{translation invariant} property \eqref{def:tr-in},
we can write the symbol $a_{p}$, expanding $u$ as in \eqref{complex-uU}, \eqref{notaFou},
as
\begin{equation}\label{espandoFousimbo}
a_{p}(U;x,\x)=\sum_{j\in \mathbb{Z}}e^{\ii jx}\sum_{\substack{\s_i\in\{\pm\}\,,i=1,\ldots,p \\
\sum_{i=1}^{p}\s_i j_i=j}} (a_{p})_{j_1,\ldots,j_p}^{\s_1\cdots\s_{p}}(\x)u_{j_1}^{\s_1}
\ldots u_{j_p}^{\s_p}\,,
\end{equation}
for some coefficients $(a_{p})_{j_1,\ldots,j_p}^{\s_1\cdots\s_{p}}(\x)\in 
\widetilde{\Gamma}_0^{m}$.
\end{remark}

 \subsubsection{Spaces of Smoothing operators}

We now introduce some classes of smoothing operators.

 \begin{definition}[{\bf Classes of smoothing operators}]\label{omosmoothing}
 Let $\rho\in \mathbb{R}$, with $\rho\geq 0$, $p, N \in \N$, $r>0 $.
 
 \noindent
 $(i)$ {\bf $p-$homogeneous smoothing operator.}
 We denote by $\widetilde{\RR}^{-\rho}_{p}$
 the space of $(p+1)$-linear maps
 from the space $(C^{\infty}(\TTT;\CCC^{2}))^{p}\times C^{\infty}(\TTT;\CCC)$ to 
 the space $C^{\infty}(\TTT;\CCC)$ symmetric
 in $(U_{1},\ldots,U_{p})$, of the form
 $ (U_{1},\ldots,U_{p+1})\to R(U_1,\ldots, U_p)U_{p+1},$
% \begin{equation}\label{omoresti}
% (U_{1},\ldots,U_{p+1})\to R(U_1,\ldots, U_p)U_{p+1},
% \end{equation} 
 that satisfy the following. There is $\mu\geq0$, $C>0$ such that 
  \begin{equation}\label{omoresti1}
 \|\Pi_{n_0}R(\Pi_{\vec{n}}\mathcal{U})\Pi_{n_{p+1}}U_{p+1}\|_{L^{2}}\leq
 C\frac{\max_2(\langle n_1\rangle,\ldots,\langle n_{p+1}\rangle)^{\rho+\mu}}{\max(\langle n_1\rangle,\ldots,\langle n_{p+1}\rangle)^{\rho}}
 \prod_{j=1}^{p+1}\|\Pi_{n_{j}}U_{j}\|_{L^{2}},
 \end{equation}
  for any 
 $\mathcal{U}=(U_1,\ldots,U_{p})\in (C^{\infty}(\TTT;\CCC^{2}))^{p}$, any 
 $U_{p+1}\in C^{\infty}(\TTT;\CCC)$,
 any $\vec{n}=(n_1,\ldots,n_p)\in \NNN^{p}$, any $n_0,n_{p+1}\in \NNN$.
 Moreover, if 
 \begin{equation}\label{omoresti2}
 \Pi_{n_0}R(\Pi_{n_1}U_1,\ldots,\Pi_{n_{p}}U_{p})\Pi_{n_{p+1}}U_{p+1}\neq0,
 \end{equation}
 then there is a choice of signs $\s_0,\ldots,\s_{p+1}\in\{-1,1\}$ such that 
 $\sum_{j=0}^{p+1}\s_j n_{j}=0$.
  In addition we require the translation invariance property
\be\label{def:R-trin}
R( \tau_\teta {\cal U}) [\tau_\teta U_{p+1}]  =  \tau_\theta \big( R( {\cal U})U_{p+1} \big) \, , \quad \forall \theta \in \R \, . 
\ee
 
\noindent
$(ii)$ {\bf Non-homogeneous smoothing operators.} 
We define the class of remainders 
$\mathcal{R}^{-\rho}_{N}[r]$ as the space of maps 
$(V,u)\mapsto R(V)u$ defined on 
$B_{r}(H^{s_0})\times H^{s_0}(\mathbb{T},\mathbb{C})$
which are linear in the variable $u$ and such that the following holds true. 
For any $s\geq s_0$ there exist a constant $C>0$ and $r(s)\in]0,r[$ 
such that for any 
$V\in B_{r}(H^{s_0})\times H^{s}(\mathbb{T},\mathbb{C})$, any $u\in H^{s}(\TTT,\CCC)$, any 
$0\leq k\cdot d\leq s-s_0$  the following estimate holds true
 \begin{equation}\label{porto20}
 \begin{aligned}
\|((d_{V}^{k}R(V))u)[h_1,\ldots, h_k]\|_{H^{s+\rho-dk}}&\leq 
C\Big[\max\{0,p-k\}\|{V}\|_{{s_0}}^{\max\{0,p-k-1\}}
\|{V}\|_{{s}}\|u\|_{s_0}
\prod_{j=1}^{k}\|{h_{j}}\|_{{s_0}}+\\
&+\|{V}\|_{{s_0}}^{\max\{0,p-k\}}\|u\|_{s_0}
\sum_{i=1}^{k}\prod_{j=1,j\neq i}^{k}\|{h_{j}}\|_{H^{s_0}}^{\nu}
\|{h_{i}}\|_{H^{s}} \\
&+\|{V}\|_{{s_0}}^{\max\{0,p-k\}}\|u\|_{s}\prod_{j=1}^{k}\|{h_{j}}\|_{{s_0}}
\Big]
\end{aligned}
\end{equation}
 for any $h_{j}$ and $U\in H^{s}$, where $\nu=1$ if $\,k\geq2\,$ and $\nu=0$ otherwise. 
 Here $d$ is the same number appearing in the definition of the non-homogeneous symbols, see also the Remark \ref{omissione}.
 
 \noindent
$(iii)$ {\bf Smoothing operators.}
We denote by $\Sigma\RR^{-\rho}_{p}[r,N]$
the space of maps $(V,t,u)\to R(V,t)u$
that may be written as 
\begin{equation*}
R(V;t)u=\sum_{q=p}^{N-1}R_{q}(V,\ldots,V)u+R_{N}(V;t)u,
\end{equation*}
for some $R_{q}\in \widetilde{\RR}^{-\rho}_{q}$, $q=p,\ldots, N-1$ and $R_{N}$ belongs to 
$\RR^{-\rho}_{N}[r]$.
We denote by $\Sigma\RR^{-\rho}_{p}[r,N]\otimes \mathcal{M}_{2}(\C)$ the space $2\times 2$
matrices with entries  in  $\Sigma\RR^{-\rho}_{p}[r,N]$.
 \end{definition}

\begin{remark}
Let $R_1(U)$ be a smoothing operator in $\Sigma\mathcal{R}^{-\rho_1}_{p_1}[r,N]$ and 
$R_2(U)$ in $\Sigma\mathcal{R}^{-\rho_2}_{p_2}[r,N]$, 
then the operator $R_1(U)\circ R_{2}(U)[\cdot]$ 
belongs to  $\Sigma\mathcal{R}^{-\rho}_{p_1+p_2}[r,N]$, 
where $\rho=\min(\rho_1,\rho_2)$.
\end{remark}

\begin{remark}\label{starship}
We remark that if $R$ is in $\widetilde{R}^{-\rho}_{p}$, $p\geq N$, then
$(V,U)\to R(V,\ldots,V)U$ is in $\RR^{-\rho}_{N}[r]$.
This inclusion follows by the multi-linearity of $R$ in each argument, and 
by estimate \eqref{omoresti1}. The proof of this fact is very similar  to the one given in 
the remark after Definition $2.2.3$ in \cite{BD}.
\end{remark}

\begin{remark}\label{smooth-totale}
In analogy with Remark \ref{super-simboli} we make a comparison between the \textbf{non-homogenous} smoothing remainders defined in \cite{BD} or \cite{Feola-Iandoli-Long} and the ones defined in item $(ii)$ in Def. \ref{omosmoothing}. Let $U(t,x)$ be a function in $C^{K-K'}_*(I,H^s)$ (defined in Section 2 in \cite{Feola-Iandoli-Long}) and  $(V,u)\mapsto R(V)u$ be a smoothing operator satisfying \eqref{porto20} with $s_0:=\tilde{s}_0-2(K-K')\gg1$. Let us suppose for simplicity that $p-k-1\geq 0$ and $k\geq 2$, in the other cases the proof may be easily adapted.  For any $2\leq k\leq K-K'$ we have 
\begin{equation*}
\begin{aligned}
\norm{\partial_t^kR(V)u}_{s-2k}\leq&\Big\|\sum_{k'+k"=k}\,\,\sum_{\gamma=1}^{k'}\,\,\sum_{k_1+\ldots+k_{\gamma}=k'}d_V^{k'}R(V)\Big[\partial^{k_1}_tV,\ldots,\partial^{k_{\gamma}}_tV,\partial^{k"}_tu\Big]\Big\|_{s-2k}\\
\leq&C\sum\Big\{\norm{V}_{s_0}^{p-\gamma-1}\norm{V}_{s-2k}\norm{\partial_t^{k''}u}_{s_0}\prod_{j=1}^{\gamma}\norm{\partial_t^{k_j}V}_{s_0}+\\
&+\norm{V}_{s_0}^{p-\gamma}\norm{\partial_t^{k"}u}_{s_0}\sum_{i=1}^{\gamma}\prod_{j=1,j\neq i}^{\gamma}\norm{\partial_t^{k_i}V}_{s-2k}\norm{\partial_t^{k_j}V}_{s_0}+\\
&+\norm{V}_{s_0}^{p-\gamma}\prod_{j=1}^{\gamma}\norm{\partial_t^{k_j}V}_{s_0}\norm{\partial_t^{k''}u}_{s-2k}\Big\},
\end{aligned}
\end{equation*}
where the sum in the r.h.s. is taken over the set of indices such that $k'+k''=k, k_1+\ldots k_{\gamma}=k"$ and $\gamma=1,\ldots, k'$. The r.h.s. of the above inequality may be bounded from above by the r.h.s. of $(2.7)$ in \cite{Feola-Iandoli-Long}. Therefore we recover the definition of non-homogeneous smoothing operators given in \cite{Feola-Iandoli-Long}. To recover the definition in \cite{BD} the computation is the same.
\end{remark}

\begin{remark}{\bf (Expansion of Homogeneous remainders).}\label{espansioneRRR}
Let $R_{p}\in \widetilde{\mathcal{R}}^{-\rho}_{p}\otimes \mathcal{M}_2(\mathbb{C})$.
Then by \eqref{omoresti2}, \eqref{def:R-trin} we deduce the following.
First of all write
 \begin{align}
R_{p}(U)
& =\left(
\begin{matrix}
(\mathtt{R}_p(U))_{+}^{+} & (\mathtt{R}_p(U))_{+}^{-}\vspace{0.2em}\\
(\mathtt{R}_{p}(U))_{-}^{+} & (\mathtt{R}_p(U))_{-}^{-}
\end{matrix}
\right)\,,  \quad 
(\mathtt{R}_i(U))_{\s}^{\s'}
\in \widetilde{\mathcal{R}}^{-\rho}_p \, ,  
%\quad 
% (\mathtt{R}_{i}(U))_{\s}^{\s'}=\ov{ (\mathtt{R}_{i}(U))_{-\s}^{-\s'}} \, , 
\label{smooth-terms2}
 \end{align}
 for $\s,\s'\in\{\pm\}$ and $i=1,2$. Then, expanding $u$ as in 
 \eqref{complex-uU}, \eqref{notaFou}, we have
  and
\be\label{R2epep'}
(\mathtt{R}_{p}(U))_{\s}^{\s'} z^{\s'} =
\frac{1}{\sqrt{2\pi}}
\sum_{j\in\mathbb{Z}} 
\Big( \sum_{k\in \mathbb{Z}}  
 (\mathtt{R}_{p}(U))_{\s,j}^{\s',k} 
 z_{k}^{\s'} \Big) e^{\ii \s jx} 
\ee
with entries
 \begin{align} \label{BNF5}
 (\mathtt{R}_{p}(U))_{\s,j}^{\s',k} :=\frac{1}{(2\pi)^{p}}
 \sum_{\substack{\s_i\in\{\pm\}, j_i\in\Z \\ 
 \sum_{i=1}^{p}\s_i j_i=\s j-\s'k }}
 \big( (\mathtt{r}_{p})_{j_1,\ldots,j_p}^{\s_1\cdots \s_{p}}\big)_{\s,j}^{\s',k}
%  (\mathtt{r}_{2,\ep,\ep'})^{\s,\s'}_{n_1,n_2,k}
  u_{j_1}^{\s_1}\ldots u_{j_p}^{\s_{p}}  \, ,   \quad  
  j,k\in \Z\setminus\{0\}   \, , 
 \end{align}
 and suitable scalar coefficients 
 $  \big( (\mathtt{r}_{p})_{j_1,\ldots,j_p}^{\s_1\cdots \s_{p}}\big)_{\s,j}^{\s',k} \in \C $. 
\end{remark}

\subsubsection{Spaces of Maps}
Below we deal with classes of operators without keeping track of the number of lost derivatives in a precise way 
(see Definition 3.9  in \cite{BD}).
The class $\widetilde{\mathcal{M}}^{m}_{p}$ denotes multilinear maps that lose $m$ derivatives
and are $p$-homogeneous in $U$, 
while the class $\mathcal{M}_{p}^{m}$ contains non-homogeneous maps which lose $m$ derivatives,
vanish at degree at least $p$ in $U$.

\begin{definition}{\bf (Classes of maps).} \label{smoothoperatormaps}
Let $p,N\in \N $, with $p\leq N$, $N\geq1$
and $ m \geq 0 $. 

\noindent
$(i)$ {\bf $p$-homogeneous maps.} We denote by 
$\widetilde{\mathcal{M}}^{m}_{p}$
 the space of $(p+1)$-linear maps $M$ 
 from $({H}^{\infty}(\T;\C^{2}))^{p}\times {H}^{\infty}(\T;\C)$ to 
 the space ${H}^{\infty}(\T;\C)$ which are symmetric
 in $(U_{1},\ldots,U_{p})$, of the form
\[
 (U_{1},\ldots,U_{p+1})\to M(U_1,\ldots, U_p)U_{p+1} 
 \]
and that satisfy the following. There is $C>0$ such that 
 \[
 \|\Pi_{n_0}M(\Pi_{\vec{n}}\mathcal{U})\Pi_{n_{p+1}}U_{p+1}\|_{L^{2}}\leq
 C( n_0 +  n_1 +\cdots+ n_{p+1})^{m} \prod_{j=1}^{p+1}\|\Pi_{n_{j}}U_{j}\|_{L^{2}}
\] 
  for any 
 $\mathcal{U}=(U_1,\ldots,U_{p})\in ({H}^{\infty}(\T;\C^{2}))^{p}$, any 
 $U_{p+1}\in {H}^{\infty}(\T;\C)$,
 any $\vec{n}=(n_1,\ldots,n_p)\in (\N^*)^{p}$, any $n_0,n_{p+1}\in \N^*$.
 Moreover the properties \eqref{omoresti2}-\eqref{def:R-trin} hold.
 
 \noindent
$(ii)$ {\bf Non-homogeneous maps.}
  We denote by  $\mathcal{M}^{m}_{N}[r]$ 
  the space of maps $(V,u)\mapsto M(V) U $ defined on 
  $B_{r}(H^{s_0})\times {H}^{s_0}(\T,\C)$ 
  which are linear in the variable $ U $ and such that the following holds true. 
  For any $s\geq s_0$ there exist a constant $C>0$ and 
  $r(s)\in]0,r[$ such that for any 
  $V\in B_{r}(H^{s})\cap {H}^{s}(\T,\C^2)$, 
  any $ U \in {H}^{s}(\T,\C)$, any $0\leq d k\leq s-s_0$, we have 
  $\|{d_{V}^{k}\left(M(V)U\right)[h_1,\ldots,h_{k}]}\|_{{H}^{s-m-dk}}$ 
  is bounded by the right hand side  of \eqref{porto20}.
  
  \noindent
$(iii)$ {\bf Maps.}
We denote by $\Sigma\mathcal{M}^{m}_{p}[r,N]$
the space of maps $(V,t,U)\to M(V,t)U$
that may be written as 
\[
M(V;t)U=\sum_{q=p}^{N-1}M_{q}(V,\ldots,V)U+M_{N}(V;t)U
\]
for some $M_{q} $ in $ \widetilde{\mathcal{M}}^{m}_{q}$, $q=p,\ldots, N-1$ and 
$M_{N}$  in  
$\mathcal{M}^{m}_{N}[r]$.
Finally we set $\widetilde{\mathcal{M}}_{p}:=\cup_{m\geq0}\widetilde{\mathcal{M}}_{p}^{m}$,
$\mathcal{M}_{p}[r]:=\cup_{m\geq0}\mathcal{M}^{m}_{p}[r]$ and 
$\Sigma\mathcal{M}_{p}[r,N]:=\cup_{m\geq0}\Sigma\mathcal{M}^{m}_{p}[r]$.

\noindent
We denote by  $\Sigma\mathcal{M}_{p}^{m}[r,N]\otimes\mathcal{M}_2(\C)$
 the space of  $\,\,2\times 2$ matrices whose entries are maps in
the class $\Sigma\mathcal{M}^{m}_{p}[r,N]$.
We also set $\Sigma\mathcal{M}_{p}[r,N]\otimes\mathcal{M}_2(\C)=\cup_{m\in \R}
\Sigma\mathcal{M}_{p}^{m}[r,N]\otimes\mathcal{M}_2(\C)$.
\end{definition}

%We have the following (see the remarks after Definition %  3.7   3.9 in \cite{BD}):

%\vspace{0.3em}
%\noindent
%$\bullet$
%if $R$ is in $\widetilde{\mathcal{R}}^{-\rho}_{p}$, $p\geq N$, then
% $(V,U)\to R(V,\ldots,V)U$ is in $\mathcal{R}^{-\rho}_{K,0,N}[r]$. 

\vspace{0.3em}
\noindent
$\bullet$
If  $M$ is in $\widetilde{\mathcal{M}}^{m}_{p}$, $p\geq N$, then
$(V,U)\to M(V,\ldots,V)U$ is in $\mathcal{M}^{m}_{N,0}[r]$.

\vspace{0.3em}
\noindent
$\bullet$
If $a\in \Sigma\Gamma^{m}_{K,K',p}[r,N]$ for $p\geq1$, then 
$ \opbw(a(V;t,\cdot))U$ 
is in $\Sigma\mathcal{M}^{m'}_{p}[r,N]$ for some $m'\geq m$.

\vspace{0.3em}
\noindent
$\bullet$
 Any  
$R\in \Sigma\mathcal{R}^{-\rho}_{p}[r,N]$
defines an element of $\Sigma\mathcal{M}^{m}_{p}[r,N]$ for some $m\geq0$.

\vspace{0.3em}
\noindent
$\bullet$
 If $ M \in \Sigma\mathcal{M}_{p}[r,N]$, $\tilde{M}\in \Sigma\mathcal{M}_{1}[r,N-p]$,
then  
$(V,t,U)\to M(V+\tilde{M}(V;t)V ;t)[U]$ is in $\Sigma\mathcal{M}_{p}[r,N]$. 

\vspace{0.3em}
\noindent
$\bullet$
If $ M \in \Sigma\mathcal{M}^{m}_{p}[r,N]$ and $\tilde{M}\in\Sigma\mathcal{M}^{m'}_{q}[r,N]$,
then 
$M(U;t)\circ \tilde{M}(U;t)$ is in $\Sigma\mathcal{M}^{m+m'}_{p+q}[r,N]$.

\subsection{Quantization of symbols}

Given a smooth symbol $(x,\x) \to a(x,\x)$,
we define, for any $\s\in [0,1]$,  the quantization of the symbol $a$ as the operator 
acting on functions $u$ as 
\begin{equation}\label{operatore}
{\rm Op}_{\s}(a(x,\x))u=\frac{1}{2\pi}\int_{\RRR\times\RRR}
e^{\ii(x-y)\x}a(\s x+(1-\s)y,\x)u(y)dy d\x\,.
\end{equation}
This definition is meaningful in particular if $u\in C^{\infty}(\TTT)$ (identifying $u$ to a $2\pi$-periodic function). By decomposing 
$u$ in Fourier series  as $u=\sum_{j\in\ZZZ}\hat{u}(j)(1/\sqrt{2\pi})e^{\ii jx}$, we may calculate the oscillatory integral in \eqref{operatore} obtaining
\begin{equation}\label{bambola}
{\rm Op}_{\s}(a)u:=\frac{1}{\sqrt{2\pi}}\sum_{k\in \ZZZ}\left(\sum_{j\in\ZZZ}\hat{a}\big(k-j,(1-\s)k+\s j\big)\hat{u}(j)\right)\frac{e^{\ii k x}}{\sqrt{2\pi}}\,, \quad \forall\;\; \s\in[0,1]\,,
\end{equation}
where $\hat{a}(k,\xi)$ is the $k^{th}-$Fourier coefficient of the $2\pi-$periodic function $x\mapsto a(x,\xi)$.
%We call the operators of the form $A:={\rm Op}_{\s}(a(x,\x))$ \emph{pseudo-differential} operators.
For convenience in the paper we shall use two particular quantizations:

\noindent
{\bf Standard quantization.}
We define the standard quantization by specifying formula \eqref{bambola} for $\s=1$:
\begin{equation*}
{\rm Op}(a)u:={\rm Op}_{1}(a)u=\frac{1}{\sqrt{2\pi}}\sum_{k\in \ZZZ}\left(\sum_{j\in\ZZZ}\hat{a}\big(k-j, j\big)\hat{u}(j)\right)\frac{e^{\ii k x}}{\sqrt{2\pi}}\,;
\end{equation*}

\noindent
{\bf Weyl quantization.}
We define the Weyl quantization by specifying formula \eqref{bambola} for $\s=\frac{1}{2}$:
\begin{equation}\label{bambola202}
{\rm Op}^{W}(a)u:={\rm Op}_{\frac{1}{2}}(a)u=\frac{1}{\sqrt{2\pi}}\sum_{k\in \ZZZ}
\left(\sum_{j\in\ZZZ}\hat{a}\big(k-j, \frac{k+j}{2}\big)\hat{u}(j)\right)
\frac{e^{\ii k x}}{\sqrt{2\pi}}\,.
\end{equation}
Moreover
the above formulas allow to transform the symbols between different quantizations, in particular we have
\begin{equation*}
{\rm Op}(a)={\rm Op}^{W}(b)\,, 
\qquad {\rm where} \quad \hat{b}(j,\x)=\hat{a}(j,\x-\frac{j}{2})\,.
\end{equation*}
We want to define a \emph{para-differential} quantization.
First we give the following definition.
\begin{definition}{\bf (Admissible cut-off functions).}\label{cutoff1}
Fix $p\in \NNN$ with $p\geq1$.
We say that  $\chi_{p}\in C^{\infty}(\RRR^{p}\times \RRR;\RRR)$ and $\chi\in C^{\infty}(\R\times\R;\R)$ are  admissible cut-off functions if  they are even with respect to each of their arguments and there exists $\delta>0$ such that
\begin{equation*}
\begin{aligned}
&{\rm{supp}}\, \chi_{p} \subset\set{(\xi',\xi)\in\R^{p}\times\R; |\xi'|\leq\delta \langle\xi\rangle}\,,\qquad \chi_p (\xi',\xi)\equiv 1\,\,\, \rm{ for } \,\,\, |\xi'|\leq\frac{\delta}{2} \langle\xi\rangle\,,\\
&\rm{supp}\, \chi \subset\set{(\xi',\xi)\in\R\times\R; |\xi'|\leq\delta \langle\xi\rangle}\,,\qquad \chi(\xi',\xi) \equiv 1\,\,\, \rm{ for } \,\,\, |\xi'|\leq\frac{\delta}{2} \langle\xi\rangle\,.
\end{aligned}
\end{equation*}
We assume moreover that for any derivation indices $\alpha$ and $\beta$
\begin{equation*}
\begin{aligned}
&|\partial_{\xi}^{\alpha}\partial_{\xi'}^{\beta}\chi_p(\xi',\xi)|\leq C_{\alpha,\beta}\langle\xi\rangle^{-\alpha-|\beta|}\,,\,\,
\forall \alpha\in \NNN, \,\beta\in\NNN^{p}\,,\\
&|\partial_{\xi}^{\alpha}\partial_{\xi'}^{\beta}\chi(\xi',\xi)|\leq C_{\alpha,\beta}\langle\xi\rangle^{-\alpha-\beta}\,,\,\,\forall \alpha\,, \,\beta\in\NNN\,.
\end{aligned}
\end{equation*}
\end{definition}
\noindent An example of function satisfying the condition above, and that will be extensively used in the rest of the paper, is $\chi(\xi',\xi):=\widetilde{\chi}(\xi'/\langle\xi\rangle)$, where $\widetilde{\chi}$ is a function in $C_0^{\infty}(\RRR;\RRR)$  having a small enough support and equal to one in a neighborhood of zero.
For any $a\in C^{\infty}(\TTT)$ we shall use the following notation
\begin{equation*}
(\chi(D)a)(x)=\sum_{j\in\ZZZ}\chi(j)\Pi_{j}{a}\,.
\end{equation*}

\begin{definition}{\bf (The  Bony quantization).}\label{quantizationtotale}
Let $\chi$ be an admissible cut-off function according to Definition \ref{cutoff1}.
If a is a symbol in $\widetilde{\Gamma}^{m}_{p}$ and $b$ is in $\Gamma^{m}_{p}[r]$,
we set, using notation \eqref{ptupla}, 
\begin{equation*}
\begin{aligned}
a_{\chi}(\mathcal{U};x,\x)&=\sum_{\vec{n}\in \NNN^{p}}\chi_{p}\left(\vec{n},\x\right)a(\Pi_{\vec{n}}\mathcal{U};x,\x)\,,\qquad
b_{\chi}(U;x,\x)=\frac{1}{2\pi}\int_{\TTT}  \chi\left(\h,\x\right)\hat{b}(U;\h,\x)e^{\ii \h x}d \h\,.
\end{aligned}
\end{equation*}
We define the  Bony quantization as
\begin{equation*}
\begin{aligned}
\opb(a(\mathcal{U};\cdot))&={\rm Op}(a_{\chi}(\mathcal{U};\cdot))\,,\qquad
\opb(b(U;\cdot))={\rm Op}(b_{\chi}(U;t,\cdot))\,.
\end{aligned}
\end{equation*}
and the Bony-Weyl quantization as
\begin{equation}\label{simbotroncati2}
\begin{aligned}
\opbw(a(\mathcal{U};\cdot))&=\opw(a_{\chi}(\mathcal{U};\cdot)),\qquad 
\opbw(b(U;t,\cdot))=\opw(b_{\chi}(U;\cdot)).
\end{aligned}
\end{equation}
Finally, if $a$ is a symbol in the class $\Sigma\Gamma^{m}_{p}[r,N]$, that we decompose as in \eqref{simbotot1},
we define its Bony quantization as 
\begin{equation*}
\opb(a(U;\cdot))=\sum_{q=p}^{N-1}\opb(a_{q}(U,\ldots,U;\cdot))+\opb(a_{N}(U;\cdot))\,,
\end{equation*}
and its Bony-Weyl quantization as
\begin{equation*}
\opbw(a(U;\cdot))=\sum_{q=p}^{N-1}\opbw(a_{q}(U,\ldots,U;\cdot))+\opbw(a_{N}(U;\cdot))\,.
\end{equation*}
%For symbols belonging to the autonomous subclass
%$\Sigma\Gamma_{p,0}^{m}[r,N,{\rm aut}]$
%we shall
%not write the time dependence in  \eqref{simbotroncati4} and  \eqref{simbotroncati5}.
\end{definition}

\begin{remark}
Let $a\in \Sigma\Gamma^{m}_{p}[r,N]$.
We note that
\begin{equation*}
\ov{\opb(a(U;x,\x)[v]}=\opb(\ov{a^{\vee}(U;x,\x)})[\bar{v}]\,, \quad
\ov{\opbw(a(U;x,\x)[v]}=\opbw(\ov{a^{\vee}(U;x,\x)})[\bar{v}]\,,
\end{equation*}
where 
\begin{equation*}
a^{\vee}(U;x,\x):=a(U;x,-\x)\,.
\end{equation*}
Moreover if we define the operator $A(U)[\cdot]:=\opbw(a(U;x,\x))[\cdot]$
we have that $A^{*}(U)$, its adjoint operator w.r.t. the $L^{2}(\TTT;\CCC)$ scalar product,
can be written as
\begin{equation}\label{aggiunto}
A^{*}(U)[v]=\opbw\Big(\ov{a(U;x,\x)}\Big)[v]\,.
\end{equation}
\end{remark}

\begin{remark}\label{aggiungoaggiunto}
By formula \eqref{aggiunto} one has %that a para-differential operator 
$\opbw(a(U;x,\x))[\cdot]$
is self-adjoint, w.r.t. the $L^{2}(\TTT;\CCC)$ scalar product, if and only if the symbol $a(U;x,\x)$ is real valued
for any $x\in \TTT$, $\x\in \RRR$. 
\end{remark}

\smallskip
The next proposition states
boundedness properties on Sobolev spaces of the  para-differential operators
(see Proposition 3.8  in \cite{BD}).

\begin{proposition}{\bf (Action of para-differential operator 1).} \label{azionepara}
Let $r>0$, $m\in \R$, $p\in \N$. Then:
\smallskip
$(i)$ There is $ s_0 > 0 $ such that for any symbol
 $a\in \widetilde{\Gamma}_{p}^{m}$, for any $s\geq s_0$ 
there is a constant $C>0$, depending only on $s$ and on \eqref{pomosimbo1} with $\alpha=\beta=0$,
such that for any $\mathcal{U}=(U_1,\ldots,U_{p})$
\begin{equation}\label{stimapar}
\|\opbw(a(\mathcal{U};\cdot))U_{p+1}\|_{{H}^{s-m}}\leq 
C\prod_{j=1}^{p}\|U_{j}\|_{{H}^{s_0}}\|U_{p+1}\|_{{H}^{s}} \, ,
\end{equation}
for $p\geq 1$, while for $p=0$ the \eqref{stimapar} holds by 
replacing the right hand side with $C\|U_{p+1}\|_{{H}^{s}}$.

\smallskip
$(ii)$
There is $ s_0 > 0 $ such that  for any symbol $a\in \Gamma^{m}_{p}[r]$ and any $s\geq s_0$
there is  a constant $C>0$, depending  only on $s,r$ and \eqref{maremma2}
with $0\leq \alpha\leq 2$, $\beta=0$, such that, for any $t\in I$, any $0\leq dk+2\leq s-s_0$,
\begin{equation}\label{stimapar2}
\|\opbw((\pa_{u}^{k}a)(U;\cdot)[h_1,\ldots,h_k])\|_{\mathcal{L}({H}^{s},{H}^{s-m})}
\leq C\|{U}\|_{H^{s_0}}^{\max\{0,p-k\}}
%\|U\|_{k+K',s_0}^{p} 
\prod_{i=1}^{k}\|h_{i}\|_{s_0}\,.
\end{equation}
\end{proposition}

\begin{proof}
The proof of item $(i)$ may be found at pag. 47 of the book \cite{BD}, the item $(i)$ of this theorem is indeed the same of item $(i)$ of Prop. 3.8 in \cite{BD}. Note that the proof of item $(ii)$ in the case that $k=0$ is the same given in item $(ii)$ of Prop 3.8 in \cite{BD}, in the case $k>0$ is very similar, however we give the proof for completeness. Let $\chi$ an admissible cut-off function ion the sense of Def. \ref{cutoff1}. We set
\begin{equation*}
b_k(x,\xi):=\big(d_u^ka(U;\cdot)[h_1,\ldots, h_k]\big)_{\chi}=\sum_{n\in\ZZZ}\chi(n,\xi)\Pi_n d_{u}^{k}a(U;x,\xi)[h_1,\ldots, h_k]
\end{equation*}
for any fixed $\xi\in\RRR$. Let $V$ be a regular function on the torus, we have the following
\begin{equation*}
\begin{aligned}
\opbw(\partial_u^ka(U;x,\xi)[h_1,\ldots,h_k])v=
\frac{1}{\sqrt{2\pi}}\sum_{n'\in\ZZZ}
\sum_{k'\in\ZZZ}\hat{b}_k(k'-n',\frac{k'+n'}{2})\hat{v}(n')
\frac{e^{\ii k'x}}{\sqrt{2\pi}}\,,
\end{aligned}
\end{equation*}
where by $\hat{b}_k(\ell,n)$ we have denoted the $\ell$-Fourier coefficient with respect to $x$ of the function $b_k(x,\xi)$ restricted at $\xi=n$. We need to estimate the general Fourier coefficient $\hat{b}_k(\ell,\xi)$. By definition we have 
\begin{equation*}
\hat{b}_k(\ell,\xi)=\int_{\TTT}b_k(x,\xi)e^{-\ii\ell x}dx\,,
\end{equation*}
therefore, by integrating twice in $x$ and by using \eqref{maremma2} with $\alpha=2$ (relabelling $s_0\rightsquigarrow s_0+2$), we obtain
\begin{equation*}
|\hat{b}_k(\ell,\xi)|\leq C\jap{\xi}^{m}\frac{1}{\jap{\ell}^2}\norm{U}_{s_0}^{p-k}\prod_{j=1}^k\norm{h_j}_{s_0}\,.
\end{equation*}
Having this inequality one can conclude the proof as done in the case of Prop. 3.8 in \cite{BD}.
\end{proof}

\begin{remark}
Notice that the following holds.
\begin{itemize}
\item
 Let $m\leq 0$ and $p\geq 1$ and let $a$ be a symbol in $\Sigma\Gamma^m_p[r,N]$, then 
the map $(U,V)\mapsto \opbw(a(U;x,\xi))V$ is in $\Sigma\mathcal{R}^{m}_p[r,N]$;

\item
 let $a$ be a symbol in $\Sigma\Gamma^m_p[r,N]$, then 
the map $(U,V)\mapsto \opbw(a(U;x,\xi))V$ is in $\Sigma\mathcal{M}^{m'}_p[r,N]$ for any $m'\geq m$;

\item
 any smoothing operator in $\Sigma\mathcal{R}^{-\rho}_p[r,N]$ defines an element of $\Sigma\mathcal{M}^{m'}_p[r,N]$ for some $m'>0$.
\end{itemize}
\end{remark}
\noindent We now state a \emph{classical} version of the \emph{action}-theorem on Sobolev spaces for para-differential operators, whose proof can be found in the book by Metivier \cite{Met} (see Theorem 5.1.15 and formula (5.1.25) therein). 
\begin{proposition}{\bf (Action of para-differential operators 2).}\label{AzioneParaMet}
Consider a symbol $a\in\Gamma^{m}_{p}[r]$, then it defines a bounded operator from $H^s$ to $H^{s-m}$ with the following estimate
\begin{equation*}
\| \opbw(a(U;x,\xi))\|_{\mathcal{L}(H^s,H^{s-m})}\leq 
\sup_{\xi\in\mathbb{R}}\langle\xi\rangle^{-m}
\|a(U;x,\xi)\|_{L^{\infty}_x}\,.
\end{equation*}
\end{proposition}
\begin{remark}
Note that in Prop. \ref{AzioneParaMet} we have better estimates in terms of regularity of the function $a(U;\cdot)$ in the r.h.s. compared to Prop. \ref{azionepara}. In Prop. \ref{azionepara} we have more information on the smallness of the symbol in terms of its dependence of the function $U$ but we have to pay the price of losing some derivatives in the r.h.s.. In Section \ref{FLOWS} we shall need the optimal estimate in term of regularity given by Prop. \ref{AzioneParaMet} in several proofs.
\end{remark}

%\\[1mm]
%{ \bf Composition theorems.}

\subsection{Symbolic calculus and Compositions theorems }
We introduce the following differential operator
\[ 
\s(D_{x},D_{\x},D_{y},D_{\eta}) := D_{\x}D_{y}-D_{x}D_{\eta}\,, 
\]
where $D_{x}:=\frac{1}{\ii}\pa_{x}$ and $D_{\x},D_{y},D_{\eta}$ are similarly defined. Given two symbols $a$ and $b$, in the following we define a new symbol $a\#_{\rho}b$ which turns out to be the symbol of the composition of the para-differential operators generated by $a$ and $b$ modulo smoothing operators of order $-\rho$.

\begin{definition}{\bf (Asymptotic expansion of composition symbol).}
Let $\rho,p,q$ be in $\N$, $m,m'\in \R$, $r>0$.
Consider $a\in \Sigma\Gamma_{p}^{m}[r,N]$ and $b\in \Sigma \Gamma^{m'}_{q}[r,N]$. 
For $U$ in $B_{r}(H^{\s})$ 
we define, for $\rho< \s- s_0$, the symbol
\begin{equation}\label{espansione2}
(a\#_{\rho} b)(U;x,\x):=\sum_{k=0}^{\rho}\frac{1}{k!}
\left(
\frac{\ii}{2}\s(D_{x},D_{\x},D_{y},D_{\eta})\right)^{k}
\Big[a(U;x,\x)b(U;y,\eta)\Big]_{|_{\substack{x=y, \x=\eta}}}
\end{equation}
modulo symbols in $ \Sigma\Gamma^{m+m' - \rho }_{p+q}[r,N] $.
\end{definition}

\noindent
$ \bullet $
%By  \eqref{prodottodisimboli2} the 
The symbol $ a\#_{\rho} b $ belongs  to $\Sigma\Gamma^{m+m'}_{p+q}[r,N]$.

\noindent
$ \bullet $
We have  the expansion  $ a\#_{\rho}b =ab+\frac{1}{2 \ii }\{a,b\} + \cdots $, 
up to a symbol in $\Sigma\Gamma^{m+m'-2}_{p+q}[r,N]$,     
where 
$
\{a,b\}:=\pa_{\x}a\pa_{x}b-\pa_{x}a\pa_{\x}b
$ 
denotes the Poisson bracket. 

\begin{proposition}{\bf (Composition of Bony-Weyl operators).} \label{teoremadicomposizione}
Let $ \rho,p,q$ be in $\N$, $m,m'\in \R$, $r>0$. 
Consider 
$a\in \Sigma {\Gamma}^{m}_{p}[r,N] $ and $b\in \Sigma {\Gamma}^{m'}_{q}[r, N]$. 
Then
\begin{equation}\label{calcolosimbolico}
R(U):=\opbw(a(U;x,\x))\circ\opbw(b(U;x,\x))-\opbw\big(
(a\#_{\rho} b)(U;x,\x)
\big)
\end{equation}
is a non-homogeneous smoothing remainder in 
$ \Sigma {\mathcal{R}}^{-\rho+m+m'}_{p+q}[r,N]$.
\end{proposition}

\begin{proof}
We give the proof of the theorem  in the case that the symbols $a$ and $b$ are non-homogenous in the classes $\Gamma_p^m[r]$ and $\Gamma_q^{m'}[r]$ respectively. In the case of composition between operators generated by homogeneous symbols we refer to the proof of Prop. 3.12 in \cite{BD} since our classes coincide with the ones therein. We set
\begin{equation*}
\begin{aligned}
&\tilde{a}_{\chi}(x,\xi)=d_U^{k_1}a(U;x,\xi)\big[h_{n_1},\ldots,h_{n_{k_1}}\big]\\
&\tilde{b}_{\chi}(x,\xi)=d_U^{k_2}b(U;x,\xi)\big[h_{n_{k_1}},\ldots,h_{n_{k_1+k_2}}\big]\,.
\end{aligned}
\end{equation*}
We have that the $(k_1+k_2)$-differential with respect to $U$ of the expression in \eqref{calcolosimbolico} applied to the vector $[h_{n_1},\ldots, h_{n_{k_1+k_2}}]$ may be written as $\Op^W(r_1(U;x,\xi)+r_2(U;x,\xi))$ where
\begin{equation*}\begin{aligned}
&r_1(U;\cdot)=\tilde{a}_{\chi}\#\tilde{b}_{\chi}-(\tilde{a}_{\chi}\#\tilde{b}_{\chi})_{\rho}\,,\\
&r_2(U;\cdot)=(\tilde{a}_{\chi}\#\tilde{b}_{\chi})_{\rho}-(\tilde{a}\#\tilde{b})_{\rho,\chi}\,,
\end{aligned}
\end{equation*}
where $a\#b$ is the symbol of the composition given in Lemma 3.14 in \cite{BD}. We estimate the term coming from $r_1(U;\cdot)$, the other one is similar. At this point one has to use Lemma 3.13 in \cite{BD} by using the fact the the symbols $a$ and $b$ satisfy the estimate \eqref{maremma2}. By following the proof of the \emph{action-}theorem \ref{azionepara} one can estimate the $H^{s-m-m'+\rho-2}$ of $\Op^W(r_1(U;\cdot))v$ for some regular function $v$ obtaining the estimate of smoothing remainder \eqref{porto20} (up to renaming $\rho-2\rightsquigarrow \rho$).
The translation invariance property \eqref{def:R-trin} follows as in \cite{BFP}.
\end{proof}

\begin{remark}
As proved in the remark after the proof of Proposition 3.12 in \cite{BD}, 
the remainder obtained by the composition of para-differential operators in Proposition 
\ref{teoremadicomposizione} has actually better estimates than
\eqref{porto20}, i.e. it is bounded from 
$ H^{s} $  to $  H^{s + \rho - (m + m') } $ for {\it any} $ s $,
with operator norm bounded by $ \| U \|_{ s_0}^{p+q} $.
\end{remark}

\begin{proposition}{\bf (Compositions).} \label{composizioniTOTALI}
Let $m,m',m''\in \R$, $N,p_1,p_2,p_{3},\rho\in \N$, $p_1+p_{2}<N$,
$p_2+p_{3}<N$ $\rho\geq0$ and $r>0$.
Let $a\in \Sigma\Gamma^{m}_{p_1}[r,N]$, 
$R\in\Sigma\mathcal{R}^{-\rho}_{p_{2}}[r,N]$ and $M\in \Sigma\mathcal{M}^{m''}_{p_{3}}[r,N]$.
Then

\begin{itemize}
\item[(i)]
$ R(U)\circ \opbw(a(U;x,\x)) $, $  \opbw(a(U;x,\x))\circ R(U;t) $
are   in $\Sigma\mathcal{R}^{-\rho+m}_{p_1+p_{2}}[r,N]$.

\item[(ii)]
$ R(U)\circ M(U) $ and $ M(U;t)\circ R(U) $ 
are smoothing operators in $\Sigma\mathcal{R}^{-\rho+m''}_{p_2+p_{3}}[r,N]$.

\item[(iii)]
If  $R_{2} \in \widetilde{\mathcal{R}}_{p_2}^{-\rho}$  
then $R_{2}(U,M(U)U)$ belongs to $\Sigma\mathcal{R}^{-\rho+m''}_{p_{2}+p_3}[r,N]$.

\item[(iv)] Let $c$ be in $\widetilde{\Gamma}_p^{m}$, $p\in \N$. 
Then 
$$
U \rightarrow c_{M}(U;x,\x):= c(U,\ldots,U,M(U)U;x,\x)
$$
is in $\Sigma\Gamma^{m}_{p+p_3}[r,N]$. If the symbol $c$ is independent of $\x$ (i.e.
$c$ is in $\widetilde{\mathcal{F}}_p$), so is the symbol $c_{M}$
(thus it is a  function in $\Sigma\mathcal{F}_{p+p_3}[r,N]$). 
Moreover if $c$ is a symbol in $\Gamma^{m}_{N}[r]$
then the symbol $c_{M}$
is in $\Gamma^{m}_{N}[r]$.

\item[(v)]
$ \opbw( c(U,\ldots,U,W;x,\x))_{|W=M(U)U}=\opbw(b(U;x,\x))+R(U) $
where  
$$
b(U;x,\x):= c(U,\ldots,U,M(U)U;x,\x)
$$
and $R(U) $ is in $\Sigma\mathcal{R}^{-\rho}_{K,K',p+p_1}[r,N]$.
\end{itemize}
\end{proposition}
\begin{proof}

We prove $(i)$ in the case of the operator $H(U):=R(U)\circ\opbw(a(U;x,\xi))$. Let $s_0-m\gg 1$, we have to estimate the quantity $\|d_U^k H(U)(h_1,\ldots , h_k)\|_{s+\rho-dk-m}$. We give the proof in the case that the symbol $a$ is in $\Gamma^m_{p_1}[r,N]$ and $R$ is in $\mathcal{R}^{-\rho}_{p_2}[r,N]$. By using estimate \eqref{porto20} for the operator $R(U)$ we may bound from above, modulo the constant appearing in \eqref{porto20}, the previous quantity by
\begin{align}
&\sum_{k_1+k_2=k}\Big\{ \mathtt{1}_{\{p-k\geq 0\}}\norm{U}_{s_0-m}^{\max{\{0,p-k-1\}}}\norm{U}_{s-m}\norm{\opbw(d_U^{k_2}a(U;x,\xi)[h_1,\ldots,h_{k_2}])V}_{s_0-m}\prod_{j=k_2+1}^k\norm{h_{k_j}}_{s_0-m}\nonumber\\
&+\norm{U}_{s_0-m}^{\max\{0,p-k\}}
\norm{\opbw(d_U^{k_2}a(U;x,\xi)[h_1,\ldots,h_{k_2}])V}_{s_0-m}
\sum_{i=k_2+1}^k\prod_{j=k_2+1,j\neq i}^k\norm{h_j}^{\nu}_{s_0-m}
\norm{h_i}_{s-m}\label{lafaccio}\\
&+\norm{U}^{\max\{0,p-k\}}_{s_0-m}
\norm{\opbw(d_U^{k_2}a(U;x,\xi)[h_1,\ldots,h_{k_2}])V}_{s-m}
\prod_{j=k_2+1}^k\norm{h_j}_{s_0-m}\,,\nonumber
\end{align}
where by $\mathtt{1}_A$ we denoted the characteristic function of the set $A$. We prove the result in for the addendum in  \eqref{lafaccio}, the others may be similarly bounded. We just have to use Theorem \ref{azionepara} obtaining
\begin{equation*}
\begin{aligned}
&\mathtt{1}_{\{p-k\geq 0\}}
\norm{U}_{s_0-m}^{\max\{0,p-k-q\}}\norm{U}_{s-m}
\prod_{j=k_2+1}^k\norm{h_j}_{s_0-m}
\\
&\,\,
\times\norm{U}_{s_0-m}^{\max{\{0,p-k-1\}}}
\norm{V}_{s_0}\prod_{j=1}^{k_2}\norm{h_j}_{s_0-m}\,,
\end{aligned}\end{equation*}
which is the first line of \eqref{porto20} up to renaming $s_0\rightsquigarrow s_0-m$.
See Proposition 3.16,  3.17, 3.18 in \cite{BD}.
The translation invariance properties for the composed operators and symbols 
in items (i)-(v) follow as in the proof of Proposition 
\ref{teoremadicomposizione} in \cite{BFP}.
% using \eqref{def:tr-in}, \eqref{tautheta}, \eqref{def:R-trin}.
\end{proof}

\subsection{Real to real and self-adjoint matrices of operators}
We discuss some algebraic properties of matrices of operators.

\noindent
{\bf Real-to-real operators.} 
Given a linear  operator  
$ R(U) [ \cdot ]$ acting on $ \C^2 $ (it may be a 
smoothing operator in  $\Sigma\mathcal{R}^{-\rho}_{1}[r,N]$ or
$\opbw(a(U;x,\x))$ with 
$a\in \Sigma\Gamma^{m}_{1}[r,N]$) 
we associate the linear  operator  defined by the relation 
\begin{equation}\label{opeBarrato}
\ov{R}(U)[V] := \ov{R(U)[\ov{V}]} \, ,   \quad \forall V \in \C^2 \, .
\end{equation}
We say that a matrix of operators acting in $ \C^2 $ is   
\emph{real-to-real}, if it has the form 
\begin{equation}\label{vinello}
R(U) =
\left(\begin{matrix} R_{1}(U) & R_{2}(U) \\
\ov{R_{2}}(U) & \ov{R_{1}}(U)
\end{matrix}
\right) \, .
\end{equation}
Notice that 

\noindent
$ \bullet $
 if  $R(U)$  is a  real-to-real matrix of operators 
then, given $V=\vect{v}{\ov{v}}$, the vector  $Z:=R(U)[V]$ has the form  $Z=\vect{z}{\bar{z}}$, namely the second component is the complex conjugated of the first one.

\noindent
$ \bullet $
If a  matrix of symbols $A(U;x,\x)$, in some class 
$\Sigma{\Gamma}^{m}_{1}[r,N]\otimes\mathcal{M}_2(\C)$,  
has the form 
\begin{equation}\label{prodotto}
A(U;x,\x) =
\left(\begin{matrix} {a}(U;x,\x) & {b}(U;x,\x)\\
{\ov{b(U;x,-\x)}} & {\ov{a(U;x,-\x)}}
\end{matrix}
\right)\,,
\end{equation}
then the matrix of operators $ \opbw(A(U;x,\x))$ is real-to-real. 

%\vspace{0.5em}
%\noindent
%{\bf Notation.}
% $A\lesssim_{s} B$
%means $A \leq C(s) B$
% where $C(s) > 0 $ 
% is a  constant depending on $s \in \R $.
%\setlength{\leftmargini}{2em}
%\begin{itemize}
%
%\item 
%$A\lesssim_{s} B$
%means $A \leq C(s) B$
% where $C(s) > 0 $ 
% is a  constant depending on $s \in \R $.
% \end{itemize}

\begin{defn}{\bf (Classical symbol).}\label{classicSimbo}
A symbol $a$ of order $d$ is called \emph{classical} if $a(x,\xi)\sim \sum a_j(x,\xi)$ and $a_j$ is positive homogeneous with respect to $\xi$ of order $d-j$.
\end{defn}
\noindent
{\bf Self-adjoint para-differential operators.}
We now study self-adjoint matrices para-differential operators. We shall restrict to the case that such matrices are reality preserving, i.e. matrices of the form \eqref{vinello}. 
Consider an operator $\mathfrak{F}$  of the form %\eqref{vinello} 
 \begin{equation}\label{barrato4}
\mathfrak{F}:=\left(\begin{matrix} A & B \\ \ov{B} & \ov{A}\end{matrix}\right), 
\end{equation}
for $A,B$ linear operators
and denote by $\mathfrak{F^*}$ its adjoint with respect to the scalar product %$\eqref{comsca}$
\begin{equation*}
(U,V)_{{\bf{ H}}^0}:=%\frac12
\int_{\TTT}U\cdot\ov{V}dx\,,\qquad U=\vect{u}{\bar{u}}, V=\vect{v}{\bar{v}}\in H^{s}(\mathbb{T},\mathbb{C}^{2})\,,
\end{equation*}
i.e.
\begin{equation*}
(\mathfrak{F}U,V)_{{\bf{ H}}^0}=(U,\mathfrak{F}^{*}V)_{{\bf{ H}}^0}\,. 
%, \quad \forall\,\, U,\, V\in {\bf{ H}}^s.
\end{equation*}
One can check that 
\begin{equation*}
\mathfrak{F}^*:=\left(\begin{matrix} A^*(U;t) & \ov{B}^*(U;t) \\ {B}^*(U;t) & \ov{A}^*(U;t)\end{matrix}\right),
\end{equation*}
where $A^*$ and $B^*$ are respectively the adjoints of the operators $A$ and $B$ with respect to
the complex scalar product on $L^{2}(\TTT;\CCC)$ in \eqref{prodottoorologio}.
%\[
%(u,v)_{L^{2}}:=\int_{\TTT}u\cdot \bar{v}dx, \quad u,v\in L^{2}(\TTT;\CCC).
%\]

\begin{definition}{\bf (Self-adjointness).}\label{selfi}
Let $\mathfrak{F}$   be a reality preserving linear operator
of the form \eqref{barrato4}.
We say that  $\mathfrak{F}$ 
is \emph{self-adjoint} if $A,A^*,B,B^* : H^{s}\to H^{s'}$, for some $s,s'\in \RRR$ and
\begin{equation}\label{calu}
A^{*}=A,\;\;
\;\; \ov{B}=B^{*}.
\end{equation}
\end{definition}
We  consider para-differential operators  of the  form:
\begin{equation}\label{prodottoloc}
\begin{aligned}
 \quad \opbw(A(U;x,\xi))&:=\opbw\left(\begin{matrix} {a}(U;x,\x) & {b}(U;x,\x)\\
{\ov{b(U;x,-\x)}} & {\ov{a(U;x,-\x)}}
\end{matrix}
\right) \\
&:=\left(\begin{matrix} \opbw({a}(U;x,\x)) & \opbw({b}(U;x,\x))\\
\opbw({\ov{b(U;x,-\x)}}) & \opbw({\ov{a(U;x,-\x))}}
\end{matrix}
\right)\,,
\end{aligned}
\end{equation}
where $a$ and $b$ are symbols in $\Gamma^{m}_{p}[r]$ and $U$ 
is a function belonging to $B^{K}_{r}(H^{s_0})$ for some $s_0$ large enough. Note that the matrix of operators in \eqref{prodottoloc} has the form \eqref{barrato4}. Moreover it is self-adjoint  if and only if
\begin{equation}\label{quanti801}
a(U;x,\xi)=\ov{a(U;x,\xi)}\,,\quad b(U;x,-\xi)= b(U;x,\xi)\,,
\end{equation}
indeed conditions \eqref{calu} on these operators read
\begin{equation}\label{marieanto}
\begin{aligned}
\left(\opbw(a(U;x,\xi))\right)^{*}&=\opbw\left(\ov{a(U;x,\xi)}\right)\, ,\\
\quad \ov{\opbw(b(U;x,\xi))}&=\opbw\left(\ov{b(U;x,-\xi)}\right)\,.
\end{aligned}
\end{equation}

\subsection{Flows of para-differential operators}\label{FLOWS}

The main result of this section is Theorem \ref{flussononlin}.
In this Theorem we analyze the well-posedness of some 
non linear flows generated by para-differential operators. 
These flows will be used in the next Sections as auxiliary flows in order to 
generate non linear
changes of coordinates.
More precisely we study the system \eqref{flusso} 
in the cases of generators as in \eqref{sim1}, \eqref{sim2}, \eqref{sim3} and \eqref{sim4}.
We shall define a sequence of linear problems 
approximating the non linear one. 
In section \ref{linflowflow}
we study the linear problem associated to problem \eqref{flusso}.
In section \ref{sec:nonlinflows} we prove Theorem \ref{flussononlin}.

\subsubsection{Linear flows}\label{linflowflow}
Let $0<r\ll1$, $u\in B_{r}(H^{s})$  we define  the symbols 
\begin{equation*}\begin{aligned}
&\tilde{A}_1(\tau,u;x,\xi):= A(\tau,u;x)\x,\,\,\, \mbox{with}\,\,\, A\in \mathcal{F}_{1}^{\mathbb{R}}[2r],\\
&A_j(\tau,u;\xi) \in \Sigma\Gamma^{j}_p[r,N], j\geq 0, \,\,\, \mbox{and}  \,\,\, \begin{cases}
 &1\leq j, \,\,\,\mbox{real and independent of}\,\, x\\
 &j<1, \,\,\, A_j-\ov{A_j} \in \Sigma\Gamma^{0}_p[r,N],
 \end{cases}
\end{aligned}
\end{equation*}
and we consider the problem
\begin{equation}\label{pioggia}
\left\{\begin{aligned}
&\pa_{\tau}z=\opbw\big(\ii \mathfrak{A}(\tau,u;x,\xi)\big)[z]\,,\\
&z(0)=u\,,
\end{aligned}\right.
\end{equation}
where $\mathfrak{A}$ equals either $\tilde{A}_1$ or $A_j$.
We have the following.
%Recall that $A\in\mathcal{F}_1^{\mathbb{R}} $
%implies that there is $s_0\gg1$ such that,
%for any $\alpha+dk\leq s-s_0$, $s\geq s_0$, we have
%\[
%rC(s)|A|_{\alpha,k}^{\mathcal{F},s}\ll1
%\]
%for some constant $C(s)>0$ depending only on $s$.
\begin{lemma}\label{flusso-differenziale}
Let $d$ be the number appearing in the definition of symbols \ref{pomosimb} item (ii).
Consider the problem \eqref{pioggia} with $\mathfrak{A}=A_j$, then for any $0\leq j\leq d$ the following holds true.
For any $s\geq s_0$ there is $r_0, C>0$ such that for $0\leq r\leq r_0$, one has:
\begin{equation}\label{pioggia2}
\begin{aligned}
\|z(\tau)\|_{H^{s}}&\leq \|u\|_{H^{s}}(1+C\|u\|_{H^{s_0}})\qquad \forall \, 0\leq \tau\leq 1\,,\\
\|\pa_{\tau}^{k}z(\tau)\|_{H^{s-jk}}&\leq C\|u\|_{H^{s}}\|u\|_{H^{s_0}}^{k}\,,\quad
1\leq dk\leq s-s_0\,.
\end{aligned}
\end{equation}
Moreover
\begin{equation}\label{pioggia3}
\|(d_{u}z)(u)[h]\|_{H^{s-j}}\leq \|h\|_{H^{s}}(1+C\|u\|_{H^{s}})\,,\qquad 
\forall \, 0\leq \tau\leq 1\,\quad \forall h\in H^{s}\,,
\end{equation}
and, for any $2\leq dk\leq s-s_0$, we have that
\begin{equation}\label{pioggia4}
\|(d_{u}^{k}z)(u)[h_1,\ldots,h_{k}]\|_{H^{s-jk}}\leq 
C 
\|h_1\|_{H^{s}}\cdots 
\|h_{k}\|_{H^{s}}\,,\qquad 
\forall \, 0\leq \tau\leq 1\,\quad \forall h_{i}\in H^{s}\,\;\; i=1,\ldots,k\,.
\end{equation}
for some constant $C=C(s,k)>0$ independent of $\mathfrak{A}$.  The same estimates hold true with $j=1$ in the case that $\mathfrak{A}=\tilde{A}_1$.
\end{lemma}

\begin{proof}
We give  the proof in the case of $\tilde{A}_1$, for $A_j$ with $j>0$ the proof is similar and for $j=0$ it is standard theory of Banach spaces ODEs.
 By reasoning as in Lemma 3.22 in \cite{BD}
 one can prove that the flow $\Phi^{\tau}=\Phi^{\tau}(u)$
 \begin{equation}\label{pioggia2bis}
\pa_{\tau}\Phi^{\tau}(u)=\opbw(\ii A(\tau,u;x)\x)[\Phi^{\tau}(u)]\,,\quad \Phi^{0}(u)={\rm Id} 
 \end{equation}
is well posed on $H^{s}$
and satisfies 
\begin{equation}\label{pioggia2tris}
\|\Phi^{\tau}(u)h\|_{H^{s}}\leq C\|h\|_{H^{s}}(1+\|u\|_{H^{s_0}})\qquad \forall \, 0\leq \tau\leq 1\,.
\end{equation}
 In particular setting $z(\tau)=\Phi^{\tau}(u)u$ the \eqref{pioggia2} holds.

\noindent To prove \eqref{pioggia3} we argue as follows.
By differentiating in $u$ the problem \eqref{pioggia} we have

\begin{equation}\label{pioggia5}
\left\{\begin{aligned}
&\pa_{\tau}(d_{u}z)(u)[h]=\opbw\big(\ii A(\tau,u;x)\x\big)[(d_{u}z)(u)[h]]+
f(\tau;u)\,,\\
&(d_{u}z(0))(u)[h]=h\,,
\end{aligned}\right.
\end{equation}
where 
\begin{equation}\label{termnoto}
f(\tau;u):=\opbw\big(\ii (d_{u}A)(\tau,u;x)[h]\x\big)[z]\,.
\end{equation}
By estimate \eqref{maremma2} and the Lemma of action of para-differential operators on Sobolev spaces
we have
\begin{equation}\label{termnoto2}
\|f(\tau,u)\|_{H^{s-1}}\lesssim_{s} \|h\|_{s_0}\|z\|_{H^{s}}
\stackrel{\eqref{pioggia2}}{\lesssim_{s}}\|u\|_{H^{s}}\|h\|_{H^{s_0}}\,,
\qquad \forall \, 0\leq \tau\leq 1\,.
\end{equation}
By Duhamel formula we have (recall \eqref{pioggia2bis})
\[
(d_{u}z)(u)[h]=\Phi^{\tau}h+\Phi^{\tau}\int_0^{\tau} (\Phi^{\s})^{-1}f(\s;u) d\s
\]
which, together with \eqref{pioggia2tris}, \eqref{termnoto2} , implies \eqref{pioggia3}.
Iterating this reasoning one gets the \eqref{pioggia4}.
\end{proof}

\subsubsection{Non-linear flows}\label{sec:nonlinflows}
Consider the Cauchy problem
\begin{equation}\label{flusso}
\left\{\begin{aligned}
&\pa_{\tau}z=\opbw\big(\ii f(\tau,z;x,\x)\big)[z]\,,\\
&z(0)=u\,,\qquad u\in B_{r}(H^{s}) \,,
\end{aligned}\right.
\end{equation}
for some $r>0$ small and $s>0$ large, 
where $ f $ is a symbol assuming one of the following forms: 
\begin{align}
& f(\tau, u;x,\x) := B(\tau, u;  x) \xi \,,
%:= 
%\frac{\beta (u; x)}{1 + \tau \beta_x (u; x)} \xi \, , 
\quad \, B (\tau,u; x) \in 
\Sigma {\mathcal F}_{1}^\R[r,N] \, , \label{sim1} \\
&  f(\tau, u;x, \x) \in \Sigma\Gamma^{m}_{1}[r,N]\,,\;\;0<m<1\,,
\quad f(\tau, u;x, \x)-\ov{f(\tau, u;x, \x)}\in 
\Sigma\Gamma^{0}_1[r,N]\, , \label{sim2}\\
&  f(\tau, u;x, \x)  \in \Sigma\Gamma^{m}_{1}[r,N] \, \, , 
\quad \quad
\ \qquad m \leq 0 \, , \label{sim3}\\
& f(\tau,u;\xi) \in \Sigma\Gamma^{m}_{1}[r,N]\,, 
\quad m\geq 0, \quad \mbox{$f(\tau,u;x)$ 
real and independent of $x$\,,}\label{sim4}
\end{align}
with $\tau\in[0,1]$. 
We also assume that the symbol $f(\tau,u;x,\x)$ satisfies the estimates 
\eqref{pomosimbo1}-\eqref{maremma2} uniformly in $\tau\in[0,1]$.
%Let us consider the symbol
%\begin{equation}\label{gene1}
%%A(\tau,u;x,\x):=\ii \x B(\tau,u;x)\,,\quad \tau\in [0,1]\,, 
%%\quad 
%B(\tau,u;x):=\frac{\beta(u;x)}{1+\tau\beta_{x}(u;x)}\,,\;\;
%\beta\in \Sigma\mathcal{F}^{\mathbb{R}}_{1,0}[2r,N]\,, \quad \tau\in [0,1]\,.
%\end{equation}
%Recall that $B\in\Sigma\mathcal{F}_{1,0}^{\mathbb{R}}[2r,N]$
%implies that there is $s_0\gg1$ such that,
%for any $\alpha+k\leq s-s_0$, $s\geq s_0$, we have
%\[
%|B|_{\alpha,k}^{\mathcal{F}_1}\leq C(s)r
%\]
%for some constant $C(s)>0$ depending only on $s$.
The key result of this section is the following.

\begin{theorem}{\bf (The non-linear transport flow).}\label{flussononlin}
For any $s\geq s_0$ there is $r_0, C>0$ such that for $0\leq r\leq r_0$,
the following holds. There exists a unique solutions of \eqref{flusso}
(with generator in \eqref{sim1}, \eqref{sim2}, \eqref{sim3},\eqref{sim4}) 
$z(\tau):=\Phi^{\tau}(u)$, defined for $\tau\in [0,1]$ such that
\begin{equation}\label{spazio1}
z(\tau)\in \bigcap_{k=0}^{K} C^{k}([0,1]; H^{s-jk})\,,\quad 0\leq jK\leq s-s_0\,.
\end{equation}
In particular we have
\begin{align}
\sup_{\tau\in[0,1]}\|z(\tau)\|_{H^{s}}&\leq \|u\|_{H^{s}}(1+C\|u\|_{H^{s_0}})\,,\label{stimaflusso1}\\
\sup_{\tau\in[0,1]}\|\pa_{\tau}^{k}z(\tau)\|_{H^{s-jk}}&\leq C\|u\|_{H^{s}}\|u\|_{H^{s_0}}^{k}\,,\quad
1\leq k\leq s-s_0\,,\label{stimaflusso2}\\
\sup_{\tau\in[0,1]}\|(d_{u}^{k}\Phi^{\tau})(u)[h_1,\ldots,h_{k}]\|_{H^{s-jk}}
&\leq C \|h_1\|_{H^{s}}\cdots 
\|h_{k}\|_{H^{s}}\,,\quad \forall h_{i}\in H^{s}\,\;\; i=1,\ldots,k\,,
\label{stimaflusso3}
\end{align}
with $j=1$ in the case of \eqref{sim1}, $j=m$ for the cases \eqref{sim2}, \eqref{sim3},\eqref{sim4}.
Finally we have that  $\Phi^{\tau}(u)=u+M(\tau;u)[u]$ with $M(\tau;u)\in \Sigma\mathcal{M}_{1}[r,N]$
with estimates uniform in $\tau\in[0,1]$.
\end{theorem}

The proof of Theorem \ref{flussononlin} relies  on an 
iterative scheme based on the ideas used 
in \cite{FIloc}. 
We give the proof of the result of the case $f(\tau,Z;x,\x)$ in \eqref{sim1}, the others are similar.
Let us introduce the following sequence of linear problems.
Let $u^{(0)}\in H^{s}$ such that $\|u^{(0)}\|_{H^s}\leq r$ for some $r>0$.
For $n=0$ we set
\begin{equation}\label{rondine0}
\mathcal{A}_0:=\left\{
\begin{aligned}
&\pa_{\tau}u_0=0\,,\\
&u_0(0)=u^{(0)}\,.%\quad U^{(0)}\in {\bf H}^{s}.
\end{aligned}\right.
\end{equation}
The solution of this problem  exists and it is unique, defined for any $\tau\in \mathbb{R}$ 
by standard linear theory.
For $n\geq1$, 
 assuming  $u_{n}$ satisfies \eqref{spazio1} 
for some $s_{0},K>0$ and $s\geq s_0$, 
we define %for any $n\in\N\setminus \{0\}$ 
the  Cauchy problem
\begin{equation}\label{rondinen}
\mathcal{A}_n:=\left\{
\begin{aligned}
&\pa_{\tau}u_n-\opbw(\ii B(\tau,u_{n-1};x)\xi)u_n=0 \,,\\
&u_n(0)=u^{(0)}\,,%\quad U^{(0)}\in {\bf H}^{s},
\end{aligned}\right.
\end{equation}
where the symbol $B(\tau,z;x)$ is  defined   in  \eqref{sim1}.
One has to show that each problem $\mathcal{A}_{n}$ 
admits a unique solution $U_{n}$ defined for $\tau\in [0,1]$.
We use  Lemma \ref{flusso-differenziale}
 in order to prove the following.

\begin{lemma}\label{esistenzaAN}
If $r$ is sufficiently small, then there exists $s_0>0$  such that for all $s\geq s_0$ 
the following holds. 
There exists a constant 
$\theta$, depending on $r$ and $s$,
such that 
for any $n\geq0$ one has:

\begin{itemize}
\item[${\bf (S1)}_{n}$] for $0\leq m\leq n$ there exists a function $u_{m}$ in
$u_{m}\in \bigcap_{k=0}^{K} C^{k}([0,1]; H^{s-k})$, $0\leq K\leq s-s_0$
such that
\begin{equation}\label{uno}
\begin{aligned}
\|u_m(\tau)\|_{H^{s}}&\leq \|u_0\|_{H^{s}}(1+C\|u_0\|_{H^{s_0}})\qquad \forall \, 0\leq \tau\leq 1\,,\\
\|\pa_{\tau}^{k}u_m(\tau)\|_{H^{s-k}}&\leq C\|u_0\|_{H^{s}}\|u_0\|_{H^{s_0}}^{k}\,,\quad
1\leq dk\leq s-s_0\,,\\
\|(d_{u}^{k}u_m)(u)[h_1,\ldots,h_{k}]\|_{H^{s-k}}&\leq 
C 
\|h_1\|_{H^{s}}\cdots 
\|h_{k}\|_{H^{s}}\,,\qquad 
\forall \, 0\leq \tau\leq 1\,\quad \forall h_{i}\in H^{s}\,\;\; i=1,\ldots,k\,.
%&u_{m}\in \bigcap_{k=0}^{K} C^{k}([0,1]; H^{s-k})\,,\quad 0\leq K\leq s-s_0\,,
%&\sum_{k=0}^{K}\|\pa_{t}u_{m}\|_{H^{s-k}}\leq \theta\,,
\end{aligned}
\end{equation}
for some $C>0$ independent of $m,n$, 
which is the unique solution of the problem $\mathcal{A}_{m}$;

\item[${\bf (S2)}_{n}$] for $0\leq m\leq n$ one has
\begin{equation}\label{due}
\sum_{k=0}^{K}\| \pa_{\tau}(u_{m} -u_{m-1}) \|_{H^{s'-k}}\leq 2^{-m} r,\quad s_0\leq s'\leq s-1\,,
\qquad \forall 0\leq \tau\leq 1\,,
\end{equation}
where $U_{-1}:=0$.
\end{itemize}
\end{lemma}

\begin{proof}
We argue by induction. 
The $(S1)_0$ and $(S2)_0$ are trivial (see the problem \eqref{rondine0}).
Suppose  that  
$(S1)_{n-1}$,$(S2)_{n-1}$ hold
with a constant $C=C(s)\gg1$.
%$\theta=\theta(s,r)\gg1$.
% and
% a time $T=T(s,r,\theta)\ll 1$.
We show that $(S1)_{n}$,$(S2)_{n}$ hold with the same constant $C$. 
%and $T$.
By estimates \eqref{uno}  on $u_{n-1}$ we deduce that $\|u_{n-1}\|_{H^{s}}\leq 2r$
(if $r>0$ is small enough) and 
the symbol $B(\tau,u_{n-1},x)\x$ 
satisfies the hypotheses of Lemma \ref{flusso-differenziale}.
Then the \eqref{uno} on $u_n$ follows by estimates \eqref{pioggia2}-\eqref{pioggia4}.
Let us check $(S2)_n$. 
Setting $v_n=v_n-v_{n-1}$ we have that
\begin{equation}\label{eqdiff}
\left\{\begin{aligned}
&\pa_{\tau}v_n-\opbw(\ii B(\tau,u_{n-1};x)\xi)v_n + f_{n}=0\,,\\
&v_n(0)=0\,,
\end{aligned}\right.
\end{equation}
where 
\begin{equation}\label{eqdiff2}
f_n:=\ii \opbw\Big(B(\tau,u_{n-1};x)\xi-B(\tau,u_{n-2};x)\xi\Big)U_{n-1}\,.
\end{equation}
Notice that
\[
B(\tau,u_{n-1};x)-B(\tau,u_{n-2};x)=(d_{u}B)(\tau,u_{n-1}+\s v_{n-1})[v_{n-1}]
\]
for some $\s\in[0,1]$. Moreover, 
by \eqref{maremma2}, we have
\[
\|(d_{u}B)(\tau,u_{n-1}+\s V_{n-1})[V_{n-1}]\|_{L^{\infty}_x}\lesssim_{s}
C\|v_{n-1}\|_{H^{s_0}}\,.
\]
Therefore, by Proposition \ref{AzioneParaMet},
 we have 
\begin{equation}\label{eqdiff3}
\begin{aligned}
\|f_{n}\|_{H^{s'}}&\leq\|\ii \opbw\Big(B(u_{n-1};x)\xi-B(u_{n-2};x)\xi\Big)U_{n-1}\|_{H^{s'}}
\lesssim_{s}C\|v_{n-1}\|_{H^{s_0}}\|u_{n-1}\|_{H^{s'+1}}\,.
\end{aligned}
\end{equation}
Let 
 $\psi_{u_{n-1}}(\tau)$ be the flow of system \eqref{eqdiff} with $f_n=0$, 
 which is given by Lemma \ref{flusso-differenziale}.
The Duhamel formulation of \eqref{eqdiff}
is
\begin{equation}
v_{n}(\tau)=\psi_{u_{n-1}}(\tau)\int_0^{\tau} 
(\psi_{u_{n-1}}(\sigma))^{-1}f_{n}(\sigma)d\sigma\,.
\end{equation}
Then using  the inductive hypothesis \eqref{uno}, inequality \eqref{pioggia2} 
we get
\begin{equation}
\|v_n\|_{H^{s'}}\leq  C_{s} r \|v_{n-1}\|_{H^{s'}}, \quad \forall \; t\in [0,1]\,,
\end{equation}
where $C_s>0$ is a constant depending $s$.
If  $C_s r\leq1/2$ then we have 
$\|v_n\|_{H^{s'}}\leq 2^{-n}r$ for any  $ t\in [0,1]$ which is the $(S2)_n$.
\end{proof}

\begin{proof}[{\bf Proof of Theorem \ref{flussononlin}}]
By Lemma \ref{esistenzaAN} we know that the sequence $u_n$ defined 
by the problem \eqref{rondinen} 
converges strongly to a function 
$z$ in $C^0([0,1],H^{{s'}})$ for any ${s'}\leq s-1$ and, up to subsequences, 
\begin{equation}\label{debolesol}
\begin{aligned}
&u_{n}(\tau) \rightharpoonup U(\tau), \;\;\; {\rm in } \;\;\; H^s\,,\qquad
%\\&
\pa_{\tau}u_{n}(\tau)  \rightharpoonup \pa_{\tau}u(\tau) , \;\;\; {\rm in } \;\;\; H^{s-1},
\end{aligned}
\end{equation}
for any $\tau\in [0,1]$, moreover the function $u$ is in 
$L^{\infty}([0,1],H^s)\cap{\rm  Lip}([0,1],H^{s-1})$.
We claim that, $z$ solves the \eqref{flusso}, it belongs to $C^{0}([0,1]; H^{s})\cap
C^{1}([0,1]; H^{s-1})$ and it is unique.
This can be proved by classical arguments, for instance following the proof of Theorem 
$1.1$ in section $6$ in  \cite{FIloc}.
The \eqref{stimaflusso2}, \eqref{stimaflusso3} can be deduced by differentiating the equation 
\eqref{flusso} (or using that $z$  is weak limit of the sequence $u_{n}$
satisfying the estimates \eqref{uno}). The  theorem in the  cases  \eqref{sim2}, \eqref{sim3}, \eqref{sim4} may be proved exactly in the same way modifying Lemma \ref{esistenzaAN} according to Lemma \ref{flusso-differenziale} in the case that $\mathfrak{A}=A_j=f$ in \eqref{sim2}, \eqref{sim3}, \eqref{sim4}, with $j\rightsquigarrow m$.
\end{proof}

\section{Main results and applications to PDEs}\label{regularization}
In this section we state the main results of this paper.
Consider
\begin{equation}\label{ipotesim}
%m\in \mathbb{N} \quad {\rm or}\quad 
m=\frac{k}{2}\,, \quad k\in \mathbb{N}\,,
\end{equation}
and let 
 $f_{m}\in \Gamma_0^{m}$ 
 be a real valued,  even in $\x\in \mathbb{R}$  \emph{classical} symbol,
 i.e. it admits and expansion in decreasing homogeneous symbols.
%$d$-homogeneous,  ${C^{\infty}(\mathbb{R}^+,\mathbb{R})}$
%and even in $\x\in\mathbb{R}$. 
Let us define the operator $\Omega$ as 
\begin{equation}\label{omegone}
\Omega e^{\ii jx}=\omega_{j}e^{\ii jx} \,,\quad \omega_{j}:=f_{m}(j)\,,\;\; 
\forall\, j\in\mathbb{Z}\,.
\end{equation}
In other words $\Omega:=\opbw(f_{m}(\x))$.
We consider symbols of the following form
\begin{equation}\label{Forma-di-Ainizio}
\begin{aligned}
&a(U;x,\x)=(1+a_{m}(U;x))f_m(\x)+a_{m'}(U;x,\x)\,,\qquad m>1\,,\;\;\; m'=m-\frac{1}{2}
\;\;{\rm or}\;\; m'=m-1\,,\\
&b(U;x,\x)=b_{m}(U;x)f_m(\x)+b_{m'}(U;x,\x)\,,
\\
&{a}_m\in \Sigma\mathcal{F}_1^{\mathbb{R}}[r,N]\,,\quad 
{b}_m\in \Sigma\mathcal{F}_1[r,N]\,,\quad 
a_{m'}\,,\; b_{m'}\in\Sigma\Gamma^{m'}_{1}[r,N]\,,\\
&a_{m'}(U;x,\x)-\ov{a_{m'}(U;x,\x)}\in \Sigma\Gamma^{0}_1[r,N]\,,
\end{aligned}
\end{equation}
and the system
\begin{equation}\label{Nonlin1inizio}
\left\{
\begin{aligned}
&\dot{U}=X(U):=\ii E\opbw(A(U;x,\x))[U]+R(U)[U]\\
&U(0)=U_0\in H^{s}\times H^{s}
\end{aligned}\right.
\end{equation}
with $R\in \Sigma\mathcal{R}^{-\rho}_{1}[r,N]\otimes\mathcal{M}_2(\mathbb{C})$
and
\begin{equation}\label{operatoreEgor1inizio}
\begin{aligned}
&A(U;x,\x):=\left(\begin{matrix}
a(U;x,\x) & b(U;x,\x) \\
\ov{b(U;x,-\x)}&\ov{a(U;x,-\x)}
\end{matrix}\right)\,.
\end{aligned}
\end{equation}

\subsection{Regularization of para-differential vector fields}
The main result of the paper is  the following.

\begin{theorem}{\bf (Non-linear Egorov).}\label{thm:main}
There exist $s_0>0$, $r_0>0$ such that, 
for $s\geq s_0$, $r\leq r_0$ the following holds true.
There exist an invertible map
\[
\Psi : B_{r}(H^{s}(\mathbb{T};\mathbb{C}^{2}))\cap\mathcal{U}\to
H^{s}(\mathbb{T};\mathbb{C}^{2})\cap \mathcal{U}\,,
\]
such that, setting 
\begin{equation}\label{nuovotutto}
Z:=\Psi(U)
\,,
\qquad
\mathcal{Y}(Z):=d\Psi\big(\Psi^{-1}(Z)\big)\big[ {X}(\Psi^{-1}(Z))\big]\,,
\end{equation}
we have that
\begin{equation}\label{Nonlin1fine}
\left\{
\begin{aligned}
&\dot{Z}=\mathcal{Y}(Z):=\ii E\mathcal{L}(Z)[Z]
%\opbw(A(U;x,\x))[U]
+\mathcal{Q}(Z)[Z]\\
&Z(0)=\Psi(U_0)
\end{aligned}\right.
\end{equation}
where $\mathcal{Q}\in\Sigma\mathcal{R}^{-\rho}_1[r,N]
\otimes\mathcal{M}_2(\mathbb{C})$ and (see \eqref{omegone})
\begin{equation}\label{opfinale}
\begin{aligned}
&\mathcal{L}(Z):=
\opbw\left( \begin{matrix} f_{m}(\x)  & 0 \\ 0 & f_{m}(\x)\end{matrix}\right)+
\opbw\big(\mathfrak{M}(Z;\x)\big)\,, \quad
\mathfrak{M}(Z;\x):=\left(
\begin{matrix}
\mathfrak{m}(Z;\x) & 0\\
0 & \ov{\mathfrak{m}(Z;-\x) }
\end{matrix}
\right)\,,\\
&\mathfrak{m}(Z;\x):=\mathfrak{m}_{m}(Z)f_{m}(\x)
%(1+\mathfrak{m}_{m}(Z))f_{m}(\x)
+\mathfrak{m}_{m'}(Z;\x)\,,
\qquad
\mathfrak{m}_{m}(Z)\in\Sigma\mathcal{F}^{\mathbb{R}}_1[r,N]\,,m'=m-\frac12
\\&
 \mathfrak{m}_{m'}(Z;\x)\in \Sigma\Gamma^{m'}_1[r,N]\,,
 \qquad
  \mathfrak{m}_{m'}(Z;\x)-\ov{ \mathfrak{m}_{m'}(Z;\x)}\in \Sigma\Gamma^{0}_1[r,N]\,.
\end{aligned}
\end{equation}
Moreover, for any $s\geq s_0$, the maps $\Psi^{\pm1}$ satisfy
\begin{equation}\label{curecure}
\|\Psi^{\pm1}(U)\|_{H^{s}}\leq \|U\|_{H^{s}}(1+C\|U\|_{H^{s_0}})\,,
\end{equation}
for some constant $C>0$ depending  on $s$.
\end{theorem}

\begin{proof}
We shall apply iteratively Theorems 
 \ref{conjOrdMaxhighoff}, 
\ref{conjOrdMaxloweroff},
\ref{conjOrdMax} and 
\ref{conjOrdMaxredu}.
\end{proof}
Some comments on the theorem above are in order
\begin{itemize}
\item The Theorem above shows that a \emph{system} 
as \eqref{Nonlin1inizio}
can be reduced to  a \emph{diagonal} system with constant coefficients
plus a smoothing remainder. This will be achieved into two steps:
$(i)$ a \emph{block-diagonalization} of the system (which is the content of 
Theorems  \ref{conjOrdMaxhighoff}, 
\ref{conjOrdMaxloweroff}); $(ii)$ a \emph{reduction to constant coefficients}
of the diagonal terms (which is the content of Theorems \ref{conjOrdMax},
\ref{conjOrdMaxredu}). 

\item The parity assumption of the Fourier multiplier  $f_{m}(\x)$ is used only in the
\emph{block-diagonalization} procedure in sections \ref{diago-lineare}.
Therefore, Theorem \ref{thm:main} 
applies also to \emph{scalar} equations of the form
\begin{equation}\label{posacenere}
\dot{u}=\ii \opbw\big(a(u;x,\x)\big)u+Q(u)[u]\,,
\end{equation}
where $a(u;x,\x)$ is a scalar symbol as in \eqref{Forma-di-Ainizio} and
$Q\in \Sigma\mathcal{R}^{-\rho}_1[r,N]$.
In this case, we assume that $f_{m}(-\x)=f_{m}(\x)$, i.e.
it is \emph{odd} in $\x\in \mathbb{R}$,
then we have the following result which is a consequence of Theorem \ref{thm:main}.
\end{itemize}
\begin{corollary}\label{thm:mainscalar}
There exist $s_0>0$, $r_0>0$ such that, 
for $s\geq s_0$, $r\leq r_0$ the following holds true.
There exist an invertible map $\Psi: B_{r}(H^{s}(\mathbb{T};\mathbb{R}))\to
B_{r}(H^{s}(\mathbb{T};\mathbb{R}))$
satisfying estimates like \eqref{curecure} and
%guarantes that \eqref{posacenere} can be written as
\begin{equation}\label{posacenere2}
\dot{z}=\ii \opbw\big(f_{m}(\x)+\mathfrak{m}(z;\x)\big)z+\widetilde{Q}(z)[z]\,,
\end{equation}
where $z=\Psi(u)$,
 $\mathfrak{m}(z;\x)$ is  as in \eqref{opfinale} and 
 $\widetilde{Q}\in \Sigma\mathcal{R}^{-\rho}_1[r,N]$.
\end{corollary}

In the case that the vector field $X(U)$ in \eqref{Nonlin1inizio}
has an Hamiltonian structure (see \eqref{X_H})
we also have a version of Theorem \ref{thm:main}
which preserves the symplectic structure of the vector field.

\begin{theorem}{\bf (Symplectic structure).}\label{thm:main2}
Assume that the vector field $X(U)$ in \eqref{Nonlin1inizio}
is \emph{Hamiltonian},
i.e.
\[
X(U):=X_{H}(U):=\ii J\nabla H(U)\,,
\]
for some Hamiltonian $H(U) : 
B_R(H^{s}(\mathbb{T};\mathbb{C}))\to \mathbb{R}$.
Then the result of Theorem \ref{thm:main}
holds with  a  symplectic map $\Psi$
and the vector field $\mathcal{Y}$ in \eqref{nuovotutto}
is Hamiltonian with respect to the symplectic form $\lambda$ in
\eqref{symform}.
Moreover the operator $\mathcal{L}(Z)$ in \eqref{opfinale}
is self-adjoint.
\end{theorem}

\subsection{Poincar\'e-Birkhoff normal forms}
In this section we state an abstract Birkhoff normal form result
for vector a field $\mathcal{Y}(Z)$ 
as in \eqref{Nonlin1fine}
%Consider the system \eqref{Nonlin1fine} 
given by Theorem \ref{thm:main} assuming that the starting vector field $X$ is Hamiltonian.
%Let us define the operator $\Omega$ as 
%\begin{equation}\label{omegone}
%\Omega e^{\ii jx}=\omega_{j}e^{\ii jx} \,,\quad \omega_{j}:=f_{m}(j)\,,\;\; \forall\, j\in\mathbb{Z}\,,
%\end{equation}
%where $f_{m}(\x)$ is given at the beginning of section \ref{regularization}.
Recalling \eqref{omegone}, 
%With this formalism 
thanks to Theorems \ref{thm:main} and \ref{thm:main2}
the system \eqref{Nonlin1fine} is rewritten as
\begin{equation}\label{NonlinBNF1}
\dot{Z}=\mathcal{Y}(Z)=\ii E\Omega Z+\ii E\opbw\big(\mathfrak{M}(Z;\x)\big)Z+\mathcal{Q}(Z)[Z]\,,
\end{equation}
where 
$\mathfrak{M}(Z;\x)\in \Sigma\Gamma^{m}_{1}[r,N]\otimes\mathcal{M}_2(\mathbb{C})$
is given in \eqref{opfinale},
$\mathcal{Q}$ is in $\Sigma\mathcal{R}^{-\rho}_1[r,N]
\otimes\mathcal{M}_2(\mathbb{C})$ and $\mathcal{Y}(Z)$ in 
\eqref{NonlinBNF1} is an Hamiltonian vector field.
Furthermore  the matrix $\mathfrak{M}(Z;\x)$ is self-adjoint, i.e.
it satisfies \eqref{quanti801}.
Recalling the Definitions \ref{pomosimb}, \ref{omosmoothing} and the remarks
under the Definition \ref{smoothoperatormaps} we have that
\[
\mathcal{M}(Z):=\ii E\opbw\big(\mathfrak{M}(Z;\x)\big)+\mathcal{Q}(Z)
\in \Sigma\mathcal{M}_{1}[r,N]\otimes\mathcal{M}_2(\mathbb{C})\,,
\]
hence 
$
\dot{Z}=\ii E\Omega Z+\mathcal{M}(Z)[Z]\,.
%=\ii E\Omega Z+O(Z^{2})\,.
$
We also assume that the frequencies \eqref{omegone}
are such that
\begin{equation}\label{omegone0}
\omega_0\neq0\,.
\end{equation}
In order to state the main result of the section we need some further definitions.
%We give the following definition.
\begin{definition}{\bf (Non-resonance conditions).}\label{nonresOmegaCOND}
We say that the linear frequencies 
$\omega_{j}$ in \eqref{omegone} are \emph{not resonant}, 
at order $N\geq 1$, if the following holds.
There are $N_0>0$ and $c>0$ such that, for any $1\leq p\leq N$, one has
\begin{equation}\label{nonresOMEGA}
|\s_1 \omega_{j_1}+\ldots+\s_p \omega_{j_p}|\geq  c
\max\{\langle j_1\rangle,\ldots, \langle j_p\rangle\}^{-N_0}\,, \quad\forall \, \s_i=\pm\,,\; 
j_i\in \mathbb{Z}\,,\; i=1,\ldots,p\,,
\end{equation}
unless $p$ is even and, up to permutations, one has
\begin{equation}\label{nonresOMEGA2}
\s_i=\s_{\frac{p}{2}+i}\,,\qquad |j_i|=|j_{\frac{p}{2}+i}|\,,\quad i=1,\ldots,p\,.
\end{equation}
Let $\vec{\s}:=(\s_1,\ldots,\s_{p})\in \{\pm\}^{p}$, $\vec{j}=(j_1,\ldots,j_{p})\in \mathbb{Z}^{p}$,
we define the resonant set $\mathcal{S}_{p}$ as
\begin{equation}\label{nonresOMEGA3}
\mathcal{S}_{p}:=\big\{(\vec{\s}, \vec{j})\in \{\pm\}^{p}\times\mathbb{Z}^{d}\; : \; \eqref{nonresOMEGA2}\; {\rm holds} \big\}
\end{equation}
for $p$ even and $\mathcal{S}_{p}=\emptyset$ for $p$ odd. 
\end{definition}

The aim of this section is to conjugate, if $\Omega$ is non-resonant, 
the system in \eqref{NonlinBNF1}
to another para-differential system of the same form
whose symbols and smoothing remainders are 
\emph{resonant}, up to terms of  degree of homogeneity  $N$,
according to the following definition.

\begin{definition}\label{def:resonant}
Let $a_{p}\in \widetilde{\Gamma}_{p}^{m}$ independent of $x\in \mathbb{T}$, 
$R_{p}\in \widetilde{\mathcal{R}}^{-\rho}_{p}\otimes\mathcal{M}_{2}(\mathbb{C})$ 
and recall the homogeneity expansions in Remarks
\ref{espansioneSSS}, 
\ref{espansioneRRR}.

\noindent
$(i)$ Given the symbol
 $a_{p}$ (recall the expansion \eqref{espandoFousimbo})
 % is \emph{resonant} if 
 we define the symbol $\bral a_{p}\brar$ as
\begin{equation}\label{espandoFousimbo100}
\bral a_{p}\brar(U;\x)=
\sum_{\substack{\sum_{i=1}^{p}\s_i j_i=0\\
(\vec{\s},\vec{j})\in \mathcal{S}_{p}}} 
(a_{p})_{j_1,\ldots,j_p}^{\s_1\cdots\s_{p}}(\x)u_{j_1}^{\s_1}
\ldots u_{j_p}^{\s_p}\,.
\end{equation}
%for some coefficients $(a_{p})_{j_1,\ldots,j_p}^{\s_1\cdots\s_{p}}(\x)\in 
%\widetilde{\Gamma}_0^{m}$.
We say that $a_p$ is \emph{resonant} if $a_p\equiv \bral a_{p}\brar$.
Let $a\in \Sigma\Gamma^{m}_{1}[r,N]$ (independent of $x$) of the form
\[
a(U;\x)=\sum_{p=1}^{N-1}a_p(U;\x)
+a_{N}(U;\x),\quad a_{p}\in \widetilde{\Gamma}^{m}_{p}, \;\; a_{N}\in 
\Sigma\Gamma^{m}_{N}[r],
\]
we define the symbol $\bral{a}\brar(U;\x)$ as
\[
\bral{a}\brar(U;\x):=\sum_{p=1}^{N-1}\bral{a_p}\brar(U;\x)+a_{N}(U;\x)\,,
\]
where $\bral a_{p}\brar(U;\x)$ is in % the r.h.s. of 
\eqref{espandoFousimbo100}.
 For a diagonal matrix of symbols  
$A\in\Sigma{\Gamma}_{1}^{m}[r,N]\otimes\MM_{2}(\CCC)$ 
of the form
\begin{equation*}
A(U;\x)=\left(
\begin{matrix}
a(U;\x) & 0\vspace{0.2em}\\
0 & \ov{a(U;-\x)}
\end{matrix}
\right), 
\end{equation*}
we define
\begin{equation}\label{matsim}
\bral{A}\brar(U;\x):=\left(
\begin{matrix}
\bral{a}\brar(U;\x) & 0\\ 0 & \bral{ \,\ov{a} \,}\brar(U;-\x)
\end{matrix}
\right).
\end{equation}

\noindent
$(ii)$ Given an operator $R_{p}$ (recall the expansion 
 \eqref{smooth-terms2}, \eqref{R2epep'}, \eqref{BNF5})
%We say that the operator $R_{p}$ is \emph{resonant} if
%is has the form \eqref{smooth-terms2}, \eqref{R2epep'}
%and 
we define the operator $\bral R_{p}\brar$
as the operator with the form  \eqref{smooth-terms2}, \eqref{R2epep'},
with coefficients 
 \begin{align} \label{BNF5100}
 (\bral\mathtt{R}_{p}\brar(U))_{\s,j}^{\s',k} :=\frac{1}{(2\pi)^{p}}
 \sum_{
 \substack{
 \sum_{i=1}^{p}\s_i j_i=\s j-\s'k \\
 (\vec{\mu},\vec{J})\in \mathcal{S}_{p+2}
 }
 }
 \big( (\mathtt{r}_{p})_{j_1,\ldots,j_p}^{\s_1\cdots \s_{p}}\big)_{\s,j}^{\s',k}
%  (\mathtt{r}_{2,\ep,\ep'})^{\s,\s'}_{n_1,n_2,k}
  u_{j_1}^{\s_1}\ldots u_{j_p}^{\s_{p}}  \, ,   \quad  
  j,k\in \Z\setminus\{0\}   \, , 
 \end{align}
%for some suitable scalar coefficients 
% $  \big( (\mathtt{r}_{p})_{j_1,\ldots,j_p}^{\s_1\cdots \s_{p}}\big)_{\s,j}^{\s',k} \in \C $ and 
 where
 \[
 \vec{\mu}:=(\vec{\s},\s,\s')=(\s_1,\ldots,\s_{p},\s,\s')\,,\quad \vec{J}:=(\vec{j},j,k)=
 (j_1,\ldots,j_{p},j,k)\,.
 \]
 We say that $R_{p}$ is \emph{resonant} if $R_{p}\equiv\bral R_{p}\brar$.
 Let $R\in \Sigma\mathcal{R}^{-\rho}_{1}[r,N]
 \otimes\mathcal{M}_{2}(\mathbb{C})$ 
of the form
\[
R(U)=\sum_{p=1}^{N-1}R_p(U)
+R_{N}(U),\quad R_{p}\in \widetilde{\mathcal{R}}^{-\rho}_{p}, \;\; 
R_{N}\in 
\Sigma\mathcal{R}^{-\rho}_{N}[r],
\]
we define the operator $\bral{R}\brar(U)$ as
\[
\bral{R}\brar(U):=\sum_{p=1}^{N-1}\bral{R_p}\brar(U)+R_{N}(U)\,,
\]
where $\bral R_{p}\brar(U)$ are matrices of operators
with entries given by the r.h.s. of \eqref{BNF5100}.
 \end{definition}
\begin{remark}\label{montefuffa}
Consider a multilinear and constant coefficients in $x$ symbol $a_p$ in $\widetilde{\Gamma}^m_p$, 
and consider the case of an Hamiltonian vector-field  of the form 
\begin{equation}\label{hakkinen1}
\ii E [\opbw(\bral A_p\brar(U,\ldots,U;\xi))U+R(U,\ldots,U)U]\end{equation}
for a smoothing reminder $R$ in $\widetilde{R}^{-\rho}_p$, where the matrix 
$\bral A_p\brar(U,\ldots,U;\xi)$ is defined as in \eqref{matsim}. 
Then its Hamiltonian function has the form
\begin{equation}\label{hakkinen}
\int_{\TTT}\opbw(\bral{A_p(U,\ldots,U;\xi)}\brar)U\cdot\bar{U}dx
+\int_{\TTT}M(U)U\cdot\bar{U}dx,
\end{equation}
where $M(U)=M_p(U,\ldots,U)$ for a multilinear map $M_p$ in 
$\widetilde{\mathcal{M}}_p\otimes\mathcal{M}_2(\CCC)$.
We know by Proposition \ref{stimedifferent3} that there exists 
$\widetilde{R}$ smoothing remainder in 
$\widetilde{\mathcal{R}}_p$ such that the Hamiltonian vector field of 
\eqref{hakkinen} equals  (recall \eqref{X_H})
\[
\ii E\left(\opbw(\bral A_p\brar(U,\ldots,U))U
+\widetilde{R}(U,\ldots,U)U\right)
+\ii J\nabla_{\bar{U}}\left(\int_{\TTT}M(U)U\cdot\bar{U}dx\right)\,.
\]
By \eqref{hakkinen1} we must have 
$R=\ii E\widetilde{R}(U,\ldots,U)U
+\ii J\nabla_{\bar{U}}(\int_{\TTT}M(U)U\cdot\bar{U}dx)$. 
Since we are considering a matrix whose entries are resonant 
symbols (they are $\bral a_p\brar$, see Definition \ref{def:resonant}), 
we have $\widetilde{R}=\bral\widetilde{R}\brar$ 
and therefore $R^{\perp}:=R-\bral R\brar=\bral\ii J\nabla_{\bar{U}}(\int_{\TTT}M(U)U\cdot\bar{U}dx)\brar$ is an Hamiltonian vector field.
%Then for the above reasoning we may rewrite it as
%$$\opw(\bral a_p(U,\ldots,U;\xi)\brar)U+\tilde{R}(U,\ldots,U)U+R(U,\ldots,U)U.$$
%In the above formula we have that the two terms $\opw(\bral a_p(U,\ldots,U;\xi)\brar)U$ and $R+\tilde{R}$ are individually Hamiltonian vectorfields.
\end{remark}

More precisely we prove the following result.
\begin{theorem}{\bf (Poincar\'e-Birkhoff normal form).}\label{thm:mainBNF}
Assume that $\omega_{j}$ in \eqref{omegone} are 
non-resonant at order $N$ according to Definition \ref{nonresOmegaCOND}
and that \eqref{omegone0} holds.
There exist $s_0>0$, $r_0>0$ (possibly larger and smaller resp. with respect to  the ones in 
Theorem 
\ref{thm:main}) such that, 
for $s\geq s_0$, $r\leq r_0$ the following holds true.
There exist an invertible and symplectic map
\[
\mathfrak{B} : B_{r}(H^{s}(\mathbb{T};\mathbb{C}^{2}))\cap\mathcal{U}\to
H^{s}(\mathbb{T};\mathbb{C}^{2})\cap \mathcal{U}\,,
\]
such that, setting  (recall \eqref{nuovotutto})
\begin{equation*}
W:=\mathfrak{B}(Z)
\,,
\qquad
\mathcal{Y}_N(Z):=d\mathfrak{B}\big(\mathfrak{B}^{-1}(Z)\big)
\big[ \mathcal{Y}(\mathfrak{B}^{-1}(Z))\big]\,,
\end{equation*}
we have that
\begin{equation}\label{YYYN}
\left\{
\begin{aligned}
&\dot{W}=\mathcal{Y}_N(W):=\ii E \Omega W+
\ii E\opbw\big(\bral \mathfrak{M}^{(N)}\brar(W;\x)\big)[W]
%\opbw(A(U;x,\x))[U]
+\bral\mathcal{Q}_N\brar(W)[W]\\
&W(0)=\mathfrak{B}(Z_0)
\end{aligned}\right.
\end{equation}
where $\mathcal{Q}_N\in\Sigma\mathcal{R}^{-\rho}_2[r,N]
\otimes\mathcal{M}_2(\mathbb{C})$,
$\mathfrak{M}^{(N)}\in 
\Sigma\Gamma^{m}_2[r,N]\otimes\mathcal{M}_2(\mathbb{C})$
is independent of $x\in \mathbb{T}$ and it has the form
\begin{equation}\label{sarosaro}
\begin{aligned}
%&\mathcal{L}(Z):=
%\opbw\left( \begin{matrix} f_{m}(\x)  & 0 \\ 0 & f_{m}(\x)\end{matrix}\right)+
%\opbw\big(\mathfrak{M}(Z;\x)\big)\,, \quad
&\mathfrak{M}^{(N)}(Z;\x):=\left(
\begin{matrix}
\mathfrak{m}^{(N)}(W;\x) & 0\\
0 & {\mathfrak{m}^{(N)}(W;-\x) }
\end{matrix}
\right)\,,\\
&\mathfrak{m}^{(N)}(W;\x):=\mathfrak{m}^{(1)}_{m}(W)f_{m}(\x)
%(1+\mathfrak{m}_{m}(Z))f_{m}(\x)
+\mathfrak{m}^{(N)}_{m'}(W;\x)\,,
\qquad
\mathfrak{m}^{(1)}_{m}(W)\in\Sigma\mathcal{F}^{\mathbb{R}}_1[r,N]\,,
\\&
 \mathfrak{m}^{(N)}_{m'}(W;\x)\in \Sigma\Gamma^{m'}_1[r,N]\,,
% \qquad
  %\mathfrak{m}^{(N)}_{m'}(Z;\x)-\ov{ \mathfrak{m}^{(N)}_{m'}(Z;\x)}
  %\in \Sigma\Gamma^{0}_1[r,N]\,.
\end{aligned}
\end{equation}
with $\mathfrak{m}^{(N)}(W;\x)$ real valued. Moreover the vector field 
$\mathcal{Y}_{N}$ is Hamiltonian.
Finally, for any $s\geq s_0$, the maps $\mathfrak{B}^{\pm1}$ satisfy
\begin{equation}\label{stimamappaBNF}
\|\mathfrak{B}^{\pm1}(Z)\|_{H^{s}}\leq \|Z\|_{H^{s}}(1+C\|Z\|_{H^{s_0}})\,,
\end{equation}
for some constant $C>0$ depending  on $s$.
\end{theorem}

\begin{proof}%[{\bf Proof of Theorem \ref{thm:mainBNF}}]
It follows by applying the result of section \ref{quadtermBNF} and then, iteratively, 
the results of Lemmata \ref{lem:Pollinijth}, \ref{prop:Pollinisalsa}.
\end{proof}

%The proofs of this theorem is given in Section \ref{sec:BNF}.
The Theorem above is the key step to obtain a long time
existence and stability result for a system of the form \eqref{Nonlin1inizio}.
In particular the following result is consequence of Theorem \ref{thm:mainBNF}.

\begin{corollary}{\bf (A priori energy estimate).}\label{thm:energy}
Let $W=\vect{w}{\bar{w}}\in C^{0}([0,T); H^{s}(\mathbb{T};\mathbb{C}^{2}))\cap
C^{1}([0,T); H^{s-m}(\mathbb{T};\mathbb{C}^{2}))$
for $T>0$ be a solution of the system \eqref{YYYN}
with initial condition $W_0=\vect{w_0}{\ov{w_0}}\in B_r(H^{s})(\mathbb{T};\mathbb{C}^{2}))$,
$0<r\ll1$.
Then there exists $C=C(s)>0$ such that
\begin{equation}\label{cor:energy}
\|w(t)\|_{H^{s}}^{2}\leq \|w_0\|^{2}_{H^{s}}+C\int_{0}^{t}\|w(\s)\|_{H^{s}}^{N+2}d\s\,,
\qquad \forall t\in[0,T)\,.
\end{equation}
\end{corollary}
The proof of the corollary is given in Section \ref{sec:BNF}.

\subsection{Applications to some PDEs}
 In this section, by using Theorems
\ref{thm:main}, \ref{thm:main2} and \ref{thm:mainBNF}, we shall prove Theorems \ref{thm:mainNLS}, \ref{thm:mainBeam}, \ref{teototale}.
\subsubsection{The quasi-linear Sch\"odinger equation}\label{applica:NLS}
We give the proof of Theorem \ref{thm:mainNLS}.

\begin{proof}[{\bf Proof of Theorem \ref{thm:mainNLS}}]
In order to prove Theorem \ref{thm:mainNLS} we shall apply Theorems \ref{thm:main},
\ref{thm:main2}, \ref{thm:mainBNF} and Corollary \ref{thm:energy}.
Let us show that equation \eqref{NLS} satisfies 
 the assumptions of these abstract results.
  
 \noindent
 First of all we recall that the nonlinearity  $f(u,u_x,u_{xx})$
 is a polynomial of maximum degree $\bar{q}\geq 2$.
 Therefore, by Lemmata $3.2$, $3.3$ in \cite{Feola-Iandoli-Long}
 and Lemma $4.1$ in \cite{FIloc} we have
 \begin{equation}\label{paraNLS1}
 \begin{aligned}
 f(u,u_x,u_{xx})&=\opbw(g_2(U;x))\pa_{xx}u+\opbw(h_2(U;x))\pa_{xx}\bar{u}\\
 &+\opbw(g_1(U;x))\pa_{x}u+\opbw(h_1(U;x))\pa_{x}\bar{u}\\
 &+\opbw(g_0(U;x))u+\opbw(h_0(U;x))\bar{u}+R(U)[U]\,,
 \end{aligned}
 \end{equation}
 where $ R(U)=\sum_{j=1}^{\bar{q}}R_j(U)$ and $R_{j}$ 
 are $1\times2$ matrices of operators in $\widetilde{\mathcal{R}}^{-\rho}_j$,
 and 
 \[
 \begin{aligned}
 g_i(U;x)&=\sum_{j=1}^{\bar{q}}g_{i}^{(j)}(U;x)\,,
 \qquad
  h_i(U;x)&=\sum_{j=1}^{\bar{q}}h_{i}^{(j)}(U;x)\,,\qquad g_i^{(j)},h_i^{(j)}\in \widetilde{\mathcal{F}}_{j}\,,\;\;\;i=0,1,2\,.
 \end{aligned}
 \]
 Moreover
 \begin{equation*}
 \begin{aligned}
  g_i(U;x):=\big( \pa_{\pa_{x}^{i}u} f\big)(u,u_x,u_{xx})\,, \qquad
  h_i(U;x):=\big( \pa_{\ov{\pa_{x}^{i}u}} f\big)(u,u_x,u_{xx})\,.
 \end{aligned}
 \end{equation*}
 More precisely, using the \eqref{HamhypNLS}, we have
 \begin{equation}\label{paraNLS10}
 \begin{aligned}
 &g_2:=-\pa_{u_x \ov{u_{x}}}F\,,\qquad  h_2:=-\pa_{\ov{u_x}\, \ov{u_{x}}}F\,,\\
 &g_1:=-\frac{d}{dx}\big(\pa_{u_x \ov{u_x}}F\big)+\pa_{u_x \bar{u}}F-\pa_{u \ov{u_x}}F\,,
 \qquad
 h_1:=-\frac{d}{dx}\big(\pa_{\ov{u_x} \,\ov{u_x}}F\big)+\pa_{\ov{u_x}\, \bar{u}}F
 -\pa_{\bar{u}\, \ov{u_x}}F\,,\\
 &g_0=\pa_{u\bar{u}}F-\frac{d}{dx}\big( \pa_{u\ov{u_x}}F\big)\,,
 \qquad
 h_0=\pa_{\bar{u}\, \bar{u}}F-\frac{d}{dx}\big( \pa_{\bar{u}\,\ov{u_x}}F\big)\,.
 \end{aligned}
 \end{equation}
 Recall now that $\pa_{x}^{p}:=\opbw\big((\ii\x)^{p}\big)$, $p=0,1,2$.
 Then, using the composition Proposition \ref{teoremadicomposizione} 
 (see also \eqref{espansione2}) and the formul\ae\, \eqref{HamhypNLS}, \eqref{paraNLS10}, 
 we obtain
  \begin{equation}\label{paraNLS2}
 \begin{aligned}
 f(u,u_x,u_{xx})&=\opbw\Big(\widetilde{a_2}(U;x)(\ii \x)^{2}+
 \widetilde{a_1}(U;x)(\ii\x)+\widetilde{a_0}(U;x)\Big)u\\
 &+\opbw\Big(\widetilde{b_2}(U;x)(\ii \x)^{2}+\widetilde{b_1}(U;x)(\ii\x)
 +\widetilde{b_0}(U;x)\Big)\bar{u}+R(U)U\,,
 \end{aligned}
 \end{equation}
 where $R(U)$ is a $1\times2$ matrix of operators in $\Sigma\mathcal{R}^{-\rho}_{1}[r,N]$ (for any $N\geq1$),
 and 
\begin{equation}\label{paraNLS3}
\begin{aligned}
&\widetilde{a_{2}}:=-\pa_{u_{x}\ov{u_x}}F\,,
\qquad \widetilde{b_{2}}:=-\pa_{\ov{u_{x}}\,\ov{u_x}}F\,,\\
&\widetilde{a_1}:=\pa_{u_x \bar{u}}F-\pa_{u \ov{u_x}}F\,,
\qquad \widetilde{a_0}:=\pa_{u\bar{u}}F-\frac{1}{2}\pa_{x}\Big(\pa_{u_x \bar{u}}F
+\pa_{u\ov{u_x}}F\Big)\,,\\
&
\widetilde{b_1}:=h_1-\frac{1}{2}h_2\,,
\qquad \widetilde{b_0}= h_0+\frac{1}{2}\pa_{xx}h_2-\frac{1}{2}\pa_{x}h_1\,.
\end{aligned}
\end{equation}
Notice that $\widetilde{a}_i, \widetilde{b_i}\in \Sigma\mathcal{F}_{1}[r,N]$, $i=0,1,2$.
Moreover, by \eqref{potenziale1}, we can also note that
\begin{equation}\label{simboPotenz}
\mathtt{p}(\x):=\widehat{p}(\x)\,,\quad\x\in \mathbb{R} \qquad 
\mathtt{p}\in\widetilde{\Gamma}_0^{0}\,.
\end{equation}
Then we write
\begin{equation}\label{Omegonepseudo}
\Omega:=\opbw(f_{2}(\x))\,,\qquad f_{2}(\x):=\x^{2}+\mathtt{p}(\x)
\in \widetilde{\Gamma}_0^{2}\,.
\end{equation}
The symbol $f_2(\x)$ is real valued, even in $x$ and \emph{classical}, namely 
satisfies the properties of $f_{m}$ in Theorem \ref{thm:main}.
By \eqref{Omegonepseudo}, \eqref{paraNLS2}, we have that equation
 \eqref{NLShamver} reads
 \begin{equation}\label{NLShamver2}
 \begin{aligned}
 \dot{u}&=\ii \opbw\Big( (1-\widetilde{a}_2(U;x))f_2(\x)+
 \widetilde{a}_1(U;x)\ii\x +\widetilde{a}_0(U;\x)+\widetilde{a}_2(U;x)\mathtt{p}(\x)\Big)u\\
 &+\opbw\Big( -\widetilde{b}_2(U;x)f_2(\x)+
 \widetilde{b}_1(U;x)\ii\x +\widetilde{b}_0(U;\x)+\widetilde{b}_2(U;x)\mathtt{p}(\x)\Big)\bar{u}
 +R(U)U
 \end{aligned}
 \end{equation}
 where $R$ is a $1\times 2$ matrix of operators in $\Sigma\mathcal{R}^{-\rho}_{1}[r,N]$. 
 Then, setting
 \begin{equation}\label{simbolifinale}
 \begin{aligned}
 &a_2:=-\widetilde{a}_2\,,\quad  b_2:=-\widetilde{b}_2\,,  
% \quad a_1:=\widetilde{a}_1\,, \quad  b_1:=\widetilde{b}_1\,,\\&
\quad
a_0:=\widetilde{a}_0+\widetilde{a}_2\mathtt{p}(\x)\,,\quad
 b_0:=\widetilde{b}_0+\widetilde{b}_2\mathtt{p}(\x)\,,\\
 &a=a(U;x,\x):=(1+a_{2}(U;x))f_2(\x)+a_{1}(U;x,\x)\,,
 \quad a_1(U;x,\x):=\widetilde{a}_1(U;x)(\ii\x)+a_0(U;x,\x)
 \\&b(U;x,\x)=b_{2}(U;x)f_2(\x)+b_{1}(U;x,\x)\,,
 \quad b_1(U;x,\x):=\widetilde{b}_1(U;x)(\ii\x)+b_0(U;x,\x) \,,
 \end{aligned}
 \end{equation}
 we have that \eqref{NLShamver2} is equivalent to the system on the 
 variables $U=\vect{u}{\bar{u}}$
 \begin{equation}\label{NLShamver3}
 \dot{U}=\ii E\opbw\big(A(U;x,\x)\big)U+R(U)U\,,\quad 
 A(U;x,\x)=\left(
 \begin{matrix}
 a(U;x,\x) & b(U;x,\x)\vspace{0.2em}\\
 \ov{b(U;x,-\x)} & \ov{a(U;x,-\x)}
 \end{matrix}
 \right)
 \end{equation}
 where  $R(U)$ is some smoothing remainder in 
 $\Sigma\mathcal{R}^{-\rho}_1[r,N]\otimes\mathcal{M}_2(\CCC)$
 which is \emph{real-to-real} (see \eqref{vinello}). 
By \eqref{simbolifinale}, \eqref{paraNLS3} and using that $F(u,u_x)$
 is real valued, we  deduce that
 \begin{equation}\label{realtaAA}
 a(U;x,\x)=\ov{a(U;x,\x)}\,.
 \end{equation}
 Therefore system \eqref{NLShamver3} has the same form of 
 \eqref{Nonlin1inizio}, \eqref{Forma-di-Ainizio} and Theorem \ref{thm:main}
 applies. 
 Moreover the equation \eqref{NLShamver} (and hence \eqref{NLShamver3})
 is Hamiltonian with respect to the symplectic form \eqref{symform}.
 Hence  Theorem  \ref{thm:main2}
 guarantees that the map $\Psi$ given by  Theorem 
 \ref{thm:main} is symplectic. Then, setting $Z=\Psi(U)$, we have that
system 
 \eqref{NLShamver3} conjugates 
 to a system of the form \eqref{nuovotutto}, \eqref{Nonlin1fine}
 with $\mathcal{Y}(Z)$ an Hamiltonian vector field.
 The local well-posedness on the system \eqref{Nonlin1fine}
 can be deduced as in  Theorem \ref{flussononlin}.
 Furthermore the linear frequencies of oscillations in \eqref{OmegoneNLS}
 are \emph{non-resonant} according to Definition \ref{nonresOmegaCOND}
 for any choice of parameters $\vec{m}=(m_1,\ldots,m_{M})$ 
 (see \eqref{potenziale1}) in $[-1/2,1/2]^{M}\setminus\mathcal{N}$
where  
$\mathcal{N}$ has zero Lebesgue measure. This is a consequence
of Proposition $5.5$  in  \cite{Feola-Iandoli-Long}.
Theorem \ref{thm:mainBNF} applies and provide a symplectic map 
$\mathfrak{B}$ such that the system for the variables $W:=\mathfrak{B}(Z)$
has the form \eqref{YYYN}. Namely we conjugate, in a symplectic way 
the system \eqref{NLShamver3}
to its resonant \emph{Poincar\'e-Birkhoff normal form}
up to order $N$ (see \eqref{YYYN}).
We used the map $\mathfrak{B}\circ\Psi$. In particular, by estimates 
\eqref{curecure},\eqref{stimamappaBNF}, we deduce that
\begin{equation}\label{equivNormeUW}
\|W\|_{H^{s}}\sim \|U\|_{H^{s}}\,,\qquad W=\mathfrak{B}\circ\Psi(U)\,,
\end{equation}
if $r>0$ is small enough.
Let  $W$ be the solution of 
problem \eqref{YYYN} and assume it is defined on a time interval
$[0,T)$, $T>0$.  By the estimate \eqref{cor:energy} in Corollary \ref{thm:energy}
and using a standard bootstrap 
argument (see for instance the proof of Theorem $5.1$ in \cite{Feola-Iandoli-Long})
one can prove that actually
\[
\|W(t)\|_{H^{s}}\lesssim_s\|W(0)\|_{H^{s}}\,, t\in[0,T)\,, \qquad T\gtrsim r^{-N}\,.
\]
The latter estimate combined with \eqref{equivNormeUW} proves the estimate
\eqref{stimaNLS} over a time scale as in \eqref{spazioNLS}. 
This concludes the proof.
\end{proof}

\subsubsection{Quasi-linear perturbations of the beam equation}
Introducing the variable $v=\dot{\psi}=\pa_{t}\psi$ we can rewrite equation \eqref{beam1}
as
\begin{equation}\label{beam4}
\left\{\begin{aligned}
&\dot{\psi}=-v\,,\\
&\dot{v}=\Omega^{2}\psi+p(\psi)\,,\qquad \Omega:=\big(\pa_{xx}^{2}+m\big)^{\frac{1}{2}}\,.
\end{aligned}\right.
\end{equation}
Notice that 
the operator $\Omega$ is the Fourier multiplier defined as
\begin{equation}\label{OmegoneBeam}
\Omega e^{\ii jx}=\omega_{j} e^{\ii jx}\,,\quad \omega_{j}=\omega_{j}({m})
:=\sqrt{|j|^{4}+m}\,,\quad
j\in\mathbb{Z}\,.
\end{equation}
We define the complex variable
\begin{equation}\label{beam5}
u:=\frac{1}{\sqrt{2}}\big( \Omega^{\frac{1}{2}}\psi+\ii \Omega^{-\frac{1}{2}}v\big)\,.
\end{equation}
Therefore the \eqref{beam4} reads
\begin{equation}\label{beam6}
\dot{u}=\ii \Omega u+\frac{\ii}{\sqrt{2}}\Omega^{-\frac{1}{2}}
p\Big(\Omega^{-\frac{1}{2}}\Big(\frac{u+\bar{u}}{\sqrt{2}}\Big)\Big)\,.
\end{equation}
Notice that \eqref{beam6} has the form $\dot{u}=\ii \pa_{\bar{u}}H(u,\bar{u})$, i.e. 
is the Hamiltonian equation 
(w.r.t. the symplectic form \eqref{symform}) of the Hamiltonian 
\begin{equation}\label{beam7}
H(u,\bar{u})=\int_{\mathbb{T}} \Omega u\cdot\bar{u} dx
+\int_{\mathbb{T}} P\Big( \Omega^{-\frac{1}{2}}\Big(\frac{u+\bar{u}}{\sqrt{2}}\Big)\Big)dx
\end{equation}
where
\begin{equation}\label{beam8}
P(\psi):=G(\psi,\psi_{x},\psi_{xx})\,.
\end{equation}
We now prove the following.

\begin{lemma}{\bf (Paralinearization of the beam equation).}\label{paraBeam}
The equation \eqref{beam6} can be written in the form
$\dot{U}=X(U)$, $U=\vect{u}{\bar{u}}$ where $X(U)$
is an Hamiltonian vector filed of the form \eqref{Nonlin1inizio}, \eqref{Forma-di-Ainizio}
where $f_{m}\rightsquigarrow f_{2}(\x)$ defined as
\begin{equation}\label{FF2beam}
f_{2}(\x):=\sqrt{\x^{4}+m}\in \widetilde{\Gamma}^{2}_0\,.
\end{equation}
\end{lemma}

\begin{proof}
We reason as done in section \ref{applica:NLS}.
First of all consider the function $G(\psi,\psi_x,\psi_xx)$ appearing in \eqref{beam3}.
Since $G$ is a real valued 
polynomial in the variables $(\psi,\psi_x,\psi_{xx})$ it is easy to check 
that
\begin{equation}\label{paraBeam2}
c_{jk}(\psi;x):=\big(\pa_{\pa_{x}^{k}\psi\pa_{x}^{j}\psi}G\big)(\psi,\psi_x,\psi_{xx})
\in \Sigma\mathcal{F}^{\mathbb{R}}_1[r,N]\,,\quad N\geq 1\,,\;\;\; k,j=0,1,2\,.
\end{equation}
Then, using the Bony paralinearization formula (see also 
Lemmata $3.2$, $3.3$ in \cite{Feola-Iandoli-Long}
 and Lemma $4.1$ in \cite{FIloc}), we can deduce that
 (see \eqref{beam3})
 \begin{equation}\label{vascoblasco}
 \begin{aligned}
 &g(\psi,\psi_x,\psi_{xx},\psi_{xxx},\psi_{xxxx})=C(\psi)\psi+Q(\psi)\psi\,,\\
 &
 C(\psi):= \sum_{k,j=0}^{2}C_{kj}(\psi)\,,\qquad
 C_{kj}(\psi):=
(-1)^{k}\pa_{x}^{k} \opbw\big(c_{jk}(\psi;x)\big)\pa_{x}^{j}
 \end{aligned}
 \end{equation}
 for some $ Q(\psi)\in \Sigma\mathcal{R}^{-\rho}_1[r,N]$.
 We note that the operator $C(\psi)$ is self-adjoint.
 Indeed $C_{00}=\opbw(c_{00}(\psi;x))$ is self-adjoint since the symbol is real valued
 (recall \eqref{marieanto}). Reasoning similarly we have that
 the operators
 \[
 \begin{aligned}
 &C_{01}(\psi)+C_{10}(\psi)=\opbw\big(c_{01}(\psi;x)\big)\pa_{x}-\pa_{x}
 \opbw\big(c_{10}(\psi;x)\big)\,,\\
 &C_{02}(\psi)+C_{11}(\psi)+C_{20}(\psi)=
 \opbw\big(c_{02}(\psi;x)\big)\pa_{xx}-\pa_{x}
 \opbw\big(c_{11}(\psi;x)\big)\pa_{x}
 +\pa_{xx} \opbw\big(c_{20}(\psi;x)\big)\\
 &C_{12}(\psi)+C_{21}(\psi)=-\pa_{x} \opbw\big(c_{12}(\psi;x)\big)\pa_{xx}
 +\pa_{xx} \opbw\big(c_{21}(\psi;x)\big)\pa_{x}\\
 &C_{22}(\psi)=\pa_{xx} \opbw\big(c_{22}(\psi;x)\big)\pa_{xx}\,,
 \end{aligned}
 \]
 are self-adjoint since $c_{jk}=c_{kj}$.
 Using the Definition \ref{pomosimb}, \ref{omosmoothing} and the \eqref{beam5}, 
 one can check
 that
  \begin{equation}\label{vascoblasco6}
 \begin{aligned}
& \widetilde{c}_{jk}(U;x):=c_{jk}\Big(\Omega^{-\frac{1}{2}}\frac{u+\bar{u}}{\sqrt{2}} ;x\Big)\in 
 \Sigma\mathcal{F}^{\mathbb{R}}_1[r,N]\,,\quad 
 \widetilde{Q}(U):=Q\Big(\Omega^{-\frac{1}{2}}\frac{u+\bar{u}}{\sqrt{2}}\Big)\in 
 \Sigma\mathcal{R}^{-\rho}_1[r,N]\,.
 \end{aligned}
 \end{equation}
 We define 
 \begin{equation}\label{vascoblasco3}
  \begin{aligned}
  &B(U):=\sum_{j,k=0}^{2}B_{kj}(U)\,,
  \qquad B_{kj}(U):=\frac{(-1)^{k}}{{2}}\Omega^{-\frac{1}{2}}
  \pa_{x}^{k}\opbw\big( \widetilde{c}_{kj}(U;x)\big)\pa_{x}^{j}\Omega^{-\frac{1}{2}}\,,\\
  &\widehat{Q}(U):=\frac{1}{2}\Omega^{-\frac{1}{2}} \widetilde{Q}(U)\Omega^{-\frac{1}{2}} 
 \end{aligned}
 \end{equation}
 With this notation, recalling \eqref{beam2}, \eqref{vascoblasco},
 we have that equation \eqref{beam6} reads
 \begin{equation}\label{vascoblasco2}
 \dot{u}=\ii \Omega u+\ii B(U)u+\ii B(U)\bar{u}
 +\ii \widehat{Q} (U)[u]+\ii \widehat{Q} (U)[\bar{u}]\,.
 \end{equation}
In order to show that equation \eqref{vascoblasco2} can be written in the form \eqref{Nonlin1inizio}
we provide a more explicit description of the operator $B(U)$ at the highest order.
Let us write $\Omega:=\opbw(f_2(\x))$, $\pa_{x}^{p}=\opbw((\ii\x)^{p})$
where $f_{2}(\x)$ is in \eqref{FF2beam}.
Notice also that 
\[
(\x^{4}+m)^{-1}(\ii\x)^{4}=1-m(\x^{4}+m)^{-1}\,.
\]
Then, using \eqref{vascoblasco3} and the expansion \eqref{espansione2},
we have
\begin{equation}\label{vascoblasco7}
B(U):=\opbw\Big(a_{2}(U;x)f_{2}(\x)+a_{1}(U;x,\x)\Big)\,,\quad
a_2(U;x)=-\frac{1}{2}\widetilde{c}_{22}(U;x)
\end{equation}
up to smoothing remainders in $ \Sigma\mathcal{R}^{-\rho}_1[r,N]$
and where $a_{1}(U;x,\x)$ is a symbol in 
$ \Sigma\Gamma^{1}_1[r,N]$.
Moreover, since $B(U)$ is self-adjoint and $\widetilde{c}_{22}(U;x)$
is real-valued (see \eqref{vascoblasco6}),
we must have that $a_{1}(U;x,\x)$ is real-valued.
Using \eqref{vascoblasco7} we rewrite \eqref{vascoblasco2} as 
\[
\dot{U}=\ii E \opbw\left(
\begin{matrix}
(1+a_2)f_2(\x) & a_2 f_{2}(\x)\vspace{0.2em}\\
a_2 f_{2}(\x) & (1+a_2)f_2(\x)
\end{matrix}
\right)U+\ii E
\opbw\left(
\begin{matrix}
a_{1}(U;x,\x)& a_{1}(U;x,\x)\vspace{0.2em}\\
a_{1}(U;x,-\x)& a_{1}(U;x,-\x)
\end{matrix}
\right)U+R(U)U
\]
which has the form \eqref{Nonlin1inizio}.
\end{proof}

We now state a result regarding the \emph{non-resonance}  
of the linear frequencies of oscillations.
\begin{lemma}\label{nonresBEAM}
There exists a zero Lebesgue measure set $\mathcal{N}\subset[1,2]$
such that, for any $m\in[1,2]\setminus\mathcal{N}$, 
the frequencies $\omega_{j}=\omega_j({m})$ in \eqref{OmegoneBeam}
are \emph{non-resonant} according to Definition \ref{nonresOmegaCOND}.
\end{lemma}
\begin{proof}
It follows by Proposition $3.1$
 in \cite{EGK} reasoning as in the proof of Proposition $5.5$
 in \cite{Feola-Iandoli-Long}
 \end{proof}

Following almost word by word 
the proof of the long time existence Theorem \ref{thm:mainNLS}
using Lemmata \ref{paraBeam}, \ref{nonresBEAM}
one can deduce the proof of Theorem \ref{thm:mainBeam}.

\subsubsection{Benjamin-Ono type equations}
We prove the following.
\begin{lemma}{\bf (Paralinearization of the Benjamin-Ono equation)}
\label{lem:paraBen}
Let $0<r\ll1$ and $\rho>0$. 
Then there exist a remainder $R\in \mathcal{R}^{-\rho}_1[r]$ and a
symbol $a(u;x,\x)$ in $\Gamma^{2}_1[r]$ of the form
\begin{equation}\label{simbobenono}
\begin{aligned}
&a(u;x,\x)=a_2(u;x)|\x|\x+a_1(u;x,\x)\,,\quad a_2(u;x):=\big(\pa_{z_4}g\big)
(u,\mathcal{H}u, u_x,\mathcal{H}u_x, \mathcal{H}u_{xx})\in \mathcal{F}^{\mathbb{R}}_1[r]\,,\\
&a_1(u;x,\x)\in \Gamma^{1}_1[r]\,,\qquad 
a_1(u;x,\x)-\ov{a_1(u;x,\x)}\in \Gamma^{0}_1[r]\,,
\end{aligned}
\end{equation}
such that the following holds. The equation \eqref{benono} can be written as
\begin{equation}\label{benonopara}
u_{t}=-\ii \opbw\Big((1+a_2(u;x))|\x|\x+a_1(u;x,\x)\Big)u+R(u)[u]\,.
\end{equation}
\end{lemma}
\begin{proof}
Since $g(z_0,z_1,z_2,z_3,z_4)$ 
is a polynomial in the variables
$(z_0,z_1,z_2,z_3,z_4)$, it is easy to check (recall Def. \ref{pomosimb})
that
\begin{equation}\label{duchessa1}
\begin{aligned}
b_{j}(u;x):=\big(\pa_{z_j}g\big)(u,\mathcal{H}u, u_x,\mathcal{H}u_x, \mathcal{H}u_{xx})
\in \mathcal{F}^{\mathbb{R}}_1[r]\,.
\end{aligned}
\end{equation}
Then, using the Bony paralinearization formula (see also 
Lemmata $3.2$, $3.3$ in \cite{Feola-Iandoli-Long}
 and Lemma $4.1$ in \cite{FIloc}), we can deduce that
 (see \eqref{nonlineBenono})
\begin{equation*}
\begin{aligned}
\mathcal{N}(u)&=B(u)u+Q(u)u\,,\qquad Q\in \mathcal{R}^{-\rho}_1[r]\,,\\
B(u)&=\opbw(b_4(u;x))\mathcal{H}\pa_{xx}+
\opbw(b_3(u;x))\mathcal{H}\pa_{x}+
\opbw(b_2(u;x))\pa_{x}\\
&+\opbw(b_1(u;x))\mathcal{H}+\opbw(b_0(u;x))\,.
\end{aligned}
\end{equation*}
Recalling \eqref{benono2} we write $\mathcal{H}=\opbw(-\ii\sign(\x))$.
Using the expansion \eqref{espansione2} we get,
up to smoothing remainders in $\mathcal{R}^{-\rho}_1[r]$
\[
\begin{aligned}
B(u)&=\opbw\Big(b_{4}\ii |\x|\x+\frac{1}{2\ii}\{b_4,\ii|\x|\x\}+b_{3}|\x|
+b_2\ii\x+
\widetilde{b}(u;x,\x)\Big)
\end{aligned}
\]
for some symbol $\widetilde{b}\in \Gamma^{0}_1[r]$.
We also have 
\[
\frac{1}{2\ii}\{b_4(u;x),\ii|\x|\x\}+b_{3}(u;x)|\x|=
-\pa_{x}\Big(b_4(u;x)\Big)|\x|+b_3(u;x)|\x|
\stackrel{\eqref{duchessa1}, \eqref{benonoassump}}{=}0\,.
\]
We define $a_2(u;x):=b_4(u;x)$ and 
\begin{equation}\label{duchessa2}
a_1(u;x,\x):=(u+b_2(u;x))\x+u_x+\widetilde{b}(u;x,\x)\,.
\end{equation}
Therefore the equation \eqref{benono}
assumes the form
\[
u_t=-\mathcal{H}\pa_{xx}u-\ii \opbw(a_2(u;x)|\x|\x+a_1(u;x,\x))+R(u)u\,,
\]
for some $R\in \mathcal{R}^{-\rho}_1[r]$.
Since 
$-\mathcal{H}\pa_{xx}=-\ii\x|\x|$ we get the
\eqref{benonopara}. The symbol $a_1$ in \eqref{duchessa2}
satisfies the property in  \eqref{simbobenono} by explicit computation.
\end{proof}
We are in position to prove Theorem \ref{teototale}.
\begin{proof}[{\bf Proof of Theorem \ref{teototale}}]
By Lemma \ref{lem:paraBen} we have that equation \eqref{benono} has the form
\eqref{posacenere} (see \eqref{benonopara})
for some symbol $a(u;x,\x)$ satisfying the properties in \eqref{Forma-di-Ainizio}.
Therefore the Corollary \ref{thm:mainscalar} of Theorem \ref{thm:main} applies.
Then there is a map $\Psi$ such that the equation for the variable $z=\Psi(u)$
has the form \eqref{posacenere2}. The existence result for such an equation
over a time interval $[0,T)$, 
can be deduce by Theorem \ref{flussononlin} in the case of generator 
as in \eqref{sim4}.
Moreover, reasoning as in the proof of Corollary \ref{thm:energy}
one can prove an a priori energy estimate 
for the equation \eqref{posacenere2} of the form \eqref{cor:energy}
with $N=1$.
Then, reasoning as in the proof of Theorem \ref{thm:mainNLS},
one can prove that estimate \eqref{stimaBenono} holds
over a time interval $[0,T)$ with $T\gtrsim r^{-1}$.
\end{proof}

\section{Some non linear para-differential equations
%Preliminary result for Egorov Theory
}\label{sec:PreEgorov}
This section is the core of our paper.
We study four non linear problems arising in 
implementing an iterative procedure which diagonalizes and conjugates 
to constant coefficients a system like \eqref{Nonlin1inizio}.
These problems involve symbols transported along the flows of 
some para-differential equations and can be considered as a non linear 
counterpart 
of problems arising the Egorov Theory for \emph{pseudo-differential} operators.

%In this section we provide some preliminary results which will be used 
%in sections \ref{sec:egoHigh} and \ref{sec:BlockDiago}
%where we reduce to constant coefficients (at the highest order)
%a para-differential system.
%In subsection \ref{sec:egohigh}
%we dealt with terms on the diagonal. We shall prove 
%a fundamental step to implement an \emph{Egorov Theory}.
%In section \ref{sec:egohighoff}
%we shall deal with of diagonal terms.

\subsection{Diagonal terms at highest order}\label{sec:egohigh}
The equation that we study in this subsection is the one
appearing in conjugating to constant coefficient 
the principal symbol on the diagonal of the system \eqref{Nonlin1inizio}.
The equation we 
need to solve, and that will be used in section \ref{sec:egoHigh},
is the \eqref{equa814}. This equation 
involves a non linear auxiliary flow which is the 
solution of the system \eqref{Ham111prova}. Such a system is  a highly non-linear system of coupled equations.
 Given a function $b(\tau,u;x)$, we study the well-posedness 
of such flow  in Section \ref{wellwell52} (see Theorem \ref{Buonaposit}).
In section \ref{expexpFlow} (see Theorem \ref{formuleesplici})
we provide a more explicit expression  of such flow
in the case that the function $b(\tau,u;x)$
is the one associated to a torus diffeomorphism $x\to x+\beta(u;x)$
(see equations \eqref{def:betanNN}, \eqref{eq:b}).
This is necessary in order to solve \eqref{equa814} which depends on the flow generated by $\beta$. A similar problem is faced also in the paper \cite{BD}. In such paper the authors do not look for an invertible change of coordinates of the phase space, as a consequence the equation they need to solve is "essentially linear". We explain  in Remark \ref{paragoneBD} the link between our and their solution, another comparison between the two methods
is made in Remark \ref{rmk:Rambo}.\\
Finally, in section \ref{constEGOsec}, we look for a solution 
of equation \eqref{equa814}. Being \eqref{equa814} a non-linear equation we use an iterative scheme (see system \eqref{Ham111provan1}-\eqref{Ham111provan3}, which are well-posed thanks to Section \ref{wellwell52}) which converges to a solution of \eqref{equa814}. In the proof of the convergence of the aforementioned iterative scheme one can note  the $n^{th}$  approximate solution of the non-linear system \eqref{def:gamman}, \eqref{def:Mn} are "close" (up h.o.t. in degree of homogeneity) to the "linear" solutions found in \cite{BD}, see equations (5.1.8) and (5.1.9) therein.
In analogy with the papers  \cite{BD}, \cite{FIloc}, \cite{Feola-Iandoli-Long}
this is the most delicate 
part of our analysis.

\smallskip
\noindent Before entering in the core of the section we fix the following notation.
\begin{notation}\label{notazioni}
Throughout this section (and also the throughout the next ones) we fix a number $1<m$ in $\frac12\mathbb{N}$ and we work with symbols in $\Sigma\Gamma^{m'}_{p}[r,N]$ with $m'\leq m$. We fix $d=m$, where $d$ is the number appearing in item $(ii)$ of Definition \ref{pomosimb}, in other words, in the notation of item $(ii)$ of Definition \ref{pomosimb} we shall work with symbols in the classes $\Sigma\Gamma^{m',m}_{p}[r,N]$ with $m'\leq m$. The role of this number is discussed in Remark \ref{omissione}.
\end{notation}
\noindent Consider a symbol $a(u;x,\z)$ in the class $\Sigma\Gamma_1^{m}[r,N]$, $m\in \mathbb{R}$, 
$0<r\ll1$ and assume that it is \emph{classical}
according to Definition \ref{classicSimbo}.
Assume also that its principal part $a_m$ has the following structure
\begin{equation}\label{strutturaSimb}
a_m(z_0,x_0;\xi_0):=(1+\tilde{a}_m(z_0,x_0))f_0(\xi_0)\,, \quad 
\tilde{a}_m(z_0,x_0)\in \Sigma\mathcal{F}_1^{\mathbb{R}}[r,N]\,,
\end{equation}
where $f_0$ is a $m$-homogeneous ${\C^{\infty}(\mathbb{R}^+,\mathbb{R})}$   function   
and $z_0$ is in $H^s$ for $s$ large enough.
Let $b(\tau,w,y)\in \Sigma\mathcal{F}_1^{\mathbb{R}}[r,N]$, $\tau\in[0,1]$, 
and consider the system
 \begin{equation}\label{Ham111prova}
\left\{
\begin{aligned}
&\pa_{\tau}x(\tau)=-b(\tau,z(\tau);x(\tau)) \\
& \pa_{\tau}\x(\tau)=b_x(\tau,z(\tau);x(\tau))\x(\tau)\\
&\pa_{\tau}z(\tau)=\opbw\big(\ii b(\tau,z(\tau);x(\tau))\x(\tau)\big)[z(\tau)]\,,
\end{aligned}\right.
\end{equation}
with initial condition $(z(0),x(0),\xi(0))=(z_0,x_0,\x_0)$.
We have the following.

\begin{theorem}\label{constEgo}
Assume \eqref{strutturaSimb}.
For $r>0$ small enough there exists a symbol $b(\tau,w;y)\in\Sigma\mathcal{F}_1^{\mathbb{R}}[r,N]$, 
$\tau\in[0,1]$,  such that, the following holds.

\noindent
(i)
The flow 
 \begin{equation}\label{soluzione}
 (z(\tau), x(\tau),\x(\tau))={\bf \Phi}_{b}(\tau, z_0,x_0,\x_0)
 =(\Phi_{b}^{(z)}(\tau),\Phi_{b}^{(x)}(\tau),\Phi_{b}^{(\x)}(\tau))(z_0,x_0,\x_0)
 \end{equation}
of \eqref{Ham111prova} is well-posed and 
\begin{align}
z(\tau)&=\Phi_{b}^{(z)}(\tau,z_0)\in \cap_{k=0}^{K}C^{k}([0,1];H^{s-k})\,,\quad 0\leq K\leq s\,,
\label{nota1}\\
x(\tau)&=\Phi_{b}^{(x)}(\tau,z_0,x_0)=x_0+\Psi_{b}^{(x)}(\tau,z_0,x_0)\,,\qquad 
\Psi_{b}^{(x)}\in \Sigma\mathcal{F}^{\mathbb{R}}_{1}[r,N]\label{nota2}\\
\x(\tau)&=\Phi_{b}^{(\x)}(\tau,z_0,x_0,\x_0)=\x_0(1+\Psi_{b}^{(\x)}(\tau,z_0,x_0))\,, 
\qquad  \Psi_{b}^{(\x)}\in\Sigma \mathcal{F}^{\mathbb{R}}_{1}[r,N]\,.\label{nota3}
\end{align}
It is invertible, we denote by  
$(\widetilde{\Phi}_{b}^{(z)}(\tau,\tilde{z}_0), \widetilde{\Phi}_{b}^{(x)}(\tau,\tilde{z}_0,\tilde{x}_0),
\widetilde{\Phi}_{b}^{(z)}(\tau,\tilde{z}_0,\tilde{x}_0,\tilde{\x}_0))$
%\[
%\widetilde{\Phi}_{b}^{(z)}(\tau,\tilde{z}_0,\tilde{x}_0,\tilde{\x}_0)=
%(1+\widetilde{\Psi}_{b}^{(z)}(\tau,\tilde{z}_0,\tilde{x}_0))\tilde{\x}_0
%\]
its  inverse where
\begin{equation*}
\begin{aligned}
\tilde{z}_0=\Phi_{b}^{(z)}(1,z_0)\,, \,\,\,
\tilde{x}_0=\Phi_{b}^{(x)}(1,z_0,x_0)\,,\,\,\,
\tilde{\x}_0=\Phi_{b}^{(\x)}(1,z_0,x_0,\x_0)\,. 
\end{aligned}
\end{equation*}
We have that
\begin{equation}\label{losperobene}
\widetilde{\Phi}_{b}^{(\xi)}(\tau,\tilde{z}_0,\tilde{x}_0,\tilde{\x}_0)=
(1+\widetilde{\Psi}_{b}^{(\xi)}(\tau,\tilde{z}_0,\tilde{x}_0))\tilde{\x}_0.
\end{equation}

\noindent
(ii) There is  $m_b$ in the class 
$\Sigma\mathcal{F}^{\mathbb{R}}_{0}[r,N]$ \emph{independent} of $x\in \mathbb{T}$
such that
we have
\begin{equation}\label{equa814}
F(b):=(1+\tilde{a}_{m}(\widetilde{\Phi}_{b}^{(z)}(1,\tilde{z}_0), 
\widetilde{\Phi}_{b}^{(x)}(1,\tilde{z}_0,\tilde{x}_0)))
\big(1+\widetilde{\Psi}^{(\x)}_{b}(1,\tilde{z}_0,\tilde{x}_0)\big)^{d}
=m_{b}\,.
\end{equation}

%\noindent
%(iii) the following estimates hold true:

\end{theorem}

The proof of Theorem \ref{constEgo}
involves many different arguments 
that we shall study in the following subsections.

\subsubsection{Well-posedness of the flow (\ref{Ham111prova})}\label{wellwell52}
In this subsection we study the existence of the flow of \eqref{Ham111prova}
for \emph{any} generator $b(\tau;w,y)$ in the class $\mathcal{F}^{\mathbb{R}}_{1}[r]$.

\begin{theorem}{\bf (WP of \eqref{Ham111prova}).}\label{Buonaposit}
Consider the problem \eqref{Ham111prova} with $b$ in the class $\mathcal{F}^{\RRR}_{1}[r]$ for some small enough $r>0$. Then there exists $0<\tilde{r}\ll r $ such that if $z_0$ is in $B_{\tilde{r}}(H^s)$ the following holds true. There exists a solution of the problem \eqref{Ham111prova} 
with initial condition $z(0)=z_0$, $x(0)=x_0$, $\xi(0)=\xi_0$ 
of the form \eqref{nota1}-\eqref{nota3} 
for $\tau\in [0,1]$. 
In particular (recall \eqref{seminormagamma}, \eqref{semi-norma-totale}) one has
\begin{align}
& \sup_{\tau\in [0,1]}|\Psi_{b}^{(x)}(\tau)|^{\mathcal{F},s-1}_{\alpha,k}
\leq C|b|^{\mathcal{F}}_s\,; \label{ind:1TOT}\\
&  \sup_{\tau\in [0,1]}|\Psi^{(\xi)}_{b}(\tau)|^{\mathcal{F},s-2}_{\alpha,k}\leq C|b|^{\mathcal{F}}_{s-1}\,;\label{ind:2TOT}\\
& z (\tau)\in C^0(I,H^s)\cap C^1(I,H^{s-1}), \,\,\, \sup_{\tau\in [0,1]} \|z(\tau)\|_{s}\leq C\|z_0\|_{s}\label{ind:3TOT}\,,
%\\
%& b(\tau,z_{n-1}(\tau),x_{n-1}(\tau))\,\, \mbox{satisfies \eqref{maremma2} for any } \alpha \mbox{ and } k \mbox{ such that } \alpha+k\leq s-s_0.\label{ind:4}
\end{align}
\begin{equation}\label{pioggia4teorema}
\|(d_{u}^{k}z)(z_0)[h_1,\ldots,h_{k}]\|_{H^{s-mk}}\leq 
C 
\|h_1\|_{H^{s}}\cdots 
\|h_{k}\|_{H^{s}}\,,\qquad 
\forall \, 0\leq \tau\leq 1\,\quad \forall h_{i}\in H^{s}\,\;\; i=1,\ldots,k\,.
\end{equation}
\end{theorem}

\noindent
The proof of the theorem above 
is based on the following iterative scheme.
Define   for any $n\geq 1$
the following system of equations
\begin{align}
&\pa_{\tau}x_n=-b(\tau,z_{n-1};x_{n}) \label{eq:Xn}\\
&\pa_{\tau}\x_{n}=(\pa_{x}b)(\tau,z_{n-1};x_{n})\x_{n}\label{eq:XXn}\\
&\pa_{\tau}z_{n}=\opbw\big(\ii b(\tau,z_{n-1};x_{n-1})\x_{n-1}\big)[z_{n}]\label{eq:Zn}\,,\\
&\left(
x_n(0), \x_n(0) , z_n(0)
\right)=\left(
x_0 , \x_0 , z_0\right)\,.\nonumber
\end{align}
We shall prove that the sequence of solutions of the problem above converges to a solution of the system \eqref{Ham111prova}.
To start, in the following lemma, we shall prove that, if $z_0$ is small enough in $H^s$ for some $s\gg1$, for any $n$ in $\N^*$ there are functions $\Psi_n^{(x)}$, $\Psi_n^{(\xi)}$ in $\mathcal{F}^{\R}_{1}[r]$ such that the solutions of  \eqref{eq:Xn} and \eqref{eq:XXn} are of the form
\begin{equation}\label{ansatzXn}
x_{n}(\tau)=x_0+\Psi_n^{(x)}(\tau,z_0,x_0)\,, \quad \Psi_n^{(x)}\in \mathcal{F}^{\mathbb{R}}_{1}[r],
\end{equation}
\begin{equation}\label{ansatzXXn}
\x_n(\tau)
=\x_0\Big(1+\Psi_n^{(\x)}(\tau,z_0,x_0)\Big)\,, \quad \Psi_n^{(\x)}\in \mathcal{F}^{\mathbb{R}}_{1}[r]\,.
\end{equation}
Moreover we shall prove that $\{\Psi_n^{(x)}\}_{n\in\N^*}$, $\{\Psi_n^{(\xi)}\}_{n\in\N^*}$ are converging  sequences in the  space $\mathcal{F}^{\R}_{1}[r]$. 
Furthermore we prove that for any $n$ in $\N^*$ the equation \eqref{eq:Zn} admits a solution $z_n$ such that $\{z_n\}_{n\in\N^*}$ is a Cauchy  sequence in $H^{s-1}$ and it is bounded in $H^s$.

\begin{lemma}{\bf (Iterative Lemma).}\label{iterative}
Consider $b$ a  function in $\mathcal{F}^{\R}_{1}[r]$. There exist $\tilde{r}>0$ and $s>0$ respectively small and  big enough,  such that if $z_0\in B_{\tilde{r}}(H^s)$ the following holds true. For any $n\in\N^*$ there exist a unique solution $(x_n(\tau),\xi_n(\tau),z_n(\tau))$ of the system made of equations \eqref{eq:Xn}, \eqref{eq:XXn}, \eqref{eq:Zn} with initial condition $(x_n(0)=x_0,\xi_n(0)=\xi_0,z_n(0)=z_0)$ satisfying  the following properties.
\begin{description}
\item [$(S1)_{n}$] $x_n(\tau)$ and $\xi_n(\tau)$ have the form \eqref{ansatzXn} and \eqref{ansatzXXn} respectively.
Moreover, recalling the notation \eqref{semi-norma-totale}, there exists a constant $C>0$ depending on $\tilde{r}$  and $s$ (independent on $n$) such that  
\begin{align}
& \sup_{\tau\in [0,1]}|\Psi^{(x)}_n(\tau)|^{\mathcal{F},s}_{\alpha,k}\leq C|b|^{\mathcal{F}}_s; \label{ind:1}\\
&  \sup_{\tau\in [0,1]}|\Psi^{(\xi)}_n(\tau)|^{\mathcal{F},s-1}_{\alpha,k}\leq C|b|^{\mathcal{F}}_{s};\label{ind:2}\\
& z_{n} (\tau)\in C^0(I,H^s)\cap C^1(I,H^{s-1}), \,\,\, \sup_{\tau\in I} \|z_n(\tau)\|_{s}\leq C\|z_0\|_{s}\label{ind:3};\\
& b(\tau,z_{n-1}(\tau),x_{n-1}(\tau))\,\, \mbox{satisfies \eqref{maremma2} for any } \alpha \mbox{ and } k \mbox{ such that } \alpha+mk\leq s-s_0.\label{ind:4}
\end{align}
\item[$(S2)_{n}$] We have the following estimates, recall \eqref{semi-norma-totale2},
\begin{align}
&|\Psi_{n}^{(x)}-\Psi^{(x)}_{n-1}|^{\mathcal{F}}_{s-1,0} \leq 2^{-n} ;\label{ind_2:1}\\
&|\Psi_{n}^{(\xi)}-\Psi_{n}^{(\xi)}|_{s-2, 0}^{\mathcal{F}} \leq 2^{-n} ;\label{ind_2:2}\\
& \|z_{n}-z_{n-1}\|_{H^{s-1}}\leq 2^{-n} \tilde{r}\label{ind_2:3}.
\end{align}
\end{description}
\end{lemma}
\begin{proof}
We proceed by induction over $n$ in $\N^*$. In the case that $n=1$ we have
\begin{equation}\label{x1xi1}
\begin{aligned}
x_1(\tau)=x_0-\int_0^{\tau}b(s,z_0;x_0)ds, \quad
\xi_1(\tau)=\xi_0\Big(1+\int_0^{\tau}(\partial_xb)(s,z_0;x_0)ds\Big).
\end{aligned}
\end{equation}
Concerning the solution of \eqref{eq:Zn}, in the case that $n=1$, one has to reason as done in Lemma 3.22 in \cite{BD}, obtaining a solution $z_1(\tau)$ 
belonging to $C^0([0,1],H^s)\cap C^1([0,1],H^{s-1})$ and such that 
\begin{equation*}
\sup_{\tau\in I}\|z_1(\tau)\|_{H^s}\leq C_{s,r}\|{z_0}\|_{H^s},
\end{equation*}
therefore if $z_0$ is small enough in $H^s$ the statement $(S1)_{1}$ is proved by setting
\begin{equation*}
\begin{aligned}
\Psi_1^{(x)}(\tau,z_0;x_0)&=-\int_0^{\tau}b(s,z_0;x_0)ds,\\
\Psi_1^{(\xi)}(\tau,z_0;x_0)&=\int_0^{\tau}(\partial_x b)(s,z_0;x_0)ds.
\end{aligned}\end{equation*} 
We recall that the \eqref{ind:4} is trivial and follows immediately from the fact that  
function $b$ belongs to $\mathcal{F}_{1}^{\mathbb{R}}[r]$.
We now prove $(S2)_1$. The function $v(\tau):=z_1(\tau)-z_0$ solves the problem
\begin{equation*}
\partial_{\tau}v(\tau)=\opbw(\ii b(\tau;z_0,x_0)\xi_0)v(\tau)+\opbw(\ii b(\tau;z_0,x_0)\xi_0)z_0,
\end{equation*}
therefore, by using the Duhamel formulation and Lemma 3.22 in \cite{BD} we obtain
\begin{equation*}
\|v(t)\|_{s-1}\leq C\|z_0\|_{s_0}\|z_0\|_{s},
\end{equation*}
therefore it is enough to choose $\|z_0\|_{s_0}$ small enough to satisfy the third condition in $(S2)_1$. We obtain the first line in $(S2)_1$  by using equation \eqref{x1xi1}. The reasoning for the second line in $(S2)_1$ is similar. 

We suppose that $(S1)_{n}$, $(S2)_{n}$ hold true and we prove $(S1)_{n+1}$, $(S2)_{n+1}$. We start by showing the \eqref{ind:1} of $(S1)_{n+1}$, we have that 
\begin{equation*}
\begin{aligned}
x_{n+1}(\tau)-x_0&= -\int_{0}^{\tau}b(\sigma,z_n(\sigma),x_n(\sigma))d\s\\
&=-\int_{0}^{\tau}b\big(\sigma,z_n(\sigma),x_0+\Psi^{(x)}_n(\tau,z_0,x_0)\big)d\s
:=\Psi_{n+1}^{(x)}(\tau,z_0,x_0).
\end{aligned}
\end{equation*}
We want to bound the semi-norm 
$|\Psi_{n+1}^{(x)}(\tau,z_0,x_0)|^{\mathcal{F}}_{\alpha,k}$ 
for any $\alpha$ and $k$ satisfying 
$\alpha+mk\leq s-s_0$.
We shall do the computation, for simplicity, in the case $\alpha=0$. 
If $\alpha>0$ the reasoning is similar 
but the computation is much more tedious. 
Let $K\leq k$ where $k$ is as in \eqref{ind:4}, we have
\begin{equation*}
D_{z}^K\Psi_{n+1}^{(x)}(\tau,z,x)=-\int_0^{\tau}D_z^{K}
b(\sigma,z_n;x+\Psi_{n}(z,x))[h_1,\ldots,h_K]d\s\,.
\end{equation*}
We expand the term inside the integral in the equation 
above by using the formula for the derivatives of the composition of functions, obtaining
\begin{equation}\label{faadibruno}
\begin{aligned}
\sum_{k=1}^K&\sum_{k_1+k_2=K}\sum_{\nu_1=1}^{k_1}\sum_{\nu_2=1}^{k_2}\sum_{p_1+\ldots+p_{\nu_1}=k_1}\sum_{q_1+\ldots+q_{\nu_2}=k_2}C_{q_{1},\ldots,q_{\nu_2}}^{p_1,\ldots,p_{\nu_1}}
\prod_{j=1}^{\nu_1}\Big(D_{z_0}^{p_j}(x_0+\Psi^{(x)}_n(x_0,z_0))\left[h_{p_j,1},\ldots,h_{p_j,p_j}\right]\Big)\\
&\partial_y^{k_1}D_{z_n}^{k_2}b(z_n,x_0+\Psi_n^{(x)}(x_0,z_0))
\left[D_{z_0}^{q_1}z_n\left[h_{q_1,1},\ldots,h_{q_1,q_1}\right],\ldots,D_{z_0}^{q_{\nu_2}}z_n\left[h_{q_{\nu_2},1},\ldots,h_{q_{\nu_2},q_{\nu_2}}\right]\right]\,,
\end{aligned}\end{equation}
where we denoted 
$y=x_0+\Psi_n^{(x)}(x_0,z_0)$ and by 
$C_{q_{1},\ldots,q_{\nu_2}}^{p_1,\ldots,p_{\nu_1}}$ 
some combinatorial coefficients. 
We estimate the absolute value of the general term in the sum above. 
The first factor may be bounded from above by
\begin{equation}\label{fattore:1}
\prod_{j=1}^{\nu_1}
\Big|D_{z}^{p_j}(x_0+\Psi_n^{(x)}(x_0,z_0))[h_{p_j,1},\ldots h_{p_j,p_j}]
\Big|\leq
C(1+|\Psi_n^{(x)}|_{p_j,0}^{\mathcal{F},s})
\norm{z_0}_{s_0}^{\max\{0,1-p_j\}}
\prod_{j=1}^{\nu_1}\norm{h_j}_{s_0}\,,
\end{equation}
here we have used just the definition of symbol 
(more precisely the definition of a function independent on $\xi$) 
and of semi-norm, i.e.  \eqref{maremma2} and \eqref{seminormagamma}. 
For the second factor we have
\begin{equation}\label{fattore:2}
\begin{aligned}
\Big|\partial_y^{k_1}&D_{z_n}^{k_2}b(z_n;x_0+\Psi_n^{(x)}(x_0,z_0))
\Big[D_{z_0}^{q_1}z_n[h_{q_1,1},\ldots, h_{q_1,q_1}],\ldots, 
D_{z_0}^{q_{\nu_2}}z_n[h_{q_{\nu_2},1},\ldots, 
h_{q_{\nu_2},q_{\nu_2}}]\Big]\Big|\leq\\
&|b|^{\mathcal{F},s}_{k_1,k_2}\Big\{\max\{0,1-k_2\}\norm{z_n}_{s_0}^{\max\{0,p-k_2-1\}}
\norm{z_n}_{s_0+k_1}\prod_{j=1}^{k_2}
\norm{D_{z_0}^{q_j}z_n[h_{q_{j},1},\ldots,h_{q_j,q_j}]}_{s_0}\\
+&\norm{z_n}_{s_0}^{\max\{0,1-k_2\}}\sum_{i=1}^{k_2}
\prod_{j=1,j\neq i}^{k_2}\norm{D_{z_0}^{q_j}z_n[h_{q_{j},1},\ldots, 
h_{q_{j},q_{j}}]}_{s_0}\norm{D_{z_0}^{q_i}z_n[h_{q_{i},1},
\ldots, h_{q_{i},q_{i}}]}_{s_0+k_1}\Big\}\,.
\end{aligned}
\end{equation}
Thanks to the inductive hypothesis \eqref{ind:4} we can apply 
the Lemma \ref{flusso-differenziale}, therefore we can 
estimate from above the r.h.s. of \eqref{fattore:2} by 
\begin{equation}\label{fattore:3}
\begin{aligned}
C|b|_{k_1,k_2}^{\mathcal{F},s}&\Big\{\max\{0,1-k_2\}
\norm{z_0}_{s_0}^{\max\{0,p-k_2-1\}}\norm{z_0}_{s_0+k_1}
\prod_{j=1}^{k_2}\prod_{i=1}^{j}\norm{h_{j,i}}_{s_0}\\
&+\norm{z_0}_{s_0}^{\max\{0,1-k_2\}}\prod_{j=1}^{k_2}
\prod_{i=1}^{j}\norm{h_{j,i}}_{s_0+k_1}\Big\}\,.
\end{aligned}
\end{equation}
Putting together \eqref{fattore:1}, \eqref{fattore:2} 
and \eqref{fattore:3} we obtain the bound 
\eqref{maremma2} for the general term of the sum in 
\eqref{faadibruno} up to renaming, with abuse of notation, 
$s_0\rightsquigarrow s_0+k_1$ with the constant $C$ in \eqref{maremma2} replaced by 
\begin{equation*}
C\prod_{j=1}^{\nu_1}\big(1+|\Psi_n^{(x)}|^{\mathcal{F}}_{p_j,0}\big)
|b|_{k_1,k_2}^{\mathcal{F}}\,,
\end{equation*}
which, by using the inductive hypothesis is bounded by 
$C(1+|b|^{\mathcal{F}}_{k_1,0})^{\nu_1}|b|^{\mathcal{F}}_{k_1,k_2}$. 
Therefore one obtains \eqref{ind:1} by using the smallness of $r$.

The proof of \eqref{ind:2} of $(S1)_{n+1}$ is similar. One can also deduce the \eqref{ind:4} of  $(S1)_{n+1}$ by equation \eqref{faadibruno}. Concerning \eqref{ind:3} of $(S1)_{n+1}$ 
one has to reason as done in Section 3 of \cite{BD} 
by  recalling that \eqref{ind:4} and 
\eqref{ind:3} of $(S1)_n$ hold true.

We now pass to the proof of $(S2)_{n+1}$, starting from \eqref{ind_2:1}. By using the fundamental calculus' theorem we obtain
\begin{equation*}
\begin{aligned}
x_{n+1}-x_n=&\int_0^{\tau}b(\sigma,z_{n-1};x_{n-1})-b(\sigma,z_n;x_n) d\s\\
=&\int_0^{\tau}\int_0^1\partial_x b\big(\sigma,z_{n-1};x_{n-1}
+\gamma(x_n-x_{n-1})\big)(x_{n-1}-x_n)\\
&\quad\quad+D_{z}b\big(\sigma,z_{n-1}
+\gamma(z_n-z_{n-1});x_n\big)[z_{n-1}-z_n]d\gamma d\s\,.
\end{aligned}
\end{equation*}
For  the first addendum inside the integral of the r.h.s. 
of the equation above we can proceed as follows. 
Let $\alpha$ such that $\alpha\leq s-s_0-1$, we have
\begin{equation}\label{uffa}
\begin{aligned}
\partial_{x_0}^{\alpha}&\big[\partial_x b\big(\sigma,z_{n-1};
x_{n-1}+\gamma(x_n-x_{n-1})\big)(x_{n-1}-x_n)\big]=\\
&\sum_{\alpha_1+\alpha_2=\alpha}C_{\alpha_1,\alpha_2}
\Big(\partial_{x_0}^{\alpha_1}\partial_xb\big(\sigma,z_{n-1};
x_{n-1}+\gamma(x_n-x_{n-1})\big)\partial_{x_0}^{\alpha_2}(x_{n-1}-x_n)\,,
\end{aligned}
\end{equation}
Therefore we may deduce, by using 
$\eqref{ind_2:1}$ of $(S2)_n$ and the formula for the 
derivatives of the composition of functions, that
\begin{equation*}
\begin{aligned}
|\partial_xb\big(\sigma,z_{n-1};x_{n-1}+&\gamma(x_n-x_{n-1})\big)
(x_{n-1}-x_n)|^{\mathcal{F},s}_{\alpha,0} \\
& \leq |b|_{s-1}^{\mathcal{F}} \norm{z_{n-1}}_{s_0}|x_{n-1}-x_n|_{\alpha,0}^{\mathcal{F},s} 
\leq |b|_{s-1}^{\mathcal{F}}Cr2^{-n}\,,
\end{aligned}
\end{equation*}
from which the thesis follows if $r$ is chosen small enough in such a way that $rC|b|_s^{\mathcal{F}}\leq1/2$. The computation is similar for the other addendum of \eqref{uffa}, therefore \eqref{ind_2:1} of $(S2)_{n+1}$ is proved.
The proof of \eqref{ind_2:2} of $(S2)_{n+2}$ is analogous. 

For the proof of \eqref{ind_2:3} we define the following quantities and we reason as follows:
\begin{equation*}
\begin{matrix}
&b_{n}:= b(\tau, z_{n};x_0+\Psi^{(x)}_{n}(x_0,z_0))(1+\Psi^{(\xi)}_{n}(x_0,z_0)),  &b_{n-1}:= b(\tau, z_{n-1};x_0+\Psi^{(x)}_{n-1}(x_0,z_0))(1+\Psi^{(\xi)}_{n-1}(x_0,z_0)),\\
&f_n:=\opbw\big(\ii(b_n-b_{n-1})\xi_0\big)z_{n}, &v_{n+1}:=z_{n+1}-z_n.
\end{matrix}
\end{equation*}
Since the function $z_{n+1}$ solves the problem
\begin{equation*}
\partial_{\tau}z_{n+1}=\opbw\Big(\ii b(\tau,z_{n};x_0+\Psi^{(x)}_n(x_0,z_0))(1+\Psi^{(\xi)}_{n}(x_0,z_0))\xi_0\Big)[z_{n+1}],
\end{equation*}
(and since the function $z_n$ solves the same problem up to relabelling $n\rightsquigarrow n-1$) we have that 
\begin{equation}\label{v_n}
\partial_{\tau}v_{n+1}=\opbw(\ii b_n\xi_0)v_{n+1} + f_n.
\end{equation}
Therefore, by using Theorem 3.5.2 in \cite{BD} and the Duhamel principle, we obtain 
\begin{equation*}
v_{n+1}(\tau)=\Theta_n^{\tau}\int_0^{\tau}\big[\Theta^{\tau'}_n\big]^{-1}f_n(\tau')d\tau',
\end{equation*} 
where $\Theta^\tau_{n}$ is the flow of the equation \eqref{v_n} with $f_n=0$.  From the equation above   one infers that 
\begin{equation*}
\begin{aligned}
\norm{v_{n+1}(\tau)}_{s-1}&\leq C\norm{z_n}_{s}\norm{b_n-b_{n-1}}_{s_0}
%\\&
\leq C r\norm{z_n-z_{n-1}}_{s_0}\leq Cr2^{-n}r\,,
\end{aligned}
\end{equation*}
where in the last inequality we have used the \eqref{ind_2:3} of $(S2)_n$. 
The thesis is proved if one choses $r$ such that $rC\leq1/2$.
\end{proof}

We are ready to prove the existence of the flow of equation \eqref{Ham111prova}.

\begin{proof}[{\bf Proof of Theorem \ref{Buonaposit}}]
For any $n\geq 1$ we consider the system made by the equations \eqref{eq:Xn}, \eqref{eq:XXn} and \eqref{eq:Zn}, with initial conditions $z_n(0)=z_0$, $x_n(0)=x_0$, $\xi_n(0)=\xi_0$. By Lemma \ref{iterative} and Remark \ref{frechet} we have, up to extracting a subsequence, a sequence of solutions of the form \eqref{nota2} and \eqref{nota3}
%\eqref{ansatzX} and \eqref{ansatzXX} 
such that
\begin{equation*}
\begin{aligned}
&\Psi^{(x)}_n\rightharpoonup^* \widetilde{\Psi}^{(x)} \mbox{ in } 
\mathcal{F}^{\mathbb{R}}_{1}[r]\,, \,\,\, \partial_x^{\alpha}\Psi^{(x)}_n\rightarrow\partial_{x}^{\alpha} \Psi^{(x)} 
\mbox{ uniformly in $x_0$ for } \alpha\leq s-s_0-1\,,\\
&\Psi^{(\x)}_n\rightharpoonup^* \widetilde{\Psi}^{(\x)}
\mbox{ in } \mathcal{F}^{\mathbb{R}}_{1}[r], \,\,\, 
\partial_{x}^{\alpha}\Psi^{(\x)}_n\rightarrow\partial_{x}^{\alpha} 
\Psi^{(\x)} \mbox{ uniformly in $x_0$ for } \alpha\leq s-s_0-2\,,\\
&z_n\rightharpoonup^* \tilde{z} \mbox{ in } L^{\infty}([0,1],H^s)\,, 
\,\,\, z_n\rightarrow z \mbox{ in } L^{\infty}([0,1],H^s)\,.
\end{aligned}
\end{equation*}
We deduce that $\widetilde{\Psi}^{(x)}=\Psi^{(x)}$, $\widetilde{\Psi}^{(\x)}=\Psi^{(\x)}$ and 
$\tilde{z}=z$. We claim that the triple in 
\eqref{soluzione} solves respectively the third, 
the first and the second  equation in \eqref{Ham111prova}. 
Let us consider the third equation for instance. 
We want to show that there exists $s'>0$ such that 
\begin{equation*}
\norm{\opbw(b(\tau,z;x)\ii\xi)z-\opbw(b(\tau,z_n;x_n)\ii\xi_n)z_n}_{s'}\rightarrow 0\,,
\end{equation*}
when $n$ goes to infinity. 
We can estimate from above the preceding inequality by
\begin{equation*}
\begin{aligned}
&\norm{\opbw(b(\tau,z;x)\ii\xi)[z-z_n]}_{s'}
+ \norm{\opbw(b(\tau,z;x)\ii\xi)[z_n]-\opbw(b(\tau,z;x)\ii\xi_n)[z_n]}_{s'}\\
+&\norm{\opbw(b(\tau,z;x)\ii\xi_n)[z_n]-\opbw(b(\tau,z;x_n)\ii\xi_n)[z_n]}_{s'}\\+&\norm{\opbw(b(\tau,z;x_n)\ii\xi_n)[z_n]-\opbw(b(\tau,z_n;x_n)\ii\xi_n)[z_n]}_{s'}\,.
\end{aligned}
\end{equation*}
The first summand is bounded by $C\norm{z}_{s_0}\norm{z-z_n}_{s'+1}$. The second one may be bounded by $C\norm{z}_{s_0}\|\Psi^{(\x)}_n-\Psi^{(\xi)}\|_{s_0}\norm{z_n}_{s'}$, and similarly the others. Therefore it is enough to choose $s'+1\leq s-1$. The regularity of the solution $z$ may be deduced in a classical way by using the third equation in \eqref{Ham111prova}.
One can  prove the \eqref{pioggia4teorema}
by using  the \eqref{ind:1TOT}, \eqref{ind:2TOT}, by differentiating the third 
equation in \eqref{Ham111prova} and reasoning as in Lemma \ref{flusso-differenziale}.
\end{proof}

\begin{remark}\label{differenzioInver}
One can note that the inverse flow $\widetilde{\Phi}^{(z)}_{b}(1,z)$
of \eqref{Ham111prova} satisfies estimates  similar to \eqref{pioggia4teorema}.
This follows by differentiating the relation 
$\widetilde{\Phi}^{(z)}_{b}\big(1,{\Phi}^{(z)}_{b}(1,z_0)\big)=z_0$.
\end{remark}

\subsubsection{Explicit expressions
for the flow (\ref{Ham111prova})}\label{expexpFlow}

As shown in subsection \ref{wellwell52}, the Theorem \ref{Buonaposit}
guarantees the well-posedness of the flow \eqref{Ham111prova}
for \emph{any} given 
symbol $b(\tau,w;x)$ in $\mathcal{F}_1^{\mathbb{R}}[r]$.
In particular such a flow has the form \eqref{soluzione}-\eqref{nota3}.
In order to solve the equation \eqref{equa814}
and prove the main Theorem \ref{constEgo} we shall provide a more explicit expression for the flow
\eqref{Ham111prova} in the case that $b(\tau,w;x)$ has a special structure
(see formul\ae\, \eqref{def:betanNN}, \eqref{eq:b}). 
We have the following.

\begin{theorem}\label{formuleesplici}
Let $\beta(W;x)$ be a function in $\mathcal{F}_1^{\R}[r]$. Then if $\|W\|_{{s}}\ll 1$
for a sufficiently large ${s}$
then the following holds.

\noindent
(i) There is a unique $\gamma(\tau,W;x)$ in $\mathcal{F}^{\mathbb{R}}_{1}[r]$,  $\tau\in [0,1]$
such that
\begin{equation}\label{def:betanNN}
x_0+\gamma(\tau,W,x_0)+\tau\beta(W,x_0+\gamma(\tau,W,x_0))=x_0\,,\qquad \forall \tau\in [0,1]\,, x_0\in\mathbb{T}\,,\quad W\in H^{s}\,.
\end{equation}

\noindent
(ii)
 There exists a unique  $b(\tau,W;x)$ 
in $\mathcal{F}_1^{\R}[r]$ solution of the following equation
\begin{equation}\label{eq:b}
b(\tau,W;x)=\frac{\beta(W;x)+\tau d_{W}\beta(W;x)[\opbw(b(\tau,W;x)\xi)W]}{1+\tau(\beta)_x(W,x)}\,.
\end{equation}
In particular there is a constant $\mathtt{C}>0$
depending only on $s$ and $|\beta|_{s+1}^\mathcal{F}$
such that (recall \eqref{semi-norma-totale} and \eqref{semi-norma-totale2})
\begin{align}
&|b|_{s}^\mathcal{F}\lesssim\mathtt{C}\,,\label{constgiusta1}\\
&|b|_{s,0}^{\mathcal{F}}\lesssim_s
|\beta|_{s+1,0}^\mathcal{F}\,.\label{constgiusta2}
\end{align}
%$|b|_{s}^\mathcal{F}\lesssim_{s}|\beta|_{s+1}^\mathcal{F}$.

\noindent
(iii) Recalling \eqref{nota1}-\eqref{nota3} (the flow of \eqref{Ham111prova}) 
we have
\begin{align}
x(\tau)&:=x_0+\Psi_{b}^{(x)}(\tau,z_0,x_0)
=x_0+\gamma(\tau, \Phi_{b}^{(z)}(\tau,z_0),x_0)=
x_0+\gamma(\tau, z(\tau),x_0)\,,\label{flussoLINE1}\\
\x(\tau)&=\x_0(1+\Psi_{b}^{(\x)}(\tau,z_0,x_0))=\x_0(1+\tau(\beta)_{x}(z(\tau),x(\tau)))\,,
\label{flussoLINE2}
\end{align}
where $b(\tau,W,y)$ is given by formula \eqref{eq:b} and $\gamma(\tau;W,y)$
by formula \eqref{def:betanNN}.
\end{theorem}

\begin{proof}[{\bf Proof of Item $(i)$ of Theorem \ref{formuleesplici}}]
One can reason as in Section 2.5 in \cite{BD}.
\end{proof}

\begin{proof}[{\bf Proof of Item $(ii)$ of Theorem \ref{formuleesplici}}]

We proceed inductively, we define 
\begin{equation}\label{rambo10}
b_0=\frac{\beta(W;x)}{1+\tau\beta_x(W;x)}
\end{equation}
and the $n$-th problem as
\begin{equation}\label{eq:n-th}
\begin{aligned}
b_n(\tau,W;x)=&\frac{\beta(W;x)+\tau d_{W}\beta(W;x)[\opbw(b_{n-1}(\tau,W;x)\xi)W]}{1+\tau(\beta)_x(W,x)}\\
								=&b_0(W,x)+\tau\frac{d_{W}\beta(W;x)[\opbw(b_{n-1}(\tau,W;x)\xi)W]}{1+\tau(\beta)_x(W,x)},
\end{aligned}
\end{equation}
we shall prove that for any $n\geq 0$ the solution \eqref{eq:n-th} 
is a symbol in $\mathcal{F}_1^{\R}[r]$. 
More precisely we show that if $\beta(W;x)$ satisfies the estimate 
\eqref{maremma2} for a certain $s_0$ 
then for any $n$ the symbol $b_n$ defined in \eqref{eq:n-th} satisfies 
the same estimate with $s_0+1$. 

By Remark \ref{prodottoSimboli}
we deduce that
$b_0$ satisfies the \eqref{constgiusta1}.
Using the smallness of $r>0$ (see Remark \ref{smallnessSemi})
one can also check the bound \eqref{constgiusta2}.
%\eqref{maremma2}, let us suppose that so does  $b_{n-1}$.

Then 
let us assume that $b_{n-1}$ satisfies \eqref{constgiusta1}, \eqref{constgiusta2}.
For simplicity we show that $b_n$ satisfies \eqref{maremma2},
with estimates as in  \eqref{constgiusta1}, \eqref{constgiusta2}
in the case that $\alpha=0$ and $k=0$. 
Moreover we note that it is enough to prove 
the claim for $d_{W}\beta(W;x)[\opbw(b_{n-1}(\tau,W;x)\xi)W]$. We have
\begin{equation*}
\big|d_{W}\beta(W;x)[\opbw(b_{n-1}(\tau,W;x)\xi)W]\big|\leq C\|\opbw(b_{n-1}(W;x)\xi)W\|_{s_0}\leq C\|W\|_{s_0}\|W\|_{s_0+1},
\end{equation*}
where we have used in the first estimate the \eqref{maremma2} for the symbol $\beta(W;x)$ with $\alpha=0$ and $k=1$ and in the second we have used the Theorem \ref{azionepara}.

\noindent We now prove that the sequence of $b_n$ converges for any $\tau>0 $ in the space $\mathcal{F}_1^{\R}[r]$, therefore  its limit is a symbol and solves the equation \eqref{eq:b}. More precisely we prove, by using that $\|W\|_{s_0+1}\ll 1$, that for any there exists a constant $C>0$ such that  (recall \eqref{semi-norma-totale2})
\begin{equation}\label{eq:n-iter}
|b_{n+1}-b_{n}|^{\mathcal{F}}_{s,0}\leq C 2^{-(n+1)},
\end{equation}
for any $n\geq 1$ and $\tau>0$.
%, where $|\cdot|^{\mathcal{F}}_{s}$ is defined in \eqref{semi-norma-totale}. 
We note that
\begin{equation*}
(b_1-b_0)(\tau,W;x)=\tau d_u\beta(W;x)\big[\opbw(b_0(\tau,W;x)\xi) W\big],
\end{equation*}
therefore thanks to the smallness of $W$ and to the fact that $b_0$ is a function we obtain the \eqref{eq:n-iter} with $n=1$. We proceed with the inductive step
\begin{equation*}
(b_{n+1}-b_n)(\tau,W;x)=\tau \frac{d_u\beta(W;x)\big[\opbw((b_{n}-b_{n-1})(\tau,W;x)\xi) W\big]}{1+\tau(\beta)_x(W;x)},
\end{equation*}
by using formula \eqref{maremma2}, the smallness of $\norm{W}_{s_0+1}$ (therefore of $(\beta)_x(W;x)$) and Proposition \ref{AzioneParaMet}
%\ref{azione-ottimale} 
we obtain
\begin{equation*}
\begin{aligned}
|(b_{n+1}-b_n)(\tau,W;x)|&\leq C \|b_n-b_{n-1}\|_{L^{\infty}}\norm{W}_{s_0+1}\leq C |b_{n+1}-b_n|_s^{\mathcal{F}}\norm{W}_{s_0}\norm{W}_{s_0+1}\\
&\leq C2^{-n}\norm{W}_{s_0}\norm{W}_{s_0+1}\leq C2^{-n-1}\norm{W}_{s_0+1},
\end{aligned}
\end{equation*}
where we have used the inductive hypothesis and the smallness of $\norm{W}_{s_0}$. 
To prove the \eqref{eq:n-iter} it is enough to apply 
$\partial_x^{\alpha}$ %and $d_u^k$ 
%for $\alpha+k\leq s-(s_0+1)$ 
for $\alpha\leq s-(s_0+1)$
and reason as above by using the chain rule. 
%\comment{forse bisogna dire qualcosa riguardo la convergenza. De \eqref{eq:n-iter} 
%ho la convergenza come simbolo di $x$. La \eqref{constgiusta1}
%mi dice anche che i $b_{n}$ sono tutti uniformemente limitati in $\mathcal{F}_1^{\mathbb{R}}$}

\noindent 
Reasoning as done in the proof of 
Theorem \ref{Buonaposit}, one deduces that
 the sequence $b_n(W;x)$ converges in the space $\mathcal{F}_1^{\mathbb{R}}[r]$ 
 to a symbol $b(W;x)$ solving the equation \eqref{eq:b}. 
 Let us prove that such a solution is unique. Suppose that there exists another solution $\tilde{b}(W;x)$ of \eqref{eq:b}. Then we have
\begin{equation*}
(b-\tilde{b})(\tau,W;x)=\tau \frac{d_u\beta(W;x)\big[\opbw((b-\tilde{b})(\tau,W;x)\xi) W\big]}{1+\tau(\beta)_x(W;x)},
\end{equation*}
which, as before, implies 
\begin{equation*}
\norm{b-\tilde{b}}_{L^{\infty}_x}\leq \tau C\norm{W}_{s_0+1}\norm{b-\tilde{b}}_{L^{\infty}_x},
\end{equation*}
therefore a contradiction if $\norm{W}_{s_{0}+1}$ is small enough.
\end{proof}

\begin{remark}\label{rmk:Rambo}
Notice that the first order of approximation of 
our symbol $b$ (see the \eqref{rambo10}) coincide with the definition 
of $b$ given in the paper \cite{BD}.
\end{remark}

\begin{proof}[{\bf Proof of Item $(iii)$ of Theorem \ref{formuleesplici}}]
First of all  we recall that by Item (ii)
 we have that $b\in\mathcal{F}_1^{\R}[r]$. Hence 
 the flow in \eqref{soluzione} of \eqref{Ham111prova} is well-posed 
 and unique 
 and satisfies \eqref{nota1}-\eqref{nota3} by Theorem \ref{Buonaposit}.
So we just have to show that
 $x(\tau)$ defined in \eqref{flussoLINE1}, with $b$ given 
in item $(ii)$, 
solve the first equation in \eqref{Ham111prova}.
We recall that, by \eqref{flussoLINE1} and \eqref{nota1}, we have 
$x(\tau)=x_0+\gamma(\tau, \Phi_{b}^{(z)}(\tau,z_0),x_0)$.
By differentiating \eqref{def:betanNN} we get
\[
\pa_{\tau}x(\tau)+\beta(z(\tau),x(\tau))
+\tau(d_u\beta)\big(z(\tau),x(\tau)\big)[\pa_{\tau}(\tau)]
+\tau(\pa_{x}\beta)(z(\tau),x(\tau))\pa_{\tau}x(\tau)=0\,.
\]
Hence
\begin{equation}\label{nota13}
\pa_{\tau}x(\tau)=-\frac{\beta(z(\tau),x(\tau))
+\tau(d_u\beta)\big(z(\tau),x(\tau)\big)[\pa_{\tau}z(\tau)]}{1
+\tau(\pa_{x}\beta)(z(\tau),x(\tau))}\,,
\end{equation}
where $u=z(\tau)$. We shall use this notation also in the rest of the proof.
Recalling that (see the third equation in \eqref{Ham111prova})
one has 
$\pa_{\tau}z(\tau)=\opbw\big(\ii b(\tau,z(\tau);x(\tau))\x(\tau)\big)[z(\tau)]$,
and using the definition of $b$ in \eqref{eq:b}
the \eqref{nota13} implies 
\[
\pa_{\tau}x(\tau)=\pa_{\tau}\big(x_0+\gamma(\tau,z(\tau),x_{0})\big)=
-b(\tau,z(\tau),x(\tau))\,,
\]
which is the first equation in \eqref{Ham111prova}.

Consider now $\x(\tau)$ in \eqref{flussoLINE2}. Hence
\begin{equation}\label{nota14}
\begin{aligned}
\pa_{\tau}\x(\tau)&=\x_0(\beta)_{x}(z(\tau),x(\tau))+
\tau\x_0(\beta)_{xx}(z(\tau),x(\tau))\pa_{\tau}x(\tau)\\
&+\x_0\tau(d_{z}\beta)_x(z(\tau),x(\tau))[\pa_{\tau}z(\tau)]\,.
\end{aligned}
\end{equation}
This is true because
\begin{equation}\label{speranza}
\pa_x\Big(d_{w}\big(\beta\big)(w,x)[h]\Big)=
(d_w(\beta)_x)(w,x)%_{|at \;(w,x)}
[h]\,.
\end{equation}
By \eqref{nota14}, using \eqref{flussoLINE2}, 
the first and the third equations in \eqref{Ham111prova}
%\eqref{Ham111provan1}, \eqref{Ham111provan3} , 
we also deduce
\[
\begin{aligned}
\pa_{\tau}\x(\tau)&=
\x(\tau)\frac{(\beta)_{x}(z(\tau),x(\tau))}{1+\tau(\beta)_x(z(\tau),x(\tau))}
-
\tau\x(\tau)\frac{(\beta)_{xx}(z(\tau),x(\tau))\pa_{\tau}x(\tau)}
{1+\tau(\beta)_x(z_n(\tau),x_n(\tau))}
b(\tau,z(\tau),x)\\
&+\tau \x(\tau)
\frac{(d_{u}\beta)_x(z(\tau),x(\tau))\big[\opbw\big(\ii b(\tau,z(\tau);x(\tau))\x(\tau)\big)[z(\tau)]\big]}{1+\tau(\beta)_x(z(\tau),x(\tau))}\,.
\end{aligned}
\]
Now, using \eqref{eq:b}, we have
\[
\begin{aligned}
&(b)_x(\tau,z(\tau),x(\tau))=\\
&=-\frac{\beta(z(\tau),x(\tau))+\tau(d_u\beta)(z(\tau),x(\tau))
[\opbw(\ii b(\tau,z(\tau),x(\tau))\x(\tau))[z(\tau)]]}{(1+\tau(\beta)_{x}(w,x))^2}
(\beta)_{xx}(z(\tau),x(\tau))\\
&+
\frac{(\beta)_x(z(\tau),x(\tau))}{1+\tau(\beta)_x(z(\tau),x(\tau))}+
\frac{\tau(d_{u}\beta)_x(z(\tau),x(\tau))\big[\opbw\big(\ii b(\tau,z(\tau);x(\tau))\x(\tau)\big)[z(\tau)]}{1+\tau(\beta)_x(z(\tau),x(\tau))}\,,
\end{aligned}
\]
again recalling \eqref{speranza}. Therefore
\[
\pa_{\tau}\x(\tau)=\x(\tau)(b)_x(\tau,z(\tau),x(\tau))\,.
\]
This means that \eqref{flussoLINE2} solves he second equation in  \eqref{Ham111prova}.
\end{proof}

\begin{remark}\label{paragoneBD}
Notice that, by Taylor expanding in $z$ the flow $\Phi_{b}^{(z)}(\tau,z)$ in 
\eqref{flussoLINE1}, \eqref{flussoLINE2}, one obtains the formula $(3.5.31)$ in \cite{BD}
up to higher order  homogeneity terms.
\end{remark}

\subsubsection{Proof of Theorem \ref{constEgo}}\label{constEGOsec}

In this subsection we conclude the proof of Theorem \ref{constEgo}.
The key point is to understand  how to choose the function $b(\tau,W;y)$ in such a way
the equation \eqref{equa814} is satisfied.

Notice that by Theorem  \ref{formuleesplici} we also deduce information 
on the inverse flow
\eqref{Ham111prova}. 
In particular, using \eqref{flussoLINE1}, \eqref{flussoLINE2}
we can rewrite the equation \eqref{equa814}.
It is easy to check that
\begin{equation}\label{equa814bis}
F(b):=
(1+\widetilde{a}_{m}(z_0,y))\big(1+\gamma_{y}(1,z,y)\big)^{m}_{|y=x+\beta(z,x)}=m_{b}
\end{equation}
where $\gamma, b$ are given (in terms of $\beta$) by \eqref{def:betanNN} and \eqref{eq:b}
and 
\[
z=\Phi_{b}^{(z)}(1,z_0) \,, \qquad \longleftrightarrow \qquad z_0=\widetilde{\Phi}_{b}^{(z)}(1,z)\,.
\]
The equation \eqref{equa814bis} is non linear in the symbol $b(\tau,W;y)$.
So, roughly speaking, we shall construct the solution $b=b_{\infty}$ of \eqref{equa814bis}
as limit of a sequence of approximate solutions 
$b_{n}(\tau,W;y)\in \mathcal{F}_{1}^{\mathbb{R}}[r]$.
 More precisely, 
let $\gamma_0=\beta_0=b_0=0$ and
for any $n\geq 1$ consider the problem
\begin{align}
&\pa_{\tau}x_n(\tau)=-b_n(\tau,z_n(\tau);x_n(\tau)) \,, \label{Ham111provan1}\\
& \pa_{\tau}\x_n(\tau)=(b_n)_x(\tau,z_n(\tau);x_n(\tau))\x_n(\tau)\,,\label{Ham111provan2}\\
&\pa_{\tau}z_n(\tau)=\opbw\big(\ii b_n(\tau,z_n(\tau);x_n(\tau))\x_n(\tau)\big)[z_n(\tau)]\,,\label{Ham111provan3}
\end{align}
where $b_{n}(\tau,w,x)$ it the symbol defined by 
\begin{equation}\label{def:bn}
b_n(\tau,w,x):=\frac{\beta_n(w,x)+\tau(d_u\beta_n)(w,x)[\opbw(\ii b_n(\tau,w,x)\x)[w]]}{1+\tau(\beta_n)_{x}(w,x)}
\end{equation}
and $\beta_{n}(w,x)$ is defined by 
\begin{equation}\label{def:betan}
x_0+\gamma_{n}(\tau,w,x_0)+\tau\beta_n(w,x_0+\gamma_{n}(\tau,w,x_0))=x_0\,,\qquad \forall \tau\in [0,1]\,, x_0\in\mathbb{T}\,,\quad w\in H^{s}\,,
\end{equation}
and
\begin{align}
\gamma_{n}(1,w,x)&:=\pa_{x}^{-1}\left(\left(\frac{m_n(w)}{1
+\tilde{a}_m(\widetilde{\Phi}^{(z)}_{b_{n-1}}(1,w),x)}\right)^{\frac{1}{m}}-1\right)\,,\label{def:gamman}\\
m_{n}(w)&:=\left[2\pi\left( \int_{\mathbb{T}} \frac{1}{(1+\tilde{a}_{m}(\widetilde{\Phi}^{(z)}_{b_{n-1}}(1,w),x))^{\frac{1}{m}}}dy\right)^{-1} \right]^{m}-1\,.\label{def:Mn}
\end{align}

We shall prove inductively the following:
for any $n\geq 1$ 

\begin{itemize}

\item[$({\bf S1})_{n}$] 
One has that $m_{n}$ in \eqref{def:Mn} belongs to $\mathcal{F}_{0}^{\mathbb{R}}[r]$ and it is independent of $x$; the functions 
$\gamma_n,\beta_n,b_{n}$ given respectively by \eqref{def:gamman}, \eqref{def:betan}
and \eqref{def:bn}
belong to  $\mathcal{F}_1^{\mathbb{R}}$. In particular we have
that there exists a constant $\mathtt{C}>0$ depending only on 
$|\tilde{a}_{d}|^{\mathcal{F}}_{s}$ such that
\begin{equation}\label{stimeSeminormeMMM}
|m_{n}|_{s}^{\mathcal{F}}+|\gamma_{n}|_{s}^{\mathcal{F}}+
|\beta_{n}|_{s}^{\mathcal{F}}+|b_{n}|_{s}^{\mathcal{F}}\lesssim_{s}\mathtt{C}\,.
\end{equation}
Moreover
\begin{equation}\label{stimeSeminormeMMMXX}
|m_{n}|_{s,0}^{\mathcal{F}}+|\gamma_{n}|_{s+1,0}^{\mathcal{F}}+
|\beta_{n}|_{s+1,0}^{\mathcal{F}}+|b_{n}|_{s,0}^{\mathcal{F}}\lesssim_{s}
|\tilde{a}_{d}|^{\mathcal{F}}_{s,0}\,.
\end{equation}
%\begin{align}
%|m_{n}|_{s}^{\mathcal{F}}&\lesssim_{s}|\tilde{a}_{d}|^{\mathcal{F}}_{s}\,,\label{stimeSeminormeM}\\
%|\gamma_{n}|_{s}^{\mathcal{F}}&\lesssim_{s}|\tilde{a}_{d}|^{\mathcal{F}}_{s-1}\,,\label{stimeSeminormegamma}\\
%|\beta_{n}|_{s}^{\mathcal{F}}&\lesssim_{s}|\tilde{a}_{d}|^{\mathcal{F}}_{s-1}\,,\label{stimeSeminormebeta}\\
%|b_{n}|_{s}^{\mathcal{F}}&\lesssim_{s}|\tilde{a}_{d}|^{\mathcal{F}}_{s}\,,\label{stimeSeminormeB}
%\end{align}

\item[$({\bf S2})_{n}$] 
 the flow of \eqref{Ham111provan1}-\eqref{Ham111provan3} with $b_n$ given by \eqref{def:bn} 
 is well-posed, has the form
\begin{equation}\label{nota1n}
\begin{aligned}
z_{n}(\tau)&=\Phi_{b_{n}}^{(z)}(\tau,z_0)\in  \cap_{k=0}^{K}C^{k}([0,1];H^{s-k})\,,\\
x_{n}(\tau)&=x_0+\gamma_n(\tau, z_n(\tau),x_0)\,,
%x_0+\Psi_{b_{n}}^{(x)}(\tau,z_0,x_0)\,,
\\
\x_{n}(\tau)&=\x_0\big(1+\tau(\beta_n)_{x}(z_n(\tau),x_n(\tau))\big)\,.
%\Phi_{b_{n}}^{(\x)}(\tau,z_0,x_0,\x_0)=\x_0(1+\Psi_{b_{n}}^{(\x)}(\tau,z_0,x_0))\,, 
\end{aligned}
\end{equation}

%\item[$({\bf S3})_{n}$]  
%One has 
%\begin{equation}\label{differeNN}
%|\Phi_{b_{n}}^{(z)}-\Phi_{b_{n-1}}^{(z)}|^{\mathcal{F}}_{s-1}\lesssim_{s}2^{-n}\,.
%\end{equation}

%
%\item[$({\bf S3})_{n}$] % The functions $\gamma_n,\beta_n, b_n$ are Cauchy
%%sequence in $\mathcal{F}_1^{\mathbb{R}}$. In particular w
%We have  (recall \eqref{semi-norma-totale2})
%\[
%|\gamma_n-\gamma_{n-1}|_{s,0}^{\mathcal{F}}
%+|\beta_n-\beta_{n-1}|_{s,0}^{\mathcal{F}}+|b_n-b_{n-1}|_{s-1,0}^{\mathcal{F}}\leq C2^{-n}
%\]
%for some $C>0$ independent of $n$.

%\item[$({\bf S3})_{n}$]   One has that
%\begin{equation}\label{FBN}
%|F(b_{n})|^{\mathcal{F}}_{s-2,0}\lesssim_{s} 2^{-n}\,.
%\end{equation}
\end{itemize}

We argue by induction. 

\noindent
{\bf Inizialization.} The $({\bf S1})_{0}$, $({\bf S2})_{0}$ are trivial.

\noindent
So we assume that
$({\bf Sk})_{j}$, for $k=1,2$, hold true with $0\leq j\leq n-1$.

\begin{proof}[{\bf Proof of} $({\bf S1})_{n}$.]
By the inductive 
hypothesis $b_{n-1}$ is a symbol in $\mathcal{F}_1^{\mathbb{R}}[r]$, hence, by Theorem
\ref{Buonaposit}, the flow of \eqref{Ham111prova} with $b\rightsquigarrow b_{n-1}$ is well-posed.
Therefore, using the formula of Faa di Bruno and \eqref{pioggia4teorema}, one can check that
(reasoning as in \eqref{faadibruno}-\eqref{fattore:3})
\[
| \tilde{a}_{d}(\widetilde{\Phi}^{(z)}_{b_{n-1}}(1,w),x) |^{\mathcal{F}}_{s}\lesssim_{s}
|\tilde{a}_{d}|^{\mathcal{F}}_{s}\,.
\]
Hence, using the \eqref{def:Mn} and Remark \ref{prodottoSimboli}, 
one deduces the \eqref{stimeSeminormeMMM} for $m_{n}$.
Recalling Remark \ref{smallnessSemi} one obtains the \eqref{stimeSeminormeMMMXX}
for $m_{n}$. The estimates \eqref{stimeSeminormeMMM}, \eqref{stimeSeminormeMMMXX}
for $\gamma_{n}$ follow in the same way using 
\eqref{def:gamman} and the smoothing effect of the Fourier multiplier $\pa_{x}^{-1}$. 

Using the relation \eqref{def:betan} at $\tau=1$ one construct the function $\beta_{n}$
(independent of $\tau$) as the inverse diffeomorphism of 
$y=x_0+\gamma_{n}(1,w,x_0)$. One can check that $|\beta_{n}|_{s}^{\mathcal{F}}
\lesssim_{s}C(|\gamma_{n}|^{\mathcal{F}}_{s})$
and $|\beta_{n}|_{s,0}^{\mathcal{F}}
\lesssim_{s}|\gamma_{n}|^{\mathcal{F}}_{s,0}$.
Hence the \eqref{stimeSeminormeMMM}, \eqref{stimeSeminormeMMMXX}
hold for $\beta_{n}$.
The family $[0,1]\ni\tau\to \gamma_{n}(\tau,w;x_0)$ is given by item $(i)$ of 
Theorem \ref{formuleesplici}.
Since $\beta_{n}$ satisfies \eqref{stimeSeminormeMMM}, \eqref{stimeSeminormeMMMXX}
 then, item $(ii)$
of Theorem \ref{formuleesplici} implies that the function $b_{n}$ in \eqref{def:bn}
is well-posed and satisfies \eqref{stimeSeminormeMMM}, \eqref{stimeSeminormeMMMXX}.
\end{proof}

\begin{proof}[{\bf Proof of} $({\bf S2})_{n}$.]
Thanks to \eqref{stimeSeminormeMMM}, \eqref{stimeSeminormeMMMXX}.
% \eqref{stimeSeminormeB} 
we can apply Theorems \ref{Buonaposit} and
\ref{formuleesplici} (see item $(iii)$)
which imply the \eqref{nota1n}.
\end{proof}

%In order to prove the $({\bf S3})_{n}$ %and $({\bf S4})_{n}$ 
%we first need the following result.

In order to conclude the proof of Theorem \ref{constEgo}
we have to check the \eqref{equa814}.
We need some preliminary results which are consequences of 
 $({\bf S1})_{n}$, $({\bf S2})_{n}$.

\begin{lemma}\label{diffeFlussiZ}
%For any $n\geq 1$, if $b_{n},b_{n-1}\in \mathcal{F}_1[r]$ with 
For $r>0$ satisfying
\begin{equation}\label{piccoloRRR}
rC_s\sup_{\tau\in[0,1]}\Big(|b_n|^{\mathcal{F}}_{s}+|b_{n-1}|^{\mathcal{F}}_{s}\Big)\ll 1\,,
\end{equation}
for some $C_{s}\gg1$ 
we have %(recall \eqref{nota31}-\eqref{nota33})
\begin{equation}\label{stimaNN}
\sup_{\tau\in[0,1]}\|\Phi_{b_{n}}^{(z)}-\Phi_{b_{n-1}}^{(z)}\|_{H^{s-1}}
%\Big(|Z_{n}|^{\mathcal{F}}_{s}+|X_{n}|^{\mathcal{F}}_{s}
%+|Y_{n}|^{\mathcal{F}}_{s}\Big)
\lesssim_{s}
r\sup_{\tau\in[0,1]}|b_{n}-b_{n-1}|_{s-1,0}^{\mathcal{F}}\,,
\end{equation}
where $\Phi_{b_{j}}^{(z)}$, $j=n,n-1$ is the solution of \eqref{Ham111provan3}.
\end{lemma}

We postpone the proof of Lemma \ref{diffeFlussiZ} and we first show some consequences.
\begin{remark}\label{DiffeFlussiZinv}
If $\widetilde{\Phi}_{b_{n}}^{(z)}$ and $\widetilde{\Phi}_{b_{n-1}}^{(z)}$
are respectively the inverse flows of 
$\Phi_{b_{n}}^{(z)}$ and $\Phi_{b_{n-1}}^{(z)}$, then on can write
\begin{equation*}
\widetilde{\Phi}_{b_{n}}^{(z)}-\widetilde{\Phi}_{b_{n-1}}^{(z)}=
-\widetilde{\Phi}_{b_{n}}^{(z)}
\Big({\Phi}_{b_{n}}^{(z)}-{\Phi}_{b_{n-1}}^{(z)}\Big)\widetilde{\Phi}_{b_{n-1}}^{(z)}\,.
\end{equation*}
Then, since $\|\widetilde{\Phi}_{b_{j}}^{(z)}\|_{H^{s}}\lesssim_{s}\|z_0\|_{H^{s}}$, $j=n,n-1$ 
(it satisfies the same estimates of the flow ${\Phi}_{b_{j}}^{(z)}$ in \eqref{ind:3TOT}),
and using the \eqref{stimaNN} we deduce
\begin{equation}\label{stimaNNinv}
\sup_{\tau\in[0,1]}\|\widetilde{\Phi}_{b_{n}}^{(z)}-\widetilde{\Phi}_{b_{n-1}}^{(z)}\|_{H^{s-1}}
%\Big(|Z_{n}|^{\mathcal{F}}_{s}+|X_{n}|^{\mathcal{F}}_{s}
%+|Y_{n}|^{\mathcal{F}}_{s}\Big)
\lesssim_{s}
r\sup_{\tau\in[0,1]}|b_{n}-b_{n-1}|_{s-1,0}^{\mathcal{F}}\,.
\end{equation}
\end{remark}

Notice that, by \eqref{stimeSeminormeMMM}, we have
\[
b_n\rightharpoonup^* \widetilde{b} \mbox{ in } \mathcal{F}^{\mathbb{R}}_{1}[r]\,,
\]
with $\widetilde{b}$ still satisfying \eqref{stimeSeminormeMMM}. 
Moreover, by Ascoli-Arzel\`a theorem, we deduce that 
\begin{equation}\label{convforteBn}
b_{n}\to b\qquad {\rm in}\qquad C^{k}(\mathbb{T},\mathbb{R})\,,
\end{equation}
%$b_{n}$ converges to some function $b$
%in the space $C^{k}(\mathbb{T},\mathbb{R})$ 
for $0\leq k\leq s-s_0-1$ uniformly in $x_0\in \mathbb{T}$.
Hence we deduce $\widetilde{b}\equiv b$, which means that
$b\in\mathcal{F}_1^{\mathbb{R}}[r]$ with bounded $|\cdot|_{s-1}^{\mathcal{F}}$-norm. In the same way one deduces that $m_n(w)$ converges in norm $|\cdot|_{s,0}^{\mathcal{F}}$ to another constant function $m_{b}(w)$.
We are now ready to prove 
%\eqref{nota1n}
\eqref{equa814}.
%$({\bf S3})_{n}$. % and $({\bf S4})_{n}$.

\begin{lemma}\label{tendeazero}
One has that
$F(b_n)\to m_{b}(w)$
in the space $C^{k}(\mathbb{T},\mathbb{R})$ with $0\leq k\leq s-s_0-1$.
\end{lemma}

\begin{proof}
We set 
\[
z=z_{n}(1)=\Phi_{b_n}^{(z)}(1,z_0)\,,\quad x=x_{n}(1)=\Phi_{b_n}^{(x)}(1,x_0)\,.
\]
Then, using the \eqref{nota1n} in $({\bf S2})_{n}$,
we have that
\begin{equation}\label{nota10}
\begin{aligned}
F(b_n)&=\Big(1+\tilde{a}_{m}\big( \widetilde{\Phi}_{b_n}^{(z)}(1,z),
\widetilde{\Phi}_{b_n}^{(x}(1,z,x) \big)\Big)(1+\widetilde{\Psi}_{b_n}^{(\x)}(1,z,x))^{m}\\
&=\Big(1+\tilde{a}_{m}\big( \widetilde{\Phi}_{b_n}^{(z)}(1,z),
y \big)\Big)\big(1+(\gamma_{n})_{y}(1,z,y)\big)^{m}_{|y=x+\beta_{n}(z,x)}\\
&=\underbrace{\Big(1+\tilde{a}_{m}\big( \widetilde{\Phi}_{b_{n-1}}^{(z)}(1,z),
y \big)\Big)\big(1+(\gamma_{n})_{y}(1,z,y)
\big)^{m}_{|y=x+\beta_{n}(z,x)}}_{\stackrel{\eqref{def:gamman}}{=}m_n}\\
&+\underbrace{\Big[
\tilde{a}_{m}\big( \widetilde{\Phi}_{b_{n}}^{(z)}(1,z),
y \big)-
\tilde{a}_{m}\big( \widetilde{\Phi}_{b_{n-1}}^{(z)}(1,z),
y \big)
\Big]\big(1+(\gamma_{n})_{y}(1,z,y)\big)^{m}_{|y=x+\beta_{n}(z,x)}}_{{\rm to \; be\; bounded}}
\end{aligned}
\end{equation}
To bound the last term in \eqref{nota10} we reason as follows.
Recalling that $\tilde{a}_m\in \mathcal{F}_1^{\mathbb{R}}$ we
have
\[
\begin{aligned}
\|(d_u\tilde{a}_{m})(w)\big[ \widetilde{\Phi}_{b_{n}}^{(z)}(1,z)
- \widetilde{\Phi}_{b_{n-1}}^{(z)}(1,z) \big]\|_{L^{\infty}_x}&\lesssim_{s}
\| \widetilde{\Phi}_{b_{n}}^{(z)}(1,z)
- \widetilde{\Phi}_{b_{n-1}}^{(z)}(1,z)\|_{H^{s_0}}\\
&\stackrel{\eqref{stimaNNinv}}{\lesssim_s}r\sup_{\tau\in[0,1]}|b_{n}-b_{n-1}|_{s-1,0}^{\mathcal{F}}\,.
\end{aligned}
\]
By \eqref{convforteBn} we have that the estimate above
implies $F(b_n)\to m_{b}(w)$ in $C^{0}(\mathbb{T},\mathbb{R})$.
For the derivatives in $x_0$ one can reason similarly.
\end{proof}

In order to prove  Lemma \ref{diffeFlussiZ},
we  need the following result.

\begin{lemma}\label{jack}
Let $f\in \mathcal{F}_{1}[r]$ and define
\[
G(\tau,z_0;x_0):=f(z_{j}(\tau),x_{j}(\tau))\x_{j}(\tau)\frac{1}{\x_0}
\]
where $(z_{j}(\tau),x_{j}(\tau), \x_{j}(\tau))$ is the flow of 
\eqref{Ham111provan1}-\eqref{Ham111provan3} generated by $b_{j}$.
Then, for  $r$ small enough, one has
\begin{equation*}
\sup_{\tau\in[0,1]}|G|^{\mathcal{F}}_{s,0}\lesssim_{s}
\sup_{\tau\in[0,1]}|f|^{\mathcal{F}}_{s,0}(1+
r\sup_{\tau\in[0,1]}|b_{j}|_{s,0}^{\mathcal{F}})\,.%+\sup_{\tau\in[0,1]}|f|^{\mathcal{F}}_{s}r\,.
\end{equation*}
\end{lemma}
\begin{proof}
First of all we have
\[
\pa_{x_0}G=f(z_j,x_j)\pa_{x_0}\Psi_{b_j}^{(\x)}+\x_{j}(\tau)\frac{1}{\x_0}(\pa_{x}f)(z_j,x_j)\pa_{x_0}x_j
%+(d_uf)(z_j,x_j)\x_{j}(\tau)\frac{1}{\x_0}\pa_{x_0}z_{j}
\]
Hence (recall \eqref{maremma2})
\begin{equation}\label{aurora}
\begin{aligned}
\|\pa_{x_0}G\|_{L^{\infty}_x}&\lesssim_s
\|f(z_j,x_j)\|_{L^{\infty}_x}\|\pa_{x_0}\Psi_{b_j}^{(\x)}\|_{L^{\infty}_x}+
(1+\|\Psi_{b_j}^{(\x)}\|_{L^{\infty}_x})\|(\pa_{x}f)(z_j,x_j)\|_{L^{\infty}_x}
(1+\|\pa_{x_0}\Psi_{b_j}^{(x)}\|_{L^{\infty}_x})
\\&\stackrel{\eqref{ind:1TOT}, \eqref{ind:2TOT}}{\lesssim_s}
|f|^{\mathcal{F}}_{0,0}|b_j|^{\mathcal{F}}_{1,0}\|z_{j}\|_{H^{s_0}}\|z_{0}\|_{H^{s_0+1}}\\&
\qquad +
(1+|\Psi_{b_j}^{(\x)}\|z_{0}\|_{H^{s_0}}|^{\mathcal{F}}_{0,0})|f|^{\mathcal{F}}_{1,0}(1
+|\Psi_{b_j}^{(x)}|^{\mathcal{F}}_{1,0}\|z_{0}\|_{H^{s_0+1}})\|z_{j}\|_{H^{s_0+1}}\,.
%+|f|^{\mathcal{F}}_{0,1}(1+|\Psi_{b_j}^{(\x)}|^{\mathcal{F}}_{0,0})\|z_j\|_{H^{s_0+1}}\,.
\end{aligned}
\end{equation}
Using \eqref{ind:3TOT} we obtain
\[
\sup_{\tau\in[0,1]}|G|^{\mathcal{F}}_{1,0}\lesssim_{s}
\sup_{\tau\in[0,1]}|f|^{\mathcal{F}}_{1,0}(1+
r\sup_{\tau\in[0,1]}|b_{j}|_{1,0}^{\mathcal{F}})\,.
\]
To estimate $\|\pa_{x_0}^{\alpha}G\|_{L^{\infty}_x}$ one can  
apply the formula of Faa di Bruno
and reasoning as done in \eqref{aurora}.
\end{proof}

\begin{proof}[{\bf Proof of Lemma \ref{diffeFlussiZ}}]\label{orologio100}
 Consider the equations \eqref{Ham111provan1}-\eqref{Ham111provan3} and set,
 for any $n\geq1$,
 \begin{align}
X_{n}(\tau)&:=x_{n}(\tau)-x_{n-1}(\tau)\,,\label{nota31}\\
Y_{n}(\tau)&:=\x_{n}(\tau)-\x_{n-1}(\tau)\,,\label{nota32}\\
Z_{n}(\tau)&:=z_{n}(\tau)-z_{n-1}(\tau)\,.\label{nota33}
\end{align}
We claim that
if \eqref{piccoloRRR} holds
then, for $s\geq s_0+2$, 
\begin{equation}\label{stimaincredibile}
\sup_{\tau\in[0,1]}\Big(|X_{n}|^{\mathcal{F}}_{s-1,0}+|Y_{n}\x_0^{-1}|_{s-2,0}^{\mathcal{F}}+
\|Z_{n}\|_{H^{s-1}}\Big)\lesssim_{s} r\sup_{\tau\in[0,1]}|b_{n}-b_{n-1}|_{s-1,0}^{\mathcal{F}}\,.
\end{equation}
The \eqref{stimaincredibile} implies the \eqref{stimaNN}.

Let us prove the \eqref{stimaincredibile}.
Notice that, for some $\s_i\in[0,1]$, $i=1,\ldots,6$, 

%\vspace{0.5em}
%\noindent
%{\bf The equation \eqref{Ham111provan1}.}
\begin{equation}\label{XN}
\begin{aligned}
-\pa_{\tau}X_{n}&\stackrel{\eqref{Ham111provan1}}{=}
%b_n(\tau,z_n;x_n) -b_{n-1}(\tau,z_{n-1};x_{n-1}) \\
%&=
\big(\pa_{x}b_{n}\big)(\tau,z_n,x_{n-1}+\s_1 X_{n})\big[X_{n}\big]+\big(d_ub_{n}\big)(\tau,z_{n-1}
+\s_2Z_{n},x_{n-1})\big[Z_{n}\big]\\
&\quad+ \Big(b_{n}-b_{n-1}\Big)(\tau,z_{n-1},x_{n-1})\,.
\end{aligned}
\end{equation}

\vspace{0.5em}
%\noindent
%{\bf The equation \eqref{Ham111provan2}.}
%We have, for some $\s\in [0,1]$,
\begin{equation}\label{YN}
\begin{aligned}
\pa_{\tau}Y_{n}&\stackrel{\eqref{Ham111provan2}}{=}
%(b_n)_{x}(\tau,z_n,x_n)\x_{n} -(b_{n-1})_{x}(\tau,z_{n-1};x_{n-1})\x_{n-1} \\
%&=
(b_{n})_{x}(\tau,z_n,x_n) Y_{n}+\x_{n-1}\big(b_{n}\big)_{xx}(\tau,z_n,x_{n-1}+\s_3 X_{n})\big[X_{n}\big]\\
&\quad+\x_{n-1}\big(d_ub_{n}\big)_{x}(\tau,z_{n-1}+\s_4 Z_{n},x_{n-1})\big[Z_{n}\big]+ \x_{n-1}\Big(b_{n}-b_{n-1}\Big)_{x}(\tau,z_{n-1},x_{n-1})\,.
\end{aligned}
\end{equation}

\vspace{0.5em}
%\noindent
%{\bf The equation \eqref{Ham111provan3}.}
%We have, for some $\s\in [0,1]$,
\begin{equation}\label{ZN}
\begin{aligned}
\pa_{\tau}Z_{n}&\stackrel{\eqref{Ham111provan3}}{=}
%\opbw\big(\ii b_n(\tau,z_n;x_n)\x_n\big)[z_n]-
%\opbw\big(\ii b_{n-1}(\tau,z_{n-1};x_{n-1})\x_{n-1}\big)[z_{n-1}]\\
%&=
\opbw\big(\ii b_n(\tau,z_n;x_n)\x_n\big)[Z_{n}]+\opbw\big(\ii b_n(\tau,z_n;x_n) Y_{n}\big)[z_{n-1}]\\
&+\opbw\Big(\ii \big(\pa_{x}b_n\big)(\tau,z_n;x_{n-1}+\s_5 X_n)[X_{n}] \x_{n-1}\Big)[z_{n-1}]\\
&+\opbw\Big(\ii \big(d_ub_n\big)(\tau,z_{n-1}+\s_6 Z_{n};x_{n-1})[Z_{n}] \x_{n-1}\Big)[z_{n-1}]\\
&+\opbw\Big(\big(b_{n}-b_{n-1}\big)(\tau,z_{n-1},x_{n-1})\x_{n-1}\Big)[z_{n-1}]\,.
\end{aligned}
\end{equation}
We now estimate $X_{n}(\tau)$ in \eqref{nota31}. First we note that, by \eqref{XN},
\begin{equation}\label{XNint}
\begin{aligned}
X_{n}(\tau)
&=\int_{0}^{\tau}\big(\pa_{x}b_{n}\big)(t,z_n(t),x_{n-1}(t)+\s_1 X_{n}(t))\big[X_{n}(t)\big]dt\\
&+\int_{0}^{\tau}\big(d_ub_{n}\big)(t,z_{n-1}(t)+\s_2 Z_{n}(t),x_{n-1}(t))\big[Z_{n}(t)\big]dt\\
&+ \int_{0}^{\tau}\Big(b_{n}-b_{n-1}\Big)(t,z_{n-1}(t),x_{n-1}(t))dt\,.
\end{aligned}
\end{equation}
Notice that
\[
x_{n-1}(t)+\s_1 X_{n}(t)=x_{n-1}(t)+\s_1 (x_{n}-x_{n-1})(t)
\]
satisfies, by using \eqref{ind:1TOT}-\eqref{ind:2TOT} (with $b\rightsquigarrow b_{n-1}$ and $b \rightsquigarrow b_{n}$),
\begin{equation}\label{aurora2}
 \sup_{\tau\in I}|x_{n-1}+\s_1 X_{n}|^{\mathcal{F},s}_{\alpha,0}\lesssim_s
 |b_n|^{\mathcal{F}}_{s,0}+|b_{n-1}|^{\mathcal{F}}_{s,0}
\end{equation}
Similarly we can get bounds on $\x_{n-1}(t)+\s Y_{n}(t)$
and $z_{n-1}(t)+\s Z_{n}(t)$ for some $\s\in[0,1]$.
Then, by differentiating the \eqref{XNint}, using Lemma \ref{jack} (see also \eqref{aurora2})
and recalling Remark \ref{smallnessSemi},
we get
\begin{equation}\label{stimaXN}
\begin{aligned}
\sup_{\tau\in[0,1]}|X_{n}|^{\mathcal{F}}_{s,0}&\lesssim_{s}
r A_{n}|X_{n}|^{\mathcal{F}}_{s,0}
+
%\sup_{\tau\in[0,1]}|b_{n}|_{s+1}^{\mathcal{F}}\sup_{\tau\in[0,1]}
C_{n}\|Z_{n}\|_{H^{s}}%^{\mathcal{F}}_{s}
+r\sup_{\tau\in[0,1]}|b_{n}-b_{n-1}|_{s,0}^{\mathcal{F}}(1+r\sup_{\tau\in[0,1]}|b_{n-1}|^{\mathcal{F}}_{s,0})\,,
\end{aligned}
\end{equation}
where
\begin{align}
A_{n}&:=\sup_{\tau\in[0,1]}|b_{n}|^{\mathcal{F}}_{s+1,0}
\Big(1+r |b_n|^{\mathcal{F}}_{s,0}+r|b_{n-1}|^{\mathcal{F}}_{s,0}\Big)\,,\label{piccoloR11}\\
C_{n}&:=\sup_{\tau\in[0,1]}|b_{n}|^{\mathcal{F}}_{s+1}\,.\label{stimaXNfine100}
\end{align}
Therefore, for $r>0$ such that 
\begin{equation}\label{piccoloR10}
r A_{n}\leq 1/2
\end{equation}
we get
we obtain
\begin{equation}\label{stimaXNfine}
\sup_{\tau\in[0,1]}|X_{n}|^{\mathcal{F}}_{s,0}\lesssim_{s}
C_{n}\|Z_{n}\|_{H^{s}}%^{\mathcal{F}}_{s}
+r\sup_{\tau\in[0,1]}|b_{n}-b_{n-1}|_{s,0}^{\mathcal{F}}
\end{equation}
We recall that, by \eqref{nota3}, the symbol $Y_{n}\x_0^{-1}$ is actually a function
in $\mathcal{F}_1^{\mathbb{R}}[r]$. Then
reasoning as done for $X_n$ and  using equation \eqref{YN}, we deduce
\begin{equation}\label{stimaYN}
\begin{aligned}
\sup_{\tau\in[0,1]}|\x_0^{-1}Y_{n}|^{\mathcal{F}}_{s,0}&\lesssim_{s}
r\sup_{\tau\in[0,1]}|b_{n}|_{s+2,0}^{\mathcal{F}}\sup_{\tau\in[0,1]}|X_{n}|^{\mathcal{F}}_{s}+\sup_{\tau\in[0,1]}|b_{n}|_{s+2}^{\mathcal{F}}\sup_{\tau\in[0,1]}\|Z_{n}\|_{H^{s}}\\
&+r\sup_{\tau\in[0,1]}|b_{n}-b_{n-1}|_{s+1,0}^{\mathcal{F}}
(1+r\sup_{\tau\in[0,1]}|b_{n-1}|^{\mathcal{F}}_{s})\\
&\stackrel{\eqref{stimaXNfine}}{\lesssim_{s}}
\sup_{\tau\in[0,1]}|b_{n}|_{s+2}^{\mathcal{F}}\sup_{\tau\in[0,1]}\|Z_{n}\|_{H^{s}}+
r\sup_{\tau\in[0,1]}|b_{n}-b_{n-1}|_{s+1,0}^{\mathcal{F}}
\end{aligned}
\end{equation}
where we used the smallness condition 
\begin{equation}\label{piccoloR12}
r\sup_{\tau\in[0,1]}\Big(|b_{n}|_{s+2,0}^{\mathcal{F}}+|b_{n-1}|_{s+2,0}^{\mathcal{F}}\Big)\ll1\,.
\end{equation}
We now consider the equation \eqref{ZN}.
Using Proposition \ref{AzioneParaMet} (and $r>0$ small enough)we deduce
\begin{equation}\label{picture}
\begin{aligned}
\|\opbw\big(\ii &b_n(\tau,z_n;x_n) Y_{n}\big)[z_{n-1}]\|_{H^{s}}\lesssim_{s}
\|z_{n-1}\|_{H^{s+1}}\|b_n(\tau,z_n;x_n) Y_{n}\x_0^{-1}\|_{L^{\infty}_x}\\
&\stackrel{Lem. \ref{jack}}{\lesssim_{s}}
\|z_{n-1}\|_{H^{s+1}}\|z_{n}\|_{H^{s_0}}
\sup_{\tau\in[0,1]}|b_{n}|^{\mathcal{F}}_{0,0}2\sup_{\tau\in[0,1]}|\x_0^{-1}Y_{n}|^{\mathcal{F}}_{0,0}\\
&\stackrel{\eqref{stimaYN}}{\lesssim_{s}}
\|z_{n-1}\|_{H^{s+1}}
\sup_{\tau\in[0,1]}|b_{n}|_{s_0+2}^{\mathcal{F}}\sup_{\tau\in[0,1]}\|Z_{n}\|_{H^{s_0}}
+\|z_{n-1}\|_{H^{s+1}}
r\sup_{\tau\in[0,1]}|b_{n}-b_{n-1}|_{s_0+1,0}^{\mathcal{F}}
\end{aligned}
\end{equation}
Reasoning in the same way and using \eqref{stimaXNfine}, \eqref{stimaXNfine100},
we get
\begin{equation}\label{picture2}
\begin{aligned}
\|\opbw\Big(\ii &\big(\pa_{x}b_n\big)(\tau,z_n;x_{n-1}+\s_5 X_n)[X_{n}] \x_{n-1}\Big)
[z_{n-1}]\|_{H^{s}}+\\
&\|\opbw\Big(\ii \big(d_ub_n\big)(\tau,z_{n-1}+\s_6 Z_{n};x_{n-1})[Z_{n}] \x_{n-1}\Big)[z_{n-1}]\|_{H^{s}}
\\
&\lesssim_{s}
\|z_{n-1}\|_{H^{s+1}}
\sup_{\tau\in[0,1]}|b_{n}|_{s_0+2}^{\mathcal{F}}\sup_{\tau\in[0,1]}\|Z_{n}\|_{H^{s_0}}
+\|z_{n-1}\|_{H^{s+1}}
r\sup_{\tau\in[0,1]}|b_{n}-b_{n-1}|_{s_0+1,0}^{\mathcal{F}}
\end{aligned}
\end{equation}
and
\begin{equation}\label{picture3}
\begin{aligned}
\|\opbw\Big(\big(b_{n}-b_{n-1}\big)(\tau,z_{n-1},x_{n-1})\x_{n-1}\Big)[z_{n-1}]\|_{H^{s}}
\lesssim_{s}\|z_{n-1}\|_{H^{s+1}}
\sup_{\tau\in[0,1]}|b_{n}-b_{n-1}|^{\mathcal{F}}_{0,0}r\,.
\end{aligned}
\end{equation}
Finally, recalling \eqref{Sobnorm2} and using symbolic calculus, one can check
that
\begin{equation}\label{picture4}
\begin{aligned}
{\rm Re}\Big(\langle D\rangle^{s} \opbw\big(\ii b_n(\tau,z_n;x_n)\x_n\big)
[Z_{n}],\langle D\rangle^{s} Z_{n}\Big)_{L^{2}}
\lesssim_{s}\|Z_n\|_{H^{s}}^{s}\sup_{\tau\in[0,1]}|b_{n}|^{\mathcal{F}}_{s,0}r\,.
\end{aligned}
\end{equation}
Therefore, putting together the \eqref{picture}, \eqref{picture2}, \eqref{picture3}
and \eqref{picture4}, by equation \eqref{ZN} we deduce
\begin{equation}\label{picture5}
\begin{aligned}
\pa_{\tau}\|Z_{n}\|_{H^{s}}^{2}&\lesssim_{s}
\|Z_n\|_{H^{s}}^{2}\sup_{\tau\in[0,1]}|b_{n}|^{\mathcal{F}}_{s,0}r+
\|z_{n-1}\|_{H^{s+1}}\|Z_{n}\|_{H^{s}}
\sup_{\tau\in[0,1]}|b_{n}|_{s_0+2}^{\mathcal{F}}\sup_{\tau\in[0,1]}\|Z_{n}\|_{H^{s_0}}\\
&+\|z_{n-1}\|_{H^{s+1}}\|Z_{n}\|_{H^{s}}
r\sup_{\tau\in[0,1]}|b_{n}-b_{n-1}|_{s_0+1,0}^{\mathcal{F}}\,.
\end{aligned}
\end{equation}
We now recall that $\|z_{n-1}\|_{H^{s}}\lesssim_{s}r$. We use formula \eqref{picture5}
with $s\rightsquigarrow s-1$. Then, for $r>0$ (small enough) such that
$r\sup_{\tau}|b_{n}|^{\mathcal{F}}_{s,0}\ll1$, we get
\begin{align}
\sup_{\tau\in[0,1]}\|Z_{n}\|^{2}_{H^{s-1}}&\lesssim_{s}
\sup_{\tau\in[0,1]}\|z_{n-1}\|_{H^{s+1}}\sup_{\tau\in[0,1]}\|Z_{n}\|_{H^{s}}
r\sup_{\tau\in[0,1]}|b_{n}-b_{n-1}|_{s_0+1,0}^{\mathcal{F}}\,, \quad \Rightarrow\nonumber\\
\sup_{\tau\in[0,1]}\|Z_{n}\|_{H^{s-1}}&\lesssim_{s}
r^2\sup_{\tau\in[0,1]}|b_{n}-b_{n-1}|_{s_0+1,0}^{\mathcal{F}}\,.\label{picture11}
\end{align}
Since condition \eqref{piccoloRRR} implies 
\eqref{piccoloR10}-\eqref{piccoloR11} with $s\rightsquigarrow s-1$,
and the \eqref{piccoloR12} with $s\rightsquigarrow s-2$,  then
by \eqref{picture11}, \eqref{stimaYN} with $s\rightsquigarrow s-2$ and
\eqref{stimaXNfine} with $s\rightsquigarrow s-1$
we obtain the \eqref{stimaincredibile}. 

In order to conclude the proof of Theorem \ref{constEGOsec} we need to prove that actually the functions $b(w;x)$ and $m_b(w)$ belong to the class $\Sigma\mathcal{F}^{\RRR}_1[r,N]$. In order to prove such a fact one makes the following ansatz: $b=\sum_{j=1}^{N}b_j$ with $b_j\in\widetilde{\mathcal{F}}^{\RRR}_1$ and $b_N$ in $\mathcal{F}^{\RRR}_N[r]$. Then one may prove, by using Taylor expansions, that the flows in \eqref{nota2} and \eqref{nota3} are verified, while the flow $z(\tau,z_0)=z_0+M(\tau,z_0)z_0$ in \eqref{nota1} with $M$ in the class $\Sigma\mathcal{M}_{1}[r,N]$. One find, by using  \eqref{equa814}, some recursive equation on the terms $b_j$, which may be expressed at each depending only on $b_k$ with $k<j$.
\end{proof}

\subsection{Off-diagonal terms at highest order}\label{sec:egohighoff}
In this subsection we study a non-linear problem we need to solve in order to conjugate system \eqref{Nonlin1inizio} to another one whose matrix of symbols is diagonal at the highest order.  A similar problem has been solved in \cite{Feola-Iandoli-Long}. In such a paper the matrix is diagonalized by means of a parametrix generated by the matrix of eigenvectors of the original matrix. Unfortunately it is not easy to transform such parametrix in a change of coordinates of the phase space. At a linear level the diagonalization problem has achieved in \cite{FIloc} by means of auxiliary linear flows. We adapt the last strategy to the non-linear problem.

The main result of this subsection is Theorem \ref{constEgohighoff}.
This result shows that the equations \eqref{sistema01} and \eqref{sistema02}
 have solutions in the our classes of symbols. Furthermore  it guarantees that it is possible to choose the generator of the flow in \eqref{Ham111provahighoff} in such a way the new off-diagonal symbol $b^+$ equals 0.
The proof is quite long and it is divided in several steps. In Section \ref{esistenzaSOLS} we study the well-posedness of equations \eqref{sistema01} and \eqref{sistema02}. In Section \ref{choice314} we solve \eqref{equaLowerhighoff}, this is a non-linear problem and the unknown is the generator $C$ of the flow \eqref{Ham111provahighoff}, the techniques used are similar to the one used in Section \ref{constEGOsec}.

\noindent Consider functions  
\begin{equation} \label{low1highoff}
a(U;x)\in \Sigma\mathcal{F}^{\mathbb{R}}_{1}[r,N]\,,\qquad b(U;x)\in 
\Sigma\mathcal{F}_{1}[r,N]\,.
\end{equation}
Let 
\begin{equation}\label{simboCChighoff}
C(\tau,U;x)\in \Sigma\mathcal{F}_1[r,N]\,, \;\;\; \tau\in[0,1]\,,
\end{equation}
%Assume that if $\delta>0$ then $c(\tau,w,y)$ is real valued 
and consider the equation
 \begin{equation}\label{Ham111provahighoff}
\left\{
\begin{aligned}
&\pa_{\tau}Z(\tau)=\opbw\big({\bf C}(\tau, Z(\tau);x)\big)[Z(\tau)]\,,\\
&Z(0)=Z_0=\vect{z_0\vspace{0.2em}}{\ov{z_0}}\,,
\end{aligned}\right.\qquad
{\bf C}(\tau,U;x):=
\left(
\begin{matrix}
0 & C(\tau,U;x) \\
\ov{C(\tau,U;x)} & 0
\end{matrix}
\right)\,.
\end{equation}
In the following we shall write $d_{U}$ to denote 
the differential of a symbol
$a(U;x,\x)$ with respect to the variable $U=\vect{u}{\bar{u}}$, i.e.
\[
(d_{U}a)(U;x,\x)[H]=(d_u a)(U;x,\x)[h]+(d_{\bar{u}} a)(U;x,\x)[\bar{h}]\,,
\qquad H=\vect{h}{\bar{h}}\,.
\]
We  consider the two problems
\begin{equation}\label{sistema01}
\left\{
\begin{aligned}
&\pa_{\tau}a^{+}(\tau,Z;x)=-2{\rm Re}\Big(C(\tau,Z;x)\ov{b^{+}(\tau,Z;x)}\Big)
%-b^{+}(\tau,Z;x)\ov{C(\tau,Z;x)}
-(d_{U}a^{+})(\tau,Z;x)\big[
\opbw\big({\bf C}(Z;x)\big)[Z]\big]\,,\\
&1+a^{+}(0,Z;x)=1+a(Z;x)\,,
\end{aligned}\right.
\end{equation}

\begin{equation}\label{sistema02}
\left\{
\begin{aligned}
&\pa_{\tau}b^{+}(\tau,Z;x)=-2(1+a^{+}(\tau,Z;x))C(\tau,Z;x)-(d_{U}b^{+})(\tau,Z;x)\big[
\opbw\big({\bf C}(Z;x)\big)[Z]\big]\,,\\
&b^{+}(0,Z;x)=b(Z;x)\,.
\end{aligned}\right.
\end{equation}
The following holds true.

\begin{theorem}\label{constEgohighoff}
Assume \eqref{low1highoff}. 
For $r>0$ small enough there exists 
a symbol $C(\tau,U;x)$ as in \eqref{simboCChighoff} such that 
the symbols defined by \eqref{sistema01}, \eqref{sistema02}
are such that
\begin{equation}\label{offdiag1}
a^{+}(\tau,U;x)\in \Sigma\mathcal{F}^{\mathbb{R}}_1[r,N]\,,\qquad 
b^{+}(\tau,U;x)\in \Sigma\mathcal{F}_1[r,N]
\end{equation}
with estimates uniform in $\tau\in[0,1]$.
Moreover one has
%there exists a symbol $C(\tau,U;x)\in\Sigma\mathcal{F}_1[r,N]$, 
%$\tau\in[0,1]$,
%such that
\begin{equation}\label{equaLowerhighoff}
F(C):=b^{+}(1,Z;x)\equiv0\,.
\end{equation}
%where $\Phi^{\tau}(u)$, $\tau\in[0,1]$ is the flow of \eqref{Ham111provaloweroff}.
\end{theorem}
The rest of the section is devoted to the proof of Theorem \ref{constEgohighoff}.
At the beginning we shall work with non-homogeneous symbols in the class 
$\mathcal{F}_1[r]$.

\subsubsection{Solutions of (\ref{sistema01}), (\ref{sistema02})}\label{esistenzaSOLS}

First of all we note that, if the symbol $C$ belongs to 
$ \mathcal{F}_1[r]$,  the flow in \eqref{Ham111provahighoff}
is well-posed by Theorem \ref{flussononlin} applied with generator 
as in \eqref{sim3}.
We denote by $\Psi^{\tau}_{C}$, $(\Psi_{C}^{\tau})^{-1}$
respectively the flow and the inverse flow of \eqref{Ham111provahighoff}.
Set 
\begin{equation}\label{defLungoFlussohighoff}
g_1(\tau,x):=a^{+}(\tau,Z(\tau);x)\,,\qquad
g_2(\tau,x):=b^{+}(\tau,Z(\tau);x)\,,\quad Z(\tau):=\Psi_{C}^{\tau}(U)\,.
\end{equation}
We have that (recall (\ref{sistema01}), (\ref{sistema02}))
\begin{equation}\label{sistema1022}
\begin{aligned}
\pa_{\tau}g_{1}(\tau,x)&=-2{\rm Re}\big(g_2(\tau,x)\ov{C(\tau,Z(\tau);x)}\big)\,,\qquad
g_1(0)=a(U;x)\,,\\
\pa_{\tau}g_{2}(\tau,x)&=-2(1+g_1(\tau))C(\tau,Z(\tau);x)\,,\qquad
g_2(0)=b(U;x)\,,
\end{aligned}
\end{equation}
which implies 
\begin{equation}\label{sistema1023}
\begin{aligned}
g_1(\tau)&:=a(U;x)-2\int_{0}^{\tau}{\rm Re}\big(g_2(\s,x)\ov{C(\s,Z(\s);x)}\big)d\s\,,\\
g_2(\tau)&:=b(U;x)-2\int_{0}^{\tau}(1+g_1(\s,x)){C(\s,Z(\s);x)}\big)d\s\,.
\end{aligned}
\end{equation}
By \eqref{defLungoFlussohighoff} we write $Z(\sigma)=\Psi_C^{\sigma}(U)=\Psi_C^{\sigma}\circ{(\Psi_C^{\tau})}^{-1}(Z)$ and by \eqref{sistema1023} we deduce that 
equations \eqref{sistema01}, \eqref{sistema02} are equivalent to
\begin{align}
a^{+}(\tau,Z;x)&:=a\big((\Psi_{C}^{\tau})^{-1}(Z);x\big)
-2\int_{0}^{\tau}{\rm Re}\Big(b^{+}(\s,\Psi_{C}^{\s}(\Psi_{C}^{\tau})^{-1}(Z) ;x)
\ov{C(\s,\Psi_{C}^{\s}(\Psi_{C}^{\tau})^{-1}(Z);x)}\Big)d\s\,,\label{sistema0101}\\
b^{+}(\tau,Z;x)&:=b\big((\Psi_{C}^{\tau})^{-1}(Z);x\big)
-2\int_{0}^{\tau}\big(1+a^{+}(\s,\Psi_{C}^{\s}(\Psi_{C}^{\tau})^{-1}(Z) ;x)\big)
C(\s,\Psi_{C}^{\s}(\Psi_{C}^{\tau})^{-1}(Z);x)d\s\,.\label{sistema0102}
\end{align}
To solves the \eqref{sistema0101}, \eqref{sistema0102} we reason as follows.
Let $a^{+}_{-1}=b^{+}_{-1}\equiv0$ and consider, for $n\geq0$, the following problems:
 \begin{align}
a_{n}^{+}(\tau,Z;x)&:=a\big((\Psi_{C}^{\tau})^{-1}(Z);x\big)
-2\int_{0}^{\tau}{\rm Re}\Big(b_{n-1}^{+}(\s,\Psi_{C}^{\s}(\Psi_{C}^{\tau})^{-1}(Z) ;x)
\ov{C(\s,\Psi_{C}^{\s}(\Psi_{C}^{\tau})^{-1}(Z);x)}\Big)d\s\,,\label{sistema0101n}\\
b_{n}^{+}(\tau,Z;x)&:=b\big((\Psi_{C}^{\tau})^{-1}(Z);x\big)
-2\int_{0}^{\tau}\big(1+a_{n-1}^{+}(\s,\Psi_{C}^{\s}(\Psi_{C}^{\tau})^{-1}(Z) ;x)\big)
C(\s,\Psi_{C}^{\s}(\Psi_{C}^{\tau})^{-1}(Z);x)d\s\,.\label{sistema0102n}
\end{align}
We have the following.

\begin{lemma}{\bf (Iterative Lemma).}\label{iterationOff}
Let $C(\tau,U;x)\in  \mathcal{F}_1[r]$
and assume that (recall \eqref{semi-norma-totale})
\[
|C|_{s}^{\mathcal{F}}=\sum_{\alpha+mk\leq s-s_0 }|C|_{\alpha,k}^{\mathcal{F},s}\leq \mathtt{C}\,,
\]
for some $\mathtt{C}>0$ (depending on $ s$) and $s_0\gg1$. Then for $r>0$ small enough
one has the following for any $n\geq 0$:

\smallskip
\noindent
${\bf (S1)}_n $ one has that (recall \eqref{sistema0101n}, \eqref{sistema0102n} )
$a_n^{+}\in \mathcal{F}_1^{\mathbb{R}}[r]$, $b_n^{+}\in \mathcal{F}_1[r]$.
Moreover 
there are constants $C_{\alpha,k}$, 
independent of $n$, increasing in $k$ such that
\begin{equation}\label{stima0101}
|a^{+}_{n}|_{\alpha,k}^{\mathcal{F},s}
+
|b^{+}_{n}|_{\alpha,k}^{\mathcal{F},s}\leq C_{\alpha,k}\,,\quad \alpha+mk\leq s-s_0\,.
\end{equation} 
Moreover the constants $C_{\alpha,k}$ depends on $s$ and
on the norms $|a|_{s}^{\mathcal{F}}$, $|b|_{s}^{\mathcal{F}}$, $|c|_{s}^{\mathcal{F}}$.

\smallskip
\noindent
${\bf (S2)}_n $ We have that
\begin{equation}\label{stima0102}
 |a^{+}_{n}-a^{+}_{n-1}|^{\mathcal{F}, s-1}_{\alpha,0}+
  |b^{+}_{n}-b^{+}_{n-1}|^{\mathcal{F}, s-1}_{\alpha,0}
  \leq 2^{-n}\,,\qquad \alpha\leq s-s_0-1\,.
  \end{equation}
\end{lemma}

\begin{proof}
Assume inductively that ${\bf (S1)}_{n-1} $, ${\bf (S2)}_{n-1} $ hold.
We prove the ${\bf (S1)}_n $, ${\bf (S2)}_n $ for the symbol $a_{n}^{+}$.
The proof of the properties for the symbol $b_{n}^{+}$ is the same.

\smallskip
\noindent
${\bf (S1)}_n $. Using the estimates \eqref{stimaflusso3} on the flow 
$(\Psi^{\tau}_{C})^{-1}$ one can prove 
 (using Faa di Bruno's formula) that the  term 
$a\big((\Psi_{C}^{\tau})^{-1}(Z);x\big)$ in 
\eqref{sistema0101n} satisfies the \eqref{stima0101}.
So we study the term
\[
d(\tau,\s,Z;x):=b_{n-1}^{+}(\s,\Psi_{C}^{\s}(\Psi_{C}^{\tau})^{-1}(Z) ;x)
\ov{C(\s,\Psi_{C}^{\s}(\Psi_{C}^{\tau})^{-1}(Z);x)}\,,
\]
providing estimates independent of $\s$ and $\tau$.
 We prove the result only for $\alpha=0$, $mk\leq s-s_0$. 
 First of all notice that 
 \begin{equation}\label{bottone}
 |b_{n-1}^{+}(\s,\Psi_{C}^{\s}(\Psi_{C}^{\tau})^{-1}(Z) ;x)|_{0,p}^{\mathcal{F},s}
 \lesssim_{s}
 \widetilde{\mathtt{C}}
 |b^{+}_{n-1}|_{0,p}\lesssim_{s} \widetilde{\mathtt{C}} C_{0,p}\,,\qquad p\leq k\,,
 \end{equation}
 where in the last inequality we used the inductive hypothesis.
 Here the constant $ \widetilde{\mathtt{C}}$ depends also on the constant appearing 
 in \eqref{stimaflusso3}.
 The first inequality can be deduced by reasoning 
exactly as in  \eqref{faadibruno}-\eqref{fattore:3} using \eqref{stimaflusso3}
and  recalling Remarks \ref{smallnessSemi}, \ref{prodottoSimboli} .
In the same way we have
 \[
 |C(\s,\Psi_{C}^{\s}(\Psi_{C}^{\tau})^{-1}(Z) ;x)|_{0,p}^{\mathcal{F},s}\lesssim_{s}
 \widetilde{\mathtt{C}}
 |C|_{0,p}\,,\qquad p\leq k\,.
 \]
Therefore
we have
\begin{equation*}
\begin{aligned}
|d|_{0,k}^{\mathcal{F},s}
&\lesssim_{s}\widetilde{\mathtt{C}} C_{0,k}|C|_{0,0}^{\mathcal{F},s}r+
\widetilde{\mathtt{C}} C_{0,k-1}|C|_{0,k}^{\mathcal{F},s}\leq C_{0,k}
\end{aligned}
\end{equation*}
for $C_{0,k}$
 suitably chosen and $r$ small enough. 
 This implies the \eqref{stima0101} with $\alpha=0$.
 For $\alpha>0$ one can reason in the same way.
% \comment{ La costante in \eqref{stima0101} non DEVE dipendere da $n$. 
% Se faccio il rozzo trovo (usando ip. induttiva)
% $a_n\sim b_{n-1}C\sim C_{0,k}|C|_{0,k}$. Ma allora per $|a_{n}|$
% ho una stima perggiore che per $b_n$.
% Questo non accade, 
% infatti $C_{k}$
% (tutte le derivate su $b$) compare scontrata con $r$. Invece le $C_{k-1}$ ovviamente no, perche' delle derivate stanno candendo anche su $C(\tau,Z;x)$. 
% quindi $C_{0,p}<C_{0,k}$ per $p<k$ ma non dipende da $n$.
% }

 \smallskip
\noindent
${\bf (S2)}_n $. We prove the result for $a_{n}^{+}$.
By \eqref{sistema0101n} we have
\[
a_{n}^{+}-a_{n-1}^{+}=
-2\int_{0}^{\tau}{\rm Re}\Big(\Big(b_{n-1}^{+}-b_{n-2}^{+}\Big)(\s,\Psi_{C}^{\s}
(\Psi_{C}^{\tau})^{-1}(Z) ;x)
\ov{C(\s,\Psi_{C}^{\s}(\Psi_{C}^{\tau})^{-1}(Z);x)}\Big)d\s\,.
\]
Reasoning as in \eqref{bottone} one gets
\[
|\Big(b_{n-1}^{+}-b_{n-2}^{+}\Big)(\s,\Psi_{C}^{\s}(\Psi_{C}^{\tau})^{-1}(Z) ;x)|_{\alpha,0}^{\mathcal{F},s}\lesssim_{s} |b_{n-1}^{+}-b_{n-2}^{+}|_{\alpha,0}^{\mathcal{F},s}
\stackrel{\eqref{stima0102}}{\lesssim_{s}} 2^{-(n-1)}\,.
\]
Therefore (recall Remark \ref{prodottoSimboli})
\[
|a_{n}^{+}-a_{n-1}^{+}|_{\alpha,0}^{\mathcal{F},s}\lesssim_{s}|C|_{\alpha,0}^{\mathcal{F},s}r2^{-(n-1)}\leq
2^{-n}
\]
for $r$ small enough. This is the \eqref{stima0102} at the step $n$.
 \end{proof}

 By Lemma \ref{iterationOff} and Remark \ref{frechet} we have, 
 up to extracting a subsequence, 
 a sequence of solutions of \eqref{sistema0101n}, \eqref{sistema0102n}
such that
\begin{equation*}
\begin{aligned}
&a^{+}_n\rightharpoonup^* \widetilde{a^{+}} \mbox{ in } \mathcal{F}_{1}[r]\,, \,\,\, 
\partial_x^{\alpha}a^{+}_n\rightarrow\partial_{x}^{\alpha} a^{+} 
\mbox{ uniformly in $x_0$ for } \alpha\leq s-s_0-1\,,\\
&b^{+}_n\rightharpoonup^* \widetilde{b^{+}} 
\mbox{ in } \mathcal{F}_{1}[r], \,\,\, 
\partial_{x}^{\alpha}b^{+}_n\rightarrow\partial_{x}^{\alpha} 
b^{+} \mbox{ uniformly in $x_0$ for } \alpha\leq s-s_0-1\,.
\end{aligned}
\end{equation*}
Therefore we deduce $\widetilde{a^{+}} ={a^{+}} $, $\widetilde{b^{+}} ={b^{+}} $.
%Then the \eqref{offdiag1} is proved. 

\subsubsection{The choice of the generator}\label{choice314}
We want to exhibit a symbol $C(\tau,U;x)$ in such a way 
%It remains to show the 
\eqref{equaLowerhighoff} holds true.
We define the following 
\begin{equation*}
\begin{aligned}
\mathcal{G}_{C}(\tau,U)&:=b\big(U;x\big)
-2\int_{0}^{\tau}\big(1+a(U;x)\big)
C(\s,\Psi_{C}^{\s}(U);x)d\s\\
&
-4\int_{0}^{\tau}
\int_0^{\s}{\rm Re}\Big(b^{+}(\theta,\Psi_{C}^{\theta}(U);x)\ov{
C(\theta,\Psi_{C}^{\theta}(U);x)
}\Big) 
C(\s,\Psi_{C}^{\s}(U);x)d\theta d\s\,.
\end{aligned}
\end{equation*}
Then, by using  \eqref{sistema0101} and  \eqref{sistema0102}, we note that
\begin{equation*}
b^{+}(\tau,Z;x)=\mathcal{G}_{C}(\tau,(\Psi_{C}^{\tau})^{-1}(Z))\,.
\end{equation*}
We look for $C(\tau,U;x)$ such that (see \eqref{equaLowerhighoff})
\begin{equation*}
b^{+}(1,Z;x)=\mathcal{G}_{C}(1,(\Psi_{C}^{1})^{-1}(Z))=0\,.
\end{equation*}
We reason as follows.
Given functions $f,g,h\in \mathcal{F}_1[r]$
we define the functional
\begin{equation}\label{def:funzGG}
\begin{aligned}
\mathcal{T}(f,g,h)&:=
b\big(U;x\big)
-2\int_{0}^{1}\big(1+a(U;x)\big)
f(\s,\Psi_{g}^{\s}(U);x)d\s\\
&
-4\int_{0}^{1}
\int_0^{\s}{\rm Re}\Big(h(\theta,\Psi_{g}^{\theta}(U);x)\ov{
g(\theta,\Psi_{g}^{\theta}(U);x)
}\Big) 
f(\s,\Psi_{g}^{\s}(U);x)d\theta d\s\,.
\end{aligned}
\end{equation}
According to this notation we have
\[
\mathcal{G}_{C}(1,U)\equiv \mathcal{T}(C,C,b^{+})\,.
\]
We set $C_0=0$ and 
for any $n\geq 1$ we consider  the symbol $C_{n}(\tau,U;x,\x)$ as the solution of
\begin{equation}\label{def:CN}
\mathcal{T}(C_{n},C_{n-1},b^{+}_{n-1})=0
\end{equation}
where $b_{n-1}^{+}$ is defined as 
\begin{equation}\label{def:GGG5}
b^{+}_{n-1}(\tau,Z;x)=\mathcal{G}_{C_{n-1}}(\tau,(\Psi_{C_{n-1}}^{\tau})^{-1}(Z))\,.
\end{equation}
We shall prove the following lemma.
\begin{lemma}\label{bastaIterazioni}
 For $r>0$ small enough
one has the following for any $n\geq 1$:

\smallskip
\noindent
${\bf (S1)}_n $ The symbols 
 $C_{n}$
 defined by \eqref{def:CN} belong to
$\mathcal{F}_{1}[r]$. 
In particular there are constants $C_{\alpha,k}$, independent of $n$, 
increasing in $k$ such that
\begin{equation}\label{stima0101GGG}
|C_{n}|_{\alpha,k}^{\mathcal{F},s}
\leq C_{\alpha,k}\,,\quad \alpha+mk\leq s-s_0\,.
\end{equation} 
Moreover the constants $C_{\alpha,k}$ depends on $s$ and
on the norms $|a|_{s}^{\mathcal{F}}$, $|b|_{s}^{\mathcal{F}}$.

\smallskip
\noindent
${\bf (S2)}_n $ We have that
\begin{equation}\label{stima0102CC}
 |C_{n}-C_{n-1}|^{\mathcal{F}, s-1}_{\alpha,0}
 \leq 2^{-n}\,,\qquad \alpha\leq s-s_0-1\,.
  \end{equation}
\end{lemma}

In order to prove the above Lemma we need some preliminary results.
\begin{lemma}\label{diffeFlussiZhighoff}
Assume that $C_{n},C_{n-1}\in \mathcal{F}_1[r]$.
For $r>0$ small enough
\begin{equation*}
\sup_{\tau\in[0,1]}\|\Psi^{\tau}_{C_{n}}-\Psi_{C_{n-1}}^{\tau}\|_{H^{s}}
\lesssim_{s}
r\sup_{\tau\in[0,1]}|C_{n}-C_{n-1}|_{0,0}^{\mathcal{F},s}\,,
\end{equation*}
where $\Psi_{C_{j}}^{\tau}$, $j=n,n-1$ is the solution 
of \eqref{Ham111provahighoff} with $C_{j}$ in place of $C$.
\end{lemma}

\begin{proof}
We set $v_{n}(\tau):=Z_{n}(\tau)-Z_{n-1}(\tau)$. 
%By \eqref{Ham111provahighoff} w
We get
\[
\begin{aligned}
\pa_{\tau}v_{n}&=\ii \opbw({\bf C}_{n}(Z_{n};x))[v_{n}]
+\ii \opbw({\bf C}_{n}(Z_n;x)-{\bf C}_{n}(Z_{n-1};x))[Z_{n-1}]\\
&+\ii \opbw\big(({\bf C}_{n}-{\bf C}_{n-1})(Z_{n-1};x)\big)[Z_{n-1}]\,,
\end{aligned}
\]
where ${\bf C}_{n}$ is the matrix in \eqref{Ham111provahighoff} with $C\rightsquigarrow C_{n}$.
Since the generator is bounded one can reason essentially 
as in the proof of the estimates for $X_{n}, Y_{n}$ in the proof of Lemma  \ref{diffeFlussiZ}.
\end{proof}

\begin{proof}[{\bf Proof of Lemma \ref{bastaIterazioni}}]
We argue by induction.
Assume that ${\bf (S1)}_{n-1} $, ${\bf (S2)}_{n-1} $ hold.

\noindent
\emph{Proof of} ${\bf (S1)}_{n}$.
Since $C_{n-1}$ is a symbol satisfying \eqref{stima0101GGG},
then by subsection \ref{esistenzaSOLS} we have that the symbol
$b_{n-1}^{+}(\tau,U;x)$ defined by \eqref{def:GGG5}
is in $\mathcal{F}_{1}[r]$
with estimates 
\begin{equation*}
|b_{n-1}^{+}|_{\alpha,k}^{\mathcal{F},s}
\leq C_{\alpha,k}\,,\quad \alpha+mk\leq s-s_0\,,
\end{equation*} 
for some constants $C_{\alpha,k}$ depending only 
on the norms $|a|_{s}^{\mathcal{F}}$, $|b|_{s}^{\mathcal{F}}$ and 
$|C_{n-1}|_{s}^{\mathcal{F}}$.
Then, by reasoning as in the proof of \eqref{stima0101}
and using \eqref{def:CN},
we obtain the \eqref{stima0101GGG} for $C_{n}$.

\smallskip
\noindent
\emph{Proof of} ${\bf (S2)}_{n}$.
By \eqref{def:funzGG} we have
\begin{equation}\label{def:hh1}
\begin{aligned}
0&=\mathcal{T}(C_{n},C_{n-1},b^{+}_{n-1})-\mathcal{T}(C_{n-1},C_{n-2},b^{+}_{n-2})=\\
&=
-2\int_{0}^{1}\big(1+a(U;x)\big)\Big(
C_{n}(\s,\Psi_{C_{n-1}}^{\s}(U);x)-C_{n-1}(\s,\Psi_{C_{n-1}}^{\s}(U);x)\Big)d\s
\\
&
-2\int_{0}^{1}\big(1+a(U;x)\big)\Big(
C_{n-1}(\s,\Psi_{C_{n-1}}^{\s}(U);x)-C_{n-1}(\s,\Psi_{C_{n-2}}^{\s}(U);x)\Big)d\s\\
&
-4\int_{0}^{1}
\int_0^{\s}{\rm Re}\Big(b_{n-1}^{+}(\theta,\Psi_{C_{n-1}}^{\theta}(U);x)\ov{
C_{n-1}(\theta,\Psi_{C_{n-1}}^{\theta}(U);x)
}\Big) 
C_{n}(\s,\Psi_{C_{n-1}}^{\s}(U);x)d\theta \\
&
+4\int_{0}^{1}
\int_0^{\s}{\rm Re}\Big(b_{n-2}^{+}(\theta,\Psi_{C_{n-2}}^{\theta}(U);x)\ov{
C_{n-2}(\theta,\Psi_{C_{n-2}}^{\theta}(U);x)
}\Big) 
C_{n-1}(\s,\Psi_{C_{n-2}}^{\s}(U);x)d\theta \,.
\end{aligned}
\end{equation}
First of all notice that, using Lemma \ref{diffeFlussiZhighoff}
and estimates \eqref{stimaflusso3} 
on the flow $\Psi^{\tau}_{C_{j}}$,
\begin{equation}\label{def:hh2}
\begin{aligned}
|C_{n-1}(\s,\Psi_{C_{n-1}}^{\s}(U);x)
-&C_{n-1}(\s,\Psi_{C_{n-2}}^{\s}(U);x)|_{\alpha,0}^{\mathcal{F},s-1}
\lesssim_{s}
|C_{n-1}|_{s}^{\mathcal{F}}r |C_{n-1}-C_{n-2}|^{\mathcal{F}}_{s-1,0}\,,\\
|C_{n-1}(\s,\Psi_{C_{n-1}}^{\s}(U);x)
-&C_{n-2}(\s,\Psi_{C_{n-1}}^{\s}(U);x)|_{\alpha,0}^{\mathcal{F},s-1}
\lesssim_{s}
|C_{n-1}|_{s}^{\mathcal{F}}r |C_{n-1}-C_{n-2}|^{\mathcal{F}}_{s-1,0}\,.
\end{aligned}
\end{equation}
One can prove the following inequality
\begin{equation}\label{claimBBn}
|b_{n-1}^{+}-b_{n-2}^{+}|_{\alpha,0}^{\mathcal{F}}\lesssim_{s}\mathtt{C}
|C_{n-1}|_{s}^{\mathcal{F}}r |C_{n-1}-C_{n-2}|^{\mathcal{F}}_{s-1,0}\,,
\end{equation}
for some $\mathtt{C}>0$ depending on $|a|^{\mathcal{F}}_{s}$, $|b|^{\mathcal{F}}_{s}$.
In order to prove this fact one has to use the definition \eqref{def:GGG5}, triangular inequality, \eqref{def:hh2} and the smallness of $r$.
\noindent
By \eqref{def:hh1}, \eqref{def:hh2}, \eqref{claimBBn},  for $r>0$ small enough we deduce the
\eqref{stima0102CC}.
\end{proof}
We have the following.
\begin{lemma}\label{lem:GGG}
One has
$\mathcal{T}(C_{n},C_{n},b^{+}_{n})\to0$ as  $n\to\infty$
in the norm $|\cdot|_{\alpha,0}^{\mathcal{F},s-1}$.
\end{lemma}

\begin{proof}
We write
\begin{equation*}
\begin{aligned}
\mathcal{T}(C_{n},C_{n},b^{+}_{n})&=\mathcal{T}(C_{n},C_{n-1},b^{+}_{n-1})+
\mathcal{T}(C_{n},C_{n},b^{+}_{n})-\mathcal{T}(C_{n},C_{n-1},b^{+}_{n-1})\nonumber\\
&\stackrel{\eqref{def:CN}}{=}
\mathcal{T}(C_{n},C_{n},b^{+}_{n})-\mathcal{T}(C_{n},C_{n-1},b^{+}_{n-1})\,.
\end{aligned}\end{equation*}
By reasoning as in the proof
of ${\bf (S2)}_{n}$
in  Lemma \ref{bastaIterazioni} one can check
\[
|\mathcal{T}(C_{n},C_{n},b^{+}_{n})-\mathcal{T}(C_{n},C_{n-1},b^{+}_{n-1})|_{\alpha,0}^{\mathcal{F},s-1}
\lesssim_{s}2^{-n}
\] 
for $r>0$ small enough. Hence the result follows.
\end{proof}
The discussion above provides symbols $a^{+}(\tau,U;x), b^{+}(\tau,U;x), C(\tau,U;x) $,
in the class $\mathcal{F}_1[r]$ such that the \eqref{offdiag1}, \eqref{equaLowerhighoff} 
hold. Actually, reasoning as in the conclusion of the proof of Theorem \ref{constEgo},
one can prove that such symbols belong to the class
$\Sigma\mathcal{F}_1[r,N]$.

\subsubsection{Non-homogeneous problem}

In section \ref{sec:BlockDiago} we shall deal with the following problem.
We fix $m'<m$ and we consider symbols 
\begin{equation}\label{ipo523}
a_{m'}(U;x,\x)\,,\;b_{m'}(U;x,\x)\,, \;q^{(1)}(\tau,U;x,\x)\,,\; 
q^{(2)}(\tau,U;x,\x) \in \Sigma\Gamma^{m'}_1[r,N]\,,\quad \tau\in[0,1]
%\;\;\;m'<m\,,
\end{equation}
and the two problems
\begin{equation}\label{sistema01bis}
\begin{aligned}
\pa_{\tau}a_{m'}^{+}(\tau,Z;x,\x)&=-
2{\rm Re}\Big(C(\tau,Z;x)\ov{b_{m'}^{+}(\tau,Z;x,\x)}\Big)
+q^{(1)}(\tau,Z;x,\x)\\
&\qquad\qquad\qquad-(d_{U}a_{m'}^{+})(\tau,Z;x,\x)\big[
\opbw\big({\bf C}(Z;x)\big)[Z]\big]\,,\\
a_{m'}^{+}(0,Z;x,\x)&=a_{m'}(Z;x,\x)\,,
\end{aligned}
\end{equation}

\begin{equation}\label{sistema02bis}
\begin{aligned}
\pa_{\tau}b_{m'}^{+}(\tau,Z;x,\x)&
=-\Big(a_{m'}^{+}(\tau,Z;x,\x)+\ov{a_{m'}^{+}(\tau,Z;x,-\x)}\Big)C(\tau,Z;x)
+q^{(2)}(\tau,Z;x,\x)
\\&\qquad\qquad\qquad-(d_{U}b_{m'}^{+})(\tau,Z;x,\x)\big[
\opbw\big({\bf C}(Z;x)\big)[Z]\big]\,,\\
b_{m'}^{+}(0,Z;x,\x)&=b_{m'}(Z;x,\x)\,,
\end{aligned}
\end{equation}
where $C$ is in \eqref{simboCChighoff}. %\eqref{Ham111provahighoff}
The following holds true.

\begin{proposition}\label{constEgohighoffbis}
Assume \eqref{ipo523} and \eqref{simboCChighoff}. 
For $r>0$ small enough 
the symbols defined by \eqref{sistema01bis}, \eqref{sistema02bis}
are such that
\begin{equation*}
a^{+}_{m'}(\tau,U;x,\x)\,,\; b_{m'}^{+}(\tau,U;x,\x)\in \Sigma\Gamma^{m'}_1[r,N]\,,
\end{equation*}
with estimates uniform in $\tau\in[0,1]$.
\end{proposition}

\begin{proof}
We denote by $\Psi^{\tau}_{C}$, $(\Psi_{C}^{\tau})^{-1}$
respectively the flow and the inverse flow of 
\eqref{Ham111provahighoff} with symbol $C(\tau,U;x)$ in \eqref{simboCChighoff}
and we set
\begin{equation*}
g_1(\tau,x,\x):=a_{m'}^{+}(\tau,Z(\tau);x,\x)\,,\qquad
g_2(\tau,x,\x):=b_{m'}^{+}(\tau,Z(\tau);x,\x)\,,\quad Z(\tau):=\Psi_{C}^{\tau}(U)\,.
\end{equation*}
We have that (recall (\ref{sistema01bis}), (\ref{sistema02bis}))
\begin{equation*}
\begin{aligned}
\pa_{\tau}g_{1}(\tau,x,\x)&=-2{\rm Re}\big(g_2(\tau,x,\x)\ov{C(\tau,Z(\tau);x)}\big)
+q^{(1)}(\tau,Z(\tau);x,\x)\,,\\
\pa_{\tau}g_{2}(\tau,x,\x)&=-(g_1(\tau,x,\x)+\ov{g_1}(\tau,x,-\x))C(\tau,Z(\tau);x)+
q^{(2)}(\tau,Z(\tau);x,\x)\,,\\
g_1(0,x,\x)&=a_{m'}(U;x,\x)\,,\quad 
g_2(0,x,\x)=b_{m'}(U;x,\x)\,,
\end{aligned}
\end{equation*}
which implies 
\begin{equation}\label{sistema1023bis}
\begin{aligned}
g_1(\tau,x,\x)&:=a(U;x,\x)
-\int_{0}^{\tau}2{\rm Re}\big(g_2(\s,x,\x)\ov{C(\s,Z(\s);x)}\big)-q^{(1)}(\s,Z(\s);x,\x) 
  d\s\,,\\
g_2(\tau,x,\x)&:=b(U;x,\x)-\int_{0}^{\tau}\big(g_1(\s,x,\x)+\ov{g_1(\s,x,-\x)}\big){C(\s,Z(\s);x)}\big)
-q^{(2)}(\s,Z(\s);x,\x) 
d\s\,.
\end{aligned}
\end{equation}
By \eqref{sistema1023bis} we deduce that 
equations \eqref{sistema01bis}, \eqref{sistema02bis} are equivalent to

\begin{align}
a_{m'}^{+}(\tau,Z;x)&:=
a_{m'}\big((\Psi_{C}^{\tau})^{-1}(Z);x,\x\big)
-2\int_{0}^{\tau}{\rm Re}\Big(b_{m'}^{+}(\s,\Psi_{C}^{\s}(\Psi_{C}^{\tau})^{-1}(Z) ;x,\x)
\ov{C(\s,\Psi_{C}^{\s}(\Psi_{C}^{\tau})^{-1}(Z);x)}\Big)d\s\nonumber\\
&+\int_{0}^{\tau}q^{(1)}(\s,\Psi_{C}^{\s}(\Psi_{C}^{\tau})^{-1}(Z) ;x,\x)d\s\,,
\label{sistema0101bis}
\end{align}

\begin{align}
b_{m'}^{+}(\tau,Z;x,\x)&:=b_{m'}\big((\Psi_{C}^{\tau})^{-1}(Z);x,\x\big)\nonumber
\\&-\int_{0}^{\tau}\big(a_{m'}^{+}(\s,\Psi_{C}^{\s}(\Psi_{C}^{\tau})^{-1}(Z) ;x,\x)
+\ov{a_{m'}^{+}(\s,\Psi_{C}^{\s}(\Psi_{C}^{\tau})^{-1}(Z) ;x,-\x)}
\big)
C(\s,\Psi_{C}^{\s}(\Psi_{C}^{\tau})^{-1}(Z);x)d\s \nonumber\\
&+\int_{0}^{\tau}q^{(2)}(\s,\Psi_{C}^{\s}(\Psi_{C}^{\tau})^{-1}(Z) ;x,\x)d\s\,.\label{sistema0102bis}
\end{align}
Equations \eqref{sistema0101bis}, \eqref{sistema0102bis}
differs form 
 \eqref{sistema0101}, \eqref{sistema0102} for the non-homogeneous terms 
 depending on $q^{(i)}$, $i=1,2$. 
 In order to conclude the proof of the proposition
 it is sufficient to follow almost word by word the proof 
 of Lemma \ref{iterationOff} in subsection \ref{esistenzaSOLS}.
\end{proof}

%\section{Preliminary results for lower order terms}\label{sec:Prelower}

\subsection{Diagonal terms at lower orders}
In the following we consider two non-linear equations 
which will appear in sections \ref{blokkoLower},
\ref{diagoLower}.
Consider symbols 
\begin{align}
&f(\x)\in \Gamma_0^{m}\,,\;\;\; m>1\,,\;\;\;f(\x)\in\mathbb{R} \quad 
a(u;x,\x)\in \Sigma\Gamma^{m'}_1[r,N]\,,\;\;\; m'<m\,, \label{low1}\\
&\mathfrak{m}(\tau,u)\in \Sigma\mathcal{F}^{\mathbb{R}}_1[r,N]\,,\quad {\rm independent \,\, of}\,\, 
x\in \mathbb{T}\,,
\label{low2}\\
&a(u;x,\x)-\ov{a(u;x,\x)}\in \Sigma\Gamma^{0}_1[r,N]\,.\label{low3}
\end{align}
Assume also that 
$\mathfrak{m}(\tau,u)$ satisfies the estimates 
\eqref{pomosimbo1}-\eqref{maremma2} uniformly in $\tau\in[0,1]$.
%and that, if $0<m'<m$, the symbol $a(u;x,\x)$ is real valued.
Let 
\begin{equation}\label{simboCC}
c(\tau,w;x,\x)\in \Sigma\Gamma^{\delta}_1[r,N]\,, \;\;\; \tau\in[0,1]\,, \;\;\; \delta:=m'-m+1\,,
\qquad
c(u;x,\x)-\ov{c(u;x,\x)}\in \Sigma\Gamma^{0}_1[r,N]\,.
\end{equation}
%Assume that %if $\delta>0$ then 
%$c(\tau,w;x,\x)$ satisfies  the same condition of $a(u;x,\x)$ in  \eqref{low3}
%is real valued 
%and c
Consider the equation
 \begin{equation}\label{Ham111provalower}
\left\{
\begin{aligned}
&\pa_{\tau}z(\tau)=\opbw\big(\ii c(\tau,z(\tau);x,\x)\big)[z(\tau)]\,,\\
&z(0)=z_0\,.
\end{aligned}\right.
\end{equation}
The aim of the section is to prove the following.

\begin{theorem}\label{constEgolower}
Assume \eqref{low1}, \eqref{low2}, \eqref{low3}.
For $r>0$ small enough there exist a 
symbol $c(\tau,w;x,\x)$
as in \eqref{simboCC}
% belonging to $\Sigma\Gamma^{\delta}_1[r,N]$, 
%$\tau\in[0,1]$, $\delta=m'-m+1$ 
and a symbol $\mathfrak{C}(w,\x)\in \Sigma\Gamma^{m'}_1[r,N]$,
 independent of $x\in \mathbb{R}$,
such that the following holds.
One has
\begin{equation}\label{equaLower}
\begin{aligned}
F(c)&:=a(u;x,\x)-\int_{0}^{1}(\pa_{\x}f)(\x)\big(1+\mathfrak{m}(\s,\Phi_{c}^{\s}(u))\big)(\pa_{x}c)
(\s,\Phi_{c}^{\s}(u);x,\x)d\s\\
&=\mathfrak{C}(u,\x)
:=\frac{1}{2\pi}\int_{\mathbb{T}}a(u;x,\x)dx\,,
\end{aligned}
\end{equation}
where $\Phi^{\tau}(u)$, $\tau\in[0,1]$ is the flow of \eqref{Ham111provalower}.
Moreover
%if $\delta>0$ then 
the symbol %$c(\tau,w;y)$ and 
$\mathfrak{C}(w,\x)$ 
%are real valued. 
satisfy the same condition of $a(u;x,\x)$ in  \eqref{low3}.
\end{theorem}

%For the moment we consider non-homogeneous symbols.
By following the strategy adopted in the previous sections, it is sufficient to 
prove the result in the case of non-homogeneous symbols.
We reason inductively. Let $c_0=0$ and for any $n\geq1$
consider the problem
\begin{equation}\label{flowlower}
\pa_{\tau}z_n(\tau)=\ii \opbw(c_{n}(\tau,z_{n}(\tau);x,\x))[z_{n}(\tau)]\,,\quad z_{n}(0)=z_0\,,
\qquad z_{n}(\tau):=\Phi_{c_n}^{\tau}(z_0)\,,
\end{equation}
where the symbol $c_n(\tau,w;x,\x)$ is defined as
\begin{equation}\label{def:Cnlow}
c_{n}(\tau,w;x,\x):=
\pa_{x}^{-1}\left(
\frac{a((\Phi^{\tau}_{c_{n-1}})^{-1}(w);x,\x)-\mathfrak{C}((\Phi^{\tau}_{c_{n-1}})^{-1}(w);\x)}{\big(1
+\mathfrak{m}(\tau,w)\big)(\pa_{\x}f)(\x)}
\right)\,.
\end{equation}
%\begin{equation}\label{def:Cnlow2}
%\mathfrak{M}_{n}(\tau,w;\x):=\frac{1}{2\pi}\int_{\mathbb{T}}a\big(
%(\Phi^{\tau}_{c_{n-1}})^{-1}(w);x,\x
%\big)dx\,.
%\end{equation}
First of all notice that the constant $\mathfrak{C}(u;\x)$ defined in \eqref{equaLower}
belongs to $\Gamma^{m'}_1[r]$ and
\begin{equation}\label{stimeSeminormeMlow}
|\mathfrak{C}|_{0,\beta,k}^{\Gamma^{m'},s}\lesssim_{s} %C_1\,,
|{a}|_{0,\beta,k}^{\Gamma^{m'},s}\,,
\qquad mk\leq s-s_0\,.
\end{equation}

For any $n\geq 1$ 
we shall prove inductively that the following
conditions hold.

\begin{itemize}

\item[$({\bf S1})_{n}$] 
One has that 
%$\mathfrak{M}_{n}$ in \eqref{def:Cnlow2} 
%belongs to $\Gamma_{1}^{m'}[r]$ and is independent of $x$; 
the symbol $c_n$ in \eqref{def:Cnlow}
belongs to  $\Gamma_1^{m'-m+1}$. In particular  (see \eqref{seminormagamma})
there exists  constant $\mathtt{C}$ (independent of $n$)
depending only on $s$, $|a|_{\alpha,\beta,k}^{\Gamma^{m'},s}$ and
$|\mathfrak{m}|_{0,\beta,k}^{\Gamma^{0},s}$
such that
\begin{align}
|c_{n}|_{\alpha,\beta,k}^{\Gamma^{\delta},s}&\lesssim_{s}\mathtt{C}\,,
%|{a}|_{\alpha,\beta,k}^{\Gamma^{\delta},s}+|{a}|_{\alpha,\beta,k}^{\Gamma^{\delta},s}
%|\mathfrak{M}_n|_{\alpha,\beta,k}^{\Gamma^{\delta},s}\,,
\qquad \alpha+mk\leq s-s_0
\label{stimeSeminormeMlow2}\,,\\
|c_{n}|_{\alpha,\beta,0}^{\Gamma^{\delta},s+1}&\lesssim_{s}
|{a}|_{\alpha,\beta,0}^{\Gamma^{m'},s}\,,
\qquad \alpha\leq s+1-s_0
\label{stimeSeminormeMlow3}\,.
\end{align}
%If $\delta>0$ then 
The symbol $c_n$ satisfies
the same condition of $a(u;x,\x)$ in  \eqref{low3}.

\item[$({\bf S2})_{n}$] 
 The flow of 
 \eqref{flowlower}
 with $c_n$ given by \eqref{def:Cnlow} 
 is well-posed, has the form
\begin{equation}\label{nota1nlow}
\begin{aligned}
z_{n}(\tau)&=\Phi_{c_{n}}^{(z)}(\tau,z_0)\in  \cap_{k=0}^{K}C^{k}([0,1];H^{s-k})\,.
\end{aligned}
\end{equation}
\end{itemize}

We argue by induction. So we assume that
$({\bf Sk})_{j}$, for $k=1,2$, hold with $0\leq j\leq n-1$.

\begin{proof}[{\bf Proof of} $({\bf S1})_{n}$.]
By the inductive 
hypothesis $c_{n-1}$ is a symbol satisfying 
\eqref{simboCC}.
%in $\Gamma^{\delta}_1[r]$
%satisfying a condition like \eqref{}
% and
%if $\delta>0$ then symbol $c_{n-1}$ is real valued. 
Then we can apply Theorem \ref{flussononlin} (with generator as in \eqref{sim2} or \eqref{sim3}).
So 
the flow of \eqref{Ham111provalower} with $c\rightsquigarrow c_{n-1}$ is well-posed.
Therefore, using the formula of Faa di Bruno and  the
\eqref{stimaflusso3},
one can check 
\[
| {a}(({\Phi}^{\tau})^{-1}_{c_{n-1}}(w),x,\x) |^{\Gamma^{m'},s}_{\alpha,\beta,k}
\lesssim_{s}
|{a}|^{\Gamma^{m'},s}_{\alpha,\beta,k}\,.
\]
Moreover 
\begin{equation}\label{penna}
a((\Phi^{\tau}_{c_{n-1}})^{-1}(w);x,\x)-\mathfrak{C}((\Phi^{\tau}_{c_{n-1}})^{-1}(w);\x)
\end{equation}
has zero average in $x\in \mathbb{T}$. Indeed the symbol in \eqref{penna}
is nothing but the symbol
\[
a(u;x,\x)-\mathfrak{C}(u;\x)=
a(u;x,\x)-\frac{1}{2\pi}\int_{\mathbb{T}}a(u;x,\x)dx
\]
evaluated at $u=(\Phi^{\tau}_{c_{n-1}})^{-1}(w)$. 
Therefore $\pa_{x}^{-1}$ in \eqref{def:Cnlow} is well defined.
%Finally, by   \eqref{stimeSeminormeMlow}, we get
%the \eqref{stimeSeminormeMlow}.
Finally,  using \eqref{stimeSeminormeMlow}, \eqref{def:Cnlow} 
and Remark \ref{prodottoSimboli},
we get the \eqref{stimeSeminormeMlow2}.
Taking into account 
the smoothing effect of the 
Fourier multiplier $\pa_{x}^{-1}$, Remarks \ref{prodottoSimboli}, 
\ref{smallnessSemi}, we obtain 
the \eqref{stimeSeminormeMlow3}
 for $r$ small enough.
 By formula \eqref{def:Cnlow} and hypothesis
 \eqref{low3} we deduce that $c_n$ satisfies 
 the same condition of $a$ in  \eqref{low3}.
\end{proof}

\begin{proof}[{\bf Proof of} $({\bf S2})_{n}$.]
Thanks to \eqref{stimeSeminormeMlow2}, \eqref{stimeSeminormeMlow3}
the \eqref{nota1nlow} follows by
Theorem \ref{flussononlin}.
%If $\delta>0$ then we must have $m'>m-1>0$. Therefore, by assumption in \eqref{low1},
%the symbol $a$ is real valued. Hence also $c_n$ defined in \eqref{def:Cnlow} is real valued.
\end{proof}

\noindent
{\bf Inizialization.} The $({\bf S1})_{0}$, $({\bf S2})_{0}$ are trivial.

In order to conclude the proof of Theorem \ref{constEgolower}
we have check the \eqref{equaLower}.
We need some preliminary results which are consequences of 
 $({\bf S1})_{n}$, $({\bf S2})_{n}$.

\begin{lemma}\label{diffeFlussiZlow}
%For any $n\geq 1$, if $b_{n},b_{n-1}\in \mathcal{F}_1[r]$ with 
For $r>0$ small enough
%\begin{equation}\label{piccoloRRRlow}
%rC_s\sup_{\tau\in[0,1]}\Big(|c_n|^{\Gamma^{\delta},s}_{\alpha,\beta,k}+|c_{n-1}|^{\mathcal{F}}_{s}\Big)\ll 1\,,
%\end{equation}
%for some $C_{s}\gg1$ 
we have %(recall \eqref{nota31}-\eqref{nota33})
\begin{equation}\label{stimaNNlow}
\sup_{\tau\in[0,1]}\|\Phi^{\tau}_{c_{n}}-\Phi_{c_{n-1}}^{\tau}\|_{H^{s-1}}
%\Big(|Z_{n}|^{\mathcal{F}}_{s}+|X_{n}|^{\mathcal{F}}_{s}
%+|Y_{n}|^{\mathcal{F}}_{s}\Big)
\lesssim_{s}
r\sup_{\tau\in[0,1]}|c_{n}-c_{n-1}|_{0,0,0}^{\Gamma^{\delta}}\,,
\end{equation}
where $\Phi_{c_{j}}^{\tau}$, $j=n,n-1$ is the solution of 
\eqref{flowlower}.
%with $c_{j}$ in place of $c$.
\end{lemma}

\begin{proof}
We set $v_{n}(\tau):=z_{n}(\tau)-z_{n-1}(\tau)$. By \eqref{flowlower} we get
\[
\begin{aligned}
\pa_{\tau}v_{n}&=\ii \opbw(c_{n}(z_{n};x,\x))[v_{n}]
+\ii \opbw(c_{n}(z_n;x,\x)-c_{n}(z_{n-1};x,\x))[z_{n-1}]\\
&+\ii \opbw\big((c_{n}-c_{n-1})(z_{n-1};x,\x)\big)[z_{n-1}]\,.
\end{aligned}
\]
Let us consider the case $\delta>0$ which is the most difficult one. 
If $\delta\leq 0$ one can reason 
as in the proof of the estimates for $X_{n}, Y_{n}$ 
in the proof of Lemma  \ref{diffeFlussiZ} (see  \eqref{stimaincredibile}).

Using that $c_{n}(z_{n};x,\x)$ is real valued, Propositions 
\ref{AzioneParaMet},\ref{teoremadicomposizione}, 
the \eqref{espansione2}, Remark \ref{smallnessSemi},  and \eqref{Sobnorm2}, 
we get, for any $s\in \mathbb{R}$, 
\begin{equation}\label{picture5lower}
\begin{aligned}
\pa_{\tau}\|v_{n}\|_{H^{s}}^{2}&\lesssim_{s}
\|v_{n}\|_{H^{s}}^{2}\sup_{\tau\in[0,1]}\sum_{\alpha\leq s-s_0}|c_{n}|_{\alpha,0,0}^{\Gamma^{\delta},s}
r\\
&+\|z_{n-1}\|_{H^{s+1}}\|v_{n}\|_{H^{s}}\Big(
|c_{n}|_{s_0,0,1}^{\Gamma^{\delta},s}\|v_{n}\|_{H^{s_0}}+
|c_{n}-c_{n-1}|_{0,0,0}^{\Gamma^{\delta},s}
\Big)\,.
\end{aligned}
\end{equation}
We now recall that $\|z_{n-1}\|_{H^{s}}\lesssim_{s}r$. We use formula \eqref{picture5lower}
with $s\rightsquigarrow s-1$. 
Then, reasoning as in \eqref{picture11}, we deduce
that if
\[
r\sup_{\tau}\left(|c_{n}|^{\Gamma^{\delta}}_{s-s_0,0,0}+
|c_{n}|^{\Gamma^{\delta}}_{s_0,0,1}
\right)\ll1\,,
\]
then
\[
\sup_{\tau\in[0,1]}\|v_{n}\|_{H^{s-1}}\lesssim_{s}
r\sup_{\tau\in[0,1]}|c_{n}-c_{n-1}|_{0,0,0}^{\mathcal{F}}\,.
\]
%
%
%Then, for $r>0$ (small enough) such that
%\[
%r\sup_{\tau}\left(|c_{n}|^{\Gamma^{\delta}}_{s_0,0,0}+
%|c_{n}|^{\Gamma^{\delta}}_{s_0,0,1}
%\right)\ll1
%\]
% we get, reasoning as in \eqref{picture11}, 
%\begin{align}
%\sup_{\tau\in[0,1]}\|v_{n}\|_{H^{s-1}}&\lesssim_{s}
%r\sup_{\tau\in[0,1]}|c_{n}-c_{n-1}|_{0,0,0}^{\mathcal{F}}\,.\label{picture11lower}
%\end{align}
Then we get the \eqref{stimaNNlow}.
\end{proof}
\noindent
We now study the convergence of the symbols $c_{n}$ in \eqref{def:Cnlow}.
%Notice that, by \eqref{stimeSeminormeMlow2}, we have
%\begin{equation}\label{deboleCCC}
%c_n\rightharpoonup^* \widetilde{c} \mbox{ in } \Gamma_{1}^{\delta}[r]\,,
%\end{equation}
%with $\widetilde{c}$ still satisfying \eqref{stimeSeminormeMlow2}. 
We claim that
\begin{equation}\label{claimCCN}
\sup_{\tau\in[0,1]}|c_{n}-c_{n-1}|^{\Gamma^{\delta},s}_{\alpha,\beta,0}\lesssim_{\beta} 
\frac{1}{2^{n}}\,,
\quad \forall n\geq 1\,, \alpha\leq s-1-s_0\,.
%\qquad \Rightarrow c_n\to c\,.
\end{equation}
%Moreover, by Ascoli-Arzel\`a theorem, we deduce that 
%\begin{equation}\label{convforteBn}
%b_{n}\to b\qquad {\rm in}\qquad C^{k}(\mathbb{T},\mathbb{R})\,,
%\end{equation}
%%$b_{n}$ converges to some function $b$
%%in the space $C^{k}(\mathbb{T},\mathbb{R})$ 
%for $0\leq k\leq s-s_0-1$ uniformly in $x_0\in \mathbb{T}$.
%Hence we deduce $\widetilde{c}\equiv c$, which means that
%$b\in\Gamma_1^{\delta}[r]$.
% with bounded $|\cdot|_{s-1}^{\mathcal{F}}$-norm.
%Let us check the claim \eqref{claimCCN}.
For $n=1$ it is trivial. Assume that \eqref{claimCCN} holds for $j\leq n-1$.
By \eqref{def:Cnlow} we can write
\[
\begin{aligned}
\pa_{x}(c_{n}-c_{n-1})&:=
\frac{1}{\big(1
+\mathfrak{m}(\tau,w)\big)(\pa_{\x}f)(\x)}
\Big(
a((\Phi^{\tau}_{c_{n-1}})^{-1}(w);x,\x)-a((\Phi^{\tau}_{c_{n-2}})^{-1}(w);x,\x)
\Big)
\\&
+\frac{1}{\big(1
+\mathfrak{m}(\tau,w)\big)(\pa_{\x}f)(\x)}\Big(
\mathfrak{C}((\Phi^{\tau}_{c_{n-2}})^{-1}(w);\x)
-\mathfrak{C}((\Phi^{\tau}_{c_{n-1}})^{-1}(w);\x)
\Big)\,.
\end{aligned}
\]
Define $d_{n}(\tau):=a((\Phi^{\tau}_{c_{n-1}})^{-1}(w);x,\x)-a((\Phi^{\tau}_{c_{n-2}})^{-1}(w);x,\x)$.
We have that
\[
%\pa_{x}^{\alpha}\pa_{\x}^{\beta}
d_{n}(\tau)=
(d_ua)\Big((\Phi^{\tau}_{c_{n-2}})^{-1}(w)
+\s ((\Phi^{\tau}_{c_{n-1}})^{-1}(w)-(\Phi^{\tau}_{c_{n-2}})^{-1}(w))\Big)
[(\Phi^{\tau}_{c_{n-1}})^{-1}(w)-(\Phi^{\tau}_{c_{n-2}})^{-1}(w)]\,,
\]
for some $\s\in[0,1]$.
Hence, for any $\alpha_1\leq s-s_0$ and $\beta_1\in \mathbb{N}$, we get
(see \eqref{maremma2})
\[
\begin{aligned}
|\pa_{x}^{\alpha_1}\pa_{\x}^{\beta_1}d_{n}(\tau)|&\lesssim_{s}
C\langle \x\rangle^{m'-\beta_1}\|
(\Phi^{\tau}_{c_{n-1}})^{-1}(w)-(\Phi^{\tau}_{c_{n-2}})^{-1}(w))
\|_{H^{s_0+\alpha_1}}\\
&\stackrel{\eqref{stimaNNlow}}{\lesssim_{s}}
C\langle \x\rangle^{m'-\beta_1}r\sup_{\tau\in[0,1]}|c_{n-1}-c_{n-2}|_{0,0,0}^{\Gamma^{\delta}}\,,
\end{aligned}
\]
where $C>0$ is some constant depending 
only on $s$ and $|a|^{\Gamma^{m'},s}_{\alpha_1,\beta_1,1}$.
Let us define
$g_{n}(\tau):=\mathfrak{C}((\Phi^{\tau}_{c_{n-2}})^{-1}(w);\x)
-\mathfrak{C}((\Phi^{\tau}_{c_{n-1}})^{-1}(w);\x)$.
For any  $\alpha\leq s-s_0$ and $\beta\in \mathbb{N}$
we have
\[
\begin{aligned}
|\pa_{x}^{\alpha+1}\pa_{\x}^{\beta}(c_{n}-c_{n-1})|&=
|\sum_{\beta_1+\beta_2=\beta}\pa_{\x}^{\beta_1}\left(\frac{1}{\big(1
+\mathfrak{m}(\tau,w)\big)(\pa_{\x}f)(\x)}\right)
\pa_{x}^{\alpha+1}\pa_{\x}^{\beta_2}(d_{n}(\tau)-g_{n}(\tau))|
\\&\lesssim_{s}
K\langle\x\rangle^{m'-m+1-\beta}r\sup_{\tau\in[0,1]}|c_{n-1}-c_{n-2}|_{0,0,0}^{\Gamma^{\delta}}\,,
\end{aligned}
\]
where $K$ is some constant depending 
only on $s$ and $|a|^{\Gamma^{m'},s}_{\alpha,\beta,1}$ 
and $|\mathfrak{m}|^{\Gamma^{m'},s}_{0,\beta,0}$.
For $r>0$ small enough and by the inductive assumption we obtain the \eqref{claimCCN}.
%\[
%|c_{n}-c_{n-1}|^{\Gamma^{m'-m+1}}_{\alpha,\beta,0}\lesssim_{\beta}2^{-n}\,.
%\]
%This proves the \eqref{claimCCN}.
By \eqref{claimCCN} we deduce that
$c_{n}$ converges to a symbol $c$ in the Fr\'echet 
space defined by the semi-norms $|\cdot|^{\Gamma^{m'-m+1}}_{\alpha,\beta,0}$
with $\alpha\leq s-s_0$.
In order to prove that the symbol $c$ admits differentials with respect to the variable
$w$ (see formula \eqref{maremma2})
we reasons as follows. For fixed $N>0$, consider the Banach space
defined by the norm
\[
\sum_{
\substack{\beta\leq N \\ \alpha+m\cdot k\leq s-s_0}}
|\cdot|^{\Gamma^{m'-m+1}}_{\alpha,\beta,k}\,.
\]
For any $N>0$, the sequence $c_{n}$ is bounded, uniformly in $n$, in such a space 
thanks to estimate \eqref{stimeSeminormeMlow2}.
Therefore, up to subsequences, 
$c_{n}$ converges weak-$*$ to a function $\tilde{c}$ in the above Banach space.
This implies that $\tilde{c}$ admits estimates 
on the differentials in $w$. Since $\tilde{c}$ coincides with the symbol $c$
almost everywhere we conclude the proof.

We are now ready to prove 
\eqref{equaLower}.

\begin{lemma}\label{tendeazerolower}
One has that 
 \begin{equation}\label{nota10lower50}
 \lim_{n\to\infty}\sup_{\tau\in[0,1]}\langle\x\rangle^{m'}\|F(c_{n})-\mathfrak{C}\|_{L^{\infty}}=0
 \end{equation}
 where $\mathfrak{C}$ is given in \eqref{equaLower}.
\end{lemma}

\begin{proof}
By \eqref{def:Cnlow} with $w=\Phi^\tau_{c_{n-1}}(u)$ we can check that
\begin{equation}\label{penna2}
a(u;x,\x)-\int_{0}^{1}(\pa_{\x}f)(\x)\big(1+\mathfrak{m}(\s,\Phi_{c_{n-1}}^{\s}(u))\big)(\pa_{x}c_n)
(\s,\Phi_{c_{n-1}}^{\s}(u);x,\x)d\s=\mathfrak{C}(u,\x)\,.
\end{equation}
By \eqref{equaLower} (reasoning as in \eqref{nota10})
%
%We set 
%\[
%z=z_{n}(1)=\Phi_{b_n}^{(z)}(1,z_0)\,,\quad x=x_{n}(1)=\Phi_{b_n}^{(x)}(1,x_0)\,.
%\]
%Then, using the \eqref{nota1n} in $({\bf S2})_{n}$,
we have that
\begin{equation}\label{nota10lower2}
\begin{aligned}
F(c_n)&=a(u;x,\x)-\int_{0}^{1}(\pa_{\x}f)(\x)\big(1+\mathfrak{m}(\s,\Phi_{c_n}^{\s}(u))\big)(\pa_{x}c_n)
(\s,\Phi_{c_n}^{\s}(u);x,\x)d\s
\\&
=a(u;x,\x)-\int_{0}^{1}(\pa_{\x}f)(\x)\big(1+\mathfrak{m}(\s,\Phi_{c_{n-1}}^{\s}(u))\big)(\pa_{x}c_n)
(\s,\Phi_{c_{n-1}}^{\s}(u);x,\x)d\s\\
&\qquad \qquad +\int_{0}^{1}(\pa_{\x}f)(\x)
A_{n}(\s)
d\s\,,\\
&\!\!\!\!\!\!\stackrel{\eqref{penna2}}{=} \mathfrak{C}(u;\x) 
+\int_{0}^{1}(\pa_{\x}f)(\x)
A_{n}(\s)
d\s\,,
\end{aligned}
\end{equation}
where
\begin{equation*}
A_{n}(\s):=
\big(1+\mathfrak{m}(\s,\Phi_{c_{n-1}}^{\s}(u))\big)(\pa_{x}c_n)
(\s,\Phi_{c_{n-1}}^{\s}(u);x,\x)-
\big(1+\mathfrak{m}(\s,\Phi_{c_{n}}^{\s}(u))\big)(\pa_{x}c_n)
(\s,\Phi_{c_{n}}^{\s}(u);x,\x)\,.
\end{equation*}
We now show that $A_{n}(\s)$ goes to zero in norm $\|\cdot\|_{L^{\infty}}$
uniformly in $\s\in[0,1]$.
Notice that
\begin{align}
A_{n}(\s)&=\big(1+\mathfrak{m}(\s,\Phi_{c_{n-1}}^{\s}(u))\big)
\Big[
(\pa_{x}c_n)
(\s,\Phi_{c_{n-1}}^{\s}(u);x,\x)-
(\pa_{x}c_n)
(\s,\Phi_{c_{n}}^{\s}(u);x,\x)
\Big]\label{penna10}\\
&+
\Big[ \mathfrak{m}(\s,\Phi_{c_{n-1}}^{\s}(u))-\mathfrak{m}(\s,\Phi_{c_{n}}^{\s}(u))\Big]
(\pa_{x}c_n)
(\s,\Phi_{c_{n}}^{\s}(u);x,\x)\,.
\end{align}
Since $\mathfrak{m}\in \mathcal{F}^{\mathbb{R}}_1[r]$ and 
setting $h=\Phi_{c_{n}}^{\s}(u)+\s'(\Phi_{c_{n-1}}^{\s}(u)- \Phi_{c_{n}}^{\s}(u))$
for some $\s'\in[0,1]$, we have
\[
\begin{aligned}
| \mathfrak{m}(\s,\Phi_{c_{n-1}}^{\s}(u))-\mathfrak{m}(\s,\Phi_{c_{n}}^{\s}(u))|&\lesssim
|(d_u\mathfrak{m})(h)[\Phi_{c_{n-1}}^{\s}(u))-\Phi_{c_{n}}^{\s}(u)]|
\stackrel{\eqref{maremma2}, \eqref{stimaNNlow}, \eqref{claimCCN}}{\lesssim_{s}}2^{-n}\,,
\end{aligned}
\]
for $r$ small enough. Reasoning in the same way on the term in \eqref{penna10}
we get the claim on $A_{n}(\s)$. Hence, by \eqref{nota10lower2}, 
we deduce \eqref{nota10lower50}. 
\end{proof}
This concludes the proof of Theorem \ref{constEgolower}.

\subsection{Off-diagonal terms at lower orders}

Consider symbols 
\begin{align}
&f(\x)\in \Gamma_0^{m}\,,\;\;\; m>1\,,\;\;\;f(\x)\in\mathbb{R} \quad 
a_{m}(U;x)\in \Sigma\mathcal{F}^{\mathbb{R}}_1[r,N]\,,
%a_{m}(U;x,\x)\in \Sigma\Gamma^{m'}_1[r,N]\,,\;\;\; m'<m\,,
 \label{low1off}\\
&
a_{m'}(U;x,\x)\in \Sigma\Gamma^{m'}_1[r,N]\,,\quad m'<m\,.
\label{low2off}
\end{align}
Let 
\begin{equation*}
C(\tau,U;x,\x)\in \Sigma\Gamma^{\delta}_1[r,N]\,, \;\;\; \tau\in[0,1]\,, \;\;\; \delta:=m'-m\,.
\end{equation*}
%Assume that if $\delta>0$ then $c(\tau,w,y)$ is real valued 
and consider the equation
 \begin{equation}\label{Ham111provaloweroff}
\left\{
\begin{aligned}
&\pa_{\tau}Z(\tau)=\opbw\big(\ii E{\bf C}(\tau,z(\tau);x,\x)\big)[z(\tau)]\,,\\
&Z(0)=Z_0=\vect{z_0\vspace{0.2em}}{\ov{z_0}}\,,
\end{aligned}\right.\qquad
{\bf C}(\tau,U;x,\x):=
\left(
\begin{matrix}
0 & C(\tau,U;x,\x) \\
\ov{C(\tau,U;x,-\x)} & 0
\end{matrix}
\right)\,.
\end{equation}
The following holds true.

\begin{theorem}\label{constEgoloweroff}
Assume \eqref{low1off}, \eqref{low2off}.
For $r>0$ small enough there exists a symbol 
$C(\tau,U;x,\x)$
belonging to $\Sigma\Gamma^{\delta}_1[r,N]$, 
$\tau\in[0,1]$, $\delta=m'-m$ 
such that
\begin{equation}\label{equaLoweroff}
F(c):=a_{m'}(U;x,\x)-2\int_{0}^{1}\big(1+a_{m}(\s,\Phi_{C}^{\s}(u))\big)f(\x)
C(\s,\Phi_{C}^{\s}(u);x,\x)d\s=0\,,
\end{equation}
where $\Phi_{C}^{\tau}(u)$, $\tau\in[0,1]$ is the flow of \eqref{Ham111provaloweroff}.
\end{theorem}

\begin{proof}
One can reasons by following almost word by word the proof of Theorem \ref{constEgolower}.
\end{proof}

\section{Conjugations}\label{secconjconj}
In this section we prove four abstract theorems which will be used in order to prove Theorem \ref{thm:main}. In all the theorems we shall consider a system of the form
\begin{equation}\label{Nonlin1}
\left\{
\begin{aligned}
&\dot{U}=X(U)\\
&U(0)=U_0\in H^{s}\times H^{s}
\end{aligned}\right.\,,
\end{equation}
where the matrix of operator $X(U)$ is defined as follows
\begin{equation}\label{operatoreEgor1}
\begin{aligned}
&X(U):=\ii E\opbw(A(U;x,\x))[U]+R(U)[U]\,,\qquad E=\sm{1}{0}{0}{-1}\,, 
\quad \\
&R\in \Sigma\mathcal{R}^{-\rho}_{1}[r,N]\otimes\mathcal{M}_2(\mathbb{C})\,,\;\;
A(U;x,\x)\in\Sigma\Gamma^m_{0}[r,N]\otimes\mathcal{M}^2(\C)\,.
%\quad a(U;x,\x) \in \Sigma\Gamma^{m}_{1}[r,N]\,,
%\quad m>1\,.
\end{aligned}
\end{equation}

 In  Theorem \ref{conjOrdMax} we shall consider the system \eqref{Nonlin1}, where the matrix of operators in r.h.s. is defined in \eqref{operatoreEgor1}. Here the matrix of para-differential operators is assumed to be diagonal (its off diagonal symbols equal 0).  In such a Theorem we exhibit a change of coordinates of the phase space $H^s\times H^{s}$ 
 which removes the dependence on the space 
 variable $x$ from the symbols on the diagonal at the highest order.

The main theorem in  Section \ref{diagoLower} is the \ref{conjOrdMaxredu}. In this case the matrix of symbols $A(U;x,\xi)$ is assumed to be diagonal as before with following additional assumption. We assume that symbol on the diagonal $a$ admits an expansion in decreasing orders which is equal to a Fourier multiplier at the highest order plus a lower order term which may depend on $x$ (see equation \eqref{Forma-di-Aredu}). 
We provide a change of coordinates 
which removes the dependence on  $x$ from such a lower order term.

In the same spirit, in Section \ref{diago-lineare} we explain how diagonalize the matrix $A(U;x,\xi)$ by means of changes of coordinates. The diagonalization of the highest order is the content of Theorem \ref{conjOrdMaxhighoff}, the lower order terms are treated in Theorem \ref{conjOrdMaxloweroff}.

The general strategy that we adopt in the proofs of all the theorems of this section is the following. We shall consider changes of coordinates defined through  non-linear flows of the form 
\begin{equation}\label{sist2}
\left\{
\begin{aligned}
&\pa_{\tau}\Psi^{\tau}=G^{\tau}(\Psi^{\tau})\\
&\Psi^{0}={\rm Id},
\end{aligned}\right.
\end{equation}
where $G^{\tau}$ is some non-linear vector field
possibly depending explicitly on $\tau\in[0,1]$. 
In the application the operator 
$G^{\tau}$ is  para-differential.  
All the issues of well posedness of \eqref{sist2} 
have been analyzed in Section \ref{FLOWS}. 
We consider as a new variable the function $Z:=\Psi(U):=\Psi^{\tau}_{|\tau=1}(U)$.
The system \eqref{Nonlin1} in the new 
coordinates 
reads
\begin{equation}\label{sistnew}
\left\{
\begin{aligned}
&\dot{Z}=d\Psi\big(\Psi^{-1}(Z)\big)\big[ X(\Psi^{-1}(Z))\big]=:(\Psi)^*X(Z)
=:P^{1}(Z)\\
&Z(0)=\Psi(U_0),
\end{aligned}\right.
\end{equation}
where $P^{1}(Z)=(P^{\tau}(Z))_{\tau=1}$
with $P^{\tau}$ defined, for $\tau\in[0,1]$, as
\begin{equation}\label{pushpush}
P^{\tau}(Z):=d\Psi^{\tau}\big((\Psi^{\tau})^{-1}(Z)\big)\big[ {X}((\Psi^{\tau})^{-1}(Z))\big]\,.
\end{equation}
On can prove  that $P^{\tau}$ satisfies the non-linear Heisenberg equation 
\begin{equation}\label{Ego}
\left\{
\begin{aligned}
&\pa_{\tau}P^{\tau}(Z)=\big[ G^{\tau}(Z) , P^{\tau}(Z)\big]\\
&P^{0}(Z)=X(Z)\,,
%\opbw(a(z;x,\x))[z]\,.
\end{aligned}\right.
\end{equation}
where the non-linear commutator is defined in \eqref{nonlinCommu}. We choose this Heisenberg approach since we start from a para-differential system and we are interested in preserving such structure. The Heisenberg approach provides a systematic way to prove this fact at each step. 
This Heisenberg strategy has been already used in \cite{BD}, \cite{Feola-Iandoli-Long}, \cite{FIloc}, \cite{BFP}. As already said in other part of the manuscript the authors in those  papers do not attempt to find non-linear changes of coordinates but only some modified energies. As a consequence they face a linear version of system \eqref{Ego}, where the non linear commutators are replaced by the linear ones and the operator $G^{\tau}$ is linear. 
In this case the commutator between two para-differential operators is still a para-differential one modulo smoothing remainders, this is a consequence of
 standard para-differential calculus. In our case the non-linear commutators involves also the differential of the symbols with respect to the "non-linear" variable $U$, this is coherent with our definition of symbols and operators given in Section \ref{paraparapara}.

\subsection{Reduction  to constant coefficients}\label{super-costanti}
In this section we show how to conjugate to constant coefficients
para-differential systems which are diagonal up to smoothing remainders, we shall adopt the Notation \ref{notazioni}.

\subsubsection{Reduction at the highest orders}\label{sec:egoHigh}
We consider system \eqref{Nonlin1} with operator $X(U)$ defined in \eqref{operatoreEgor1}. We shall assume that the matrix $A(U;x,\xi)$ is as follows
\begin{equation}\label{tordo}
A(U;x,\x):=\left(\begin{matrix}
a(U;x,\x) & 0 \\
0 &\ov{a(U;x,-\x)}
\end{matrix}\right).
\end{equation}

%In this section 
%we study how a para-differential vector field conjugate under the flow in \eqref{flusso}
%with $f(\tau,u;x,\x)$ as in \eqref{sim1}.
%Let us consider 
%\begin{equation}\label{operatoreEgor1}
%\begin{aligned}
%&X(U):=\ii E\opbw(A(U;x,\x))[U]+R(U)[U]\,,\qquad E=\sm{1}{0}{0}{-1}\,, 
%\quad R\in \Sigma\mathcal{R}^{-\rho}_{1}[r,N]\otimes\mathcal{M}_2(\mathbb{C})\,,\\
%&A(U;x,\x):=\left(\begin{matrix}
%a(U;x,\x) & 0 \\
%0 &\ov{a(U;x,-\x)}
%\end{matrix}\right)\,.
%%\quad a(U;x,\x) \in \Sigma\Gamma^{m}_{1}[r,N]\,,
%%\quad m>1\,.
%\end{aligned}
%\end{equation}
\noindent The symbol $a(U;x,\x)$ has the form
\begin{equation}\label{Forma-di-A}
\begin{aligned}
&a(U;x,\x)=(1+a_{m}(U;x))f_m(\x)+a_{m'}(U;x,\x)\,,\qquad m>1,\,\, m'=m-\frac12,\\
&{a}_m(U,x)\in \Sigma\mathcal{F}_1^{\mathbb{R}}[r,N]\,,\quad 
a_{m'}(U;x,\x) \in\Sigma\Gamma^{m'}_{1}[r,N]\,,\\
&a_{m'}(U;x,\x)-\ov{a_{m'}(U;x,\x)}\in \Sigma\Gamma^{0}_1[r,N]\,.
\end{aligned}
\end{equation}
where $m$ satisfies \eqref{ipotesim},
$f_{m}\in \Gamma_0^{m}$ is an even in $x$ \emph{classical} symbol. 
In particular $f_{m}$ admits the expansion
\begin{equation}\label{expFM}
f_{m}(\x)=f_0(\x)+\widetilde{f}_{m-1}(\x)\,,\quad \widetilde{f}_{m-1}\in \Gamma_{0}^{m-1}\,,
\end{equation}
and 
$f_0(\x)\in \Gamma_{0}^{m}$ is a 
$m$-homogeneous function.
% in ${C^{\infty}(\mathbb{R}^+,\mathbb{R})}$.
%We also assume that \eqref{ipotesim}
%\begin{equation}\label{ipotesim}
%m\in \mathbb{N} \quad {\rm or}\quad m=\frac{k}{2}\quad k\in \mathbb{N}\,.
%\end{equation}
The symbol in \eqref{Forma-di-A} (recall also \eqref{expFM}) 
satisfies the assumption \eqref{strutturaSimb}, 
therefore Theorem \ref{constEgo} applies 
and provides a symbol $b$ such that 
\eqref{equa814} holds true with $\tilde{a}_m\rightsquigarrow a_m$.
Having such $b$ we now consider the  matrix of symbols
\begin{equation}\label{simboB}
\begin{aligned}
&{\bf B}(\tau,U;x,\x):=\left(\begin{matrix}
 B(\tau,U;x,\x) & 0 \\ 0 &  \ov{B(\tau,U;x,-\x)} 
\end{matrix}
\right)=
\left(\begin{matrix}
 B(\tau,U;x,\x) & 0 \\ 0 & -{B(\tau,U;x,\x)} 
\end{matrix}
\right)
\\&
B(\tau,U;x,\x):=b(\tau,U;x)\x\,,
%=\frac{\beta(U;x)}{1+\tau\beta_{x}(U,x)}\x\,,
\quad
b\in \Sigma\mathcal{F}^{\mathbb{R}}_{1}[r,N]\,,
\end{aligned}
\end{equation}
and let 
 $\Psi^{\tau}$ be the flow of \eqref{sist2} with
 \begin{equation}\label{Nonlin2}
 G^{\tau}(U):=\ii E \opbw({\bf B}(\tau,U;x,\x))U\,.
 \end{equation}
 
% \begin{equation}\label{Nonlin2}
%\left\{
%\begin{aligned}
%&\pa_\tau\Psi^{\tau}(U)=\opbw(\ii E {\bf B}(\tau,\Psi^{\tau}(U);x,\x))[\Psi^{\tau}(U)]\\
%&\Psi^{0}(U)=U=\vect{u}{\bar{u}}\,.
%\end{aligned}\right.
%\end{equation}
% $\Phi^{\tau}$ be the flow of
% \begin{equation}\label{Nonlin2}
%\left\{
%\begin{aligned}
%&\pa_\tau\Phi^{\tau}(u)=\opbw(\ii B(\tau,\Phi^{\tau}(u);x,\x))[\Phi^{\tau}(u)]\\
%&\Phi^{0}(u)=u\,.
%\end{aligned}\right.
%\end{equation}
\noindent Notice that the flow of system \eqref{sist2} with generator  \eqref{Nonlin2} is well-posed by Theorem \ref{flussononlin}
applied with generator as in \eqref{sim1}.
We define
\begin{equation}\label{variabNuova}
Z:=\vect{z}{\bar{z}}:={\bf \Psi}^{\tau}(U)_{|\tau=1}\,.
%:=\left( 
%\begin{matrix}
%\Phi^{\tau}(u) \vspace{0.2em}\\
%\ov{\Phi^{\tau}(u)}
%\end{matrix}
%\right)_{|\tau=1}
%\,,\qquad U=\vect{u}{\bar{u}}\,.
\end{equation}
Notice that, by \eqref{simboB},
\[
{\bf \Psi}^{\tau}(U)_{|\tau=1}
:=\left( 
\begin{matrix}
\Phi^{\tau}(u) \vspace{0.2em}\\
\ov{\Phi^{\tau}(u)}
\end{matrix}
\right)_{|\tau=1}\,,\quad \pa_\tau\Phi^{\tau}(u)=\opbw(\ii B(\tau,\Phi^{\tau}(u);x,\x))[\Phi^{\tau}(u)]\,.
\]
The main result of this section is the following.

\begin{theorem}{\bf (Non-linear Egorov).}\label{conjOrdMax}
For $r>0$ small enough the conjugate of $X$ in \eqref{Nonlin1} with generator \eqref{Nonlin2} has the form (see \eqref{variabNuova})
\begin{equation}\label{Nonlin3}
\dot{Z}=\ii E\opbw(A^{+}(Z;x,\x))[Z]+R^{+}(Z)[Z]\,,\qquad 
R^{+}\in R\in \Sigma\mathcal{R}^{-\rho}_{1}[r,N]\otimes\mathcal{M}_2(\mathbb{C})\,,
\end{equation}
and the matrix of symbols $A^{+}\in \Sigma\Gamma^{m}_{1}[r,N]\otimes\mathcal{M}_2(\mathbb{C})$
has the form
\begin{equation}\begin{aligned}
&A^{+}(Z;x,\x):=\left(\begin{matrix}
a^{+}(Z;x,\x) & 0 \\
0 &\ov{a^{+}(Z;x,-\x)}
\end{matrix}\right)\,,\\
&a^{+}(Z;x,\x)=\mathfrak{m}(Z)f_{m}(\x)+a^{+}_{m'}(Z;x,\x)\,,\qquad 
\mathfrak{m}\in \Sigma\mathcal{F}_0^{\mathbb{R}}[r,N]\,,  \quad m'=m-\frac12,\\
&a^{+}_{m'}(Z;x,\x) 
\in\Sigma\Gamma^{m'}_{1}[r,N]\,, \quad a_{m'}(U;x,\x)-\ov{a_{m'}(U;x,\x)}\in \Sigma\Gamma^{0}_1[r,N]\,.
\end{aligned}\end{equation}
with $\mathfrak{m}(Z)$ is \emph{independent} of $x\in \mathbb{T}$.
\end{theorem}
The rest of the section is devoted to the proof of the result above.
%In the following we shall write
%$\Psi:=\Psi^{\tau}_{|\tau=1}$, $\Phi:=\Phi^{\tau}_{|\tau=1}$.
The system \eqref{Nonlin1} in the new 
coordinates \eqref{variabNuova}
has the form \eqref{sistnew}-\eqref{Ego} with $G^{\tau}$ as in 
\eqref{Nonlin2}.
%reads
%\begin{equation}\label{sistnew}
%\left\{
%\begin{aligned}
%&\dot{Z}=d\Psi\big(\Psi^{-1}(Z)\big)\big[ \mathcal{X}(\Psi^{-1}(Z))\big]=:(\Psi)^*\mathcal{X}(Z)
%=:P^{1}(Z)\\
%&Z(0)=\Psi(U_0),
%\end{aligned}\right.
%\end{equation}
%where $P^{1}(z)=(P^{\tau}(z))_{\tau=1}$
%with $P^{\tau}$ defined as, for $\tau\in[0,1]$,
%\begin{equation}\label{pushpush}
%P^{\tau}(Z):=d\Psi^{\tau}\big((\Psi^{\tau})^{-1}(Z)\big)\big[ \mathcal{X}((\Psi^{\tau})^{-1}(Z))\big]\,.
%\end{equation}
%It is easy to note that $P^{\tau}$ satisfies 
%\begin{equation}\label{Ego}
%\left\{
%\begin{aligned}
%&\pa_{\tau}P^{\tau}(Z)=\big[\opbw(\ii E {\bf B}(\tau,Z;x,\x))[Z] , P^{\tau}(Z)\big]\\
%&P^{0}(Z)=X(Z)\,.
%%\opbw(a(z;x,\x))[z]\,.
%\end{aligned}\right.
%\end{equation}
Moreover, by the remarks under Definition \ref{smoothoperatormaps} and by Theorem \ref{flussononlin},
we note that 
\begin{equation}\label{dono1}
P^{\tau}(Z)=\ii E \Omega Z+M_1(\tau;Z)[Z]\,,\;\;\;
M_1\in \Sigma\mathcal{M}_{1}[r,N]\otimes\mathcal{M}_2(\mathbb{C})\,,
%\qquad {\rm and\;\; hence}\qquad 
%P^{\tau}\in  \Sigma\mathcal{M}_{1}[r,N]\otimes\mathcal{M}_2(\mathbb{C})\,,
\end{equation}
with estimate uniform in $\tau\in[0,1]$, where $\Omega$ is in \eqref{omegone}. 
Actually we shall prove that
\begin{equation}\label{forma-totale}
P^{\tau}(Z)
=\ii E\Omega Z+M_1(\tau;Z)[Z]=\ii E\opbw(A^{+}(\tau,Z;x,\x))[Z]+R^{+}(\tau,Z)[Z]\,,
\end{equation}
with  %as in \eqref{Nonlin3}, \eqref{nuovoSimboA}.
\begin{equation}\label{dono7}
A^{+}(\tau,Z;x,\x):=\left(\begin{matrix}
a^{+}(\tau,Z;x,\x) & 0 \\
0 &\ov{a^{+}(\tau,Z;x,-\x)}
\end{matrix}\right)\,, \qquad
R^{+}\in \Sigma\mathcal{R}^{-\rho}_1[r,N]\otimes\mathcal{M}_{2}(\mathbb{C})\,.
\end{equation}
In particular 
we make the ansatz
\begin{equation}\label{ansatz}
a^{+}(\tau,Z;x,\x)=\sum_{j=0}^{2(\rho+m)} a_{m-j}^{+}(\tau,Z;x,\x)\,,\qquad
a_{m}^{+}\in \Sigma\Gamma^{m}_{0}[r,N]\quad 
a_{m-j}^{+}\in \Sigma\Gamma^{m-\frac{j}{2}}_{1}[r,N]\,, j>0\,.
\end{equation}
Let $F(Z):=\opbw(\ii E{\bf B}(\tau,Z;x,\x))[Z]$. Then one can note that
\begin{equation}\label{strike}
(d_{Z}F)(Z)[V]= \opbw(\ii E{\bf B}(\tau,Z;x,\x))[V]+  \opbw(\ii E(d_{Z}{\bf B})(\tau,Z;x,\x)[V])[Z]\,,
\end{equation}
and we also recall that
\[
(d_{Z}{\bf B})(\tau,Z;x,\x)[V]=(\pa_{Z}{\bf B})(\tau,Z;x,\x)\cdot V\,.
\]
By the definition of the non-linear commutator 
in \eqref{nonlinCommu}
we write
\begin{equation}\label{commuExp100}
\begin{aligned}
\big[\opbw(\ii E{\bf B}(\tau,Z;x,\x))[Z] ,P^{\tau}(Z)\big]&=\opbw(\ii E{\bf B}(\tau,Z;x,\x))\big[P^{\tau}(Z)\big]\\
&+\opbw\big(\ii Ed_{Z}{\bf B}(\tau,Z;x,\x)[P^{\tau}(Z)]\big)[Z]
\\&
-d_{Z}P^{\tau}(Z)\big[\opbw(\ii E{\bf B}(\tau,Z;x,\x))[Z]\big]\,.
\end{aligned}
\end{equation}
By \eqref{dono1}, \eqref{simboB}, 
we deduce that the matrix of symbols in the second summand in the r.h.s. of  \eqref{commuExp100} has the form
\begin{equation}\label{dono2}
%d_{Z}{\bf B}(\tau,Z;x,\x)[P^{\tau}(Z)]=
\left(\begin{matrix}
 (d_{Z}B)(\tau,Z;x,\x)[P^{\tau}(Z)] & 0 \\ 0 & -(d_{Z}B)(\tau,Z;x,\x)[P^{\tau}(Z)]
\end{matrix}
\right)\in 
 \Sigma\Gamma^{1}_{1}[r,N]\otimes\mathcal{M}_2(\mathbb{C})\,. 
\end{equation}
For simplicity we shall write
\begin{equation}\label{timeDeriva}
(\pa_{t}{\bf B})(\tau,Z;x,\x):=
d_{Z}{\bf B}(\tau,Z;x,\x)[P^{\tau}(Z)]\,.
\end{equation}
By \eqref{forma-totale}, \eqref{timeDeriva}  we have that \eqref{commuExp100}
becomes
\begin{align}
\big[\opbw(\ii E{\bf B}&(\tau,Z;x,\x))[Z] ,P^{\tau}(Z) \big]=\nonumber\\
&\quad\,\opbw(\ii E{\bf B}(\tau,Z;x,\x))\big[\opbw(\ii EA^{+}(\tau,Z;x,\x))[Z]\big]\label{commuExpA}\\
&-\opbw(\ii EA^{+}(\tau,Z;x,\x))\big[\opbw(\ii E{\bf B}(\tau,Z;x,\x))[Z]\big]\label{commuExpB}\\
&+\opbw\big(\ii E(\pa_{t}{\bf B})(\tau,Z;x,\x)\big)[Z]\label{commuExpC}\\
&-\opbw\big(\ii E(d_{Z}A^{+})(\tau,Z;x,\x)[\opbw(\ii E{\bf B}(\tau,Z;x,\x))[Z]]\big)[Z]
\label{commuExpD}\\
&
+\opbw(\ii E{\bf B}(\tau,Z;x,\x))\big[R^{+}(Z)[Z]\big]
-R^{+}(Z)[\opbw(\ii E{\bf B}(\tau,Z;x,\x))[Z]]
\label{commuExpE}\\
&-(d_{Z}R^{+})(Z)[\opbw(\ii E{\bf B}(\tau,Z;x,\x))[Z]]\,.\label{commuExpF}
\end{align}
We now analyze each summand above.
By Proposition \ref{composizioniTOTALI}
we have
\[
\eqref{commuExpE}+\eqref{commuExpF}\in 
\Sigma\mathcal{R}^{-\rho}_1[r,N]\otimes\mathcal{M}_2(\mathbb{C})\,.
\]
By \eqref{ansatz}
we have that (recall \eqref{commuExpD})
\begin{equation}\label{dono8}
\begin{aligned}
&\ii E(d_{Z}A^{+})(\tau,Z;x,\x)[\opbw(\ii E{\bf B}(\tau,Z;x,\x))[Z]]=
\ii E\left(\begin{matrix}
\tilde{d}(\tau,Z;x,\x) & 0 \\ 0 &
\ov{\tilde{d}(\tau,Z;x,-\x)}
\end{matrix}
\right)\,,\\
&\tilde{d}(\tau,Z;x,\x)=
\sum_{j=0}^{\rho+m} (d_{Z}a_{m-j}^{+})(\tau,Z;x,\x)[\opbw(\ii E{\bf B}(\tau,Z;x,\x))[Z]]\,.
\end{aligned}
\end{equation}
%By \eqref{Nonlin2} we have
%\[
%\eqref{commuExpD}=
%-\opbw\big(\ii E(d_{z}A^{+})(\tau,Z;x,\x)[\pa_{\tau}\Psi^{\tau}(U)]\big)[Z]
%%[\opbw(\ii E{\bf B}(\tau,Z;x,\x))[Z]]\big)[Z]
%\]
By Proposition 
\ref{teoremadicomposizione} we also deduce
\begin{equation*}
\eqref{commuExpA}+\eqref{commuExpB}=
\opbw\Big(\ii E{\bf B }(\tau,Z;x,\x)\star_{\rho}\ii EA^{+}(\tau,Z;x,\x) \Big)[Z]+\widetilde{R}(\tau,Z)[Z]\,,
\end{equation*}
where $\widetilde{R}\in \Sigma\mathcal{R}^{-\rho}_1[r,N]\otimes\mathcal{M}_2(\mathbb{C})$
(with estimates uniform in $\tau\in[0,1]$) and
\begin{equation*}
\begin{aligned}
\ii E{\bf B }(\tau,Z;x,\x)&\star_{\rho}\ii EA^{+}(\tau,Z;x,\x):=\\
&:=\ii E{\bf B }(\tau,Z;x,\x)\#_{\rho}\ii EA^{+}(\tau,Z;x,\x)
-\ii E{ A^{+} }(\tau,Z;x,\x)\#_{\rho}\ii E{\bf B}(\tau,Z;x,\x)\\
&\stackrel{\eqref{simboB}, \eqref{dono7}}{=}
\left(\begin{matrix}
d(\tau,Z;x,\x) & 0 \\ 0 &
\ov{d(\tau,Z;x,-\x)}
\end{matrix}
\right)\,,
\end{aligned}
\end{equation*}
where 
\begin{equation}\label{dono6}
d(\tau,Z;x,\x):=\ii B(\tau,Z;x,\x)\#_{\rho}\ii a^{+}(\tau,Z;x,\x)-
\ii a^{+}(\tau,Z;x,\x)\#_{\rho}\ii B(\tau,Z;x,\x)\,.
\end{equation}
By recalling the ansatz \eqref{ansatz} and using the expansion \eqref{espansione2}
we get
\begin{equation}\label{dono11}
d(\tau,Z;x,\x)=\ii \sum_{j=0}^{2(m+\rho)}\big\{b(\tau,Z;x)\x,a_{m-j}^{+}(\tau,Z;x,\x)\big\}+r_{j}(\tau,Z;x,\x)\,,
\end{equation}
where
\begin{equation*}
r_0=r_{-1}\equiv0\,,\quad r_{j}(\tau,Z;x,\x)\in 
\Sigma\Gamma^{m-\frac{j}{2}}_1[r,N]\,,\;\; j \geq 2\,,
\end{equation*}
 the symbol $r_{j}$ depends only on $b(\tau,Z;x)\x$
and on the symbols $a^{+}_{m-k}$ with $k<j$.
Recalling \eqref{dono2}, \eqref{timeDeriva} and \eqref{dono11}
we define, for $j=0,\ldots,2(m+\rho)$,
\begin{equation}\label{dono13}
q_{j}(\tau,Z;x,\x):=\left\{\begin{aligned}
&r_{j}(\tau,Z;x,\x)\,,\;\;\; \qquad  j\neq 2(m-1)\,,\\
&r_{j}(\tau,Z;x,\x)+(\pa_{t}{B})(\tau,Z;x,\x)\,,\;\;\;j=2(m-1)\,.
\end{aligned}\right.
\end{equation}
Our aim is to prove the ansatz \eqref{forma-totale} using \eqref{Ego},
the expansion \eqref{commuExpA}-\eqref{commuExpF}.
We shall solve such equation iteratively expanding 
the symbol in \eqref{dono6} in decreasing orders. 
%
%up to smoothing remainders in $\Sigma\mathcal{R}^{-\rho}_{1}[r,N]$.
%In  symbols we have
%\begin{equation}\label{commuExp2}
%\begin{aligned}
%\ii B(\tau,z;x,\x)&\# a(\tau,z;x,\x)-a(\tau,z;x,\x)\# \ii B(\tau,z;x,\x)\\
%&-
%d_{z}a(\tau,z;x,\x)[\opbw(\ii B(\tau,z;x,\x))[z]]+\ii (\pa_{t}B)(\tau,z;x,\x)\,.
%\end{aligned}
%\end{equation}

\noindent
{\bf Order $m$.}
By \eqref{dono8}, \eqref{dono6}, \eqref{espansione2}, \eqref{dono11}
we have that, at the highest order, the equation \eqref{Ego}
reads
%Using the expansion of the symbol $a$ in \eqref{ansatz}
%we have, at the highest order,
%the equation
\begin{equation}\label{ordMax}
\left\{\begin{aligned}
&\pa_{\tau}a^{+}_m(\tau,Z;x,\x)=\big\{b(\tau,Z;x)\x, a^{+}_m(\tau,Z;x,\x)\big\}-
d_{Z}a^{+}_m(\tau,Z;x,\x)[\opbw(\ii E {\bf B}(\tau,Z;x,\x))[Z]]\\
&a^{+}_m(0,Z;x,\x)=a(Z;x,\x)\,,
\end{aligned}\right.
\end{equation}
where $\{\cdot,\cdot\}$ denotes the $\pa_{\x}\pa_{x}-\pa_{x}\pa_{\x}$.
Notice that the function 
\begin{equation*}
g(\tau)=a^{+}_m(\tau,Z(\tau);x(\tau),\x(\tau))
\end{equation*}
is constant along the solution of \eqref{Ham111prova}.
%We denote by 
%\begin{equation}\label{Ham1}
%\left\{
%\begin{aligned}
%&\pa_{\tau}x=-b(\tau,Z;x) \\
%& \pa_{\tau}\x=b_x(\tau,Z;x)\x
%\end{aligned}\right.
%\end{equation}
%and 
%\begin{equation}\label{Ham2}
%\pa_{\tau}z=\opbw\big(\ii B(\tau,z;x,\x)\big)[z]\,,
%\end{equation}
By Theorem \ref{constEgo} we have that the flow 
of \eqref{Ham111prova} is well posed and invertible.
We denote by
\begin{equation}\label{flussoInv}
\begin{aligned}
\Upsilon^{\tau}(U,x_0,\x_0)&:=\Big(\Upsilon^{(Z)},\Upsilon^{(x)}, \Upsilon^{(\x)} \Big)(\tau,U,x_0,\x_0)\,,\\
\widetilde{\Upsilon}^{\tau}(Z,x,\x)&:=\Big(\widetilde{\Upsilon}^{(Z)},
\widetilde{\Upsilon}^{(x)}, \widetilde{\Upsilon}^{(\x)} \Big)(\tau,Z,x,\x)\,,
\end{aligned}
\end{equation}
respectively the flow and the inverse flow of \eqref{Ham111prova}. 
In particular we have (recall \eqref{losperobene})
\begin{equation}\label{flussoInvInvInv}
\widetilde{\Upsilon}^{(\x)}(\tau,Z;x,\x):=(1+\widetilde{\Psi}^{(\x)}(\tau,Z;x))\x\,,\qquad
\widetilde{\Psi}^{(\x)}\in \Sigma\mathcal{F}^{\mathbb{R}}_1[r,N]\,.
\end{equation}
By the proof of Theorem \ref{constEgo}
we also deduce that
\begin{equation*}
g(Z;x,\x)\in \Sigma\Gamma^{m}_{0}[r,N]\quad
\Rightarrow 
\quad
g(\widetilde{\Upsilon}^{\tau}(Z,x,\x))\in \Sigma\Gamma^{m}_{0}[r,N]\,.
\end{equation*}
Hence the symbol
%Hence the solution of the problem \eqref{ordMax} is
\begin{equation}\label{solOrdMax}
a^{+}_m(\tau,Z;x,\x)=a(\widetilde{\Upsilon}^{\tau}(Z;x,\x))\,,
\end{equation}
belongs to $ \Sigma\Gamma^{m}_{1}[r,N]\otimes\mathcal{M}_2(\mathbb{C})$ and
solves the problem \eqref{ordMax}.
We now study the properties of the symbol $a_{m}^{+}$
by using the result of Theorem \ref{constEgo},
in particular equation \eqref{equa814}.
By \eqref{Forma-di-A}, \eqref{expFM},
and using \eqref{solOrdMax}, \eqref{flussoInv}, \eqref{flussoInvInvInv}, 
(we shall omit the dependence on $Z,x$)
we have 
\begin{equation}\label{locomotive1}
\begin{aligned}
a_{m}^{+}(\tau,Z;x,\x)&=\big(1+a_{m}
(\widetilde{\Upsilon}^{(z)},\widetilde{\Upsilon}^{(x)})\big)
f_{m}\Big( (1+\widetilde{\Psi}^{(\x)})\x\Big)+
a_{m'}
\big(\widetilde{\Upsilon}^{(z)},\widetilde{\Upsilon}^{(x)}, \widetilde{\Upsilon}^{(\x)}\big)\\
&\stackrel{ \eqref{expFM}}{=}
\big(1+a_{m}
(\widetilde{\Upsilon}^{(z)},\widetilde{\Upsilon}^{(x)})\big)
(1+\widetilde{\Psi}^{(\x)})^{m}
f_{m}( \x)+r_1+r_2\\
&\stackrel{\eqref{equa814}}{=} \mathfrak{m} f_{m}(\x)+r_1+r_2\,,
\end{aligned}
\end{equation}
where $\mathfrak{m}\in \Sigma\mathcal{F}_{0}^{\mathbb{R}}[r,N]$ is independent of $x\in \mathbb{T}$ and 
\begin{equation}\label{locomotive2}
\begin{aligned}
r_1&:=\mathfrak{m}\big( f_0(\x)-f_{m}(\x)\big)+
\mathfrak{m}
%\big(1+a_{m}
%(\widetilde{\Upsilon}^{(z)},\widetilde{\Upsilon}^{(x)})\big)
(1+\widetilde{\Psi}^{(\x)})^{-m}
\widetilde{f}_{m-1}\big(
(1+\widetilde{\Psi}^{(\x)})\x
\big)\,,\\
r_2&:=a_{m'}
\big(\widetilde{\Upsilon}^{(z)},\widetilde{\Upsilon}^{(x)}, 
\widetilde{\Upsilon}^{(\x)}\big)\,.
\end{aligned}
\end{equation}
By \eqref{locomotive1}, \eqref{locomotive2}, and using the properties 
of $\widetilde{\Upsilon}$ in Theorem \ref{constEgo}
we  deduce that
%\medskip
%Thanks to Theorem \ref{constEgo}, in particular equation \eqref{equa814},
%we have that
%
%Moreover, recalling \eqref{Forma-di-A}, 
%and using \eqref{equa814} in Theorem \ref{constEgo},
%it is possible to chose $B(\tau,Z;x,\x)$ in \eqref{simboB}
%in such a way
\begin{equation*}
a_m^{+}(1,Z;x,\x)=
%\stackrel{\eqref{ordMax}, \eqref{Forma-di-A}}{=}
\mathfrak{m}(Z)f_{m}(\x)+r(Z;x,\x)\,,
\quad
\mathfrak{m}\in \Sigma\mathcal{F}^{\mathbb{R}}_{0}[r,N]\,,
\quad r\in \Sigma\Gamma^{m'}_{1}[r,N]\,,
\end{equation*}
and $\mathfrak{m}$ is independent of $x\in \mathbb{T}$.

\noindent
{\bf Lower orders.} 
Recalling \eqref{dono11}, \eqref{dono13}, \eqref{dono8}
we have that the equation \eqref{Ego}, at order
$m-j/2$ with $j\geq1$,
reads
\begin{align}
&\pa_{\tau}a^{+}_{m-j}(\tau,Z;x,\x)=\big\{b(\tau,Z;x)\x, a^{+}_{m-j}(\tau,Z;x,\x)\big\}\nonumber
\\&
\qquad\qquad\qquad\qquad\qquad-d_{z}a^{+}_{m-j}(\tau,Z;x,\x)[\opbw(\ii E{\bf B}(\tau,Z;x,\x))[Z]]
-\ii q_{j}(\tau,Z;x,\x)\label{ord1}\\
&a_0(0,z;x,\x)=0\nonumber\,,
\end{align}
where the $q_j'$s are defined in \eqref{dono13}.
%\begin{equation}\label{ord1}
%\left\{\begin{aligned}
%&\pa_{\tau}a^{+}_{m-j}(\tau,Z;x,\x)=\big\{b(\tau,Z;x)\x, a^{+}_{m-j}(\tau,Z;x,\x)\big\}-
%d_{z}a_{m-j}(\tau,Z;x,\x)[\opbw(\ii B(\tau,z;x,\x))[z]-\ii q_{j}(\tau,Z;x,\x)\\
%&a_0(0,z;x,\x)=0\,.
%\end{aligned}\right.
%\end{equation}
Setting
$g(\tau)=a^{+}_{m-j}(\tau,Z(\tau);x(\tau),\x(\tau))\,,
$ with $z,x,\x$ satisfying \eqref{Ham111prova},
%\eqref{Ham1}, \eqref{Ham2} 
we note that
\[
\pa_{\tau}g(\tau)=-\ii q_{j}(\tau,Z(\tau);x(\tau),\x(\tau))\quad \Rightarrow\quad
g(\tau)=-\int_{0}^{\tau}\ii q_{j}(\s,Z(\s);x(\s),\x(\s))d\s.
\]
Hence we have that (recall \eqref{flussoInv})
\[
a^{+}_{m-j}(\tau,Z;x,\x)=-\int_{0}^{\tau}\ii 
q_{j}(\s,\Upsilon^{\s}\circ\widetilde{\Upsilon}^{\sigma}(Z;x,\x))d\s\in 
\Sigma\Gamma^{m-\frac{j}{2}}_1[r,N]\,
\]
solves the problem \eqref{ord1}.
%\[
%a^{+}_{m-j}(\tau,z;x,\x)=-\int_{0}^{\tau}\ii 
%q_{j}(\s,\Phi^{\s}\circ(\Phi^{\tau})^{-1}(z);\gamma^{0,\s}\circ\gamma^{\tau,0}(x,\x))d\s
%\]
By iterating  the procedure above (by solving the problems \eqref{ord1})
we construct a symbol $a^{+}$ as in \eqref{ansatz} such that the following holds.
Define \[
Q^{\tau}(Z)=\opbw(\ii E A^{+}(\tau,Z;x,\x))[Z]
\]
with $A^{+}$
of the form \eqref{dono7} 
with $a^{+}$ as in \eqref{ansatz}.
Then 
the operator $Q^{\tau}(Z)$ solves the problem (recall \eqref{Nonlin2})
\begin{equation}\label{approx}
\left\{\begin{aligned}
&\pa_{\tau}Q^{\tau}(Z)=\big[ G^{\tau}(Z),
%\opbw(\ii E{\bf B}(\tau,Z;x,\x))[Z], 
Q^{\tau}(Z)\big]+
\mathcal{G}_{\rho}(\tau;Z)\\
&Q^{0}(Z)=\ii E\opbw(A(Z;x,\x))[Z]\,,
\end{aligned}\right.
\end{equation}
where
$\mathcal{G}_{\rho}(\tau;Z):=\opbw( G_{\rho}(\tau,Z;x,\x))[Z]$
for some matrix of symbols 
$G_{\rho}\in \Sigma\Gamma^{-\rho}_{1}[r,N]\otimes\mathcal{M}_2(\mathbb{C})$.
It remains to prove that 
the difference $Q^{\tau}-P^{\tau}$ is a smoothing remainder in
in $\Sigma\mathcal{R}^{-\rho}_{1}[r,N]\otimes\mathcal{M}_2(\mathbb{C})$.
First of all we write
\begin{equation}\label{strike5}
Q^{\tau}(Z)-P^{\tau}(Z)=V^{\tau}\circ (\Psi^{\tau})^{-1}(Z)\,,
\end{equation}
where, recalling \eqref{pushpush},
\[
V^{\tau}(U):=Q^{\tau}\circ\Psi^{\tau}(U)-(d_{U}\Psi^{\tau})(U)[X(U)]\,.
\]
%Setting $G^{\tau}(U):=\opbw(\ii E{\bf B}(\tau,U;x,\x))[U]$ (see \eqref{Nonlin2})
%and r
Recalling \eqref{Nonlin2},  \eqref{approx}, \eqref{sist2}, $Z=\Psi^{\tau}(U)$,
we deduce that
\[
\begin{aligned}
&\pa_{\tau}d_{U}\Psi^{\tau}(U)[\cdot]=\big(d_Z G^{\tau}\big)(\Psi^{\tau}(U))\big[ d_{U}\Psi^{\tau}(U)[\cdot]\big]\,,\\
&\big[G^{\tau}(\Psi^{\tau}(U)), Q^{\tau}(\Psi^{\tau}(U)) \big]=
\big(d_{Z} G^{\tau}\big)(\Psi^{\tau}(U))\Big[Q^{\tau}(\Psi^{\tau}(U))\Big]-
\big(d_{Z} Q^{\tau}\big)(\Psi^{\tau}(U))\Big[G^{\tau}(\Psi^{\tau}(U))\Big]\,.
\end{aligned}
\]
Therefore we have that
%and using \eqref{Nonlin2}, \eqref{approx}, \eqref{operatoreEgor1},
%one can check that
\begin{equation}\label{probTotale}
\left\{\begin{aligned}
&\pa_{\tau}V^{\tau}(U)=(d_{Z}G^{\tau})(\Psi^{\tau})\big[ V^{\tau}\big]+\mathcal{G}_{\rho}(\Psi^{\tau})\,,\\
&V^{0}(U)=-R(U)[U]\,.
\end{aligned}\right.
\end{equation}
We claim that $V^{\tau}\in \Sigma\mathcal{R}^{-\rho}_{1}[r,N]\otimes\mathcal{M}_2(\mathbb{C})$.
We start by studying the flow generated by $(d_{Z}F^{\tau})(\Psi^{\tau})[\cdot]$,
denoted by $\Phi^{\tau}_{dG}$,
by using formula \eqref{strike}.
The existence of such flow can be deduced reasoning as in Lemma \ref{flusso-differenziale}.
We now provide estimates on Sobolev spaces.
Let us denote $\Phi^{\tau}_{B}$ the (linear) flow generated by 
$\opbw(\ii E{\bf B}(\tau,\Psi^{\tau}(U);x,\x))[\cdot]$ with $U\in H^{s}$.  We have that
\begin{equation}\label{strike2}
\|\Phi^{\tau}_{B}h\|_{H^{s}}\lesssim_{s}\|h(0)\|_{H^{s}}(1+\|\Psi^{\tau}(U)\|_{H^{s_0}})\lesssim_{s}
\|h(0)\|_{H^{s}}(1+\|U\|_{H^{s_0}})\,.
\end{equation}
Therefore we write (recall \eqref{strike})
\[
\Phi^{\tau}_{dG}h=\Phi^{\tau}_{B}h(0)+\Phi^{\tau}_{B}\int_{0}^{\tau}(\Phi^{\s}_{B})^{-1}
\opbw\big(\ii E(d_{W}{\bf B})(\s,\Psi^{\s}(U);x,\x)[\Phi^{\tau}_{dG}h]\big)[\Psi^{\s}(U)]d\s\,,
\]
where $W=\Psi^{\s}(U)$.
Hence, using \eqref{strike2} and Proposition \ref{azionepara},
\[
\|\Phi^{\tau}_{dG}h\|_{H^{s-1}}\lesssim_{s}\|h(0)\|_{H^{s-1}}+\|\Phi^{\tau}_{dG}h\|_{H^{s_0}}\|U\|_{H^{s}}\,,
\]
from which we deduce $\|\Phi^{\tau}_{dG}h\|_{H^{s-1}}\lesssim_{s}\|h(0)\|_{H^{s-1}}$, for any $s$,
if $U\in B_{r}(H^{s})$ with 
$r>0$ is small enough. 
Then we obtain, using the Duhamel formula on \eqref{probTotale},
\[
\|V^{\tau}(U)\|_{H^{s+\rho-1}}\lesssim_{s}
\|\Phi^{\tau}_{dF}R(U)U\|_{H^{s+\rho-1}}+
\|
\Phi^{\tau}_{dG}\int_{0}^{\tau}(\Phi^{\s}_{dG})^{-1}\mathcal{G}_{\rho}(\s;\Psi^{\s})d\s
\|_{H^{s+\rho-1}}\lesssim_{s}\|U\|_{H^{s_0}}\|U\|_{H^{s}}\,.
\]
This proves the \eqref{porto20} for $V^{\tau}$ with $k=0$.
The estimates on the differentials in $U$ (with $k\geq1$)
follow by 
differentiating the equation \eqref{probTotale} and reasoning in the same way.
As a consequence, using also estimates \eqref{stimaflusso3} on the flow $\Psi^{\tau}$,
we have that the operator in \eqref{strike5}
belongs to $\mathcal{R}^{-\rho}_{1}[r]\otimes\mathcal{M}_2(\mathbb{C})$. 
The fact that the operator $Q^{\tau}-P^{\tau}$
admits an expansion in homogeneous operators, i.e. it is in 
$\Sigma\mathcal{R}^{-\rho}_{1}[r,N]\otimes\mathcal{M}_2(\mathbb{C})$,
follows by using the expansion in homogeneous symbols and operators
of ${\bf B}(\tau,Z;x,\x)$ and of $\Psi^{\tau}(U)-U\in 
\Sigma\mathcal{M}_{1}[r,N]\otimes\mathcal{M}_2(\mathbb{C})$.
% which 
%is in 
% $M(\tau;u)\in \Sigma\mathcal{M}_{1,K}[r,N]$.
 This proves the \eqref{Nonlin3}.

\subsubsection{Reduction at lower orders}\label{diagoLower}

In this section 
we study how a para-differential vector field conjugate
under the flow in \eqref{flusso}
with $f(\tau,u;x,\x)$ as in \eqref{sim2} or \eqref{sim3}.
Let us consider  
\begin{equation*}
\begin{aligned}
&M(U;\x)=(1+\mathfrak{m}(U))f_{m}(\x)+\widetilde{M}(U;\x)\,,
\qquad \mathfrak{m}\in\mathcal{F}^{\R}_{1}[r,N]\,, 
\qquad \widetilde{M}(U;\x)\in \Sigma\Gamma^{m-\frac{1}{2}}_1[r,N]\,,\\
&\mathbb{R}\ni f_{m}(\x)\in\Gamma^{m}_0\,,
\qquad m>1, \quad M(U;\x)\;\; {\rm independent \;of}\; x\in\mathbb{T}\,,\\
& M(U;\xi)-\overline{M(U;\xi)}\in\Sigma\Gamma^0_1[r,N].
\end{aligned}
\end{equation*}
We also assume that \eqref{ipotesim} holds.
Consider the operator $X(U)$ as in \eqref{operatoreEgor1} with $A(U;x,\x)$ 
as in \eqref{tordo} and we assume
%satisfying
%\begin{equation}\label{operatoreEgor1redu}
%\begin{aligned}
%%&X(U):=\ii E\opbw(A(U;x,\x))[U]+R(U)[U]\,,\qquad E=\sm{1}{0}{0}{-1}\,, 
%%\quad R\in \Sigma\mathcal{R}^{-\rho}_{1}[r,N]\otimes\mathcal{M}_2(\mathbb{C})\,,\\
%&A(U;x,\x):=\left(\begin{matrix}
%a(U;x,\x) & 0 \\
%0 &\ov{a(U;x,-\x)}
%\end{matrix}\right)\,.
%%\quad a(U;x,\x) \in \Sigma\Gamma^{m}_{1}[r,N]\,,
%%\quad m>1\,.
%\end{aligned}
%\end{equation}
%We also assume 
that the symbol $a(U;x,\x)$ has the form
\begin{equation}\label{Forma-di-Aredu}
\begin{aligned}
&a(U;x,\x)=M(U;\x)+a_{m'}(U;x,\x)\,,\quad
a_{m'}(U;x,\x) \in\Sigma\Gamma^{m'}_{1}[r,N]\,,\quad m'=m-\frac{p}{2}\,,
\end{aligned}
\end{equation}
for some $p\geq 1$
and
\begin{equation*}
a_{m'}(U;x,\x)-\ov{a_{m'}(U;x,\x)}\in\Sigma\Gamma^0_1[r,N].
\end{equation*}
The symbol $a$ in \eqref{Forma-di-Aredu} satisfies the assumptions \eqref{low1}, \eqref{low2}, \eqref{low3}. Hence, 
by Theorem \ref{constEgolower},
there exists a symbol $c(\tau,U;x,\x)$
%\in \Sigma\Gamma^{\delta}_{1}[r,N]$,
%with $ \delta=m'-m+1$ 
satisfying \eqref{simboCC}, 
such that equation \eqref{equaLower}
is verified.
Let us now consider the  matrix of symbols
\begin{equation}\label{simboBredu}
\begin{aligned}
&{\bf C}(\tau,U;x,\x):=\left(\begin{matrix}
 c(\tau,U;x,\x) & 0 \\ 0 &  \ov{c(\tau,U;x,-\x)} 
\end{matrix}
\right)\,
%\quad
%C(\tau,U;x,\x)\in \Sigma\Gamma^{\delta}_{1}[r,N]\,,\quad \delta=m'-m+1\,.
\end{aligned}
\end{equation}
%If $\delta>0$ we shall assume that $C(\tau,U;x,\x)$ is real valued.
and let 
 $\Psi^{\tau}$ be the flow of
 \eqref{sist2} with 
 \begin{equation}\label{Nonlin2redu}
 G^{\tau}(U):= \ii E\opbw( {\bf C}(\tau,U;x,\x))[U]\,.
%\left\{
%\begin{aligned}
%&\pa_\tau\Psi^{\tau}(U)=\opbw(\ii E {\bf C}(\tau,\Psi^{\tau}(U);x,\x))[\Psi^{\tau}(U)]\\
%&\Psi^{0}(U)=U=\vect{u}{\bar{u}}\,.
%\end{aligned}\right.
\end{equation}
Notice that this  flow  
%\eqref{Nonlin2redu} 
is well-posed by Theorem \ref{flussononlin}.
We define
\begin{equation}\label{variabNuovaredu}
Z:=\vect{z}{\bar{z}}:={\bf \Psi}^{\tau}(U)_{|\tau=1}\,.
\end{equation}
By using  \eqref{simboBredu} one has
\[
{\bf \Psi}^{\tau}(U)_{|\tau=1}
:=\left( 
\begin{matrix}
\Phi^{\tau}(u) \vspace{0.2em}\\
\ov{\Phi^{\tau}(u)}
\end{matrix}
\right)_{|\tau=1}\,,\quad \pa_\tau\Phi^{\tau}(u)=\opbw(\ii c(\tau,\Phi^{\tau}(u);x,\x))[\Phi^{\tau}(u)]\,.
\]
The main result of this section is the following.
\begin{theorem}\label{conjOrdMaxredu}
For $r>0$ small enough 
%there exists a symbol $C(\tau,W;y,\x)
%\in\Sigma\Gamma^{\delta}_{1}[r,N]$, 
%$\tau\in[0,1]$,  such that, 
%the following holds.
the conjugate of $X$ in \eqref{Nonlin1} (with the assumption \eqref{Forma-di-Aredu}) 
has the form (see \eqref{variabNuovaredu})
\begin{equation*}
\dot{Z}=\ii E\opbw(A^{+}(Z;x,\x))[Z]+R^{+}(Z)[Z]\,,\qquad 
R^{+}\in \Sigma\mathcal{R}^{-\rho}_{1}[r,N]\otimes\mathcal{M}_2(\mathbb{C})\,,
\end{equation*}
and the matrix of symbols 
$A^{+}(Z;x,\x)$
has the form
\begin{equation*}\begin{aligned}
&A^{+}(Z;x,\x):=\left(\begin{matrix}
a^{+}(Z;x,\x) & 0 \\
0 &\ov{a^{+}(Z;x,-\x)}
\end{matrix}\right)\,,\qquad
m'':=m-\frac{p+1}{2}\,,
\\&
a^{+}(Z;x,\x)=M((\Psi^{1})^{-1}(Z);\x)+\widetilde{\mathfrak{m}}(Z;\x)
+a^{+}_{m''}(Z;x,\x)\,,%\quad m'':=m-\frac{p+1}{2}\,,
\quad
a^{+}_{m''}(Z;x,\x) \in\Sigma\Gamma^{m''}_{1}[r,N]\,,
\\
&\widetilde{\mathfrak{m}}(Z;\x)-\overline{\widetilde{\mathfrak{m}}(Z;\x)}\in\Sigma\Gamma^0_1[r,N]\, ,
%\quad \mbox{ if }\,\,\, m-\frac{p}{2}>0,
\\
&a^{+}_{m''}(Z;x,\x)-\overline{a^{+}_{m''}(Z;x,\x)}
\in\Sigma\Gamma^0_1[r,N]\,,
%\quad \mbox{ if }\,\,\, m">0.
\end{aligned}\end{equation*}
and where $\widetilde{\mathfrak{m}}\in \Sigma\Gamma_1^{m'}[r,N]$ 
 is \emph{independent} of $x\in \mathbb{T}$.
\end{theorem}
The rest of the section is devoted to the proof of the result above
and we will follow the strategy used in section \ref{sec:egoHigh}.
The system \eqref{Nonlin1} in the new 
coordinates \eqref{variabNuovaredu}
has the form \eqref{sistnew}-\eqref{Ego} with $G^{\tau}$
as in \eqref{Nonlin2redu}.
%reads (recall also \eqref{pushpush})
%\begin{equation}\label{sistnewredu}
%\begin{aligned}
%&\dot{Z}=P^{1}(Z)=(P^{\tau}(z))_{\tau=1}\,,\\
%&\pa_{\tau}P^{\tau}(Z)=\big[\opbw(\ii E {\bf C}(\tau,Z;x,\x))[Z] , P^{\tau}(Z)\big]\,,
%\qquad P^{0}(Z)=X(Z)\,.
%\end{aligned}
%\end{equation}
We recall also that, 
by the remarks under Definition \ref{smoothoperatormaps} and by Theorem \ref{flussononlin},
we have 
\begin{equation}\label{dono1redu}
P^{\tau}(Z)=\ii E\Omega Z+M_1(\tau;Z)[Z]\,,\;\;
M_1\in \Sigma\mathcal{M}_{1}[r,N]\otimes\mathcal{M}_2(\mathbb{C})
%\qquad {\rm and\;\; hence}\qquad 
%P^{\tau}\in  \Sigma\mathcal{M}_{1}[r,N]\otimes\mathcal{M}_2(\mathbb{C})\,,
\end{equation}
with estimates uniform in $\tau\in[0,1]$, where $\Omega$ is in \eqref{omegone}.
We  shall look for a solution of the Heisenberg equation
%\eqref{sistnewredu} 
of the form
\begin{equation}\label{forma-totaleredu}
P^{\tau}(Z)
=\ii E\Omega Z+M_1(\tau;Z)[Z]=\ii E\opbw(A^{+}(\tau,Z;x,\x))[Z]+R^{+}(\tau,Z)[Z]\,,
\end{equation}
with $A^{+}$, $R^{+}$ as follows %as in \eqref{Nonlin3}, \eqref{nuovoSimboA}.
\begin{equation}\label{dono7redu}
A^{+}(\tau,Z;x,\x):=\left(\begin{matrix}
a^{+}(\tau,Z;x,\x) & 0 \\
0 &\ov{a^{+}(\tau,Z;x,-\x)}
\end{matrix}\right)\,, \qquad
R^{+}\in \Sigma\mathcal{R}^{-\rho}_1[r,N]\otimes\mathcal{M}_{2}(\mathbb{C})\,.
\end{equation}
In particular 
we make the ansatz (recall \eqref{ipotesim})
\begin{equation}\label{ansatzredu}
\begin{aligned}
&a^{+}(\tau,Z;x,\x)=a_{m}^{+}(\tau,Z;\x)+
a_{m-\frac{1}{2}}^{+}(\tau,Z;\x)
+\sum_{j=0}^{2\rho+m}a^{+}_{m-\frac{p+j}{2}}(\tau,Z;x,\x)\,,
\\& a^{+}_{m-\frac{j}{2}}(\tau,Z;x,\x)\in 
\Sigma\Gamma^{m-\frac{j}{2}}_{1}[r,N] \quad \forall j\in\NNN
\\
&a_{m}^{+}(Z;\x)=(1+\mathfrak{m}^{+}(\tau,Z))f_{m}(\x)\,,
\qquad \mathfrak{m}^{+}\in\mathcal{F}^{\mathbb{R}}_{1}[r,N]\,, 
\qquad \\
&\widetilde{M}^{+}(Z;\x):=
a^{+}_{m-\frac{1}{2}}(\tau,Z;\x)\,\in \Sigma\Gamma^{m-\frac{1}{2}}_{1}[r,N]
\;\; {\rm is\; independent \;of}\; x\in\mathbb{T},\\
&\widetilde{M}^{+}(U;\x)-\overline{\widetilde{M}^{+}(U;\x)}\in\Sigma\Gamma^0_1[r,N].
\end{aligned}
\end{equation}
Expanding the non-linear commutator as in 
\eqref{commuExp100}-\eqref{commuExpF}
we get
\begin{equation}\label{commuExp100redu}
\big[\opbw(\ii E {\bf C}(\tau,Z;x,\x))[Z] , P^{\tau}(Z)\big]=
\ii E\opbw\left(\begin{matrix}
d(\tau,Z;x,\x) & 0 \\ 0 &
\ov{d(\tau,Z;x,-\x)}
\end{matrix}
\right)[Z]+\mathcal{G}_{\rho}(Z)[Z]
\end{equation}
where 
$\mathcal{G}_{\rho}\in \Sigma\mathcal{R}^{-\rho}_1[r,N]\otimes\mathcal{M}_2(\mathbb{C})$ and
\begin{equation}\label{newSimboDD}
\begin{aligned}
\ii d(\tau,Z;x,\x)&=
\ii c(\tau,Z;x,\x)\#_{\rho}\ii a^{+}(\tau,Z;x,\x)-
\ii a^{+}(\tau,Z;x,\x)\#_{\rho}\ii c(\tau,Z;x,\x)\\
&+\ii (\pa_{t}c)(\tau,Z;x,\x)-\ii(d_{Z}a^{+})(\tau,Z;x,\x)\big[\ii E \opbw({\bf C}(\tau,Z;x,\x))[Z]\big]
\end{aligned}
\end{equation}
with
\[
(\pa_{t}c)(\tau,Z;x,\x):=(d_{Z}{c})(\tau,Z;x,\x)[P^{\tau}(Z)]\,.
\]
In order to lighten the notation we shall sometimes omit  the dependence on 
$(\tau,Z;x,\x)$ of the symbols.
By recalling the ansatz \eqref{ansatzredu} and using the expansion \eqref{espansione2}
we get
\begin{equation}\label{dono11redu}
\begin{aligned}
d&=\{c, a_{m}^{+}\}+\sum_{j=1}^{m+2\rho}q_{j}
-(d_{Z}a^{+})(\tau,Z;x,\x)\big[\ii E \opbw({\bf C}(\tau,Z;x,\x))[Z]\big]\,,
\\&
q_{j}\in \Sigma\Gamma^{m-\frac{p+j}{2}}_1[r,N]\,, \;
\,j=1,\ldots,2\rho+m\,,\; m-\frac{p+j}{2}\neq
\delta\stackrel{\eqref{simboBredu}, \eqref{Forma-di-Aredu}}{=} 1-\frac{p}{2}\,,\\
&q_{j}:=(\pa_{t}c)+r_{j}\,,\;\; m-\frac{p+j}{2}:= 1-\frac{p}{2}\,,\quad r_{j}\in
\Sigma\Gamma^{\delta}_1[r,N] \,.
\end{aligned}
\end{equation}
Moreover the 
 symbols $q_{j}$ depend only on $c$, $a_{m-j/2}^{+}$, 
 %$a_{m_1}^{+}$, $a_{m'}^{+}$,
%and  
%and on the symbols 
$q_{k}$, with $k<j$.
%We define
%\begin{equation}\label{dono13redu}
%\begin{aligned}
%q_{\delta}&:=r_{\delta}+\ii (\pa_{t}C)\,,\\
%q_{j}&:=\ii\{C,a_{m}^{+}\}\,,\quad j=m+\delta-1\,,\\
%q_{j}&:=\ii\{C,a_{m_1}^{+}\}\,,\quad j=m_1+\delta-1\,,\\
%q_{j}&:=\ii\{C,a_{m'}^{+}\}\,,\quad j=m'+\delta-1\,.
%\end{aligned}
%\end{equation}
Our aim is to prove the ansatz \eqref{forma-totaleredu} using 
the expansion \eqref{newSimboDD}.
We shall solve   the Heisenberg equation
%equation \eqref{sistnewredu}
 iteratively expanding 
the symbol in \eqref{newSimboDD} in decreasing orders as in \eqref{dono11redu}. 

\noindent
{\bf Order $m$.}
By \eqref{newSimboDD}, \eqref{espansione2}, 
we have that, at the highest order,  the Heisenberg equation
%the equation \eqref{sistnewredu}
reads
%Using the expansion of the symbol $a$ in \eqref{ansatz}
%we have, at the highest order,
%the equation
\begin{equation}\label{ordMaxredu2}
\left\{\begin{aligned}
&\pa_{\tau}a^{+}_m(\tau,Z;\x)=-
d_{Z}a^{+}_m(\tau,Z;\x)[\opbw(\ii E {\bf C}(\tau,Z;x,\x))[Z]]\\
&a^{+}_m(0,Z;\x)=(1+\mathfrak{m}(Z))f_{m}(\x)\,.
\end{aligned}\right.
\end{equation}
Notice that the function 
\begin{equation}\label{eq:glowoff}
g(\tau)=a^{+}_m(\tau,Z(\tau);\x)
\end{equation}
is constant along the solution generated by  \eqref{Nonlin2redu}.
By Theorem \ref{flussononlin} such flow  is well posed and invertible.
Hence the symbol
\begin{equation}\label{solOrdMaxredu}
a^{+}_m(\tau,Z;\x)=(1+\mathfrak{m}((\Psi^{\tau})^{-1}(Z)))f_{m}(\x)
=:(1+\mathfrak{m}^{+}(\tau,Z))f_{m}(\x)\,,
\end{equation}
belongs to $ \Sigma\Gamma^{m}_{1}[r,N]$ and
solves the problem \eqref{ordMaxredu2}. 
The symbol $a_m^+(\tau,Z;\x)$ in \eqref{solOrdMaxredu} 
is constant in $x$.
We have  verified the ansatz \eqref{ansatzredu} at the highest order.

\noindent
{\bf Order $m-1/2$.}
By \eqref{newSimboDD}, \eqref{espansione2}, 
we have that, at order $m-1/2$, the Heisenberg equation
% the equation \eqref{sistnewredu}
reads
%Using the expansion of the symbol $a$ in \eqref{ansatz}
%we have, at the highest order,
%the equation
\begin{equation}\label{ordMaxredu22}
\left\{\begin{aligned}
&\pa_{\tau}a^{+}_{m-1/2}(\tau,Z;x,\x)=-
d_{Z}a^{+}_{m-\frac{1}{2}}(\tau,Z;x,\x)[\opbw(\ii E {\bf C}(\tau,Z;x,\x))[Z]]\\
&a^{+}_{m-\frac{1}{2}}(0,Z;x,\x)=\widetilde{M}(Z;\x)\,.
\end{aligned}\right.
\end{equation}
Notice that the function $g(\tau)=a^{+}_{m-1/2}(\tau,Z(\tau);x,\x)$
is constant along the solution of \eqref{Nonlin2redu}, therefore
the symbol
\begin{equation*}
a^{+}_{m-1/2}(\tau,Z;x,\x)=\widetilde{M}((\Psi^{\tau})^{-1}(Z);\x)=:\widetilde{M}^{+}(\tau,Z;\x)
\end{equation*}
belongs to $ \Sigma\Gamma^{m-\frac{1}{2}}_{1}[r,N]$ and
solves the problem \eqref{ordMaxredu22}. 
The symbol $a_{m-\frac{1}{2}}^{+}(\tau,Z;x,\x)$ in \eqref{solOrdMaxredu} clearly
is constant in $x$.
Hence it verifies the ansatz \eqref{ansatzredu}. 

\smallskip
\noindent
{\bf Order $m-p/2$.}
By \eqref{newSimboDD}, \eqref{dono11redu}, %\eqref{espansione2}, 
we have that, at order $m-p/2$,  the Heisenberg equation
%the equation %\eqref{sistnewredu} 
reads
\begin{align}
&\pa_{\tau}a^{+}_{m-\frac{p}{2}}(\tau,Z;x,\x)=
\{C(\tau,Z;x,\x),  (1+\mathfrak{m}^{+}(\tau,Z))f_{m}(\x)\}
-
d_{Z}a^{+}_{m-\frac{p}{2}}(\tau,Z;x,\x)[\opbw(\ii E {\bf C}(\tau,Z;x,\x))[Z]]
\nonumber\\
&a^{+}_{m-\frac{p}{2}}(0,Z;x,\x)=a_{m-\frac{p}{2}}(Z;x,\x)\,.\label{ordMaxredu222}
\end{align}
Notice that the function $g(\tau)=a^{+}_{m-\frac{p}{2}}(\tau,Z(\tau);x,\x)$
satisfies
\[
\begin{aligned}
\pa_{\tau}g(\tau)&=\{c(\tau,Z;x,\x),  (1+\mathfrak{m}^{+}(\tau,Z))f_{m}(\x)\}=
-(1+\mathfrak{m}^{+}(\tau,Z))(\pa_{\x}f_{m})(\x)(\pa_{x}c)(\tau,Z;x,\x)\,,\\
&g(0)=a_{m-\frac{p}{2}}(Z(0);x,\x)=a_{m-\frac{p}{2}}(U;x,\x)\,.
\end{aligned}
\]
Therefore
\[
\begin{aligned}
a_{m-\frac{p}{2}}^{+}(\tau,Z,x,\x)=&a_{m-\frac{p}{2}}((\Psi^{\tau})^{-1}(Z);x,\x)\\
&\quad -\int_{0}^{\tau}
(1+\mathfrak{m}^{+}(\s,\Psi^{\s}(\Psi^{\tau})^{-1}(Z)))
(\pa_{\x}f_{m})(\x)(\pa_{x}c)(\s,\Psi^{\s}(\Psi^{\tau})^{-1}(Z);x,\x)d\s
\end{aligned}
\]
solves the problem \eqref{ordMaxredu222}. 
Moreover, by Theorem \ref{constEgolower} (see \eqref{equaLower}),
we have
%one can choose a symbol $C(\tau,Z;x,\x)$ as in \eqref{simboBredu} 
%such that
\begin{equation*}
a^{+}_{m-\frac{p}{2}}(1,Z;x,\x)=\widetilde{\mathfrak{m}}(Z;\x)\stackrel{\eqref{equaLower}}{:=}
\frac{1}{2\pi}\int_{\mathbb{T}}a_{m-\frac{p}{2}}\big((\Psi^{1})^{-1}(Z);x,\x\big)dx\,,
\end{equation*}
i.e. $a^{+}_{m-\frac{p}{2}}(\tau,Z;x,\x)$
at $\tau=1$ is constant  in $x\in\mathbb{T}$.

\smallskip
\noindent
{\bf Lower orders.} 
Recalling \eqref{dono11redu}, %\eqref{dono13redu}, 
we have that the equation \eqref{Ego}, at order
$m-(p+j)/2$, 
reads
\begin{align}
&\pa_{\tau}a^{+}_{m-\frac{p+j}{2}}(\tau,Z;x,\x)=
-d_{Z}a^{+}_{m-\frac{p+j}{2}}(\tau,Z;x,\x)[\opbw(\ii E{\bf C}(\tau,Z;x,\x))[Z]]
+ q_{j}(\tau,Z;x,\x)\label{ord1redu}\\
&a_{m-\frac{p+j}{2}}^{+}(0,z;x,\x)=0\nonumber\,.
\end{align}
Setting
\begin{equation*}
g(\tau)=a^{+}_{m-\frac{p+j}{2}}(\tau,Z(\tau);x,\x)
\end{equation*}
with $Z(\tau)$ satisfying \eqref{Nonlin2redu}
%\eqref{Ham1}, \eqref{Ham2} 
we note that
\[
\pa_{\tau}g(\tau)= q_{j}(\tau,Z(\tau);x,\x)\quad \Rightarrow\quad
g(\tau)=\int_{0}^{\tau} q_{j}(\s,Z(\s);x,\x)d\s.
\]
Hence we have that
\[
a^{+}_{m-\frac{p+j}{2}}(\tau,Z;x,\x)=\int_{0}^{\tau}
q_{j}(\s,\Psi^{\s}(\Psi^{\tau})^{-1}(Z);x,\x)d\s\in \Sigma\Gamma^{j}_1[r,N]\,,
\]
solves the problem \eqref{ord1redu}.
%\[
%a^{+}_{m-j}(\tau,z;x,\x)=-\int_{0}^{\tau}\ii 
%q_{j}(\s,\Phi^{\s}\circ(\Phi^{\tau})^{-1}(z);\gamma^{0,\s}\circ\gamma^{\tau,0}(x,\x))d\s
%\]
To summarize by iterating  the procedure above (by solving the problems \eqref{ord1redu})
we construct a symbol $a^{+}$ as in \eqref{ansatzredu} such that the following holds.
Define \[
Q^{\tau}(Z)=\opbw(\ii E A^{+}(\tau,Z;x,\x))[Z]
\]
with $A^{+}$
of the form \eqref{dono7redu} 
with $a^{+}$ as in \eqref{ansatzredu}.
Then 
the operator $Q^{\tau}(Z)$ solves the problem
\begin{equation*}
\left\{\begin{aligned}
&\pa_{\tau}Q^{\tau}(Z)=\big[\opbw(\ii E{\bf C}(\tau,Z;x,\x))[Z], Q^{\tau}(Z)\big]+
\mathcal{G}_{\rho}(\tau;Z)\\
&Q^{0}(Z)=\ii E\opbw(A(Z;x,\x))[Z]\,,
\end{aligned}\right.
\end{equation*}
where
$\mathcal{G}_{\rho}(\tau;Z):=\opbw(G_{\rho}(\tau,Z;x,\x))[Z]$
for some matrix of symbols 
$G_{\rho}\in \Sigma\Gamma^{-\rho}_{1}[r,N]\otimes\mathcal{M}_2(\mathbb{C})$.
The fact that 
the difference $Q^{\tau}-P^{\tau}$ is a smoothing remainder 
in $\Sigma\mathcal{R}^{-\rho}_{1}[r,N]\otimes\mathcal{M}_2(\mathbb{C})$
can be proved following word by word the conclusion of the proof of Theorem
\ref{conjOrdMax}.

\subsection{Diagonalization of the para-differential matrix}\label{diago-lineare}
In this section we show how to  diagonalize
a matrix of 
para-differential operators 
 up to smoothing remainders.

\subsubsection{Diagonalization of the matrix at the highest order}\label{sec:BlockDiago}

In this section 
we study how a para-differential vector field conjugates 
under the flow in \eqref{flusso}
with $f(\tau,u;x,\x)$ as in \eqref{sim3}.
Let us consider $X(U)$ as in \eqref{operatoreEgor1} with 
the matrix $A(U;x,\x)$ as follows
\begin{equation}\label{systExample}
\begin{aligned}
%&X(U):=\ii E\opbw(A(U;x,\x))[U]+R(U)[U]\,,\qquad E=\sm{1}{0}{0}{-1}\,, 
%\quad R\in \Sigma\mathcal{R}^{-\rho}_{1}[r,N]\otimes\mathcal{M}_2(\mathbb{C})\,,\\
&
A(U;x,\x):=A_{m}(U;x,\x)+A_{m'}(U;x,\x)\,,\qquad m'=m-\frac{1}{2}\,,
\qquad 
A_{m}(U;x,\x):=(\uno+\widetilde{A}_{m}(U;x))f_{m}(\x)\,,
\\&
\widetilde{A}_{m}(U;x):=\left(
\begin{matrix}
a_{m}(U;x) & b_{m}(U;x) \\
\ov{b_{m}(U;x)} & a_{m}(U;x)
\end{matrix}
\right)\in \Sigma\mathcal{F}_{1}[r,N]\otimes\mathcal{M}_2(\mathbb{C})\,,
\quad {a}_m(U,x)\in \Sigma\mathcal{F}_1^{\mathbb{R}}[r,N]\,
\\&A_{m'}(U;x,\x):=\left(\begin{matrix}
a_{m'}(U;x,\x) & b_{m'}(U;x,\x) \\
\ov{b_{m'}(U;x,-\x)} &\ov{a_{m'}(U;x,-\x)}
\end{matrix}\right)\in  \Sigma\Gamma_{1}^{m'}[r,N]
\otimes\mathcal{M}_2(\mathbb{C})\,,\\
&a_{m'}(U;x,\x)-\ov{a_{m'}(U;x,\x)}\in \Sigma\Gamma^{0}_1[r,N]\,.
\end{aligned}
\end{equation}
where $f_{m}\in \Gamma_0^{m}$ and in particular it is a 
$m$-homogeneous ${C^{\infty}(\mathbb{R}^+,\mathbb{R})}$.
We also assume that $f_{m}(\x)$ is even in $\x$ and the \eqref{ipotesim}.
Our aim is to conjugate system \eqref{Nonlin1} with hypotheses \eqref{systExample}
%\eqref{systExample} 
with the flow $\Psi_{C}^{\tau}(U)$ of \eqref{sist2} with 
\begin{equation}\label{flussoord0}
G^{\tau}(U):=\opbw\big({\bf C}(\tau,U;x)\big)[U]
%\left\{\begin{aligned}
%&\pa_{\tau}\Psi_{C}^{\tau}(U)=
%\opbw\big({\bf C}(\tau,\Psi_{C}^{\tau}(U);x)\big)[\Psi_{C}^{\tau}(U)]\,,\\
%&\Psi_{C}^{0}(U)=U\in B_r(H^{s})\,,
%\end{aligned}
%\right.
\qquad {\bf C}(\tau,U;x):=\left(
\begin{matrix}
0 & C(\tau,U;x) \\
\ov{C(\tau,U;x)} & 0
\end{matrix}
\right)
\end{equation}
where the symbol 
%for some symbol 
$C(\tau,U;x)$ in $ \Sigma\mathcal{F}_{1}[r,N]$ is given by Theorem \ref{constEgohighoff}.
Notice that the flow of \eqref{flussoord0} is well-posed 
by Theorem \ref{flussononlin} with generator as in \eqref{sim3}.
We define
\begin{equation}\label{variabNuovahighoff}
Z:=\vect{z}{\bar{z}}:={\bf \Psi}_{C}^{\tau}(U)_{|\tau=1}\,.
\end{equation}
The main result of this section is the following.

\begin{theorem}\label{conjOrdMaxhighoff}
For $r>0$ small enough 
%there exists a symbol 
%$C(\tau,U;x)\in\Sigma\mathcal{F}_1[r,N]$, 
%$\tau\in[0,1]$,  such that, the following holds.
the conjugate of $X$ in \eqref{Nonlin1} with assumption 
\eqref{systExample} has the form (see \eqref{variabNuovahighoff})
\begin{equation*}
\dot{Z}=\ii E\opbw(A^{+}(Z;x,\x))[Z]+R^{+}(Z)[Z]\,,\qquad 
R^{+}\in R\in \Sigma\mathcal{R}^{-\rho}_{1}[r,N]\otimes\mathcal{M}_2(\mathbb{C})\,,
\end{equation*}
where
\begin{equation}\label{nuovoSimboAhighoff}
\begin{aligned}
&
A^{+}(Z;x,\x):=A_{m}^{+}(Z;x,\x)+A^{+}_{m'}(Z;x,\x)\,,
\qquad 
A_{m}^{+}(Z;x,\x):=(\uno+\widetilde{A}_{m}^{+}(Z;x))f_{m}(\x)\,,
\\&
\widetilde{A}_{m}^{+}(Z;x):=\left(
\begin{matrix}
a^{+}_{m}(Z;x) & 0 \\
0 & a^{+}_{m}(Z;x)
\end{matrix}
\right)\in \Sigma\mathcal{F}^{\mathbb{R}}_{1}[r,N]
\otimes\mathcal{M}_2(\mathbb{C})\,,
\\&A^{+}_{m'}(Z;x,\x):=\left(\begin{matrix}
a^{+}_{m'}(Z;x,\x) & b^{+}_{m'}(Z;x,\x) \vspace{0.2em}\\
\ov{b^{+}_{m'}(Z;x,-\x)} &\ov{a^{+}_{m'}(Z;x,-\x)}
\end{matrix}\right)\in  \Sigma\Gamma_{1}^{m'}[r,N]
\otimes\mathcal{M}_2(\mathbb{C})\,,\\
&a_{m'}^+(U;x,\x)-\ov{a_{m'}^+(U;x,\x)}\in \Sigma\Gamma^{0}_1[r,N]\,.
\end{aligned}
\end{equation} 
\end{theorem}
The rest of the section is devoted to the proof of the result above
and we will follow the strategy used in section \ref{sec:egoHigh}.

%In the following we shall write
%$\Psi:=\Psi^{\tau}_{|\tau=1}$.
The system \eqref{systExample} in the new 
coordinates \eqref{variabNuovahighoff}
has the form \eqref{sistnew}-\eqref{Ego} with $G^{\tau}$
as in \eqref{flussoord0}.
%reads (recall also \eqref{pushpush})
%\begin{equation}\label{sistnewhighoff}
%\begin{aligned}
%&\dot{Z}=P^{1}(Z)=(P^{\tau}(z))_{\tau=1}\,,\\
%&\pa_{\tau}P^{\tau}(Z)=\big[\opbw( {\bf C}(\tau,Z;x,\x))[Z] , P^{\tau}(Z)\big]\,,
%\qquad P^{0}(Z)=X(Z)\,.
%\end{aligned}
%\end{equation}
We note that 
because of the remarks under Definition \ref{smoothoperatormaps} and by Theorem \ref{flussononlin},
\begin{equation*}
P^{\tau}(Z)=\ii E\Omega Z+M_1(\tau;Z)[Z]\,,\;\;\;
M_1\in \Sigma\mathcal{M}_{1}[r,N]\otimes\mathcal{M}_2(\mathbb{C})\,,
\end{equation*}
with estimate uniform in $\tau\in[0,1]$, where $\Omega$ is in \eqref{omegone}.
We  shall look for a solution of the Heisenberg equation 
%\eqref{sistnewhighoff} 
of the form
\begin{equation}\label{forma-totalehighoff}
P^{\tau}(Z)
=\ii E \Omega Z+M_1(\tau;Z)[Z]=\ii E\opbw(A^{+}(\tau,Z;x,\x))[Z]+R^{+}(\tau,Z)[Z]\,,
\end{equation}
with 
$R^{+}\in \Sigma\mathcal{R}^{-\rho}_1[r,N]\otimes\mathcal{M}_{2}(\mathbb{C})$
and $A^{+}$ of the form 
\begin{equation}\label{nuovoSimboAhighoffbis}
\begin{aligned}
A^{+}(\tau,Z;x,\x):=\left(\begin{matrix}
a^{+}(\tau,Z;x,\x) & b^{+}(\tau,Z;x,\x) \vspace{0.2em}\\
\ov{b^{+}(\tau,Z;x,-\x)} &\ov{a^{+}(\tau,Z;x,-\x)}
\end{matrix}\right)\,.
%\in \Sigma\Gamma^{m}_{1}[r,N]
%\otimes\mathcal{M}_2(\mathbb{C})\,.
\end{aligned}
\end{equation} 
In particular 
we make the ansatz
\begin{equation}\label{ansatzhighoff}
\begin{aligned}
&a^{+}(\tau,Z;x,\x)=(1+a_{m}^{+}(\tau,Z;x))f_{m}(\x)+
%a^{+}_{m-\frac{1}{2}}(\tau,Z;x,\x)+
\sum_{j=1}^{2(m+\rho)}a^{+}_{m-\frac{j}{2}}(\tau,Z;x,\x)\,,\\
&b^{+}(\tau,Z;x,\x)=
b_{m}^{+}(\tau,Z;x)f_{m}(\x)+
%b_{m-\frac{1}{2}}^{+}(\tau,Z;x,\x)+
\sum_{j=1}^{2(m+\rho)}b_{m-\frac{j}{2}}^{+}(\tau,Z;x,\x)\,,
\\
&a_{m}^{+}\in\mathcal{F}^{\mathbb{R}}_{1}[r,N]\,,
\quad b_{m}^{+}\in\mathcal{F}_{1}[r,N]
 \quad
a^{+}_{m-\frac{j}{2}}, b_{m-\frac{j}{2}}^{+}\in \Sigma\Gamma^{m-\frac{j}{2}}_1[r,N]\,, \;\;j=1,\ldots, 2(m+\rho)\,.
%\quad a^{+}_{m-m'}, b^{+}_{m-m'}\in \Sigma\Gamma^{m'-m}_1[r,N]\,,
\end{aligned}
\end{equation}
%In order to lighten the notation we shall omit to write the dependence on 
%$(\tau,Z;x,\x)$ in the symbols.
Expanding the non-linear commutator as in \eqref{commuExp100}-\eqref{commuExpF}
we get
\begin{equation*}
\big[\opbw({\bf C}(\tau,Z;x,\x))[Z] , P^{\tau}(Z)\big]=\ii E\opbw
\left(\begin{matrix}
d_1(\tau,Z;x,\x) & d_2(\tau,Z;x,\x) \\ 
\ov{d_2(\tau,Z;x,-\x)} &
\ov{d_1(\tau,Z;x,-\x)}
\end{matrix}
\right)Z+\mathcal{G}_{\rho}(Z)[Z]
\end{equation*}
where 
$\mathcal{G}_{\rho}\in \Sigma\mathcal{R}^{-\rho}_1[r,N]\otimes\mathcal{M}_2(\mathbb{C})$ and
\begin{equation}\label{newSimboDDhighoff1}
\begin{aligned}
d_1&=-
C(\tau,Z;x)\#_{\rho} \ov{b^{+}(\tau,Z;x,-\x)}-
 b^{+}(\tau,Z;x,\x)\#_{\rho}\ov{C(\tau,Z;x)}\\
&
%+\ii (\pa_{t}C)(\tau,Z;x,\x)
-(d_{Z}a^{+})(\tau,Z;x,\x)\big[ \opbw({\bf C}(\tau,Z;x))[Z]\big]
\end{aligned}
\end{equation}
\begin{equation}\label{newSimboDDhighoff2}
\begin{aligned}
d_2&=-
C(\tau,Z;x)\#_{\rho} \ov{a^{+}(\tau,Z;x,-\x)}-
 a^{+}(\tau,Z;x,\x)\#_{\rho}{C(\tau,Z;x)}\\
&
+ (\pa_{t}C)(\tau,Z;x)
-(d_{Z}b^{+})(\tau,Z;x,\x)\big[ \opbw({\bf C}(\tau,Z;x))[Z]\big]
\end{aligned}
\end{equation}
with
\[
(\pa_{t}C)(\tau,Z;x):=(d_{Z}{C})(\tau,Z;x)[P^{\tau}(Z)]\,.
\]
By recalling the ansatz \eqref{ansatzhighoff} and using the expansion \eqref{espansione2}
we get
\begin{equation}\label{dono11highoff1}
\begin{aligned}
d_1&= -2{\rm Re}\Big(\ov{C(\tau,Z;x)}{b_{m}^{+}(\tau;Z,x)}\Big)f_{m}(\x)
%-C(\tau,Z;x)\ov{b_{m}^{+}(\tau;Z,x)}f_{m}(\x)-
%\ov{C(\tau,Z;x)}{b_{m}^{+}(\tau;Z,x)}f_{m}(\x)
%\\&
-2\sum_{j=1}^{2(m+\rho)}{\rm Re}\Big( 
\ov{C(\tau,Z;x)} b_{m-\frac{j}{2}}^{+}(\tau;Z,x,\x)
\Big)
\\&+\sum_{j=1}^{2(m+\rho)}q^{(1)}_{j}(\tau,Z;x,\x)
-(d_{Z}a^{+})(\tau,Z;x,\x)\big[ \opbw({\bf C}(\tau,Z;x,\x))[Z]\big]
\end{aligned}
\end{equation}

\begin{equation}\label{dono11highoff2}
\begin{aligned}
d_2&=-2C(\tau,Z;x)(1+a_{m}^{+}(\tau,Z;x))f_{m}(\x)
-\sum_{j=1}^{2(m+\rho)}\Big(a_{m-\frac{j}{2}}^{+}(\tau,Z;x,\x)
+\ov{a_{m-\frac{j}{2}}^{+}(\tau,Z;x,-\x)}
\Big)C(\tau,Z;x)
\\&+\sum_{j=1}^{2(m+\rho)}q^{(2)}_{j}(\tau,Z;x,\x)
-(d_{Z}b^{+})(\tau,Z;x,\x)\big[ \opbw({\bf C}(\tau,Z;x,\x))[Z]\big]
\end{aligned}
\end{equation}
where
\begin{equation}\label{zerohigh}
q^{(1)}_{j}\,,\; q^{(2)}_{j}\in \Sigma\Gamma^{m-\frac{j}{2}}_{1}[r,N]\,.
\end{equation}
Moreover the symbols $q_{j}^{(k)}$ depends only on the symbols
$C$, $a_{m}^{+}$, $b_{m}^{+}$,  $q_{p}^{(k)}$ with $p<j$. Note that in the notation above the term $\partial_tC$ in \eqref{newSimboDDhighoff2} is contained in $q^{(2)}_{2m}$. In the expansion for $d_2$ we used that $a_m^+$ is real valued.

Our aim is to prove the ansatz \eqref{forma-totalehighoff}, \eqref{ansatzhighoff} 
using 
the expansions \eqref{newSimboDDhighoff1}, \eqref{newSimboDDhighoff2}.

\noindent
{\bf Order $m$.}
By \eqref{newSimboDDhighoff1}, \eqref{newSimboDDhighoff2}, 
the expansions \eqref{dono11highoff1}, \eqref{dono11highoff2}
and  \eqref{ansatzhighoff},
we have that, at the highest order, the Heisenberg equation %\eqref{sistnewhighoff}
reads
%Using the expansion of the symbol $a$ in \eqref{ansatz}
%we have, at the highest order,
%the equation
\begin{equation}\label{ordMaxredu2highoff1}
\left\{\begin{aligned}
&\pa_{\tau}a^{+}_m(\tau,Z;x)=
-{\rm Re}\Big( 
\ov{C(\tau,Z;x)}{b_{m}^{+}(\tau;Z,x)}
\Big)
-
d_{Z}a^{+}_m(\tau,Z;x)[\opbw({\bf C}(\tau,Z;x))[Z]]\\
&a^{+}_m(0,Z;x)=a_{m}(Z;x)\,,
\end{aligned}\right.
\end{equation}
\begin{equation}\label{ordMaxredu2highoff2}
\left\{\begin{aligned}
&\pa_{\tau}b^{+}_m(\tau,Z;x)=
-2(1+a_{m}^{+}(\tau,Z;x))C(\tau,Z;x)
-
d_{Z}b^{+}_m(\tau,Z;x)[\opbw({\bf C}(\tau,Z;x))[Z]]\\
&b^{+}_m(0,Z;x)=b_{m}(Z;x)\,.
\end{aligned}\right.
\end{equation}
Notice that the 
problems \eqref{ordMaxredu2highoff1}, \eqref{ordMaxredu2highoff2} 
have the form 
\eqref{sistema01}, \eqref{sistema02}. Therefore, by Theorem \ref{constEgohighoff}
we have
\[
a^{+}_m(\tau,Z;x)\in \Sigma\mathcal{F}^{\mathbb{R}}_1[r,N]\,,\quad
 b^{+}_m(\tau,Z;x)\in \Sigma\mathcal{F}_1[r,N]\,,
\]
with estimates uniform in $\tau\in[0,1]$ and
%.
%Moreover (see \eqref{equaLowerhighoff}) we also have that
%one can choose $C(\tau;U,x)$ in \eqref{flussoord0}
%in such a way one has
\[
 b^{+}_m(1,Z;x)\equiv0\,.
\]

\smallskip
\noindent
{\bf Lower orders.} 
Recalling \eqref{dono11highoff1}, \eqref{dono11highoff2}, \eqref{zerohigh}
we have that the Heisenberg equation, % \eqref{sistnewhighoff}, 
at order
$m-j/2$, for $j\geq1$,
reads
\begin{equation}\label{sistema01bishigh}
\begin{aligned}
\pa_{\tau}a_{m-\frac{j}{2}}^{+}(\tau,Z;x,\x)&=-
2{\rm Re}\Big(C(\tau,Z;x)\ov{b_{m-\frac{j}{2}}^{+}(\tau,Z;x,\x)}\Big)
+q_j^{(1)}(\tau,Z;x,\x)\\
&\qquad\qquad\qquad-(d_{Z}a_{m-\frac{j}{2}}^{+})(\tau,Z;x,\x)\big[
\opbw\big({\bf C}(Z;x)\big)[Z]\big]\,,\\
a_{m-\frac{j}{2}}^{+}(0,Z;x,\x)&=a_{m-\frac{j}{2}}(Z;x,\x)\,,
\end{aligned}
\end{equation}

\begin{equation}\label{sistema02bishigh}
\begin{aligned}
\pa_{\tau}b_{m-\frac{j}{2}}^{+}(\tau,Z;x,\x)&
=-\big(a_{m-\frac{j}{2}}^{+}(\tau,Z;x,\x)
+\ov{a_{m-\frac{j}{2}}^{+}(\tau,Z;x,-\x)}\big)
C(\tau,Z;x)+q_j^{(2)}(\tau,Z;x,\x)
\\&\qquad\qquad\qquad-(d_{Z}b_{m-\frac{j}{2}}^{+})(\tau,Z;x)\big[
\opbw\big({\bf C}(Z;x)\big)[Z]\big]\,,\\
b_{m-\frac{j}{2}}^{+}(0,Z;x,\x)&=b_{m-\frac{j}{2}}(Z;x,\x)\,,
\end{aligned}
\end{equation}
where (recall \eqref{systExample}) we defined
\begin{equation}\label{datiinizhighoff}
\begin{aligned}
%&a_{m-\frac{1}{2}}(Z;x,\x):=a_{m-\frac{1}{2}}(Z;x,\x)\,,\quad
%b_{m-\frac{1}{2}}(Z;x,\x):=b_{m-\frac{1}{2}}(Z;x,\x)\,,\qquad j=1\,,\\
&a_{m-\frac{j}{2}}(Z;x,\x):=0\,,\quad
b_{m-\frac{j}{2}}(Z;x,\x):=0\,,\qquad j\geq2\,.
\end{aligned}
\end{equation}
Notice that
the problems  \eqref{sistema01bishigh}, \eqref{sistema02bishigh}
are of the form \eqref{sistema01bis}, \eqref{sistema02bis}. Since
the symbols $q_{j}^{(1)}, q_{j}^{(2)}$ depend only
on the symbols
$C$, $a_{m}^{+}$, $b_{m}^{+}$,  $q_{p}^{(k)}$ with $p<j$, we can iteratively 
solve the problems 
\eqref{sistema01bishigh}, \eqref{sistema02bishigh}
and verify that, at the step $j$,
the symbols $q_{j}^{(1)}, q_{j}^{(2)}$ satisfy the hypothesis \eqref{ipo523}.
Therefore, by Proposition \ref{constEgohighoffbis}, we have that
the solutions of \eqref{sistema01bishigh}, \eqref{sistema02bishigh} satisfy
\begin{equation*}
a^{+}_{m-\frac{j}{2}}(\tau,U;x,\x)\,,\; 
b^{+}_{m-\frac{j}{2}}(\tau,U;x,\x)\in \Sigma\Gamma^{m-\frac{j}{2}}_1[r,N]\,.
\end{equation*}
By iterating  the procedure above we construct  symbols $a^{+}, b^{+}$ 
as in \eqref{ansatzhighoff} 
such that the following holds.
Define \[
Q^{\tau}(Z)=\opbw(\ii E A^{+}(\tau,Z;x,\x))[Z]
\]
with $A^{+}$
of the form \eqref{nuovoSimboAhighoffbis} 
with $a^{+}, b^{+}$ as in \eqref{ansatzhighoff}.
Then 
the operator $Q^{\tau}(Z)$ solves the problem
\begin{equation*}
\left\{\begin{aligned}
&\pa_{\tau}Q^{\tau}(Z)=\big[\opbw({\bf C}(\tau,Z;x,\x))[Z], Q^{\tau}(Z)\big]+
\mathcal{G}_{\rho}(\tau;Z)\\
&Q^{0}(Z)=\ii E\opbw(A(Z;x,\x))[Z]\,,
\end{aligned}\right.
\end{equation*}
where
$\mathcal{G}_{\rho}(\tau;Z):=\opbw(G_{\rho}(\tau,Z;x,\x))[Z]$
for some matrix of symbols 
$G_{\rho}\in \Sigma\Gamma^{-\rho}_{1}[r,N]\otimes\mathcal{M}_2(\mathbb{C})$.
The fact that 
the difference $Q^{\tau}-P^{\tau}$ is a smoothing remainder 
in $\Sigma\mathcal{R}^{-\rho}_{1}[r,N]\otimes\mathcal{M}_2(\mathbb{C})$
can be proved following word by word the conclusion of the proof of Theorem
\ref{conjOrdMax}.

\subsubsection{Diagonalization of the matrix at lower orders}\label{blokkoLower}
 
Let us consider the operator $X(U)$ as in \eqref{operatoreEgor1}
with 
\begin{equation*}%\label{operatoreEgor1off}
\begin{aligned}
%&X(U):=\ii E\opbw(A(U;x,\x))[U]+R(U)[U]\,,\qquad E=\sm{1}{0}{0}{-1}\,, 
%\quad R\in \Sigma\mathcal{R}^{-\rho}_{1}[r,N]\otimes\mathcal{M}_2(\mathbb{C})\,,\\
&A(U;x,\x):=\left(\begin{matrix}
a(U;x,\x) & b(U;x,\x) \\
\ov{b(U;x,-\x)} &\ov{a(U;x,-\x)}
\end{matrix}\right)\,,
%\quad a(U;x,\x) \in \Sigma\Gamma^{m}_{1}[r,N]\,,
%\quad m>1\,.
\end{aligned}
\end{equation*}
where
\begin{equation}\label{Forma-di-Aoff}
\begin{aligned}
&a(U;x,\x)=(1+a_{m}(U;x))f_m(\x)+a_{m'}(U;x,\x)\,,\quad m>1\,,\;\;\; m'=m-\frac{1}{2}\,,
%\;\;{\rm or}\;\; m'=m-1\,,
\\ & b(U;x,\x)=b_{m''}(U;x,\x)
\in\Sigma\Gamma^{m''}_{1}[r,N]
\,,\;\; m'':=m'-\frac{p}{2}\,,\;\; {\rm for \; some }\;\;  p\geq0
\\&{a}_m(U,x)\in \Sigma\mathcal{F}_1^{\mathbb{R}}[r,N]\,,\quad 
a_{m'}(U;x,\x) \in\Sigma\Gamma^{m'}_{1}[r,N]\,,\\
&a_{m'}(U;x,\x)-\ov{a_{m'}(U;x,\x)}\in \Sigma\Gamma^{0}_1[r,N]\,.
\end{aligned}
\end{equation}
where $f_{m}\in \Gamma_0^{m}$ and in particular is a 
$m$-homogeneous ${\C^{\infty}(\mathbb{R}^+,\mathbb{R})}$
and even in $\x\in \mathbb{R}$. 
Assume also that \eqref{ipotesim} holds.
%Let us consider the system
%\begin{equation}\label{Nonlin1off}
%\left\{
%\begin{aligned}
%&\dot{U}=X(U)\\
%&U(0)=U_0\in H^{s}\times H^{s}
%\end{aligned}\right..
%\end{equation}
Let us now consider the  matrix of symbols
\begin{equation*}
\begin{aligned}
&{\bf C}(\tau,U;x,\x):=\left(\begin{matrix}
0 & C(\tau,U;x,\x)  \\   \ov{C(\tau,U;x,-\x)} & 0
\end{matrix}
\right)\,,\quad
C(\tau,U;x,\x)\in \Sigma\Gamma^{\delta}_{1}[r,N]\,,\;\;\delta:=m''-m\,,
\end{aligned}
\end{equation*}
where $C(\tau,U;x,\x)$ is the symbol given by Theorem \ref{constEgoloweroff}
with $m'\rightsquigarrow m''$,
and let 
 $\Psi^{\tau}$ be the flow of \eqref{sist2} with 
 \begin{equation}\label{Nonlin2off}
 G^{\tau}(U):=\opbw({\bf C}(\tau,U;x,\x))[U]\,.
%\left\{
%\begin{aligned}
%&\pa_\tau\Psi^{\tau}(U)=\opbw({\bf C}(\tau,\Psi^{\tau}(U);x,\x))[\Psi^{\tau}(U)]\\
%&\Psi^{0}(U)=U=\vect{u}{\bar{u}}\,.
%\end{aligned}\right.
\end{equation}
Notice that the flow of \eqref{sist2} with generator in \eqref{Nonlin2off}
 is well-posed by Theorem \ref{flussononlin}.
We define
\begin{equation}\label{variabNuovalowoff}
Z:=\vect{z}{\bar{z}}:={\bf \Psi}^{\tau}(U)_{|\tau=1}\,.
\end{equation}
The main result of this section is the following.

\begin{theorem}\label{conjOrdMaxloweroff}
For $r>0$ small enough 
%there exists a symbol 
%$C(\tau,U;x,\x)\in\Sigma\Gamma_1^{m''-m}[r,N]$, 
%$\tau\in[0,1]$,  such that, the following holds.
the conjugate of $X$ in \eqref{Nonlin1} with assumption \eqref{Forma-di-Aoff} 
has the form (see \eqref{variabNuovalowoff})
\begin{equation*}
\dot{Z}=\ii E\opbw(A^{+}(Z;x,\x))[Z]+R^{+}(Z)[Z]\,,\qquad 
R^{+}\in R\in \Sigma\mathcal{R}^{-\rho}_{1}[r,N]\otimes\mathcal{M}_2(\mathbb{C})\,,
\end{equation*}
and the matrix of symbols $A^{+}\in \Sigma\Gamma^{m}_{1}[r,N]\otimes\mathcal{M}_2(\mathbb{C})$
has the form
\begin{align}
&A^{+}(Z;x,\x):=\left(\begin{matrix}
a^{+}(Z;x,\x) & b^{+}(Z;x,\x) \\
\ov{b^{+}(Z;x,-\x)} &\ov{a^{+}(Z;x,-\x)}
\end{matrix}\right)\,,\label{nuovoSimboAlowoff}
\end{align}
where
\begin{equation}\label{nuovoSimboA2lowoff}
\begin{aligned}
&a^{+}(Z;x,\x)=(1+a_{m}^{+}(Z;x))f_m(\x)+a^{+}_{m'}(Z;x,\x)\,,\quad 
b^{+}(Z;x,\x)=b^{+}_{m''-\frac{1}{2}}(Z;x,\x)\,,
\\&{a}_m^{+}(Z,x)\in \Sigma\mathcal{F}_1^{\mathbb{R}}[r,N]\,,\quad 
a_{m'}^{+}(Z;x,\x) \in\Sigma\Gamma^{m'}_{1}[r,N]\,,
\quad  b^{+}(Z;x,\x) \in\Sigma\Gamma^{m''-\frac{1}{2}}_{1}[r,N]\,\\
&a_{m'}^+(U;x,\x)-\ov{a_{m'}^+(U;x,\x)}\in \Sigma\Gamma^{0}_1[r,N]\,.
\end{aligned}
\end{equation}
\end{theorem}
The rest of the section is devoted to the proof of the result above
and we will follow the strategy used in section \ref{sec:egoHigh}.

%In the following we shall write
%$\Psi:=\Psi^{\tau}_{|\tau=1}$, $\Phi:=\Phi^{\tau}_{|\tau=1}$.
The system %\eqref{Nonlin1off} 
in the new 
coordinates \eqref{variabNuovalowoff}
has the form \eqref{sistnew}-\eqref{Ego} with $G^{\tau}$ as in \eqref{Nonlin2off}.
%reads (recall also \eqref{pushpush})
%\begin{equation}\label{sistnewlowoff}
%\begin{aligned}
%&\dot{Z}=P^{1}(Z)=(P^{\tau}(z))_{\tau=1}\,,\\
%&\pa_{\tau}P^{\tau}(Z)=\big[\opbw( {\bf C}(\tau,Z;x,\x))[Z] , P^{\tau}(Z)\big]\,,
%\qquad P^{0}(Z)=X(Z)\,.
%\end{aligned}
%\end{equation}
We recall also that, 
by the remarks under Definition \ref{smoothoperatormaps} and by Theorem \ref{flussononlin},
we note that 
\begin{equation*}
P^{\tau}(Z)=\ii E\Omega Z+M_1(\tau;Z)[Z]\,,\;\;
M_1\in \Sigma\mathcal{M}_{1}[r,N]\otimes\mathcal{M}_2(\mathbb{C})\,,
%\qquad {\rm and\;\; hence}\qquad 
%P^{\tau}\in  \Sigma\mathcal{M}_{1}[r,N]\otimes\mathcal{M}_2(\mathbb{C})\,,
\end{equation*}
with estimate uniform in $\tau\in[0,1]$, where $\Omega$ is in \eqref{omegone}.
We  shall look for a solution of the Heisenberg equation %\eqref{sistnewredu} 
of the form
\begin{equation*}
P^{\tau}(Z)
=\ii E\Omega Z+M_1(\tau;Z)[Z]=\ii E\opbw(A^{+}(\tau,Z;x,\x))[Z]+R^{+}(\tau,Z)[Z]\,,
\end{equation*}
with $A^{+}$, $R^{+}$ %as in \eqref{Nonlin3}, \eqref{nuovoSimboA}.
\begin{equation}\label{dono7lowoff}
A^{+}(\tau,Z;x,\x):=\left(\begin{matrix}
a^{+}(\tau,Z;x,\x) & b^{+}(\tau,Z;x,\x) \vspace{0.2em}\\
\ov{b^{+}(\tau,Z;x,-\x)} &\ov{a^{+}(\tau,Z;x,-\x)}
\end{matrix}\right)\,, \qquad
R^{+}\in \Sigma\mathcal{R}^{-\rho}_1[r,N]\otimes\mathcal{M}_{2}(\mathbb{C})\,.
\end{equation}
In particular 
we make the ansatz
\begin{equation}\label{ansatzlowoff}
\begin{aligned}
&a^{+}(\tau,Z;x,\x)=(1+a_{m}^{+}(\tau,Z;x))f_{m}(\x)+
a_{m'}^{+}(\tau,Z;x,\x)+
\sum_{j=0}^{2(m+\rho)}a^{+}_{2m''-m-\frac{j}{2}}(\tau,Z;x,\x)\,,\\
&b^{+}(\tau,Z;x,\x)=b_{m''}^{+}(\tau,Z;x,\x)+
\sum^{2(m+\rho)}_{j=1}b_{m''-\frac{j}{2}}^{+}(\tau,Z;x,\x)\,,
\\
&a_{m}^{+}\in\Sigma\mathcal{F}^{\mathbb{R}}_{1}[r,N]\,, \quad
a^{+}_{m'}\in \Sigma\Gamma^{m'}_1[r,N]\,,
\quad
a^{+}_{2m''-m-\frac{j}{2}}\in \Sigma\Gamma^{2m''-m-\frac{j}{2}}_1[r,N]\,,
\quad b^{+}_{m''-\frac{j}{2}}\in \Sigma\Gamma^{m''-\frac{j}{2}}_1[r,N]\,,
\end{aligned}
\end{equation}
with $j\geq0$.
%To lighten the notation we shall omit to write the dependence on 
%$(\tau,Z;x,\x)$ of the symbols.
Expanding the non-linear commutator as in \eqref{commuExp100}-\eqref{commuExpF}
we get
\begin{equation*}
\big[\opbw( {\bf C}(\tau,Z;x,\x))[Z] , P^{\tau}(Z)\big]=\ii E
\left(\begin{matrix}
d_1(\tau,Z;x,\x) & d_2(\tau,Z;x,\x) \\ 
\ov{d_2(\tau,Z;x,-\x)} &
\ov{d_1(\tau,Z;x,-\x)}
\end{matrix}
\right)Z+\mathcal{G}_{\rho}(Z)[Z]
\end{equation*}
where 
$\mathcal{G}_{\rho}\in \Sigma\mathcal{R}^{-\rho}_1[r,N]\otimes\mathcal{M}_2(\mathbb{C})$ and
\begin{equation}\label{newSimboDDlowoff1}
\begin{aligned}
d_1&=-
C(\tau,Z;x,\x)\#_{\rho} \ov{b^{+}(\tau,Z;x,-\x)}-
 b^{+}(\tau,Z;x,\x)\#_{\rho}\ov{C(\tau,Z;x,-\x)}\\
&
%+\ii (\pa_{t}C)(\tau,Z;x,\x)
-(d_{Z}a^{+})(\tau,Z;x,\x)\big[ \opbw({\bf C}(\tau,Z;x,\x))[Z]\big]
\end{aligned}
\end{equation}
\begin{equation}\label{newSimboDDlowoff2}
\begin{aligned}
d_2&=-
C(\tau,Z;x,\x)\#_{\rho} \ov{a^{+}(\tau,Z;x,-\x)}-
 a^{+}(\tau,Z;x,\x)\#_{\rho}{C(\tau,Z;x,\x)}\\
&
+(\pa_{t}C)(\tau,Z;x,\x)
-(d_{Z}b^{+})(\tau,Z;x,\x)\big[ \opbw({\bf C}(\tau,Z;x,\x))[Z]\big]
\end{aligned}
\end{equation}
with
\[
(\pa_{t}C)(\tau,Z;x,\x):=(d_{Z}{C})(\tau,Z;x,\x)[P^{\tau}(Z)]\,.
\]
By recalling the ansatz \eqref{ansatzlowoff} and using the expansion \eqref{espansione2}
we get
\begin{equation}\label{dono11lowoff1}
\begin{aligned}
d_1&=-2\sum_{j=0}^{2(m+\rho)}{\rm Re}
\Big( \ov{C(\tau,Z;x,\x)} b^{+}_{m''-\frac{j}{2}}(\tau,Z;x,\x)\Big)
+\sum_{j=2}^{2(m+\rho)}q^{(1)}_{j}(\tau,Z;x,\x)
\\&-(d_{Z}a^{+})(\tau,Z;x,\x)\big[ \opbw({\bf C}(\tau,Z;x,\x))[Z]\big]
\end{aligned}
\end{equation}
\begin{equation}\label{dono11lowoff2}
\begin{aligned}
d_2&=-2(1+a_{m}^{+}(\tau,Z;x))f_{m}(\x)C(\tau,Z;x,\x)
\\&
- \Big(a_{m'}^{+}(\tau,Z;x,\x)+
\ov{a_{m'}^{+}(\tau,Z;x,-\x)}
\Big)C(\tau,Z;x,\x)
%\\&
%+\sum_{j=0}^{2(m+\rho)} \Big(a_{2m''-m-\frac{j}{2}}^{+}(\tau,Z;x,\x)+
%\ov{a_{2m''-m-\frac{j}{2}}^{+}(\tau,Z;x,-\x)}
%\Big)C(\tau,Z;x,\x)
%\\&
+\sum_{j=0}^{2(m+\rho)}q^{(2)}_{j}(\tau,Z;x,\x)
\\&-(d_{Z}b^{+})(\tau,Z;x,\x)\big[\opbw({\bf C}(\tau,Z;x,\x))[Z]\big]
\end{aligned}
\end{equation}
where
\[
q_{j}^{(1)}\in \Sigma\Gamma^{2m''-m-\frac{j}{2}}_1[r,N]\,, \;\; j\geq 2\,,
\]
depends only on the symbols $C$ and $b^{+}_{m''-\frac{k}{2}}$ with $0\leq k\leq j-2$\,,
while
\[
\begin{aligned}
%&q_{j}^{(1)}\in \Sigma\Gamma^{2m''-m-\frac{j}{2}}_1[r,N]\,, \;\; j\geq 0
%\quad 
&q_{j}^{(2)}\in \Sigma\Gamma^{m''-\frac{j}{2}}_1[r,N]\,,
\;\; j\geq2\,, \; {\rm and}\;\; \frac{j}{2}\neq m\,,\\
&q_{j}^{(2)}=r_{j}+(\pa_{t}C)\,, r_j\in \Sigma\Gamma^{m''-\frac{j}{2}}_1[r,N]\,,\;\;
{\rm for }\;\; \frac{j}{2}=m\,,
\end{aligned}
\]
depends only on the symbols 
$C$, $a_{m}^{+}$, $a_{m'}^{+}$
and $a^{+}_{2m''-m-k/2}$ with $k$ such that
$2m''-m-k/2<m''-j/2$.
%Moreover 
%the symbol $q_{j}^{(1)}$ depends only on 
%$C$,  $b_{m''-\frac{k}{2}}^{+}$,
%and  $q_{k}^{(1)}$  with $k<j$.
%The symbol $q_{j}^{(1)}$ depends only on 
%$C$,  $a_{m-\frac{k}{2}}^{+}$,
%and  $q_{k}^{(2)}$  with $k<j$.
%Our aim is to prove the ansatz \eqref{forma-totalelowoff} using 
%the expansions \eqref{newSimboDDlowoff1}, \eqref{newSimboDDlowoff2} .
%Moreover 
%\begin{equation}\label{georgia}
%q_{m''-1/2}^{(2)}\equiv0
%\end{equation}
%We shall solve  equation \eqref{sistnewlowoff} iteratively expanding 
%the symbol in \eqref{newSimboDD} in decreasing orders. 

\smallskip
\noindent
{\bf High orders of the symbol ${a}^{+}$.}
%Let $k$ such that (recall $m$, $m'$, $m''$ in \eqref{Forma-di-Aoff})
%\[
%m-\frac{k}{2}<2m''-m\,.
%\]
By \eqref{newSimboDDlowoff1}, \eqref{espansione2}, 
we have that, at the highest order, the Heisenberg equation %\eqref{sistnewlowoff}
at orders $m$ and $m'$ reads
\begin{equation}\label{ordMaxredu2lowofftris}
\left\{\begin{aligned}
&\pa_{\tau}a^{+}_{m}(\tau,Z;x)=-
d_{Z}a^{+}_{m}(\tau,Z;x)[\opbw({\bf C}(\tau,Z;x,\x))[Z]]\\
&a^{+}_m(0,Z;x)=a_{m}(Z;x)f_{m}(\x)\,.
\end{aligned}\right.
\end{equation}
\begin{equation}\label{ordMaxredu2lowoff}
\left\{\begin{aligned}
&\pa_{\tau}a^{+}_{m'}(\tau,Z;x,\x)=-
d_{Z}a^{+}_{m'}(\tau,Z;x,\x)[\opbw({\bf C}(\tau,Z;x,\x))[Z]]\\
&a^{+}_{m'}(0,Z;x,\x)=a_{m'}(Z;x,\x))\,.
\end{aligned}\right.
\end{equation}
Notice that the functions
\begin{equation*}
g_1(\tau)=a^{+}_m(\tau,Z(\tau);x)\,,\quad
g_2(\tau)=a^{+}_{m'}(\tau,Z(\tau);x,\x)
\end{equation*}
are constant along the solution of \eqref{Nonlin2off}.
By Theorem \ref{flussononlin} such flow  is well posed and invertible.
Hence the symbols 
\begin{equation*}
a^{+}_m(\tau,Z;x,\x)=a_{m}((\Psi^{\tau})^{-1}(Z);x)\,,
\qquad
a^{+}_{m'}(\tau,Z;x,\x)=a_{m'}((\Psi^{\tau})^{-1}(Z);x,\x)\,,
\end{equation*}
belong respectively  to $ \Sigma\Gamma^{m}_{1}[r,N]$,  
$ \Sigma\Gamma^{m'}_{1}[r,N]$ and
solves the problem \eqref{ordMaxredu2lowoff}.

\smallskip
\noindent
{\bf High orders of the symbol ${b}^{+}$.}
%Let $k$ such that (recall $m$, $m'$, $m''$ in \eqref{Forma-di-Aoff})
%\[
%m-\frac{k}{2}<2m''-m\,.
%\]
By \eqref{newSimboDDlowoff2}, \eqref{espansione2}, 
we have that, at the highest order, the Heisenberg equation %\eqref{sistnewlowoff}
at orders $m''$ for the symbol
$b^{+}$
reads 
\begin{align}
&\pa_{\tau}b^{+}_{m''}(\tau,Z;x,\x)=-2(1+a_{m}^{+}(\tau,Z;x))f_{m}(\x)C(\tau,Z;x,\x)
-
d_{Z}b^{+}_{m''}(\tau,Z;x,\x)[\opbw({\bf C}(\tau,Z;x,\x))[Z]]\nonumber\\
&b^{+}_{m''}(0,Z;x,\x)=b_{m''}(Z;x,\x)\,.
\label{ordMaxredu22lowoff22b}
\end{align}
Notice that the function $g(\tau)=b^{+}_{m''}(\tau,Z(\tau);x,\x)$
satisfies
\[
\begin{aligned}
\pa_{\tau}g(\tau)&=-2(1+a_{m}^{+}(\tau,Z(\tau);x))f_{m}(\x)C(\tau,Z(\tau);x,\x)
\,,\\
&g(0)=b_{m''}(Z(0);x,\x)=b_{m''}(U;x,\x)\,.
\end{aligned}
\]
Therefore
\[
\begin{aligned}
b_{m''}^{+}(\tau,Z,x,\x)=&b_{m''}((\Psi^{\tau})^{-1}(Z);x,\x)\\
&\quad -2\int_{0}^{\tau}
(1+a_{m}^{+}(\s,\Psi^{\s}(\Psi^{\tau})^{-1}(Z)))
f_{m}(\x)C(\s,\Psi^{\s}(\Psi^{\tau})^{-1}(Z);x,\x)d\s
\end{aligned}
\]
solves the problem \eqref{ordMaxredu22lowoff22b}. 
Moreover, by Theorem \ref{constEgoloweroff} (see \eqref{equaLoweroff}),
one has that
\begin{equation*}
b^{+}_{m''}(1,Z;x,\x)\stackrel{\eqref{equaLoweroff}}{=}0\,.
\end{equation*}

\smallskip
\noindent
{\bf Lower orders.} 
We start by studying the first lower order corrections to the symbols 
$a^{+}$, $b^{+}$ in \eqref{ansatzlowoff} which are respectively of order
$2m''-m$ and $m''-\frac{1}{2}$.
Recalling \eqref{dono11lowoff1}, \eqref{dono11lowoff2},
we have that the Heisenberg equation, %\eqref{sistnewlowoff}, 
at these orders reads
\begin{equation}\label{rambo1}
\begin{aligned}
\pa_{\tau}a^{+}_{2m''-m}(\tau,Z;x)&=-
d_{Z}a^{+}_{2m''-m}(\tau,Z;x)[\opbw({\bf C}(\tau,Z;x,\x))[Z]]\\
&\qquad-2{\rm Re}\Big( \ov{C(\tau,Z;x,\x)} b^{+}_{m''}(\tau,Z;x,\x)\Big)\,,\\
a^{+}_{2m''-m}(0,Z;x)&=0\,,
\end{aligned}
\end{equation}
\begin{equation}\label{rambo2}
\begin{aligned}
\pa_{\tau}b^{+}_{m''-\frac{1}{2}}(\tau,Z;x,\x)&=-
d_{Z}b^{+}_{m''-\frac{1}{2}}(\tau,Z;x,\x)[\opbw({\bf C}(\tau,Z;x,\x))[Z]]\\
&\qquad-\Big(a_{m'}^{+}(\tau,Z;x,\x)+
\ov{a_{m'}^{+}(\tau,Z;x,-\x)}
\Big)C(\tau,Z;x,\x)\,,\\
b^{+}_{m''-\frac{1}{2}}(0,Z;x,\x)&=0\,.
\end{aligned}
\end{equation}
The symbols $b^{+}_{m''}$ and $a^{+}_{m'}$ have been already computed at the previous 
steps. Therefore, one can check that
\[
\begin{aligned}
a^{+}_{2m''-m}(\tau,Z;x,\x)&:=-2\int_{0}^{\tau}
{\rm Re}
\Big( \ov{C(\tau,\Psi^{\s}(\Psi^{\tau})^{-1}(Z);x,\x)} 
b^{+}_{m''}(\tau,\Psi^{\s}(\Psi^{\tau})^{-1}(Z);x,\x)\Big)
d\s\,,\\
b^{+}_{m''-\frac{1}{2}}(\tau,Z;x,\x)&:=\int_{0}^{\tau}
Q(\s,\Psi^{\s}(\Psi^{\tau})^{-1}(Z);x,\x)
d\s\,,
%\in \Sigma\Gamma^{k}_1[r,N]\,,
\\
Q(\tau,Z;x,\x)&:=
-\Big(a_{m'}^{+}(\tau,Z;x,\x)+
\ov{a_{m'}^{+}(\tau,Z;x,-\x)}
\Big)C(\tau,Z;x,\x)\,,
\end{aligned}
\]
solve the problems 
\eqref{rambo1}, \eqref{rambo2}.
In particular, using the estimates \eqref{stimaflusso3}
on the flow $\Psi^{\tau}$ (and reasoning as in section \ref{sec:PreEgorov}),
one can prove that $a^{+}_{2m''-m}\in \Sigma\Gamma^{2m''-m}_1[r,N]$
and $b^{+}_{m''-\frac{1}{2}}\in \Sigma\Gamma^{m''-\frac{1}{2}}_1[r,N]$.
For the lower order terms we reason similarly and iteratively.
Recalling \eqref{dono11lowoff1}, \eqref{dono11lowoff2},
we have that the Heisenberg equation, %\eqref{sistnewlowoff}, 
at lower orders
reads
\begin{align}
\pa_{\tau}a^{+}_{k}(\tau,Z;x,\x)&=
-2{\rm Re}
\Big( \ov{C(\tau,Z;x,\x)} b^{+}_{m''-\frac{j}{2}}(\tau,Z;x,\x)\Big)\nonumber
\\&\qquad\qquad-d_{z}a^{+}_{k}(\tau,Z;x,\x)[\opbw({\bf C}(\tau,Z;x,\x))[Z]]
+ q_{j}^{(1)}(\tau,Z;x,\x)\label{ord1lowoff1}\\
a_k^{+}(0,z;x,\x)&=0\nonumber\,,
\end{align}
with $k:=2m''-m-j/2$.
It is important to note that
$2m''-m-j/2< m''-j/2$. Then the corresponding equation
for the symbol $b^{+}_{m''-\frac{j}{2}}$ does not depend on
the symbol $a^{+}_{2m''-m-j/2}$. Actually,  recalling \eqref{dono11lowoff2},
we have
\begin{align}
&\pa_{\tau}b^{+}_{k}(\tau,Z;x,\x)=
-d_{Z}b^{+}_{k}(\tau,Z;x,\x)[\opbw({\bf C}(\tau,Z;x,\x))[Z]]
+ \widetilde{q}_{j}^{(2)}(\tau,Z;x,\x)\label{ord1lowoff2}\\
&b_k^{+}(0,z;x,\x)=0\nonumber\,,
\end{align}
with $k:=m''-\frac{j}{2}$
for some $\widetilde{q}_{j}^{(2)}$ depending on ${q}_{j}^{(2)}$ and
$a^{+}_{p}$ with $p<2m''-m-j/2$.
One can easily check that
\[
\begin{aligned}
a^{+}_{k}(\tau,Z;x,\x)&=-2\int_{0}^{\tau}
{\rm Re}
\Big( \ov{C(\tau,\Psi^{\s}(\Psi^{\tau})^{-1}(Z);x,\x)} 
b^{+}_{m''-\frac{j}{2}}(\tau,\Psi^{\s}(\Psi^{\tau})^{-1}(Z);x,\x)\Big)
d\s
\\&+\int_{0}^{\tau}
q_{j}^{(1)}(\s,\Psi^{\s}(\Psi^{\tau})^{-1}(Z);x,\x)d\s\in \Sigma\Gamma^{j}_1[r,N]\,,
\;\;\; k=2m''-m-\frac{j}{2}\,,
\end{aligned}
\]
\[
b^{+}_{k}(\tau,Z;x,\x)=\int_{0}^{\tau}
\widetilde{q}_{j}^{(2)}(\s,\Psi^{\s}(\Psi^{\tau})^{-1}(Z);x,\x)
d\s\in \Sigma\Gamma^{k}_1[r,N]\,,\;\;\; k=m''-\frac{j}{2}
\]
solve the problems \eqref{ord1lowoff1}, \eqref{ord1lowoff2}.
%Moreover, recalling \eqref{georgia} and the definitions of 
%$m,m',m''$ in \eqref{Forma-di-Aoff},
%one can check that actually $b^{+}_{m''-1/2}\equiv0$.
To summarize by iterating  the procedure above 
we construct  symbols $a^{+}, b^{+}$ 
as in \eqref{ansatzlowoff} 
such that the following holds.
Define \[
Q^{\tau}(Z)=\opbw(\ii E A^{+}(\tau,Z;x,\x))[Z]
\]
with $A^{+}$
of the form \eqref{dono7lowoff} 
with $a^{+}, b^{+}$ as in \eqref{ansatzlowoff}.
We define the matrix in \eqref{nuovoSimboAlowoff}
as 
$A^{+}(Z;x,\x):=A^{+}(1,Z;x,\x)$.
By the construction above we can note that the conditions 
in \eqref{nuovoSimboA2lowoff} are satisfied.
Moreover
the operator $Q^{\tau}(Z)$ solves the problem
\begin{equation*}
\left\{\begin{aligned}
&\pa_{\tau}Q^{\tau}(Z)=\big[\opbw({\bf C}(\tau,Z;x,\x))[Z], Q^{\tau}(Z)\big]+
\mathcal{G}_{\rho}(\tau;Z)\\
&Q^{0}(Z)=\ii E\opbw(A(Z;x,\x))[Z]\,,
\end{aligned}\right.
\end{equation*}
where
$\mathcal{G}_{\rho}(\tau;Z):=\opbw(G_{\rho}(\tau,Z;x,\x))[Z]$
for some matrix of symbols 
$G_{\rho}\in \Sigma\Gamma^{-\rho}_{1}[r,N]\otimes\mathcal{M}_2(\mathbb{C})$.
The fact that 
the difference $Q^{\tau}-P^{\tau}$ is a smoothing remainder 
in $\Sigma\mathcal{R}^{-\rho}_{1}[r,N]\otimes\mathcal{M}_2(\mathbb{C})$
can be proved following word by word the conclusion of the proof of Theorem
\ref{conjOrdMax}.
This concludes the proof of Theorem \ref{conjOrdMaxloweroff}.

\section{Symplectic conjugations}\label{simplettomorfo}
In this section we explain how, with mild modifications, the change of coordinates of Section \ref{regularization} can be adapted to our Hamiltonian contexts. In other words in the case that the equation \eqref{Nonlin1inizio} has the Hamiltonian structure introduced in Section \ref{intro-ham}, one may be interested in preserving such a structure for the equation in the new coordinates \eqref{nuovotutto}.
%\noindent {\bf Hamiltonian structure.} On complex couple $U=\vect{u}{\bar{u}}$, $V=\vect{v}{\bar{v}}$ we define  the non-degenerate symplectic form
%\begin{equation}\label{symform}
%\lambda(U,V):=\int U\cdot \ii J^{-1}V dx=\int -\ii (u\bar{v}-\bar{u} v)dx\,,
%\qquad 
%J:=\left(\begin{matrix} 0 & 1 \\ -1 & 0\end{matrix}\right)\,, \qquad
%J^{T}=J^{-1}=\left(\begin{matrix} 0 & -1 \\ 1 & 0\end{matrix}\right)\,.
%\end{equation}
%Given a Hamiltonian function $H(U)$  
%we define its Hamiltonian vector field $X_{H}$ as
%the only vector field such that
%\begin{equation*}
%dH(U)[\hat{h}]=\lambda(\hat{h},X_{H}(U))\,, \quad \hat{h}=\vect{{h}}{\bar{{h}}}
%\,\;\;\; U=\vect{u}{\bar{u}}\,.
%\end{equation*}
%A simple computation shows that
%\begin{equation}\label{HamField2}
%X_{H}(U)=\left(
%\begin{matrix}-\ii \pa_{\bar{u}}H\\
%\ii \pa_{u}H
%\end{matrix}
%\right)=-\ii J\nabla H\,.
%\end{equation}
%The Poisson brackets between two Hamiltonian $H,G$ are defined as
%\begin{equation}\label{Poisson}
%\{G,H\}:=\lambda(X_{G},X_{H})
%\stackrel{\eqref{symform},\eqref{HamField2}}{=}
%\int -\ii J\nabla G\cdot\nabla H dx=
%\frac{1}{\ii} \int \pa_{u}H\pa_{\bar{u}}G-  \pa_{\bar{u}}H\pa_{{u}}G dx\,.
%\end{equation}
%We have that
%\begin{equation}\label{ham-field}
%[X_{G},X_{H}]=-X_{\{G,H\}}\,,
%\end{equation}
%where $[\cdot,\cdot]$ is defined in \eqref{nonlinCommu}.

\noindent {\bf Para-differential Hamiltonian vector field.} Let $u$ be a $C^{\infty}$ function on the $\TTT$ with values in $\mathbb{C}$. We define
the following frequency  localization:
\begin{equation}\label{innercutoff}
S_{\xi} u:= \sum_{|k|\leq\varepsilon|\xi|} u_k e^{\ii kx}\,,
\end{equation}
for some $0<\e<1$.
Let $A(U;x,\x)\in \Sigma\Gamma^{m}_{1}[r,N]\otimes\mathcal{M}_{2}(\mathbb{C})$ 
be a matrix of symbols real-to-real and \emph{self-adjoint}, 
i.e. satisfying \eqref{prodotto} and \eqref{quanti801}
and consider the Hamiltonian function
\begin{equation}\label{HamFUNC}
H(U):=\frac{1}{2}\int_{\mathbb{T}}\opbw(A(S_{\x}U;x,\x))U\cdot \ov{U}dx\,,\qquad U=\vect{u}{\bar{u}}\,,
\end{equation}
where $S_{\x}U:=({S_{\x}u}, {S_{\x}\ov{u}})^{T}$\,.
We are going to show that the Hamiltonian vector field of $H$ in \eqref{HamFUNC}, which by definition is $\ii J \nabla H(U)$, equals $\ii E\opbw(A(U;x,\x))U$ modulo smoothing remainders.
We  need some preliminary lemmas.

\begin{lemma}[{\bf Differential of symbols}]\label{stimedifferent}
Fix $p\in \NNN$,  $m\in \RRR$ 
and  
consider a symbol $a\in \widetilde{\Gamma}^{m}_{p}$.
For $u_{1},\ldots,u_{p+1},h$ in $ C^{\infty}(\TTT;\mathbb{C}^{2})$
define 
a linear operator on $h$ as 
$$
F[h]=F(u_2,\ldots,u_{p+1})[h] := \opbw(a(h,\ldots,u_p;\x))[u_{p+1}].
$$
Then the operator $F^{*}$, the adjoint  with respect to the $L^2$ scalar product of $F$,  belongs to the class
$\widetilde{\RR}^{-\rho}_{p}$ for any $\rho>0$.
\end{lemma}

\begin{proof}
First we note that
\begin{equation*}%\label{bambino}
F[h]:=\sum_{n_0\in\ZZZ}\sum_{\s_1n_{1}+\s_2n_{2}+\ldots+\s_p n_{p}=\s_0n_0-\s_{p+1}n_{p+1}}
C_{n_1,\ldots,n_{p}}^{n_{p+1}}(\Pi_{n_1}h)
(\Pi_{n_2}u_2)\ldots (\Pi_{n_{p}}u_{p})(\Pi_{n_{p+1}}u_{p+1})
\end{equation*}
for some coefficients $C_{n_1,\ldots,n_{p}}^{n_{p+1}}\in \CCC$, $\s_i\in \{\pm\}$, 
which are different from zero only if
$$
\sum_{j=1}^{p}|n_{j}|\leq \e \langle n_{p+1}\rangle,\quad 0<\e\ll1.
$$
Then one has that
\begin{equation*}
\begin{aligned}
\Big(F[h],\ov{u_{p+2}}\Big)_{L^{2}}&=
\int_{\TTT}F[h]\cdot u_{p+2}=
\int_{\TTT}h\cdot L(u_{2},\ldots,u_{p+1})[u_{p+2}],
\end{aligned}
\end{equation*}
where the operator $L$ is defined as
\begin{equation*}%\label{trucco}
\begin{aligned}
L(u_{2}&,\ldots,u_{p+1})[u_{p+2}]:=\\&
\sum_{j\in\ZZZ}\sum_{\s_{2}n_{2}+\s_{3}n_{3}+\ldots+\s_{{p+1}}n_{p+1}-n_{p+2}=n_1}
C_{-j,\ldots,n_{p}}^{n_{p+1}}
(\Pi_{n_2}u_2)\ldots (\Pi_{n_{p}}u_{p})(\Pi_{n_{p+1}}u_{p+1})(\Pi_{-n_{p+2}}u_{p+2})\,.
\end{aligned}
\end{equation*}
In order to obtain the thesis it is sufficient to prove that $L$
belongs to the class of remainders, i.e. one has to show that \eqref{omoresti1} 
holds.
First we note that, using \eqref{pomosimbo1}, there exists $\mu>0$
such that 
\begin{equation}\label{fata}
\|\Pi_{j}L(\Pi_{n_{2}}u_{2},\ldots, \Pi_{n_{p+1}}u_{p+1})\Pi_{n_{p+2}}u_{p+2}\|_{L^{2}}\leq
C \max(\langle n_{2}\rangle,\ldots,\langle n_{p+1}\rangle,\langle n_{p+2}\rangle)^{\mu}\prod_{j=2}^{p+2}\|\Pi_{n_{j}}u_{j}\|_{L^{2}}.
\end{equation}
We remind that $L$ is different from zero only if the following two conditions
hold:
$$
\sum_{i=2}^{p}|n_{i}|+|j|\leq \e \langle n_{p+1}\rangle, \qquad 
\s_{2}n_{2}+\s_{3}n_{3}+\ldots+\s_{p+1}{n_{p+1}}-n_{p+2}=-\s_{1} n_1.
$$
Hence
there exist constants $c,C>0$ such that 
$$
c\max(\langle n_{2}\rangle,\ldots,\langle n_{p+1}\rangle,\langle n_{p+2}\rangle)\leq
{\rm max_{2}}(\langle n_{2}\rangle,\ldots,\langle n_{p+1}\rangle,\langle n_{p+2}\rangle)\leq C \max(\langle n_{2}\rangle,\ldots,\langle n_{p+1}\rangle,\langle n_{p+2}\rangle),
$$
which, together with $\eqref{fata}$, proves that $L$ satisfies the estimate \eqref{omoresti1} 
for any $\rho>0$.
\end{proof}

In the following lemma we prove that  a para-differential operator $\opbw(a(u;x,\xi))$ equals to the truncated one $\opbw(a(S_{\xi}u;x,\xi))$ modulo smoothing remainders.

\begin{lemma}\label{stimedifferent2}
Fix $p\in \NNN$, $m\in \RRR$, $r>0$, any $\rho>0$
and  
consider
 $a(u;\xi)\in \Gamma_{p}^{m}[r]$ and set 
 \[
 R(u):=\opbw(a(u;x,\xi)-a(S_{\xi}u;x,\xi))\,.
 \] 
 Then $R(u)$ belongs to the class $\mathcal{R}^{-\rho}_{p}[r]$.
\end{lemma}
\begin{proof}
Let $v\in H^{s}$,
expanding we obtain
\begin{equation*}
R(u)v=\sum_{k\in\ZZZ}e^{\ii kx}\sum_{j:|k-j|\leq\varepsilon|j|}\left(a(u;j)-a(S_j u;j)\right)_{k-j}v_j.
\end{equation*}
Setting $R_j u= u-S_j u$ there exists $\s\in(0,1)$ such that
$$\begin{aligned}
R(u)v &=\sum_{k\in\ZZZ}e^{\ii kx}\sum_{j:|k-j|\leq\varepsilon |j|}
\left(d_{u} a(u-\s R_ju)R_ju\right)_{k-j}v_j\\
&=\sum_{k\in\ZZZ}e^{\ii kx}\sum_{j:|k-j|\leq\varepsilon |j|}\sum_{j':|j'|>\varepsilon |j|}
\left(d_{u}a(u-\s R_ju)\right)_{k-j-j'}(R_{j}u)_{j'}v_j.
\end{aligned}$$
Therefore 
\begin{equation*}%\label{steve}
\begin{aligned}
\|R(u)v\|_{H^{s+\rho}}^2 &\leq \sum_{k\in\ZZZ}\left(\sum_{j:|k-j|\leq\varepsilon |j|}
\sum_{j':|j'|>\varepsilon |j|}\left| \left(d_{u}a(u-\s R_ju)\right)_{k-j-j'}
(R_{j}u)_{j'}v_j\right|\langle k\rangle^{s+\rho}\right)^2\\
&\leq C \sum_{k\in\ZZZ}\left(\sum_{j:|k-j|\leq\varepsilon |j|}\sum_{j':|j'|>\varepsilon |j|}
\left| \left(d_{u}a(u-\s R_ju)\right)_{k-j-j'}\right| \left|(R_{j}u)_{j'}\right|
\langle j' \rangle^{\rho}\left| v_j\right|\langle j\rangle^{s}\right)^2\\
&\leq C \|d_{u}a(u-\s R_ju)\|_{H^{s_0}}^2\|{u}\|_{H^{s_0+\rho}}^2\|{v}\|_{H^{s}}^2,
\end{aligned}
\end{equation*}
for any $s_0>1/2$, $C>0$ depending on $s$.
%The bounds for the derivative in time $t$ of $R(u)[v]$ follows reasoning 
%as in \eqref{steve}
%and by the chain rule.
This implies the bound \eqref{porto20}.
The bound on the differential in $u$ follows by \eqref{maremma2} 
on the symbol $a$ and reasoning similarly. This concludes the proof.
\end{proof}
We leave to the reader the proof of the following lemma which is
similar to the one of Lemma \ref{stimedifferent2}.
\begin{lemma}\label{stimedifferent2Multi}
Fix $p\in \NNN$, $m\in \RRR$, $r>0$, any $\rho>0$
and  
consider
 $a(u;\xi)\in \widetilde{\Gamma}_{p}^{m}$ and set 
 \[
 R(u):=\opbw(a(u;x,\xi)-a(S_{\xi}u;x,\xi))\,.
 \] 
 Then $R(u)$ belongs to the class $\widetilde{\mathcal{R}}^{-\rho}_{p}$.
\end{lemma}

In the following proposition we give an explicit structure of the Hamiltonian vector field 
of the Hamiltonian function in \eqref{HamFUNC} as a sum of a para-differential operator plus a smoothing term.
\begin{proposition}\label{stimedifferent3}
Fix $p\in \NNN$,  $m\in \RRR$, $r>0$, any $\rho>0$
and  
let $A(U;x,\x)\in\Sigma\Gamma_{p}^{m}[r,N]\otimes\mathcal{M}_2(\CCC)$ of the form \eqref{prodotto}, 
consider the Hamiltonian function $H(U)$ in \eqref{HamFUNC}.
Then there exist  $R\in\Sigma\mathcal{R}^{-\rho}_{p}[r,N]\otimes\mathcal{M}_2(\CCC)$ such that 
\begin{equation*}
X_{H}(U)=\ii J\nabla H(U)
=\ii E \opbw(A(U;x,\x))U+R(U)U\,.
\end{equation*}
\end{proposition}

\begin{proof}
By using \eqref{HamFUNC} and \eqref{prodotto} we have the following
\begin{equation*}
\begin{aligned}
H(U)=H(u,\bar{u})=&\frac12\int_{\TTT}\Big(\opbw(a(S_{\xi}(u,\bar{u});x,\xi))u\cdot\bar{u}+\opbw(b(S_{\xi}(u,\bar{u});x,\xi))\bar{u}\cdot\bar{u}\\
+&\opbw(\bar{b}(S_{\xi}(u,\bar{u});x,-\xi))u\cdot\bar{u}+\opbw(\bar{a}(S_{\xi}(u,\bar{u});x,-\xi))\bar{u}\cdot u\Big)dx.
\end{aligned}
\end{equation*}
We prove the theorem for the first addendum, one can deal with the others in the same way and deduce the thesis. Since $a(U;x,\x)\in\Sigma\Gamma_{p}^{m}[r,N]$ we may decompose it as a sum of multilinear terms in $\widetilde{\Gamma}_{j}^{m}$ and a non homogeneous one in $\Gamma_N^m[r]$. We have 
\begin{equation*}
\partial_{\bar{u}}\Big(\int_{\TTT}\opbw(a(S_{\xi}(u,\bar{u});x,\xi))u\cdot\bar{u}dx\Big)=\opbw(a(S_{\xi}(u,\bar{u});x,\xi))u+F^*[\bar{u}],
\end{equation*}
where $F^*$ is the adjoint  of the linear operator $F:h\mapsto \opbw(d_{\bar{u}}a(S_{\xi}(u,\bar{u}))h)u$. Thanks to Lemma \ref{stimedifferent2} it is enough to prove that the addendum $F^*[\bar{u}]$ is a smoothing remainder in the class $\Sigma\mathcal{R}^{-\rho}_p[r,N]$. The multilinear counterpart of this claim is true thanks to Lemma \ref{stimedifferent}. In the following we prove that the non homogeneous part of $F^*[\bar{u}]$ is a smoothing remainder. With abuse of notation let $a$ be in $\Gamma_{N}^{m}[r]$, we have
\begin{equation*}
\begin{aligned}
\int_{\TTT}\opbw \big([\partial_u a(S_ju)]S_jh\big)u\cdot\bar{u}dx &= \sum_{k\in\ZZZ}\left(\sum_{j\in\ZZZ}
(\partial_{u}a(S_j u)S_j h)_{k-j}u_j\right)(\bar{u}_{-k})\\
	&=\sum_{k\in\ZZZ}\sum_{|k-j|\leq\varepsilon|j|}
	\sum_{|j'|\leq\varepsilon|j|}(\partial_{u}a(S_j u))_{k-j-j'}( h)_{j'}u_j(\bar{u})_{-k}\\
											    &=\sum_{j'\in\ZZZ}h_{j'}\sum_{\varepsilon |j|>|j'|}\sum_{|k-j|\leq\varepsilon|j|} (\partial_{u}a(S_ju))_{k-j-j'}u_j(\bar{u})_{-k}\\
											    &:=\sum_{j'}h_{j'} (f)_{-j'}.
\end{aligned}\end{equation*}
Therefore we can continue as follows
\begin{equation*}\begin{aligned}
\norm{f}_{H^{s+\rho}}^2&\leq \sum_{m\in\ZZZ}\left(\sum_{|k-j|\leq\varepsilon|j|}\sum_{|m|\leq\varepsilon|j|}|(d_{u}a(S_j u))_{k-j+m}||u_j||(\bar{u})_{-k}|\langle m\rangle^{s+\rho}\right)^2\\
				   &	\leq \sum_{m\in\ZZZ}\left(\sum_{|k-j|\leq\varepsilon|j|}\sum_{j\in\ZZZ}|(d_{u}a(S_j u))_{k-j+m}||u_j|\langle j\rangle^s|(\bar{u})_{-k}|\langle m\rangle^{\rho}\right)^2\\
				   &	\leq \sum_{m\in\ZZZ}\left(\sum_{k\in\ZZZ}\sum_{j\in\ZZZ}|(d_{u}a(S_j u))_{k-j+m}| |u_j|\langle j\rangle^s |(\bar{u}_{-k})|\langle k\rangle^{\rho}\right)^2\\
				   &\leq\norm{d_{u}a(S_{\xi}u)}_{H^{s_0}}\norm{u}_{H^s}\norm{\bar{u}}_{H^{s_0+\rho}}.
\end{aligned}\end{equation*}
As done in Lemma \ref{stimedifferent2} one concludes that $f$ is a smoothing remainder
hence the result follows.
\end{proof}

\noindent{\bf Symplectic Flows.}
In this section we study the symplectic corrections of the flows considered in Section
\ref{sec:nonlinflows}.
Consider a symbol $f(\tau,U;x,\x)$ as in \eqref{sim1}, \eqref{sim2}, \eqref{sim3},
\eqref{sim4}
a symbol $g(\tau,U;x,\x)\in \Sigma\Gamma_1^{m}[r,N]$, $m\leq 0$, and assume that
\begin{equation*}
f(\tau,U;x,\x)=\ov{f(\tau,U;x,\x)}\,, \qquad
g(\tau,U;x,-\x)=g(\tau,U;x,\x)\,.
%\quad g(\tau,U;x,\x)\in \Sigma\Gamma_1^{m}[r,N]\,,\;\;\;
%m\leq 0\,.
\end{equation*}
We also assume that the symbols $f(\tau,U;x,\x)$, $g(\tau,U;x,\x)$ satisfy 
the estimates 
\eqref{pomosimbo1}-\eqref{maremma2} uniformly in $\tau\in[0,1]$.
Let us define the operator 
\begin{equation}\label{geneHam}
\mathcal{G}(U)[\cdot]:= \opbw\left(
\begin{matrix}
f(\tau,U;x,\x) & g(\tau,U;x,\x)\vspace{0.2em}\\
\ov{g(\tau,U;x,-\x)} & \ov{f(\tau,U;x,-\x)}
\end{matrix}
\right)[\cdot]\,,
\end{equation}
and the Hamiltonian function $G(U) : H^{s}(\mathbb{T};\mathbb{C})\cap\mathcal{U}\to \mathbb{R}$ defined as
\begin{equation}\label{geneHamiltonian}
G(U):=\int_{\mathbb{T}} \mathcal{G}(S_{\x}U)[U]\cdot\bar{U} dx\,,
\end{equation}
where $S_{\x}$ is defined in \eqref{innercutoff}.
We have the following result.

\begin{proposition}{\bf (Symplectic flow).}\label{sympflows}
Let us define $\Psi_{G}^{\tau}$, $\tau\in[0,1]$, 
the flow generated by the Hamiltonian
$G$ in \eqref{geneHamiltonian}. We have that
$U(\tau):=\Psi_{G}^{\tau}(U_0)$\,, 
$U_0\in H^{s}(\mathbb{T};\mathbb{C}^{2})\cap\mathcal{U}$, solves the problem
\begin{equation}\label{eq:HamPsi}
\pa_{\tau}U=\ii J\nabla G(U)=\ii E\mathcal{G}(U)[U]+\mathcal{R}(U)U\,,\qquad U(0)=U_0\,,
\end{equation}
where $\mathcal{R}\in 
\Sigma\mathcal{R}^{-\rho}_1[r,N]\otimes\mathcal{M}_2(\mathbb{C})$.
The map $\Psi^{\tau}_{G}$ satisfies estimates 
like \eqref{stimaflusso1}-\eqref{stimaflusso3} and it is \emph{symplectic}, 
i.e. 
\begin{equation*}%\label{psisimple}
\lambda\big(\Psi^{\tau}_{G}(U), \Psi^{\tau}_{G}(V)\big)
=\lambda\big(U,V\big)\,,\qquad \forall U\,,\;V\,\in\mathcal{U}\,,
\end{equation*}
where $\lambda(\cdot,\cdot)$ is the symplectic form in \eqref{symform}.
\end{proposition}

\begin{proof}
Notice that the operator $\mathcal{G}(U)$ in \eqref{geneHam}
is self-adjoint according to Definition \ref{selfi}.
Hence formula \eqref{eq:HamPsi} follows by Proposition \ref{stimedifferent3}.
The flow $\Psi^{\tau}_{G}$
is well-posed by Theorem \ref{flussononlin}. Moreover it is symplectic
since it is the flow on an Hamiltonian vector field.
\end{proof}

%\begin{theorem}{\bf (Symplectic structure).}\label{thm:main2}
%Assume that the vector field $X(U)$ in \eqref{Nonlin1inizio}
%is \emph{Hamiltonian},
%i.e.
%\[
%X(U):=X_{H}(U):=-\ii J\nabla H(U)\,,
%\]
%for some Hamiltonian $H(U) : 
%B_R(H^{s}(\mathbb{T};\mathbb{C}))\to \mathbb{R}$.
%Then the result of Theorem \ref{thm:main}
%holds with  a  symplectic map $\Psi$
%and the vector field $\mathcal{Y}$ in \eqref{nuovotutto}
%is Hamiltonian with respect to the symplectic form $\lambda$ in
%\eqref{symform}.
%Moreover the operator $\mathcal{L}(Z)$ in \eqref{opfinale}
%is self-adjoint.
%\end{theorem}

\begin{proof}[{\bf Proof of Theorem \ref{thm:main2}} ]
In order to prove Theorem \ref{thm:main} we apply
iteratively Theorems \ref{conjOrdMaxhighoff}, 
\ref{conjOrdMaxloweroff},
\ref{conjOrdMax} and 
\ref{conjOrdMaxredu}.
The maps provided by  such results
are constructed as flows of para-differential vector fields.
In order to obtain a symplectic correction to
such maps one can reason as follows.
Instead of considering 
the flows generated by
 \eqref{flussoord0}, \eqref{Nonlin2off}, \eqref{Nonlin2} 
and \eqref{Nonlin2redu}
one can consider their symplectic corrections
generated by Hamiltonian functions of the form
\eqref{geneHamiltonian}. By Proposition \ref{sympflows}
the maps constructed in this way are symplectic. Moreover the 
generators of such flows  have the form 
\eqref{eq:HamPsi}. In other words
the generator of the symplectic correction
to the maps  
\eqref{flussoord0}, \eqref{Nonlin2off}, \eqref{Nonlin2} 
and \eqref{Nonlin2redu} has the same form up to a smoothing
remainder. Therefore the results of Theorems 
 \ref{conjOrdMaxhighoff}, 
\ref{conjOrdMaxloweroff},
\ref{conjOrdMax} and 
\ref{conjOrdMaxredu} hold true 
with some very mild modifications of their proofs.

In order to prove the self-adjointness of the operator $\mathcal{L}(Z)$
is it sufficient to show that the symbol $\mathfrak{m}(Z;\x)$ is real valued.
We need to enter the proofs of the results in section \ref{secconjconj}.

First of all, since $X(U)$ in \eqref{Nonlin1inizio} is Hamiltonian, then
the matrix $A(U;x,\x)$ in \eqref{operatoreEgor1inizio} is self-adjoint, i.e. it satisfies 
\eqref{quanti801}.
In giving the proof of Theorem \ref{conjOrdMaxhighoff} with this additional assumption
on $A(U;x,\x)$ it is evident that the matrix $A^{+}(Z;x,\x)$ in \eqref{nuovoSimboAhighoff}
still satisfies \eqref{quanti801}.
This can be deduced from equations \eqref{ansatzhighoff}-\eqref{newSimboDDhighoff2}
and \eqref{ordMaxredu2highoff1}, \eqref{ordMaxredu2highoff2}.
This gives the self-adjointness at the highest order $m$.
One can check this property at lower orders by using equations \eqref{sistema01bishigh},
\eqref{sistema02bishigh}.
At this point rename $A^{+}=A$ and apply Theorem \ref{conjOrdMaxloweroff}.
In the same way one deduces the self-adjointness of the matrix $A^{+}$ in \eqref{nuovoSimboAlowoff} by equations
\eqref{ansatzlowoff}-\eqref{newSimboDDlowoff2}
and \eqref{ordMaxredu2lowofftris}, \eqref{ordMaxredu2lowoff},
at the highest order. Similarly for the lower orders.

One has to repeat this check in the proofs of Theorems \ref{conjOrdMax}, 
\ref{conjOrdMaxredu}.
\end{proof}

\bigskip 
\section{Poincar\'e-Birkhoff normal forms}\label{sec:BNF} 
The proof of the Theorem \ref{thm:mainBNF} 
 is based on  
 the iterative procedure which is performed in the following subsections.

\subsection{Abstract conjugation results}
In this section we provide a conjugation result 
of a vector field of the form \eqref{NonlinBNF1}
under the flow generated either by Fourier multipliers or by smoothing reminders.
We shall  we consider the  system 

\begin{equation}\label{Pollini}
\dot{U}=\mathcal{X}(U)=\ii E\Omega U+\ii E\opbw\big(\mathfrak{N}(U;\x)\big)U
+\mathcal{R}(U)[U]\,,\qquad U(0)=U_0\,,
\end{equation}
where $\Omega$ is in \eqref{omegone}, 
$\mathcal{R}\in\Sigma\mathcal{R}^{-\rho}_1[r,N]
\otimes\mathcal{M}_2(\mathbb{C})$, 
$\mathfrak{N}(Z;\x)\in \Sigma\Gamma^{m}_{1}[r,N]\otimes\mathcal{M}_2(\mathbb{C})$
is independent of $x\in \mathbb{T}$ and has the form 
\begin{equation*}%\label{Pollini2}
\mathfrak{N}(U;\x):=\sm{\mathfrak{n}(U;\x)}{0}{0}{{\mathfrak{n}(U;-\x)}}\,,
%\left(
%\begin{matrix}
%\mathfrak{n}(U;\x) & 0\\
%0 & \ov{\mathfrak{n}(U;-\x) }
%\end{matrix}
%\right)\,,
\end{equation*}
for some $\mathfrak{n}\in \Sigma\Gamma^{m}_{1}[r,N]$
real valued.

\subsubsection{Flows of Fourier multipliers}
Fix $p\in \mathbb{N}$ and 
consider  
the matrix of symbols 
\begin{equation}\label{Pollini333}
\begin{aligned}
&{ B}_{p}(U;\x):=\left(\begin{matrix}
 b_p(U;\x) & 0 \\ 0 &  {b_p(U;-\x)} 
\end{matrix}
\right)\,,
\qquad  b_p(U;\x)\in  \widetilde{\Gamma}^{m}_p\,,
\end{aligned}
\end{equation}
where $b_p$ is real, independent of $x\in \mathbb{T}$
and admits the expansion \eqref{espandoFousimbo}.
Consider the Hamiltonian function (recall \eqref{innercutoff})
\begin{equation}\label{Pollini3}
\mathcal{G}(U):=\frac{1}{2}\int_{\mathbb{T}}\opbw(B_p(S_{\x}U;\x))U\cdot\ov{U}dx\,.
\end{equation}
We are using the abuse of notation $b_p(U;\xi):=b_p(U,\ldots,U;\xi)$.
Let $\tilde{\mathcal{A}}_p^{\tau}$ be the solution of
\begin{equation}\label{Pollini4}
\left\{
\begin{aligned}
&\pa_{\tau}\tilde{\mathcal{A}}_p^{\tau}(U)
=X_{\mathcal{G}}(\tilde{\mathcal{A}}_p^{\tau}(U))\\
&\tilde{\mathcal{A}}_p^{0}(U)=U\,,
\end{aligned}\right.
\end{equation}
where $X_{\mathcal{G}}$ is the Hamiltonian vector field of \eqref{Pollini3}
and has the form (see Proposition \ref{stimedifferent3} and Lemma \ref{stimedifferent2Multi})
\begin{equation}\label{deadwood5}
X_{\mathcal{G}}(U)=\ii E \opbw(B_p(U;\x))U+ \mathcal{B}_p(U)U\,,\qquad 
\mathcal{B}_p\in \widetilde{\mathcal{R}}^{-\rho}_p\,.
\end{equation}

We prove the following.

\begin{proposition}\label{prop:Pollini}
For $r>0$ small enough the following holds.
Setting  (recall \eqref{nuovotutto}, \eqref{Pollini})
\begin{equation}\label{Pollini6}
W:=\tilde{\mathcal{A}}_{p}(U):=\tilde{\mathcal{A}}_{p}^{1}(U)
\,,
\qquad
\mathcal{X}^{+}(W):=P^{\tau}(W)_{|\tau=1}
:=d\tilde{\mathcal{A}}_{p}^{\tau}\big(\tilde{\mathcal{A}}_{p}^{-\tau}(W)\big)
\big[ \mathcal{X}(\tilde{\mathcal{A}}_{p}^{-\tau}(W))\big]_{|\tau=1}\,,
\end{equation}
we have that
\begin{equation}\label{Pollini7}
\left\{
\begin{aligned}
&\dot{W}=\mathcal{X}^{+}(W):=\ii E \Omega W+
\ii E\opbw\big( \mathfrak{N}^{+}(W;\x)\big)[W]
+\mathcal{R}^{+}(W)[W]\\
&W(0)=\mathcal{A}(Z_0)
\end{aligned}\right.
\end{equation}
where $\mathcal{R}^{+}\in\Sigma\mathcal{R}^{-\rho}_1[r,N]
\otimes\mathcal{M}_2(\mathbb{C})$ and
${\mathfrak{N}}^{+}\in 
\Sigma\Gamma^{m}_1[r,N]
\otimes\mathcal{M}_2(\mathbb{C})$
is independent of $x\in \mathbb{T}$, 
real valued  and has the form
\begin{equation}\label{Pollini8}
\begin{aligned}
&\mathfrak{N}^{+}(W;\x):=\left(
\begin{matrix}
\mathfrak{n}^{+}(W;\x) & 0\\
0 & {\mathfrak{n}^{+}(W;-\x) }
\end{matrix}
\right)\,,\\
&
\mathfrak{n}^{+}(W;\x)=\mathfrak{n}(\tilde{\mathcal{A}}_{p}^{-1}(W);\x)+
\int_{0}^{1} (d_{Z}b_p)\big(\tilde{\mathcal{A}}_{p}^{\s}
\tilde{\mathcal{A}}_{p}^{-1}(W);\x\big)
[P^{\s}\big(
\tilde{\mathcal{A}}_{p}^{\s}\tilde{\mathcal{A}}_{p}^{-1}(W)
\big) ]d\s\,.
\end{aligned}
\end{equation}
Moreover, for any $s\geq s_0$, the maps $\tilde{\mathcal{A}}^{\pm1}_{p}$ are symplectic 
and satisfy
\begin{equation}\label{Pollini9}
\|\tilde{\mathcal{A}}_{p}^{\pm1}(U)\|_{H^{s}}\leq \|U\|_{H^{s}}(1+C\|U\|^{p}_{H^{s_0}})\,,
\end{equation}
for some constant $C>0$ depending  on $s$.
\end{proposition}

\begin{proof}
The estimates \eqref{Pollini9} on the flow of \eqref{Pollini4}
follow by Theorem \ref{flussononlin}.
The vector field $\mathcal{X}$ in  \eqref{Pollini} in the new 
coordinates $W$ has the form \eqref{Pollini6}.
In particular 
 $P^{\tau}$ satisfies,  for $\tau\in[0,1]$, (recall \eqref{Pollini}, \eqref{Pollini4})
\begin{equation}\label{deadwood10}
\left\{
\begin{aligned}
&\pa_{\tau}P^{\tau}(W)=\big[X_{\mathcal{G}}(W) , P^{\tau}(W)\big]\\
&P^{0}(W)=\mathcal{X}(W)\,.
%\opbw(a(z;x,\x))[z]\,.
\end{aligned}\right.
\end{equation}
Moreover, by the remarks under Definition \ref{smoothoperatormaps} and by Theorem \ref{flussononlin},
we note that 
\begin{equation}\label{deadwood11}
P^{\tau}(W)=\ii E \Omega W+M_1(\tau;W)[W]\,,\;\;\;
M_1\in \Sigma\mathcal{M}_{1}[r,N]\otimes\mathcal{M}_2(\mathbb{C})\,,
\end{equation}
with estimate uniform in $\tau\in[0,1]$.
Actually we shall prove that
\begin{equation*}%\label{deadwood12}
P^{\tau}(W)
=\ii E\Omega W+\ii E\opbw(\mathfrak{N}^{+}(\tau,W;\x))[W]+\mathcal{R}^{+}(\tau,W)[W]\,,
\end{equation*}
with $\mathfrak{N}^{+}$, $\mathcal{R}^{+}$ as follows %as in \eqref{Nonlin3}, \eqref{nuovoSimboA}. 
\begin{equation}\label{deadwood13}
\mathfrak{N}^{+}(\tau,W;\x):=\left(\begin{matrix}
\mathfrak{n}^{+}(\tau,W;\x) & 0 \\
0 &{\mathfrak{n}^{+}(\tau,W;-\x)}
\end{matrix}\right)\,, \qquad
\mathcal{R}^{+}\in \Sigma\mathcal{R}^{-\rho}_1[r,N]\otimes\mathcal{M}_{2}(\mathbb{C})\,.
\end{equation}
In particular 
we make the ansatz that $\mathfrak{n}^{+}(\tau,W;\x)$
is real valued and $x$-independent.
By expanding the non linear commutator as in \eqref{commuExpA}-\eqref{commuExpF},
and using that $\mathfrak{N}^{+}, B_p$ are independent of $x$ we deduce
(recall \eqref{deadwood5})
\begin{align}
\big[X_{\mathcal{G}}(W) ,P^{\tau}(W) \big]&=
\opbw\big(\ii E(d_{W}B_p)(W;\x)[P^{\tau}(W)]\big)[W]\label{dead11}\\
&-\opbw\big(\ii E(d_{W}\mathfrak{M}^{+})(\tau,W;\x)[X_{\mathcal{G}}(W)]\big)[W]
\label{dead12}\\
&
+\opbw(\ii EB_p(\tau,W;\x))\big[\mathcal{R}^{+}(\tau,W)[W]\big]
+ \mathcal{B}_p(W)\mathcal{R}^{+}(\tau,W)[W]\label{dead13}
\\&
-\mathcal{R}^{+}(\tau,W)[X_{\mathcal{G}}(W) ]
-(d_{W}\mathcal{R}^{+})(\tau,W)[X_{\mathcal{G}}(W) ]\,.\label{dead14}
\end{align}
Notice that, by \eqref{deadwood11} (recall also \eqref{Pollini333}), 
we have
\begin{equation}\label{deadwood19}
(\pa_{t}b_p)(W;\x):=(d_{W}b_p)(W;\x)\big)[P^{\tau}(W)]\in \Sigma\Gamma^{m}_p[r,N]\,.
\end{equation}
\noindent
{\bf Order $m$.}
By 
 \eqref{dead11}-\eqref{dead14}
we have that, at the highest order, the equation \eqref{deadwood10}
reads
%Using the expansion of the symbol $a$ in \eqref{ansatz}
%we have, at the highest order,
%the equation
\begin{equation}\label{deadwood14}
\left\{\begin{aligned}
&\pa_{\tau}\mathfrak{n}^{+}(\tau,W;\x)=(\pa_{t}b_p)(W;\x)
-d_{W}\mathfrak{n}^{+}(\tau,W;\x)[X_{\mathcal{G}}(W)]\\
&\mathfrak{n}^{+}(0,W;\x)=\mathfrak{n}(W;\x)\,.
\end{aligned}\right.
\end{equation}
Notice that, if $\mathfrak{n}^{+}$ solves \eqref{deadwood14},  
the function 
\begin{equation*}
g(\tau)=\mathfrak{n}^{+}(\tau,W(\tau);\x)
\end{equation*}
where $W(\tau)=\tilde{\mathcal{A}}^{\tau}_p(U)$ in \eqref{Pollini4}, satisfies
\begin{equation*}
\begin{aligned}
&\pa_{\tau}g(\tau)=(\pa_{t}b_p)(W(\tau);\x)\,, \qquad g(0):=\mathfrak{n}(Z;\x)\,,
\qquad \Rightarrow \qquad
g(\tau)=g(0)+\int_{0}^{\tau}(\pa_{t}b_p)(W(\s);\x)d\s\,.
\end{aligned}
\end{equation*}
Therefore
\begin{equation}\label{deadwood18}
\mathfrak{n}^{+}(\tau,W;\x):=\mathfrak{n}(\tilde{\mathcal{A}}_{p}^{-\tau}(W);\x)
+\int_{0}^{\tau}(\pa_{t}b_p)(\tilde{\mathcal{A}}_{p}^{\s}
\tilde{\mathcal{A}}_{p}^{-\tau}(W);\x)d\s\,,
\end{equation}
solves the problem \eqref{deadwood14}. 
Reasoning as in Section \ref{sec:PreEgorov} 
(see for instance the proof of Lemma \ref{iterative})
using the estimates 
of Theorem \ref{flussononlin},
one deduces that 
$\mathfrak{n}^{+}(\tau,W;\x)$ belongs to $\Sigma\Gamma^{m}_1[r,N]$.
By \eqref{deadwood18}, \eqref{deadwood19}
%, and setting 
%$\widetilde{n}^{(1)}(W;\x)=\mathfrak{n}^{+}(1,W;\x)$
one gets the \eqref{Pollini8}.

Define \[
Q^{\tau}(W)=\ii E\Omega W+\opbw(\ii E \mathfrak{N}^{+}(\tau,W;\x))[W]
\]
with $\mathfrak{N}^{+}$
of the form \eqref{deadwood13}.
Then 
the operator $Q^{\tau}(W)$ solves the problem (recall \eqref{Pollini4},
\eqref{dead11}-\eqref{dead14}, \eqref{deadwood14})
\begin{equation}\label{deadwood20}
\left\{\begin{aligned}
&\pa_{\tau}Q^{\tau}(W)=\big[ X_{\mathcal{G}}(W),
%\opbw(\ii E{\bf B}(\tau,Z;x,\x))[Z], 
Q^{\tau}(W)\big]\\
&Q^{0}(W)=\ii E\Omega W+\ii E\opbw(\mathfrak{N}(W;\x))[W]\,.
\end{aligned}\right.
\end{equation}
It remains to prove that 
the difference $Q^{\tau}-P^{\tau}$ is a smoothing remainder in
in $\Sigma\mathcal{R}^{-\rho}_{1}[r,N]\otimes\mathcal{M}_2(\mathbb{C})$.
First of all we write
\begin{equation*}
Q^{\tau}(W)-P^{\tau}(W)=V^{\tau}\circ (\tilde{\mathcal{A}}_p^{\tau})^{-1}(W)\,,
\end{equation*}
where, recalling \eqref{Pollini6},
\[
V^{\tau}(U):=Q^{\tau}\circ\tilde{\mathcal{A}}_p^{\tau}(U)
-(d_{U}\tilde{\mathcal{A}}_p^{\tau})(U)[X(U)]\,.
\]
Reasoning as for the operator in  \eqref{strike5}
we deduce (see \eqref{Pollini})
%and using \eqref{Nonlin2}, \eqref{approx}, \eqref{operatoreEgor1},
%one can check that
\begin{equation}\label{deadwood22}
\left\{\begin{aligned}
&\pa_{\tau}V^{\tau}(U)=(d_{W(\tau)}X_{\mathcal{G}})(\tilde{\mathcal{A}}_p^{\tau})\big[ V^{\tau}\big]\,,\\
&V^{0}(U)=-\mathcal{R}(U)U\,.
\end{aligned}\right.
\end{equation}
Reasoning as done for \eqref{probTotale}
one can check  
$V^{\tau}\in \Sigma\mathcal{R}^{-\rho}_{1}[r,N]\otimes\mathcal{M}_2(\mathbb{C})$.
This implies that 
$\mathcal{R}^{+}(\tau,W)[W]:=Q^{\tau}(W)-P^{\tau}(W)$ 
(again using Theorem \ref{flussononlin})
belongs to
$
\Sigma\mathcal{R}^{-\rho}_{1}[r,N]\otimes\mathcal{M}_2(\mathbb{C})$.
%Then we have obtained \eqref{Pollini7}.
Hence we have the \eqref{Pollini7}.
%Setting 
%$\widetilde{\mathcal{Q}}^{(1)}(W):=R^{+}(1,W)$ we get the
%\eqref{deadwood7}.
\end{proof}

\subsubsection{Flows of smoothing remainders}
In this section we consider
 the flow
$\mathcal{A}_p^{\tau}$ of 
\begin{equation}\label{Pollini44bis}
\left\{
\begin{aligned}
&\pa_{\tau}\mathcal{A}_p^{\tau}(U)=\mathcal{Q}^{(p)}_{aux}(\mathcal{A}_p^{\tau}(U))\mathcal{A}_p^{\tau}(U)\\
&\mathcal{A}_p^{0}(U)=U\,,
\end{aligned}\right.
\end{equation}
where $\mathcal{Q}^{(p)}_{aux}\in \widetilde{\mathcal{R}}^{-\rho}_{p}$, this problem
 is well-posed by standard theory of ODEs on Banach spaces.
Assume also that 
$\mathcal{Q}_{aux}^{(p)}(U)U$
 is Hamiltonian, i.e. there is a map $\mathtt{Q}^{(p)}\in 
 \widetilde{\mathcal{M}}_p\otimes\mathcal{M}_2(\CCC)$
 such that (recall \eqref{X_H})
 \begin{equation}\label{samba1}
 \mathcal{Q}_{aux}^{(p)}(U)U=X_{\mathcal{C}}(U)\,,\qquad \mathcal{C}(U)
 :=\int_{\mathbb{T}}\mathtt{Q}^{(p)}(U)U\cdot\ov{U}dx\,.
 \end{equation}
Recalling the system \eqref{Pollini} we define
\begin{equation}\label{pollo22bis}
W:={\mathcal{A}}_{p}(U):={\mathcal{A}}_{p}^{1}(U)\,,
\qquad
\mathcal{X}^{+}(W):=P^{\tau}(W)_{|\tau=1}
:=d{\mathcal{A}}_{p}^{\tau}\big({\mathcal{A}}_{p}^{-\tau}(W)\big)
\big[ {\mathcal{X}}({\mathcal{A}}_{p}^{-\tau}(W))\big]_{|\tau=1}\,.
\end{equation}

We prove the following.

\begin{proposition}\label{prop:PolliniSmooth}
For $r>0$ small enough the following holds.
The function $W$ in \eqref{pollo22bis} satisfies 
\begin{equation}\label{Pollini7bis}
\left\{
\begin{aligned}
&\dot{W}=\mathcal{X}^{+}(W)
:=\ii E \Omega W+
\ii E\opbw\big( \mathfrak{N}^{+}(W;\x)\big)[W]
+\mathcal{R}^{+}(W)[W]\\
&W(0)=\mathcal{A}(Z_0)
\end{aligned}\right.
\end{equation}
where $\mathcal{R}^{+}\in\Sigma\mathcal{R}^{-\rho}_1[r,N]
\otimes\mathcal{M}_2(\mathbb{C})$ and
${\mathfrak{N}}^{+}\in 
\Sigma\Gamma^{m}_1[r,N]
\otimes\mathcal{M}_2(\mathbb{C})$
is independent of $x\in \mathbb{T}$, 
real valued  and it has the form
\begin{equation}\label{Pollini8bis}
\begin{aligned}
&\mathfrak{N}^{+}(W;\x):=\left(
\begin{matrix}
\mathfrak{n}^{+}(W;\x) & 0\\
0 & {\mathfrak{n}^{+}(W;-\x) }
\end{matrix}
\right)\,,
\qquad 
\mathfrak{n}^{+}(\tau,W;\x)
:={\mathfrak{n}}(\mathcal{A}_{p}^{-\tau}(W);\x)\,.
\end{aligned}
\end{equation}
Moreover, for any $s\geq s_0$, the maps ${\mathcal{A}}^{\pm1}_{p}$ are symplectic 
and satisfy
\begin{equation}\label{Pollini9bis}
\|{\mathcal{A}}_{p}^{\pm1}(U)\|_{H^{s}}\leq \|U\|_{H^{s}}(1+C\|U\|^{p}_{H^{s_0}})\,,
\end{equation}
for some constant $C>0$ depending  on $s$.
\end{proposition}

\begin{proof}
The field 
 $P^{\tau}$  in \eqref{pollo22bis} satisfies,  for $\tau\in[0,1]$, 
 (recall \eqref{Pollini})
\begin{equation}\label{cannon1bis}
\left\{
\begin{aligned}
&\pa_{\tau}P^{\tau}(W)=\big[\mathcal{Q}^{(p)}_{aux}(W)W , P^{\tau}(W)\big]\\
&P^{0}(W)={\mathcal{X}}(W)\,.
%\opbw(a(z;x,\x))[z]\,.
\end{aligned}\right.
\end{equation}
Moreover, by the remarks under Definition \ref{smoothoperatormaps} and by Theorem \ref{flussononlin},
we note that 
\begin{equation}\label{cannon2bis}
P^{\tau}(W)=\ii E \Omega W+M_1(\tau;W)[W]\,,\;\;
M_1\in \Sigma\mathcal{M}_{1}[r,N]\otimes\mathcal{M}_2(\mathbb{C})\,,
\end{equation}
with estimate uniform in $\tau\in[0,1]$.
Actually we shall prove that
\begin{equation*}%\label{cannon3bis}
P^{\tau}(W)
=\ii E\Omega W+
\ii E\opbw(\mathfrak{N}^{+}(\tau,W;\x))[W]+\mathcal{R}^{+}(\tau,W)[W]\,,
\end{equation*}
with $\mathfrak{N}^{+}\in \Sigma\Gamma_1^{m}[r,N]\otimes\mathcal{M}_2(\CCC)$, 
$\mathcal{Q}^{+}\in\Sigma\mathcal{R}^{-\rho}_1[r,N]\otimes\mathcal{M}_{2}(\mathbb{C})$
and 
\begin{equation}\label{cannon4bis}
\mathfrak{N}^{+}(\tau,W;\x):=\left(\begin{matrix}
\mathfrak{n}^{+}(\tau,W;\x) & 0 \\
0 &{\mathfrak{n}^{+}(\tau,W;-\x)}
\end{matrix}\right)\,.
\end{equation}
Furthermore 
we make the ansatz that $\mathfrak{n}^{+}(\tau,W;\x)$
is real valued and $x$-independent.
By expanding the non linear commutator as in \eqref{commuExpA}-\eqref{commuExpF},
we deduce
\begin{align}
\big[\mathcal{Q}^{(p)}_{aux}(W)W ,P^{\tau}(W) \big]&=
\mathcal{Q}^{(p)}_{aux}(W)\big[P^{\tau}(W)\big]+
(d_{W}\mathcal{Q}^{(p)}_{aux}(W)W)\big[P^{\tau}(W)\big]\label{cannon5bis}\\
&-
\mathcal{R}^{+}(W)\big[\mathcal{Q}^{(p)}_{aux}(W)W\big]-
(d_{W}\mathcal{R}^{+}(W)W)\big[\mathcal{Q}^{(p)}_{aux}(W)W\big]
\label{cannon6bis}\\
&-\opbw\big( \mathfrak{N}^{+}(\tau,W;\x)\big)
\big[\mathcal{Q}^{(p)}_{aux}(W)W\big]\label{cannon7bis}\\
&
-\opbw\big( (d_{W}\mathfrak{N}^{+})(\tau,W;\x)
\big[\mathcal{Q}^{(p)}_{aux}(W)W \big]  \big)[W]\,.
\label{cannon8bis}
\end{align}
Notice that, by \eqref{cannon2bis}, 
we have
\begin{equation}\label{cannon9bis}
\frac{d}{dt}\big(\mathcal{Q}_{aux}^{(p)}(W)W\big)
=\mathcal{Q}^{(p)}_{aux}(W)\big[P^{\tau}(W)\big]+(d_{W}\mathcal{Q}^{(p)}_{aux}(W)W)\big[P^{\tau}(W)\big]\in
\Sigma\mathcal{R}^{-\rho}_{1}[r,N]\,.
\end{equation}
\noindent
{\bf Order $m$.}
By 
 \eqref{cannon5bis}-\eqref{cannon9bis}
we have that, at the highest order, the equation \eqref{cannon1bis}
reads (recall \eqref{Pollini}, \eqref{cannon4bis})
\begin{equation}\label{cannon10bis}
\left\{\begin{aligned}
&\pa_{\tau}\mathfrak{n}^{+}(\tau,W;\x)=
-d_{Z}\mathfrak{n}^{+}(\tau,W;\x)[\mathcal{Q}^{(p)}_{aux}(W)W]\\
&\mathfrak{n}^{+}(0,W;\x)={\mathfrak{n}}(W;\x)\,.
\end{aligned}\right.
\end{equation}
Notice that, if $\mathfrak{n}^{+}$ solves \eqref{cannon10bis},  
the function 
$g(\tau)=\mathfrak{n}^{+}(\tau,W(\tau);\x)$ is constant 
along the solution of \eqref{Pollini44bis}. 
Therefore
\begin{equation}\label{cannon11bis}
\mathfrak{n}^{+}(\tau,W;\x)
:={\mathfrak{n}}(\mathcal{A}_{p}^{-\tau}(W);\x)\,,
\end{equation}
solves the problem \eqref{cannon10bis}. 
Reasoning as in Section \ref{sec:PreEgorov} 
(see for instance the proof of Lemma \ref{iterative})
using the estimates 
of Theorem \ref{flussononlin},
one deduces that 
$\mathfrak{n}^{+}(\tau,W;\x)$ belongs to $\Sigma\Gamma^{m}_1[r,N]$.
This proves the \eqref{Pollini8bis}.

Let us now define 
\[
Q^{\tau}(W)=\ii E\Omega W+\opbw(\ii E \mathfrak{N}^{+}(\tau,W;\x))[W]
\]
with $\mathfrak{N}^{+}$
of the form \eqref{cannon4bis}, \eqref{cannon11bis}.
Then 
the operator $Q^{\tau}(W)$ solves the problem (recall \eqref{Pollini44bis},
\eqref{cannon5bis}-\eqref{cannon8bis}, \eqref{cannon10bis})
\begin{equation}\label{cannon20bis}
\left\{\begin{aligned}
&\pa_{\tau}Q^{\tau}(W)=\big[ \mathcal{Q}_{aux}^{(p)}(W)W,
%\opbw(\ii E{\bf B}(\tau,Z;x,\x))[Z], 
Q^{\tau}(W)\big]+\mathcal{G}_{\rho}(W)W\\
&Q^{0}(W)=\ii E\Omega W+\ii E\opbw(\mathfrak{N}(W;\x))[W]\,,
\end{aligned}\right.
\end{equation}
for some $\mathcal{G}_{\rho}\in \Sigma\mathcal{R}^{-\rho}_{1}[r,N]\otimes\mathcal{M}_2(\mathbb{C})$.
It remains to prove that 
the difference $Q^{\tau}-P^{\tau}$ is a smoothing remainder in
in $\Sigma\mathcal{R}^{-\rho}_{1}[r,N]\otimes\mathcal{M}_2(\mathbb{C})$.
First of all we write
\begin{equation*}%\label{cannon21bis}
Q^{\tau}(W)-P^{\tau}(W)=V^{\tau}\circ (\mathcal{A}_p^{\tau})^{-1}(W)\,,
\end{equation*}
where, recalling \eqref{pollo22bis},
\[
V^{\tau}(U):=Q^{\tau}\circ\mathcal{A}_1^{\tau}(U)-(d_{U}\mathcal{A}_1^{\tau})(U)[X(U)]\,.
\]
Reasoning as for the operator in  \eqref{strike5}
we deduce (see \eqref{Pollini})
\begin{equation}\label{cannon22bis}
\left\{\begin{aligned}
&\pa_{\tau}V^{\tau}(Z)=(d_{Z}\mathtt{Q}_{aux})(\mathcal{A}_{p}^{\tau})\big[ V^{\tau}\big]\,,
\qquad \mathtt{Q}_{aux}(Z):=\mathcal{Q}_{aux}^{(p)}(Z)Z\,,
\\
&V^{0}(Z)=-{\mathcal{R}}(Z)Z\,.
\end{aligned}\right.
\end{equation}
Reasoning as done for \eqref{probTotale}
one can check   that
$V^{\tau}\in \Sigma\mathcal{R}^{-\rho}_{1}[r,N]\otimes\mathcal{M}_2(\mathbb{C})$.
This implies that  the remainder
$\mathcal{R}^{+}(\tau,W)[W]:=Q^{\tau}(W)-P^{\tau}(W)$ 
(again using Theorem \ref{flussononlin})
belongs to 
$ 
\Sigma\mathcal{R}^{-\rho}_{1}[r,N]\otimes\mathcal{M}_2(\mathbb{C})$.
Hence, by the discussion above, we have obtained the \eqref{Pollini7bis}.
\end{proof}

\subsection{Elimination of quadratic terms}\label{quadtermBNF}
Consider the system \eqref{NonlinBNF1}. By hypothesis we have
(see also \eqref{opfinale})
\begin{equation*}%\label{deadwood1}
\begin{aligned}
\mathfrak{M}(Z;\x)&=\mathfrak{M}_1(Z;\x)+\mathfrak{M}_{\geq 2}(Z;\x)\,,\quad
\mathfrak{M}_{1}(Z;\x)\in \widetilde{\Gamma}^{m}_{1}\otimes\mathcal{M}_2(\CCC)\,,\quad
\mathfrak{M}_{\geq 2}(Z;\x)\in \Sigma\Gamma^{m}_{2}[r,N]\otimes\mathcal{M}_2(\CCC)\,, \\
\mathcal{Q}(Z)&=\mathcal{Q}_{1}(Z)+\mathcal{Q}_{\geq2}(Z)\,,\quad
\mathcal{Q}_{1}(Z)\in \widetilde{\mathcal{R}}^{-\rho}_{1}\otimes\mathcal{M}_{2}(\CCC) 
\quad
\mathcal{Q}_{\geq2}(Z)\in \Sigma\mathcal{R}^{-\rho}_{2}[r,N]\otimes\mathcal{M}_{2}(\CCC)\,.
\end{aligned}
\end{equation*}
The aim of this section is to eliminate the quadratic vector field
\begin{equation*}%\label{pollo0}
\mathcal{Y}_{2}(Z):=\opbw(\mathfrak{M}_1(Z;\x))Z+\mathcal{Q}_1(Z)Z\,,
\end{equation*}
in \eqref{NonlinBNF1}.
This will be done into two steps.
We shall first 
reduce the symbol $\mathfrak{M}_1(Z;x)$ (see section \ref{sec:stepsimbo1}) 
and then the smoothing operator 
$\mathcal{Q}_1(Z)$ (see section \ref{sec:stepsmooth1}).
Before giving the proof we make a comment on the symbol 
$\mathfrak{m}_1(Z;\x)$.
By the expansion \eqref{espandoFousimbo}  (with $p=1$)
and the fact that
$\mathfrak{m}_1(Z;\x)$
is independent of $x\in \mathbb{T}$
we can write
\begin{equation}\label{pollo}
\mathfrak{m}_1(U;\x)=
\sum_{j\in \mathbb{Z}}e^{\ii jx}
u_{j}(\mathfrak{m}_1)^{+}_{j}(\x)+e^{-\ii j x }\ov{u_{j}} (\mathfrak{m}_1)_{j}^{-}(\x)
=
u_0  (\mathfrak{m}_1)_{0}^{+}(\x)+\ov{u_0}  (\mathfrak{m}_1)_{0}^{-}(\x)\,.
\end{equation}
Notice that $\mathfrak{m}_1(U;\x)$ depends \emph{only}
on the average $u_0$ of $u$. 
We shall assume that  (see \eqref{omegone}) $\omega_0\neq0$.

%There are two possibilities:
%
%\begin{itemize}
%\item[(1)] 
%if  (see \eqref{omegone}) $\omega_0\neq0$ then 
%we work on functions $u$ such that, possibly, $u_0\neq0$;
%
%\item[(2)] 
%if  (see \eqref{omegone}) $\omega_0=0$ then 
%we shall work on functions $u$ such that $u_0=0$.
%\end{itemize}
%
%
%
%Therefore in case $(2)$
%the  term in \eqref{pollo} is \emph{already} zero and the step in section \ref{sec:stepsimbo1}
%is not required. One can directly deal with the term $\mathcal{Q}_1$ in \eqref{pollo0}
%(see section \ref{sec:stepsmooth1}).

\subsubsection{Elimination of linear symbols}\label{sec:stepsimbo1}
%We perform this step only under the hypothesis $\omega_0\neq0$ (see \eqref{omegone}).

Notice that the system \eqref{NonlinBNF1} has the form 
\eqref{Pollini}
with $\mathfrak{N}\rightsquigarrow \mathfrak{M}$, 
$\mathcal{R}\rightsquigarrow\mathcal{Q}$,
$U\rightsquigarrow Z$.
Consider now 
a real valued, independent of $x\in \mathbb{T}$ symbol 
$b_{1}\in \widetilde{\Gamma}^{m}_{1}$
and the matrix $B_{1}(Z;\x)$
having the form \eqref{Pollini333} with $p=1$.
Let $\widetilde{\mathcal{A}}_{1}^{\tau}$ be the flow of 
\eqref{Pollini4} 
generated by 
the Hamiltonian 
\eqref{Pollini3} with $p=1$.
By Proposition \ref{prop:Pollini}
we have that the variable
\begin{equation*}
W:=\widetilde{\mathcal{A}}_{1}(Z):=\widetilde{\mathcal{A}}_{1}^{1}(Z)\,,
\qquad
\widetilde{\mathcal{Y}}^{(1)}(W):=P^{\tau}(W)_{|\tau=1}
:=d\widetilde{\mathcal{A}}_{1}^{\tau}\big(\widetilde{\mathcal{A}}_{1}^{-\tau}(W)\big)
\big[ \mathcal{Y}(\widetilde{\mathcal{A}}_{1}^{-\tau}(W))\big]_{|\tau=1}\,,
\end{equation*}
satisfies (recall \eqref{NonlinBNF1}, \eqref{Pollini7})
\begin{equation}\label{pollo7}
\left\{
\begin{aligned}
&\dot{W}=\widetilde{\mathcal{Y}}^{(1)}(W):=\ii E \Omega W+
\ii E\opbw\big( \widetilde{\mathfrak{M}}^{(1)}(W;\x)\big)[W]
+\widetilde{\mathcal{Q}}^{(1)}(W)[W]\\
&W(0)=\widetilde{\mathcal{A}}(Z_0)
\end{aligned}\right.
\end{equation}
where $\widetilde{\mathcal{Q}}^{(1)}\in\Sigma\mathcal{R}^{-\rho}_1[r,N]
\otimes\mathcal{M}_2(\mathbb{C})$ and
$\widetilde{\mathfrak{M}}^{(1)}\in 
\Sigma\Gamma^{m}_1[r,N]
\otimes\mathcal{M}_2(\mathbb{C})$
is independent of $x\in \mathbb{T}$, 
real valued  and has the form
\begin{equation}\label{pollo8}
\begin{aligned}
&\widetilde{\mathfrak{M}}^{(1)}(W;\x):=\left(
\begin{matrix}
\widetilde{\mathfrak{m}}^{(1)}(W;\x) & 0\\
0 & {\widetilde{\mathfrak{m}}^{(1)}(W;-\x) }
\end{matrix}
\right)\,,\\
&
\widetilde{\mathfrak{m}}^{(1)}(W;\x)=\mathfrak{m}(\widetilde{\mathcal{A}}_{1}^{-1}(W);\x)+
\int_{0}^{1} (d_{Z}b_1)\big(\widetilde{\mathcal{A}}_{1}^{\s}\widetilde{\mathcal{A}}_{1}^{-1}(W);\x\big)
[P^{\s}\big(
\widetilde{\mathcal{A}}_{1}^{\s}\widetilde{\mathcal{A}}_{1}^{-1}(W)
\big) ]d\s\,.
\end{aligned}
\end{equation}

\vspace{0.5em}
\noindent
{\bf The homological equation.}
We look for a linear symbol $b_1\in \widetilde{\Gamma}_1^{m}$ %as in \eqref{deadwood2}
such that the symbol $\widetilde{\mathfrak{m}}^{(1)}(W;\x)$ in \eqref{pollo8}
is at least quadratic in the variable $W$.
In order to do this we reason as follows.
First of all, by Theorem \ref{flussononlin}, 
we deduce that the flow $\widetilde{\mathcal{A}}_1^{\tau}$
of \eqref{Pollini4} is such that
\begin{align*}
&\widetilde{\mathcal{A}}_1^{\tau}(Z)=Z+N_1(\tau,Z)[Z]\,,\qquad 
\widetilde{\mathcal{A}}_1^{-\tau}(Z)=Z+N_2(\tau,Z)[Z]\,, \quad N_1,N_{2}
\in\Sigma\mathcal{M}_{1}[r,N]\otimes\mathcal{M}_2(\mathbb{C})\,,
\end{align*}
with estimates uniform in $\tau\in[0,1]$.
Recall also that, by hypothesis, one has
\[
\mathfrak{m}(Z;\x)=\mathfrak{m}_1(Z;\x)+\mathfrak{m}_{\geq2}(Z;\x)\,,\quad
\mathfrak{m}_1\in\widetilde{\Gamma}^{m}_1\,,\quad
\mathfrak{m}_{\geq2}\in\Sigma\Gamma^{m}_{2}[r,N]\,.
\]
Therefore, using also \eqref{deadwood11} and the composition Proposition \ref{composizioniTOTALI}, 
we can write $\widetilde{\mathfrak{m}}^{(1)}$
in \eqref{pollo8}
as
\begin{align}
&\widetilde{\mathfrak{m}}^{(1)}(W;\x)=\widetilde{\mathfrak{m}}^{(1)}_1(W;\x)+
\widetilde{\mathfrak{m}}^{(1)}_{\geq2}(W;\x)\,,\quad 
\widetilde{\mathfrak{m}}^{(1)}_{\geq2}\in \Sigma\Gamma^{m}_2[r,N]
\otimes\mathcal{M}_2(\mathbb{C})\,,
\nonumber\\
&\widetilde{\mathfrak{m}}^{(1)}_{1}(W;\x):=
\mathfrak{m}_1(W;\x)+(d_{W}b_1)(W;\x)\big[\ii \Omega W\big]\,.
\label{deadwood30}
\end{align}
We prove the following.

\begin{lemma}{\bf (Homological equation).}\label{omoeqBB1}
There exists a symbol $b_1\in \widetilde{\Gamma}^{m}_1$ such that
(see \eqref{deadwood30})
\begin{equation}\label{deadwood31}
\widetilde{\mathfrak{m}}^{(1)}_{1}(W;\x)=0\,.
\end{equation}
Moreover $b_1$ is real valued.
\end{lemma}

\begin{proof}
We look for a symbol $b_1\in \widetilde{\Gamma}^{m}_1$ of the form
\begin{equation}\label{pollo11}
b_1(W;\x)=u_0  (b_1)_{0}^{+}(\x)+\ov{u_0}  (b_1)_{0}^{-}(\x)\,,
\end{equation}
for some coefficients $(b_1)_{0}^{\s}$, $\s\in \{\pm\}$. 
Let us define (recall \eqref{pollo})
\begin{equation}\label{pollo10}
(b_1)_{0}^{\s}(\x):=\frac{(\mathfrak{m}_1)_{0}^{-}(\x)}{\ii \s\omega_0}\,, \quad 
\s\in\{\pm\}\,.
\end{equation}
Using the  expansion \eqref{espandoFousimbo}, 
with $p=1$ (recall also \eqref{pollo})
one can check, by an explicit computation, that the function in \eqref{pollo11}
with coefficients in \eqref{pollo10} solves the equation \eqref{deadwood31}.
Moreover, since $\mathfrak{m}_1$ is real valued, one has
\[
\ov{(\mathfrak{m}_1)_{0}^{+}(\x)}=(\mathfrak{m}_1)_{0}^{-}(\x).
\] 
Hence, by \eqref{pollo10}, also the coefficients $(b_1)_{0}^{\s}(\x)$
have the same property. This implies that $b_1$ is real valued.
\end{proof}

\subsubsection{Elimination of linear smoothing operators}\label{sec:stepsmooth1}
Consider the 
system \eqref{pollo7} obtained from the change of coordinates generated 
(see \eqref{Pollini3}, \eqref{Pollini4})
by  $b_1$ given by the Lemma \ref{omoeqBB1}.
We now consider the flow
$\mathcal{A}_1^{\tau}$ of 
\begin{equation}\label{Pollini44}
\left\{
\begin{aligned}
&\pa_{\tau}\mathcal{A}_1^{\tau}(W)=\mathcal{Q}^{(1)}_{aux}(\mathcal{A}_1^{\tau}(W))\mathcal{A}_1^{\tau}(W)\\
&\mathcal{A}_1^{0}(W)=W\,,
\end{aligned}\right.
\end{equation}
where $\mathcal{Q}^{(1)}_{aux}\in \widetilde{\mathcal{R}}^{-\rho}_{1}$
which is well-posed by Theorem \ref{flussononlin}.
Assume that $\mathcal{Q}^{(1)}_{aux}$ has the form 
\eqref{smooth-terms2}, \eqref{R2epep'} with $p=1$ and 
 coefficients  (see \eqref{BNF5}) 
 \begin{equation}\label{cannon31}
({\mathcal{Q}}_{aux}^{(1)}(W_1))_{\s,j}^{\s',k}=\sum_{\substack{
\s_1\in\{\pm\}, j_1\in\mathbb{Z} \\ \s_1j_1=\s j-\s'k}}
((\mathtt{q}_{aux}^{(1)})_{j_1}^{\s_1} )_{\s,j}^{\s',k} w_{j_1}^{\s_1}\,,\qquad
((\mathtt{q}_{aux}^{(1)})_{j_1}^{\s_1} )_{\s,j}^{\s',k}\in \mathbb{C}\,.
\end{equation}
 Finally assume that $\mathcal{Q}^{(1)}_{aux}$  is Hamiltonian, see \eqref{samba1}.
We define
\begin{equation}\label{pollo22}
W_1:={\mathcal{A}}_{1}(W):={\mathcal{A}}_{1}^{1}(W)\,,
\qquad
\mathcal{Y}^{(1)}(W_1):=P^{\tau}(W_1)_{|\tau=1}
:=d{\mathcal{A}}_{1}^{\tau}\big({\mathcal{A}}_{1}^{-\tau}(W_1)\big)
\big[ \widetilde{\mathcal{Y}}^{(1)}({\mathcal{A}}_{1}^{-\tau}(W_1))\big]_{|\tau=1}\,.
\end{equation}
In the following lemma we show that it is possible to choose the coefficients
$((\mathtt{q}_{aux}^{(1)})_{j_1}^{\s_1} )_{\s,j}^{\s',k}$ in \eqref{cannon31}
in such a way the vector field in \eqref{pollo22} does not
contain any \emph{quadratic} monomials.

\begin{lemma}\label{prop:Pollini100}
For $r>0$ small enough the following holds.
There exists $\mathcal{Q}_{aux}^{(1)}\in \widetilde{\mathcal{R}}^{-\rho}_1$, 
of the form \eqref{cannon31},
such that
the function $W_1$ in \eqref{pollo22} satisfies 
\begin{equation}\label{Pollini77}
\left\{
\begin{aligned}
&\dot{W}_1=\mathcal{Y}^{(1)}(W):=\ii E \Omega W_1+
\ii E\opbw\big( \mathfrak{M}^{(1)}(W_1;\x)\big)[W_1]
+\mathcal{Q}^{(1)}(W_1)[W_1]\\
&W_1(0)=\mathcal{A}_1(W)
\end{aligned}\right.
\end{equation}
where $\mathcal{Q}^{(1)}\in\Sigma\mathcal{R}^{-\rho}_2[r,N]
\otimes\mathcal{M}_2(\mathbb{C})$ and
${\mathfrak{M}}^{(1)}\in 
\Sigma\Gamma^{m}_2[r,N]
\otimes\mathcal{M}_2(\mathbb{C})$
is independent of $x\in \mathbb{T}$, 
real valued  and has the form
\begin{equation}\label{Pollini88}
\begin{aligned}
&\mathfrak{M}^{(1)}(W;\x):=\sm{\mathfrak{m}^{(1)}(W;\x)}{0}{0}{\mathfrak{m}^{(1)}(W;-\x)}\,,
%
%:=\left(
%\begin{matrix}
%\mathfrak{n}^{+}(W;\x) & 0\\
%0 & {\mathfrak{n}^{+}(W;-\x) }
%\end{matrix}
%\right)\,,
%\\
%&
%\mathfrak{n}^{+}(W;\x)=\mathfrak{n}(\mathcal{A}_{p}^{-1}(W);\x)+
%\int_{0}^{1} (d_{Z}b_p)\big(\mathcal{A}_{p}^{\s}\mathcal{A}_{p}^{-1}(W);\x\big)
%[P^{\s}\big(
%\mathcal{A}_{p}^{\s}\mathcal{A}_{p}^{-1}(W)
%\big) ]d\s\,.
\end{aligned}
\end{equation}
Moreover, for any $s\geq s_0$, the maps $\mathcal{A}^{\pm1}_{1}$ 
are symplectic 
and satisfy
\begin{equation}\label{Pollini99}
\|\mathcal{A}_{1}^{\pm1}(U)\|_{H^{s}}\leq \|U\|_{H^{s}}(1+C\|U\|_{H^{s_0}})\,,
\end{equation}
for some constant $C>0$ depending  on $s$.
\end{lemma}

\begin{proof}
We start by studying the conjugate of 
\eqref{pollo7} under the flow \eqref{Pollini44} assuming that 
$\mathcal{Q}^{(1)}_{aux}\in \widetilde{\mathcal{R}}^{-\rho}_1$.
Notice that the vector field $\widetilde{\mathcal{Y}}^{(1)}$ in \eqref{pollo7}
has the same form of $\mathcal{X}$ in \eqref{Pollini},
with $\mathfrak{N}\rightsquigarrow \widetilde{\mathfrak{M}}^{(1)}$,
$\mathcal{R}\rightsquigarrow \widetilde{\mathcal{Q}}^{(1)}$,
$U\rightsquigarrow W$.
The generator $\mathcal{Q}^{(1)}_{aux}$ has the same properties of the 
generator $\mathcal{Q}^{(p)}_{aux}$ in \eqref{Pollini44bis}.
Therefore Proposition \ref{prop:PolliniSmooth} applies.
As a consequence we obtain that
\begin{equation}\label{cannon23}
\dot{W}_1=\ii E \Omega W_1+
\ii E\opbw\big( \mathfrak{M}^{(1)}(W_1;\x)\big)[W_1]
+\mathcal{Q}^{(1)}(W_1)[W_1]\,,
\end{equation}
where $\mathcal{Q}^{(1)}\in\Sigma\mathcal{R}^{-\rho}_1[r,N]\otimes\mathcal{M}_2(\CCC)$, 
$\mathfrak{M}^{(1)}$ has the form \eqref{Pollini88} with
(see \eqref{Pollini8bis})
\begin{equation*}
\mathfrak{m}^{(1)}(\tau,W_1;\x)
:=\widetilde{\mathfrak{m}}^{(1)}(\mathcal{A}_{1}^{-\tau}(W_1);\x)\in 
\Sigma\Gamma^{m}_2[r,N]\,.
\end{equation*}
The symbol above is quadratic since, by \eqref{deadwood30} and Lemma \ref{omoeqBB1}, 
the symbol
$\widetilde{\mathfrak{m}}^{(1)}(\tau,W;\x)$
is at least quadratic.
Moreover we have that
\begin{equation}\label{cannon24}
\mathcal{Q}^{(1)}(W_1)[W_1]=\mathcal{Q}_1^{(1)}(W_1)[W_1]
+\mathcal{Q}_{\geq 2}^{(1)}(W_1)[W_1]\,,\qquad 
\mathcal{Q}_{\geq 2}^{(1)}\in \Sigma\mathcal{R}^{-\rho}_2[r,N]\otimes\mathcal{M}_2(\CCC)\,,
\end{equation}
with $\mathcal{Q}_1^{(1)}\in\widetilde{\mathcal{R}}^{-\rho}_1$\,.
We claim that there exists $\mathcal{Q}_{aux}^{(1)}$ of the form \eqref{cannon31} 
such that
\begin{equation}\label{omoresto1}
\mathcal{Q}_1^{(1)}(W_1)W_1=0\,.
\end{equation}
The \eqref{cannon23} and the \eqref{omoresto1} imply the \eqref{Pollini77}.

In order to solve \eqref{omoresto1} we need a more explicit expression
of the remainder $\mathcal{Q}_1^{(1)}$ in \eqref{cannon24}
in terms of the generator of the flow $\mathcal{Q}_{aux}^{(1)}$.
This could be deduced by 
developing the computation in the proof 
of Proposition \ref{prop:PolliniSmooth}. In particular by calculating
explicitly the remainder
$\mathcal{G}_{\rho}$ in \eqref{cannon20bis} and using equation 
\eqref{cannon22bis}. However the computation is quite involved.
Then we reason as follows. 

We remark that the flow $\mathcal{A}_1^{\tau}$ in \eqref{Pollini44}
is $C^{k}$ in $\tau\in[0,1]$
with values in the Banach space $H^{s}(\mathbb{T};\CCC^2)$, 
because it solves an ODE on $H^s(\mathbb{T};\CCC^2)$.
Then, by Taylor expanding $P^{\tau}$ in \eqref{pollo22}
at $\tau=0$ (and using the Heisenberg equation
 \eqref{cannon1bis}),
we get
\begin{equation}\label{cannon26}
\mathcal{Y}^{(1)}=P^{1}(W_1)=
\widetilde{\mathcal{Y}}^{(1)}(W_1)+\big[\mathcal{Q}_{aux}^{(1)}(W_1)W_1,\widetilde{\mathcal{Y}}^{(1)}(W_1)\Big]+\int_0^{1}(1-\s)\pa_{\s}^{2}P^{\s}(W_1)d\s\,.
\end{equation}
Recall \eqref{deadwood30} and that
the operator $\widetilde{\mathcal{Q}}^{(1)}$ has the form
\[
\widetilde{\mathcal{Q}}^{(1)}(W_1)=\widetilde{\mathcal{Q}}_1^{(1)}(W_1)
+\widetilde{\mathcal{Q}}_{\geq2}^{(1)}(W_1)\,,\qquad 
\widetilde{\mathcal{Q}}^{(1)}_{\geq2}
\in\Sigma\mathcal{R}^{-\rho}_2[r,N]\otimes\mathcal{M}_2(\CCC)\,,
\]
with $\widetilde{\mathcal{Q}}_1^{(1)}\in \widetilde{\mathcal{R}}^{-\rho}_1$.
Then by \eqref{cannon26} we deduce
\begin{equation}\label{cannon25}
\mathcal{Y}^{(1)}(W_1)=\ii \Omega W_1+
\underbrace{\widetilde{\mathcal{Q}}_1^{(1)}(W_1)W_1
+\big[{\mathcal{Q}}_{aux}^{(1)}(W_1)W_1,\ii E\Omega W_1\big]}_{quadratic}
+\underbrace{M_{\geq 2}(W_1)W_1}_{cubic}\,,
\end{equation}
for some maps $M_{\geq 2}\in \Sigma\mathcal{M}_2[r,N]\otimes\mathcal{M}_2(\CCC)$.
Since $\mathcal{A}_1^{\tau}$ is regular, the quadratic terms in  \eqref{cannon25}
must coincide with the quadratic terms in \eqref{cannon23}.
Therefore we have (recall \eqref{cannon24})
\begin{equation}\label{cannon27}
\mathcal{Q}^{(1)}_1(W_1)W_1=
\widetilde{\mathcal{Q}}_1^{(1)}(W_1)W_1
+\big[{\mathcal{Q}}_{aux}^{(1)}(W_1)W_1,\ii E\Omega W_1\big]\,.
\end{equation}
Recall the notation \eqref{smooth-terms2}
and the definition of the non-linear commutator  \eqref{nonlinCommu}. 
Then equation \eqref{omoresto1}, using \eqref{cannon27},
is equivalent to
\begin{equation}\label{cannon28}
(\widetilde{\mathcal{Q}}_1^{(1)}(W_1))_{\s}^{\s'}w_1^{\s'}-\ii\s\Omega 
({\mathcal{Q}}_{aux}^{(1)}(W_1))_{\s}^{\s'}w_1^{\s'}+\ii\s' 
({\mathcal{Q}}_{aux}^{(1)}(W_1))_{\s}^{\s'}\Omega w_1^{\s'}
+({\mathcal{Q}}_{aux}^{(1)}(\ii E\Omega W_1))_{\s}^{\s'}w_1^{\s'}=0\,,
\end{equation}
for $\s,\s'\in\{\pm\}$ and
where $W_1=\vect{w_1}{\ov{w_1}}$ and $w_1^{+}=w_1$, $w_1^{-}=\ov{w_1}$. 
Passing to the Fourier representation (see \eqref{R2epep'}, \eqref{omegone})
we have that equation \eqref{cannon28} reads
\begin{equation}\label{cannon29}
(\widetilde{\mathcal{Q}}_1^{(1)}(W_1))_{\s,j}^{\s',k}+(\ii\s'\omega_{k}-\ii\s\omega_{j})
({\mathcal{Q}}_{aux}^{(1)}(W_1))_{\s,j}^{\s',k}+
({\mathcal{Q}}_{aux}^{(1)}(\ii E\Omega W_1))_{\s,j}^{\s',k}=0\,,
\end{equation}
for any $j,k\in \mathbb{Z}$.
Recall that, since $\widetilde{\mathcal{Q}}_1^{(1)}(W_1)\in
\widetilde{\mathcal{R}}^{-\rho}_{1}$,
the coefficients $(\widetilde{\mathcal{Q}}_1^{(1)}(W_1))_{\s,j}^{\s',k}$
have the form (see \eqref{BNF5})
\begin{equation}\label{cannon30}
(\widetilde{\mathcal{Q}}_1^{(1)}(W_1))_{\s,j}^{\s',k}=\sum_{
\substack{\s_1\in\{\pm\},j_1\in\mathbb{Z} \\ \s_1j_1=\s j-\s'k}}
((\mathtt{q}_1^{(1)})_{j_1}^{\s_1} )_{\s,j}^{\s',k} w_{j_1}^{\s_1}\,,\qquad
((\mathtt{q}_1^{(1)})_{j_1}^{\s_1} )_{\s,j}^{\s',k}\in \mathbb{C}\,.
\end{equation}
We look for the operator ${\mathcal{Q}}_{aux}^{(1)}$ as in \eqref{smooth-terms2}, \eqref{R2epep'} with coefficients as in \eqref{cannon31}.
%\begin{equation}\label{cannon31}
%({\mathcal{Q}}_{aux}^{(1)}(W_1))_{\s,j}^{\s',k}=\sum_{\substack{\s_1=\pm,j_1\in\mathbb{Z} \\ \s_1j_1=\s j-\s'k}}
%((\mathtt{q}_{aux}^{(1)})_{j_1}^{\s_1} )_{\s,j}^{\s',k} w_{j_1}^{\s_1}\,,\qquad
%((\mathtt{q}_{aux}^{(1)})_{j_1}^{\s_1} )_{\s,j}^{\s',k}\in \mathbb{C}\,.
%\end{equation}
Using \eqref{cannon30}, \eqref{cannon31}, we rewrite \eqref{cannon29}
as
\begin{equation*}%\label{cannon32}
((\mathtt{q}_1^{(1)})_{j_1}^{\s_1} )_{\s,j}^{\s',k}+
\big(
\ii\s'\omega_{k}-\ii\s\omega_{j}+\ii \s_1 j_1
\big)((\mathtt{q}_{aux}^{(1)})_{j_1}^{\s_1} )_{\s,j}^{\s',k}  = 0\,,
\end{equation*}
and hence we define
\begin{equation}\label{cannonsol}
((\mathtt{q}_{aux}^{(1)})_{j_1}^{\s_1} )_{\s,j}^{\s',k}:=\frac{-((\mathtt{q}_1^{(1)})_{j_1}^{\s_1} )_{\s,j}^{\s',k}}{\big(
\ii\s'\omega_{k}-\ii\s\omega_{j}+\ii \s_1 j_1
\big)}\,,\qquad \forall\,\; \s_1 j_1+\s'k=\s j\,,\;\; j_1,j,k\in\mathbb{Z}\,,\;\; 
\s,\s',\s_1\in\{\pm\}\,.
\end{equation}
Notice that, by assumption, the frequencies $\omega_{j}$ are not resonant 
according to Definition \ref{nonresOmegaCOND}.
By Lemma $6.5$ in \cite{BFP} we have that, since 
$\widetilde{\mathcal{Q}}_1^{(1)}\in \widetilde{\mathcal{R}}^{-\rho}_1$
and using the bounds \eqref{nonresOMEGA}, 
the operator
${\mathcal{Q}}_{aux}^{(1)}$ with coefficients as in \eqref{cannonsol}
belongs to  $ \widetilde{\mathcal{R}}^{-\rho}_1$.

To conclude the proof of the lemma it remains to show that the flow 
$\mathcal{A}_1^{\tau}$ with generators defined above is symplectic.
We show that ${\mathcal{Q}}_{aux}^{(1)}(W_1)W_1$ is an Hamiltonian vector field.
We reason as follows. First of all we have that (recall \eqref{X_H})
\[
\ii E\Omega W_1:=X_{H_2}(W_1)\,,\qquad 
H_2(W_1):=\frac{1}{2}\int_{\mathbb{T}}\Omega W_1\cdot\ov{W_1}dx=\sum_{j\in\mathbb{Z}}
\omega_{j}|(w_{1})_{j}|^{2}
\]
is the Hamiltonian vector field of the Hamiltonian $H_2$.
Moreover, the field $\widetilde{\mathcal{Y}}^{(1)}$ in \eqref{pollo7}
is Hamiltonian
and hence its quadratic terms are Hamiltonian. 
Then
\[
\begin{aligned}
&\widetilde{\mathcal{Q}}_1^{(1)}(W_1)(W_1)=X_{A}(W_1)\,,\\
&A(W_1)=\int_{\mathbb{T}}\widetilde{\mathtt{Q}}^{(1)}(W_1)W_1\cdot\ov{W_1}dx
=\sum_{\s_1 j_1+\s' k=\s j}(\widetilde{\mathtt{Q}}^{(1)})_{j_1,k,j}^{\s_1,\s',\s}
(w_1)_{j_1}^{\s_1}(w_1)_{k}^{\s'}(w_1)_{j}^{-\s}\,,
\end{aligned}
\]
for some multilinear map $\widetilde{\mathtt{Q}}^{(1)}\in 
\widetilde{\mathcal{M}}_1\otimes\mathcal{M}_2(\CCC)$.
Since ${\mathcal{Q}}_1^{(1)}(W_1)W_1$ solves
\eqref{omoresto1} (see also \eqref{cannon27}),
then one can check that
\[
{\mathcal{Q}}_{aux}^{(1)}(W_1)W_1=X_{{\rm ad}_{H_2}^{-1}A}(W_1)\,,
\qquad
({\rm ad}_{H_2}^{-1}A)(W_1):=
\sum_{\s_1 j_1+\s' k=\s j}\frac{(\widetilde{\mathtt{Q}}^{(1)})_{j_1,k,j}^{\s_1,\s',\s}}{\ii\s'\omega_{k}-\ii\s\omega_{j}+\ii \s_1 j_1}
(w_1)_{j_1}^{\s_1}(w_1)_{k}^{\s'}(w_1)_{j}^{-\s}\,,
\]
and hence it is Hamiltonian. This implies that the flow \eqref{Pollini44}
is symplectic. 
The bounds \eqref{Pollini99} follow by Theorem \ref{flussononlin}.
This concludes the proof.
\end{proof}

\subsection{Elimination of (j+1)-homogeneous terms}
Let $j\in \mathbb{N}$, $j\geq2$ and 
consider a para-differential system of the form
\begin{equation}\label{samba3}
\dot{W}_{j}=\mathcal{Y}^{(j)}(W_{j})
:=\ii E\Omega W_{j}+\ii E\opbw\big( \mathfrak{M}^{(j)}(W_{j};\x)\big)W_{j}
+\mathcal{Q}^{(j)}(W_{j})W_{j}\,,
\qquad W_{j}:=\vect{w_{j}}{\ov{w_{j}}}
\end{equation}
and assume the following. 
The matrix of symbols $\mathfrak{M}^{(j)}$
as the form
\begin{equation*}%\label{samba4}
\mathfrak{M}^{(j)}(W_{j};\x):=\left(
\begin{matrix}
\mathfrak{m}^{(j)}(W_{j};\x) & 0 \\
0 & \mathfrak{m}^{(j)}(W_{j};-\x)
\end{matrix}
\right)\,,\qquad \mathfrak{m}^{(j)}\in \Sigma\Gamma_2^{m}[r,N]\,,
\end{equation*}
and $\mathfrak{m}^{(j)}$ is real valued and independent of $x\in \mathbb{T}$.
Moreover we assume that it has the expansion (recall Definition \ref{def:resonant})
\begin{equation}\label{samba4bis}
\begin{aligned}
&\mathfrak{m}^{(j)}(W_{j};\x):=\sum_{k=2}^{j-1}\bral\mathfrak{m}^{(j)}_{k}\brar(W_{j};\x)+
\mathfrak{m}^{(j)}_{j}(W_{j};\x)+
\mathfrak{m}^{(j)}_{\geq j+1}(W_{j};\x)\,,\\
& \mathfrak{m}^{(j)}_{k}\in \widetilde{\Gamma}^{m}_{k}\,,\;\; k=2,\ldots,j\,,\;\;\;
\mathfrak{m}^{(j)}_{\geq j+1}\in \Sigma\Gamma^{m}_{j+1}[r,N]\,.
\end{aligned}
\end{equation}
The smoothing remainder 
$\mathcal{Q}^{(j)}(W_{j})$ admits the expansion
\begin{equation}\label{samba5}
\begin{aligned}
&\mathcal{Q}^{(j)}(W_{j})=\sum_{k=2}^{j-1}\bral\mathcal{Q}^{(j)}_{k}\brar(W_{j})+
\mathcal{Q}^{(j)}_{j}(W_{j})+\mathcal{Q}^{(j)}_{\geq j+1}(W_{j})\,,\\
&\mathcal{Q}^{(j)}_{k}\in \widetilde{\mathcal{R}}^{-\rho}_{k}\,,\;\;\;k=2,\ldots,j\,,\;\;\;
\mathcal{Q}^{(j)}_{\geq j+1}\in \Sigma\mathcal{R}^{-\rho}_{j+1}[r,N]
\otimes\mathcal{M}_2(\CCC)\,.
\end{aligned}
\end{equation}
The aim of the section is to eliminate all the non resonant terms
of degree of homogeneity $j+1$ appearing in the vector field 
$\mathcal{Y}^{(j)}$ in \eqref{samba3}, i.e.
the non resonant terms of 
\begin{equation*}%\label{jthtemrs}
\mathcal{Y}^{(j)}_{j}(W_{j}):=\ii E\opbw\left(
\begin{matrix}
\mathfrak{m}^{(j)}_{j}(W_{j};\x) & 0 \\ 0 & {\mathfrak{m}^{(j)}_{j}(W_{j};-\x)}
\end{matrix}
\right)W_{j}+\mathcal{Q}^{(j)}_{j}(W_{j})W_{j}\,.
\end{equation*}
As done in section \ref{quadtermBNF}
we eliminate $\mathcal{Y}^{(j)}_{j}$
in two steps. In section \ref{jthtermBNFsimbo} we reduce the symbol
$\mathfrak{m}^{(j)}_{j}$. In subsection \ref{jthtermBNFsmooth}
we deal with the $j$-homogeneous smoothing remainders.

\subsubsection{Elimination of j-homogeneous symbols}\label{jthtermBNFsimbo}
Consider a real valued, independent of $x\in \mathbb{T}$ symbol 
$b_{j}\in \tilde{\Gamma}^{m}_{j}$ and let $B_{j}(W_{j};\x)$ be a matrix
of symbols of the form \eqref{Pollini333} with $p=j$.
Let $\tilde{\mathcal{A}}_j^{\tau}$ be the solution of
\begin{equation}\label{samba6}
\left\{
\begin{aligned}
&\pa_{\tau}\tilde{\mathcal{A}}_j^{\tau}(W_{j})
=X_{\mathcal{G}_{j}}(\tilde{\mathcal{A}}_j^{\tau}(W_{j}))\\
&\tilde{\mathcal{A}}_j^{0}(W_{j})=W_{j}\,,
\end{aligned}\right.
\end{equation}
where 
\begin{equation}\label{samba6bis}
X_{\mathcal{G}_{j}}(W_{j})=\ii E\opbw(B_{j}(W_{j};\x))W_{j}+\mathcal{B}_{j}(W_{j})W_{j}\,,
\qquad
\mathcal{B}_{j}\in \widetilde{\mathcal{R}}^{-\rho}_{j}\,,
\end{equation}
is the Hamiltonian vector field of an Hamiltonian of the form 
\eqref{Pollini3} with $p=j$.
In the following lemma we conjugate the field \eqref{samba3}
under the flow \eqref{samba6}.

\begin{lemma}\label{lem:Pollinijth}
For $r>0$ small enough there exists 
a symbol $b_{j}\in \mathcal{\Gamma}^{m}_{j}$ such that
the following holds.
Setting  (recall \eqref{samba3})
\begin{equation}\label{samba7}
\tilde{W}_{j}:=\tilde{\mathcal{A}}_{j}(W_{j})
:=\tilde{\mathcal{A}}_{j}^{1}(W_{j})
\,,
\qquad
\tilde{\mathcal{Y}}^{(j)}(\tilde{W}_{j}):=P^{\tau}(\tilde{W}_{j})_{|\tau=1}
:=d\tilde{\mathcal{A}}_{p}^{\tau}
\big(\tilde{\mathcal{A}}_{p}^{-\tau}(\tilde{W}_{j})\big)
\big[ \mathcal{Y}^{(j)}(\tilde{\mathcal{A}}_{p}^{-\tau}(\tilde{W}_{j}))\big]_{|\tau=1}\,,
\end{equation}
we have that
\begin{equation}\label{samba8}
\left\{
\begin{aligned}
&\dot{\tilde{W}}_{j}=\tilde{\mathcal{Y}}^{(j)}(\tilde{W}_{j})
:=\ii E \Omega \tilde{W}_{j}+
\ii E\opbw\big( \tilde{\mathfrak{M}}^{(j)}(\tilde{W}_{j};\x)\big)[\tilde{W}_{j}]
+\tilde{\mathcal{Q}}^{(j)}(\tilde{W}_{j})[\tilde{W}_{j}]\\
&\tilde{W}_{j}(0)=\tilde{\mathcal{A}}_{j}(W_j(0))
\end{aligned}\right.
\end{equation}
where $\tilde{\mathcal{Q}}^{(j)}\in\Sigma\mathcal{R}^{-\rho}_2[r,N]
\otimes\mathcal{M}_2(\mathbb{C})$ and
$\tilde{\mathfrak{M}}^{(j)}\in 
\Sigma\Gamma^{m}_2[r,N]
\otimes\mathcal{M}_2(\mathbb{C})$
is independent of $x\in \mathbb{T}$, 
real valued  and has the form
\begin{align}
&\tilde{\mathfrak{M}}^{(j)}(\tilde{W}_{j};\x)
:=\left(
\begin{matrix}
\tilde{\mathfrak{m}}^{(j)}(\tilde{W}_{j};\x) & 0\\
0 & \tilde{\mathfrak{m}}^{(j)}(\tilde{W}_{j};-\x) 
\end{matrix}
\right)\,,\label{samba9}\\
&\tilde{\mathfrak{m}}^{(j)}(\tilde{W}_{j};\x)
:=\sum_{k=2}^{j}\bral\tilde{\mathfrak{m}}^{(j)}_{k}\brar(\tilde{W}_{j};\x)+
\tilde{\mathfrak{m}}^{(j)}_{j+1}(\tilde{W}_{j};\x)+
\tilde{\mathfrak{m}}^{(j)}_{\geq j+2}(\tilde{W}_{j};\x)\,,\label{samba10}\\
& \tilde{\mathfrak{m}}^{(j)}_{k}\in \tilde{\Gamma}^{m}_{k}\,,\;\; k=2,\ldots,j+1\,,\;\;\;
\tilde{\mathfrak{m}}^{(j)}_{\geq j+2}\in \Sigma\Gamma^{m}_{j+2}[r,N]\,.\nonumber
\end{align}
Moreover the remainder $\tilde{\mathcal{Q}}^{(j)}$ has the form
\begin{equation}\label{samba11}
\begin{aligned}
&\tilde{\mathcal{Q}}^{(j)}(\tilde{W}_{j})
=\sum_{k=2}^{j-1}\bral\tilde{\mathcal{Q}}^{(j)}_{k}\brar(\tilde{W}_{j})+
\tilde{\mathcal{Q}}^{(j)}_{j}(\tilde{W}_{j})
+\tilde{\mathcal{Q}}^{(j)}_{\geq j+1}(\tilde{W}_{j})\,,\\
&\tilde{\mathcal{Q}}^{(j)}_{k}\in \tilde{\mathcal{R}}^{-\rho}_{k}
\otimes\mathcal{M}_2(\CCC)\,,\;\;\;k=2,\ldots,j\,,\;\;\;
\tilde{\mathcal{Q}}^{(j)}_{\geq j+1}\in \Sigma\mathcal{R}^{-\rho}_{j+1}[r,N]
\otimes\mathcal{M}_2(\CCC)\,.
\end{aligned}
\end{equation}
Finally, for any $s\geq s_0$, the maps $\tilde{\mathcal{A}}^{\pm1}_{p}$ are symplectic 
and satisfy
\begin{equation}\label{samba12}
\|\tilde{\mathcal{A}}_{j}^{\pm1}(U)\|_{H^{s}}\leq \|U\|_{H^{s}}(1+C\|U\|^{p}_{H^{s_0}})\,,
\end{equation}
for some constant $C>0$ depending  on $s$.
\end{lemma}

\begin{proof}
The map $\tilde{\mathcal{A}}_{j}$ is symplectic by \eqref{samba6}, 
\eqref{samba6bis}
and satisfies \eqref{samba12} by Theorem \ref{flussononlin}.
Notice that the vector field $\mathcal{Y}^{(j)}$ in \eqref{samba3}
has the same form of \eqref{Pollini}
with $\mathfrak{N}\rightsquigarrow \mathfrak{M}^{(j)}$, 
$\mathcal{R}\rightsquigarrow \mathcal{Q}^{(j)}$, $U\rightsquigarrow W_{j}$.
Then 
Proposition \ref{prop:Pollini} applies to $\mathcal{Y}^{(j)}$. We obtain
(recall \eqref{samba7} and \eqref{Pollini7})
\begin{equation*}%\label{samba13}
\tilde{\mathcal{Y}}^{(j)}(\tilde{W}_{j})
:=\ii E \Omega \tilde{W}_{j}+
\ii E\opbw\big( \tilde{\mathfrak{M}}^{(j)}(\tilde{W}_{j};\x)\big)[\tilde{W}_{j}]
+\tilde{\mathcal{Q}}^{(j)}(\tilde{W}_{j})[\tilde{W}_{j}]
\end{equation*}
where 
$\tilde{\mathcal{Q}}^{(j)}\in\Sigma\mathcal{R}^{-\rho}_2[r,N]
\otimes\mathcal{M}_2(\CCC)$ and 
$\tilde{\mathfrak{M}}^{(j)}\in \Sigma\Gamma^{m}_{2}[r,N]\otimes\mathcal{M}_2(\CCC)$
has the form \eqref{samba9} 
with (see formula \eqref{Pollini8})
\begin{equation}\label{samba14}
\tilde{\mathfrak{m}}^{(j)}(\tilde{W}_{j};\x)
=\mathfrak{m}^{(j)}
(\tilde{\mathcal{A}}_{j}^{-1}(\tilde{W}_j);\x)+
\int_{0}^{1} (d_{Z}b_j)\big(\tilde{\mathcal{A}}_{j}^{\s}
\tilde{\mathcal{A}}_{j}^{-1}(\tilde{W}_j);\x\big)
[P^{\s}\big(
\tilde{\mathcal{A}}_{j}^{\s}\tilde{\mathcal{A}}_{j}^{-1}(\tilde{W}_j)
\big) ]d\s\,.
\end{equation}
To conclude the proof we need to show the expansions \eqref{samba10}
and \eqref{samba11}.

\vspace{0.5em}
\noindent
{\bf The homological equation.}
We look for a symbol $b_j\in \tilde{\Gamma}_{j}^{m}$ %as in \eqref{deadwood2}
such that  $\tilde{\mathfrak{m}}^{(j)}(\tilde{W}_j;\x)$ in \eqref{samba14}
satisfies \eqref{samba10}.
First of all, by Theorem \ref{flussononlin}, 
we deduce that the flow $\tilde{\mathcal{A}}_j^{\tau}$
of \eqref{samba6} is such that
\begin{align}
&\tilde{\mathcal{A}}_j^{\tau}(Z)=Z+N_j^{(1)}(\tau,Z)[Z]\,,\qquad
\tilde{\mathcal{A}}_j^{-\tau}(Z)=Z+N_j^{(2)}(\tau,Z)[Z]\,, \label{samba15}
\\& N_j^{(1)},N_{j}^{(2)}
\in\Sigma\mathcal{M}_{j}[r,N]\otimes\mathcal{M}_2(\mathbb{C})\,,\nonumber
\end{align}
with estimates uniform in $\tau\in[0,1]$ since the generator is a map in $\widetilde{\mathcal{M}}_j$.
We also recall that the symbol $\mathfrak{m}^{(j)}$ admits the expansion \eqref{samba4bis},
i.e. it is resonant (see Def. \ref{def:resonant})
up to degree of homogeneity $j$. 
Finally, notice that, for any $k\leq j-1 $, 
\begin{equation}\label{accendino}
\bral \mathfrak{m}_k^{(j)}\brar\big( \tilde{\mathcal{A}}_j^{-\tau}(\tilde{W}_{j});\x\big)
\stackrel{\eqref{samba15}}{=}
\bral \mathfrak{m}_k^{(j)}\brar\big(\tilde{W}_{j};\x\big)+\mathtt{r}_{j+k}(\tilde{W}_j;\x)\,,
\end{equation}
for some real valued and independent of $x\in \mathbb{T}$
symbol $\mathtt{r}_{j+k}\in \Sigma\Gamma^{m}_{j+k}[r,N]$. We now expand in 
degree of homogeneity the symbol in \eqref{samba14}.
By the discussion above we obtain
\begin{equation}\label{samba30}
\tilde{\mathfrak{m}}^{(j)}(\tilde{W}_{j};\x)
:=\sum_{k=2}^{j-1}\bral{\mathfrak{m}}^{(j)}_{k}\brar(\tilde{W}_{j};\x)+
\tilde{\mathfrak{m}}^{(j)}_{j}(\tilde{W}_{j};\x)+
\tilde{\mathfrak{m}}^{(j)}_{\geq j+1}(\tilde{W}_{j};\x)\,,
\end{equation}
for some real, independent of $x$ symbol $\tilde{\mathfrak{m}}^{(j)}_{\geq j+1}\in
\Sigma\Gamma^{m}_{j+1}[r,N]$ depending on $\mathfrak{m}^{(j)}$ and $b_{j}$,
and where
\begin{equation}\label{samba19}
\tilde{\mathfrak{m}}^{(j)}_{j}(\tilde{W}_{j};\x):=
\mathfrak{m}^{(j)}(\tilde{W}_{j};\x)+(d_{Z} b_{j})(\tilde{W}_{j};\x)\big[\ii E\Omega 
\tilde{W}_{j}
\big]\,.
\end{equation}

We prove the following.
\begin{lemma}{\bf (Homological equation).}\label{omoeqsimboJ}
There exists a symbol $b_j\in \tilde{\Gamma}^{m}_j$ such that
(see \eqref{samba19})
\begin{equation*}\label{samba20}
\tilde{\mathfrak{m}}^{(j)}_{j}(W;\x)=
\mathfrak{m}_j^{(j)}(W;\x)+(d_{Z} b_{j})(W;\x)
\big[\ii E\Omega W\big]=
\bral \mathfrak{m}_{j}^{(j)}\brar(W;\x)
\end{equation*}
Moreover $b_j(W;\x)$ is real valued and independent of $x\in \mathbb{T}$.
\end{lemma}

\begin{proof}
We recall that, by \eqref{espandoFousimbo}, the symbol $\mathfrak{m}^{(j)}_j$
has the form
\begin{equation}\label{samba21}
\mathfrak{m}^{(j)}_{j}(W;\x)
=\sum_{\substack{\s_i\in\{\pm\}\,,i=1,\ldots,j \\
n_i\in \mathbb{Z}\,,\\
\sum_{i=1}^{j}\s_i n_i=0}} 
(\mathfrak{m}_{j}^{(j)})_{n_1,\ldots,n_j}^{\s_1\cdots\s_{j}}(\x)w_{n_1}^{\s_1}
\ldots w_{n_j}^{\s_j}\,.
\end{equation}
Recall \eqref{nonresOMEGA3}, \eqref{omegone}
and define
\begin{equation}\label{samba22}
(b_{j})_{n_1,\ldots,n_j}^{\s_1\cdots\s_{j}}(\x):=\frac{-(\mathfrak{m}_{j}^{(j)})_{n_1,\ldots,n_j}^{\s_1\cdots\s_{j}}(\x)}{\ii(\s_1\omega_{n_1}+\ldots+\s_{j}\omega_{n_j})}\,,
\qquad 
\begin{aligned}
&\s_1n_1+\ldots+\s_j n_j=0\,,\\
&(\s_1,\ldots,\s_j,n_1,\ldots,n_j)\notin\mathcal{S}_j
\end{aligned}
\end{equation}
and $(b_{j})_{n_1,\ldots,n_j}^{\s_1\cdots\s_{j}}(\x):=0$ otherwise. 
One can check, by an explicit computation, that the symbol
\begin{equation}\label{samba23}
b_{j}(W;\x)
=\sum_{\substack{\s_i\in\{\pm\}\,,i=1,\ldots,j \\
n_i\in \mathbb{Z}\,,\\
\sum_{i=1}^{j}\s_i n_i=0}} 
(b_{j})_{n_1,\ldots,n_j}^{\s_1\cdots\s_{j}}(\x)w_{n_1}^{\s_1}
\ldots w_{n_j}^{\s_j}\,,
\end{equation}
with $(b_{j})_{n_1,\ldots,n_j}^{\s_1\cdots\s_{j}}(\x)$ in \eqref{samba22}
solves the equation \eqref{samba21}
where the r.h.s. 
$\bral \mathfrak{m}^{(j)}\brar(W;\x)$
is defined as in \eqref{espandoFousimbo100}.

Since the symbol $\mathfrak{m}^{(j)}_j$
 is real valued, its coefficients satisfies 
 \begin{equation}\label{samba24}
 \ov{(\mathfrak{m}_{j}^{(j)})_{n_1,\ldots,n_j}^{\s_1\cdots\s_{j}}(\x)}=
 (\mathfrak{m}_{j}^{(j)})_{n_1,\ldots,n_j}^{-\s_1\cdots-\s_{n}}(\x)\,.
 \end{equation}
 By formula \eqref{samba22} one can check that 
 $(b_{j})_{n_1,\ldots,n_j}^{\s_1\cdots\s_{j}}(\x)$ satisfies the same property as in \eqref{samba24}. Therefore the symbol in \eqref{samba23} is real valued.
\end{proof}
In view of Lemma \ref{omoeqsimboJ}
we have that formula \eqref{samba30}
implies the \eqref{samba10}
by setting $\tilde{\mathfrak{m}}^{(j)}_{k}=\mathfrak{m}^{(j)}_{k}$ for $0\leq k\leq j$.
In order to prove the \eqref{samba11} we reason as follows.
Recalling \eqref{samba9}, \eqref{samba10}, we define the operator 
\[
Q^{\tau}(W):=\ii E\Omega W+\ii E\opbw(\tilde{\mathfrak{M}}^{(j)}(W;\x))W\,.
\]
By the proof of Proposition \ref{prop:Pollini} (see \eqref{deadwood20}-\eqref{deadwood22}),
we deduce that the smoothing remainder in \eqref{samba8}
has the form
\begin{equation}\label{samba34}
\tilde{\mathcal{Q}}^{(j)}(W)W:=V^{\tau}\circ\tilde{\mathcal{A}}_{j}^{-\tau}(W)_{|\tau=1}
\end{equation}
where $V^{\tau}$ solves the problem (see \eqref{samba6bis})
\begin{equation*}%\label{samba33}
\left\{\begin{aligned}
&\pa_{\tau}V^{\tau}(U)=(d_{Z}X_{\mathcal{G}_j})(\tilde{\mathcal{A}}_j^{\tau})
\big[ V^{\tau}\big]\,,\\
&V^{0}(U)=-\mathcal{Q}^{(j)}(U)U\,.
\end{aligned}\right.
\end{equation*}
Since the generator $X_{\mathcal{G}_j}$ has degree of homogeneity
equals to $j+1$, the Taylor expansion of the remainder 
$\tilde{\mathcal{Q}}^{(j)}$ in \eqref{samba34}
coincides with the expansion in degree of homogeneity of the 
initial remainder $\mathcal{Q}^{(j)}$. Therefore, recalling \eqref{samba5},
we have that \eqref{samba11} holds
by setting
$\tilde{\mathcal{Q}}^{(j)}_{k}=\mathcal{Q}^{(j)}_{k}$ for $2\leq k\leq j$.
\end{proof}

\subsubsection{Elimination of j-homogeneous smoothing 
operators}\label{jthtermBNFsmooth}
Consider the vector field $\tilde{\mathcal{Y}}^{(j)}$
in \eqref{samba8}.
The aim of this section is to eliminate the non resonant terms
of degree of homogeneity
$j+1$ of $\tilde{\mathcal{Y}}^{(j)}$.
In view of Lemma \ref{lem:Pollinijth} we have that 
such terms appear only in the smoothing 
remainder $\tilde{\mathcal{Q}}^{(j)}$ 
(see \eqref{samba10}, \eqref{samba11}).

We now consider the flow
$\mathcal{A}_j^{\tau}$ of 
\begin{equation}\label{salsa1}
\left\{
\begin{aligned}
&\pa_{\tau}\mathcal{A}_j^{\tau}(W)=\mathcal{Q}^{(j)}_{aux}(\mathcal{A}_j^{\tau}(W))\mathcal{A}_1^{\tau}(W)\\
&\mathcal{A}_1^{0}(W)=W\,,
\end{aligned}\right.
\end{equation}
where $\mathcal{Q}^{(j)}_{aux}\in \tilde{\mathcal{R}}^{-\rho}_{j}$.
Assume also that the vector field 
$\mathcal{Q}_{aux}^{(j)}(W)W$
 is Hamiltonian, i.e. there is a map $\mathtt{Q}^{(j)}\in 
 \tilde{\mathcal{M}}_j\otimes\mathcal{M}_2(\CCC)$
 such that (recall \eqref{X_H})
 \begin{equation*}
 \mathcal{Q}_{aux}^{(j)}(W)W=X_{\mathcal{C}}(W)\,,\qquad \mathcal{C}(W)
 :=\int_{\mathbb{T}}\mathtt{Q}^{(j)}(W)W\cdot\ov{W}dx\,.
 \end{equation*}
 Finally
 assume that $\mathcal{Q}^{(j)}_{aux}$ has the form 
\eqref{smooth-terms2}, \eqref{R2epep'} with $p=j$ and 
 coefficients  (see \eqref{BNF5}) 
 \begin{equation}\label{salsa3}
({\mathcal{Q}}_{aux}^{(j)}(W))_{\s,k}^{\s',k'}
=\sum_{\substack{\s_i\in\{\pm\},n_i\in\mathbb{Z} \\ \sum_{i=1}^{j}\s_in_i=\s k-\s'k'}}
((\mathtt{q}_{aux}^{(j)})_{n_1,\ldots,n_{j}}^{\s_1\cdots\s_{j}} )_{\s,k}^{\s',k'} 
w_{n_1}^{\s_1}\ldots w_{n_j}^{\s_j}\,,\qquad
((\mathtt{q}_{aux}^{(j)})_{n_1,\ldots,n_{j}}^{\s_1\cdots\s_{j}} )_{\s,k}^{\s',k'} \in \mathbb{C}\,.
\end{equation}

\noindent In the following lemma we show that it is possible to choose the coefficients
$((\mathtt{q}_{aux}^{(j)})_{n_1,\ldots,n_{j}}^{\s_1\cdots\s_{j}} )_{\s,k}^{\s',k'} $
 in \eqref{salsa3}
in such a way that the vector field in \eqref{samba8} does not
contain any non resonant  monomials of degree $j+1$.

\begin{lemma}\label{prop:Pollinisalsa}
For $r>0$ small enough 
there exists $\mathcal{Q}_{aux}^{(j)}\in \tilde{\mathcal{R}}^{-\rho}_j$, 
of the form \eqref{salsa3},
such that the following holds.
Setting (recall \eqref{samba7})
\begin{equation}\label{salsa4}
\begin{aligned}
&W_{j+1}:={\mathcal{A}}_{j}(\tilde{W}_{j}):={\mathcal{A}}_{j}^{1}(\tilde{W}_j)\,,\\
%\qquad
&\mathcal{Y}^{(j+1)}(W_{j+1}):=P^{\tau}(W_{j+1})_{|\tau=1}
:=d{\mathcal{A}}_{j}^{\tau}\big({\mathcal{A}}_{j}^{-\tau}(W_{j+1})\big)
\big[ \tilde{\mathcal{Y}}^{(j)}({\mathcal{A}}_{j}^{-\tau}(W_{j+1}))\big]_{|\tau=1}\,,
\end{aligned}
\end{equation}
we have that
\begin{equation*}%\label{salsa5}
\left\{
\begin{aligned}
&\dot{W}_{j+1}=\mathcal{Y}^{(j+1)}(W_{j+1})
:=\ii E \Omega W_{j+1}+
\ii E\opbw\big( \mathfrak{M}^{(j+1)}(W_{j+1};\x)\big)[W_{j+1}]
+\mathcal{Q}^{(j+1)}(W_{j+1})[W_{j+1}]\\
&W_{j+1}(0)=\mathcal{A}_1(\tilde{W}_j(0))
\end{aligned}\right.
\end{equation*}
where $\mathcal{Q}^{(j+1)}\in\Sigma\mathcal{R}^{-\rho}_2[r,N]
\otimes\mathcal{M}_2(\mathbb{C})$ and
${\mathfrak{M}}^{(j+1)}\in 
\Sigma\Gamma^{m}_2[r,N]
\otimes\mathcal{M}_2(\mathbb{C})$
is independent of $x\in \mathbb{T}$, 
real valued  and has the form
\begin{align}
&{\mathfrak{M}}^{(j+1)}({W}_{j+1};\x)
:=\sm{{\mathfrak{m}}^{(j+1)}({W}_{j+1};\x)}{0}{0}{{\mathfrak{m}}^{(j+1)}({W}_{j+1};-\x)}
\,,\nonumber
%\label{salsa6}
\\
&{\mathfrak{m}}^{(j+1)}({W}_{j+1};\x)
:=\sum_{k=2}^{j}\bral{\mathfrak{m}}^{(j+1)}_{k}\brar({W}_{j};\x)+
{\mathfrak{m}}^{(j+1)}_{j+1}({W}_{j};\x)+
{\mathfrak{m}}^{(j+1)}_{\geq j+2}({W}_{j};\x)\,,\label{salsa7}\\
& {\mathfrak{m}}^{(j+1)}_{k}\in \tilde{\Gamma}^{m}_{k}\,,\;\; k=2,\ldots,j+1\,,\;\;\;
{\mathfrak{m}}^{(j+1)}_{\geq j+2}\in \Sigma\Gamma^{m}_{j+2}[r,N]\,.\nonumber
\end{align}
Moreover the remainder ${\mathcal{Q}}^{(j+1)}$ has the form
\begin{equation}\label{salsa8}
\begin{aligned}
&{\mathcal{Q}}^{(j+1)}({W}_{j+1})
=\sum_{k=2}^{j}\bral{\mathcal{Q}}^{(j+1)}_{k}\brar({W}_{j+1})+
\tilde{\mathcal{Q}}^{(j+1)}_{j+1}({W}_{j+1})
+{\mathcal{Q}}^{(j+1)}_{\geq j+2}({W}_{j+1})\,,\\
&{\mathcal{Q}}^{(j+1)}_{k}\in \tilde{\mathcal{R}}^{-\rho}_{k}
\otimes\mathcal{M}_2(\CCC)\,,\;\;\;k=2,\ldots,j+1\,,\;\;\;
{\mathcal{Q}}^{(j+1)}_{\geq j+2}\in \Sigma\mathcal{R}^{-\rho}_{j+1}[r,N]
\otimes\mathcal{M}_2(\CCC)\,.
\end{aligned}
\end{equation}
Finally, for any $s\geq s_0$, the maps $\mathcal{A}^{\pm1}_{j}$ 
are symplectic 
and satisfy
\begin{equation}\label{salsa9}
\|\mathcal{A}_{j}^{\pm1}(U)\|_{H^{s}}\leq \|U\|_{H^{s}}(1+C\|U\|_{H^{s_0}})\,,
\end{equation}
for some constant $C>0$ depending  on $s$.
\end{lemma}

\begin{proof}
Notice that the vector field $\tilde{\mathcal{Y}}^{(j)}$ in \eqref{samba7}, \eqref{samba8}
has the same form of $\mathcal{X}$ in \eqref{Pollini},
with $\mathfrak{N}\rightsquigarrow \tilde{\mathfrak{M}}^{(j)}$,
$\mathcal{R}\rightsquigarrow \tilde{\mathcal{Q}}^{(j)}$,
$U\rightsquigarrow \tilde{W}_{j}$.
The generator $\mathcal{Q}^{(j)}_{aux}$ in \eqref{salsa1} 
has the same properties of the 
generator $\mathcal{Q}^{(p)}_{aux}$, $p=j$,
 in \eqref{Pollini44bis}.
Therefore Proposition \ref{prop:PolliniSmooth} applies.
As a consequence we obtain that
\begin{equation}\label{salsa10}
\dot{W}_{j+1}=\ii E \Omega W_{j+1}+
\ii E\opbw\big( \mathfrak{M}^{(j+1)}(W_{j+1};\x)\big)[W_{j+1}]
+\mathcal{Q}^{(j+1)}(W_{j+1})[W_{j+1}]\,,
\end{equation}
where $\mathcal{Q}^{(j+1)}\in\Sigma\mathcal{R}^{-\rho}_2[r,N]
\otimes\mathcal{M}_2(\CCC)$, 
$\mathfrak{M}^{(j+1)}$ has the form \eqref{salsa7} with
(see \eqref{Pollini8bis})
\begin{equation}\label{salsa11}
\mathfrak{m}^{(j+1)}(\tau,W_{j+1};\x)
:=\tilde{\mathfrak{m}}^{(j+1)}(\mathcal{A}_{j}^{-\tau}(W_{j+1});\x)\in 
\Sigma\Gamma^{m}_2[r,N]\,.
\end{equation}
Reasoning as in \eqref{samba15}, 
since the generator in \eqref{salsa1} has homogeneity $j+1$, 
we deduce that the flow ${\mathcal{A}}_j^{\tau}$
of \eqref{samba6} is such that
\begin{align}
&{\mathcal{A}}_j^{\tau}(Z)=Z+N_j^{(3)}(\tau,Z)[Z]\,,\qquad
{\mathcal{A}}_j^{-\tau}(Z)=Z+N_j^{(4)}(\tau,Z)[Z]\,, \label{salsa12}
\\& N_j^{(3)},N_{j}^{(3)}
\in\Sigma\mathcal{M}_{j}[r,N]\otimes\mathcal{M}_2(\mathbb{C})\,.\nonumber
\end{align}
Then, using \eqref{salsa11}, \eqref{salsa12} and reasoning as in \eqref{accendino}, 
we deduce that the expansion \eqref{salsa7}
holds by setting
$\mathfrak{m}^{(j+1)}_{k}=\tilde{\mathfrak{m}}^{(j)}_k$ for $2\leq k\leq j$.
Let us check the \eqref{salsa8}.
In order to provide an explicit expression of the terms
of homogeneity smaller than $j$ in the remainder $\mathcal{Q}^{(j+1)}$
we reason as in \eqref{cannon26}.
We Taylor expand the vector field $\mathcal{Y}^{(j+1)}$ by using formula \eqref{salsa4}.
We get
\begin{equation*}%\label{salsa13}
\mathcal{Y}^{(j+1)}(W_{j+1})=
\tilde{\mathcal{Y}}^{(j)}(W_{j+1})+\big[\mathcal{Q}_{aux}^{(j)}(W_{j+1})W_{j+1},
\tilde{\mathcal{Y}}^{(j)}(W_{j+1})\Big]+\int_0^{1}(1-\s)\pa_{\s}^{2}P^{\s}(W_{j+1})d\s\,.
\end{equation*}
Recalling \eqref{samba8}-\eqref{samba11} we obtain the expansion
\begin{equation}\label{salsa14}
\begin{aligned}
\mathcal{Y}^{(j+1)}(W_{j+1})&=\ii E\Omega W_{j+1}+
\sum_{k=1}^{j}
\opbw\big(
\sm{\bral \tilde{\mathfrak{m}}_{k}^{(j)}\brar(W_{j+1};\x)}{0}{0}{\bral 
\tilde{\mathfrak{m}}_{k}^{(j)}\brar(W_{j+1};-\x)}\big)W_{j+1}+
\sum_{k=1}^{j-1}\bral \tilde{\mathcal{Q}}^{(j)}_k\brar(W_{j+1})W_{j+1}\\
&+\mathcal{Q}_{j}^{(j)}(W_{j+1})W_{j+1}+\mathtt{M}_{>j}(W_{j+1})W_{j+1}\,,
\end{aligned}
\end{equation}
where $\mathtt{M}_{> j}$ is some map in 
$\Sigma\mathcal{M}_{j+1}[r,N]\otimes\mathcal{M}_2(\CCC)$ and 
\begin{equation}\label{salsa15}
\mathcal{Q}_{j}^{(j+1)}(W_{j+1})W_{j+1}:=\tilde{\mathcal{Q}}_{j}^{(j)}(W_{j+1})W_{j+1}+
\big[\mathcal{Q}_{aux}^{(j)}(W_{j+1})W_{j+1},\ii E\Omega W_{j+1}\big]\,.
\end{equation}
The expansion \eqref{salsa14} coincide with the expansion
of the vector field in \eqref{salsa10}
in multilinear maps. Notice that
the terms of homogeneity $j+1$
\[
\opbw\big(
\sm{\bral \tilde{\mathfrak{m}}_{j}^{(j)}\brar(W_{j+1};\x)}{0}{0}{\bral 
\tilde{\mathfrak{m}}_{k}^{(j)}\brar(W_{j+1};-\x)}\big)W_{j+1}
\]
are \emph{resonant} (see Def. \ref{def:resonant}), 
i.e. they are  already in Birkhoff normal form. Hence the only non-resonant
terms of degree $j+1$ belongs to 
$\mathcal{Q}_{j}^{(j+1)}(W_{j+1})W_{j+1}$ in \eqref{salsa15}.
By the definition of the non-linear commutator \eqref{nonlinCommu}
one can easily note that the term 
$\big[\mathcal{Q}_{aux}^{(j)}(W_{j+1})W_{j+1},\ii E\Omega W_{j+1}\big]$
does not contain any \emph{resonant} monomials. Then we define
\begin{equation}\label{salsa16}
\tilde{\mathcal{Q}}_{j}^{(j), \perp}(W_{j+1})W_{j+1}:=
\tilde{\mathcal{Q}}_{j}^{(j)}(W_{j+1})W_{j+1}-
\bral\tilde{\mathcal{Q}}_{j}^{(j)}\brar(W_{j+1})W_{j+1}
\end{equation} 
the non-resonant part of \eqref{salsa15}.
We claim that the vector field 
$\tilde{\mathcal{Q}}_{j}^{(j), \perp}(W_{j+1})W_{j+1}$ is Hamiltonian. 
Indeed we have that 
$\tilde{\mathcal{Y}}^{(j)}$ is Hamiltonian. 
Hence also its terms of homogeneity
$(j+1)$, i.e. (see \eqref{salsa16})
\begin{equation}\label{salsa20}
\tilde{\mathcal{Y}}^{(j)}_{j}(W_{j+1}):=\opbw\big(
\sm{\bral \tilde{\mathfrak{m}}_{j}^{(j)}\brar(W_{j+1};\x)}{0}{0}{\bral 
\tilde{\mathfrak{m}}_{k}^{(j)}\brar(W_{j+1};-\x)}\big)W_{j+1}+\bral\tilde{\mathcal{Q}}_{j}^{(j)}\brar(W_{j+1})W_{j+1}+\tilde{\mathcal{Q}}_{j}^{(j), \perp}(W_{j+1})W_{j+1}
\end{equation}
 are Hamiltonian. This means that there is a multilinear map 
 $\mathtt{M}_{j}\in \tilde{\mathcal{M}}_{j}$ such that
$  \tilde{\mathcal{Y}}^{(j)}_{j}(W_{j+1})$ %=X_{A}(W_{j})$
is the Hamiltonian vector field of the Hamiltonian function 
$A(W_{j+1}):=A_1(W_{j+1})+A_{2}(W_{j+1})$ with
 \[
 \begin{aligned}
 A_1(W_{j+1})&:=\frac{1}{2}\int_{\mathbb{T}}\opbw\big(
\sm{\bral \tilde{\mathfrak{m}}_{j}^{(j)}\brar(W_{j+1};\x)}{0}{0}{\bral 
\tilde{\mathfrak{m}}_{k}^{(j)}\brar(W_{j+1};-\x)}\big)W_{j+1}\cdot\ov{W}_{j+1}dx\,,\\
%\qquad
A_2(W_{j+1})&:=
\int_{\mathbb{T}}\mathtt{M}_j(W_{j+1})W_{j+1}\cdot\ov{W}_{j+1}dx\,.
\end{aligned}
 \]
By Proposition \ref{stimedifferent3} and Lemma \ref{stimedifferent2Multi}
the vector field of $A_1(W_{j+1})$ has the form
\[
X_{A_1}(W_{j+1})=\opbw\big(
\sm{\bral \tilde{\mathfrak{m}}_{j}^{(j)}\brar(W_{j+1};\x)}{0}{0}{\bral 
\tilde{\mathfrak{m}}_{k}^{(j)}\brar(W_{j+1};-\x)}\big)W_{j+1}+R_1(W_{j+1})W_{j+1}\,,
\qquad
R_1(W_{j+1})\equiv\bral R_1\brar(W_{j+1})\,,
\]
for some $R_{1}\in \tilde{\mathcal{R}}^{-\rho}_j\otimes\mathcal{M}_2(\CCC)$.
On the other hand the vector field of $A_2(W_{j+1})$ has the form
\[
X_{A_2}(W_{j+1})=\bral R_2\brar(W_{j+1})W_{j+1}
+ \Big(R_2(W_{j+1})W_{j+1}-\bral R_2\brar(W_{j+1})W_{j+1}\Big)\,,
\]
for some multilinear map $R_2\in \tilde{\mathcal{M}}_j\otimes\mathcal{M}_2(\CCC)$.
Then, recalling \eqref{salsa20}, we must have
\begin{equation}\label{pantera}
\begin{aligned}
&\bral \tilde{\mathcal{Q}}_j^{(j)}\brar(W_{j+1})W_{j+1}\equiv
\bral R_1\brar(W_j)W_{j+1}+\bral R_2\brar(W_{j+1})W_{j+1}\,,\\
&\tilde{\mathcal{Q}}_{j}^{(j), \perp}(W_{j+1})W_{j+1}\equiv R_2(W_{j+1})W_{j+1}
-\bral R_2\brar(W_{j+1})W_{j+1}\,,
\end{aligned}
\end{equation}
implying that $ \tilde{\mathcal{Q}}_{j}^{(j), \perp}(W_{j+1})W_{j+1}$ is Hamiltonian 
(see Remark \ref{montefuffa}).
In order to conclude the proof of Lemma \ref{prop:Pollinisalsa} we need the following.
\begin{lemma}{\bf (Homological equation).}\label{lem:salsa}
There is an Hamiltonian vector field of the form
$\mathcal{Q}^{(j)}_{aux}(W)W$ with 
$\mathcal{Q}^{(j)}_{aux}\in \tilde{\mathcal{R}}^{-\rho}_j\otimes\mathcal{M}_{2}({\CCC})$
such that (see \eqref{salsa16}, \eqref{salsa15})
\begin{equation}\label{salsa31}
\tilde{\mathcal{Q}}_{j}^{(j), \perp}(W)W
+\big[\mathcal{Q}_{aux}^{(j)}(W)W,\ii E\Omega W\big] =0\,.
\end{equation}
\end{lemma}

\begin{proof}
We look for a solution $\mathcal{Q}^{(j)}_{aux}$ in the class of multilinear 
operators $\tilde{\mathcal{R}}^{-\rho}_{j}\otimes\mathcal{M}_2(\CCC)$
of the form \eqref{smooth-terms2}, \eqref{R2epep'}
with coefficients
 \begin{align} \label{salsa333}
 ({\mathcal{Q}}_{aux}^{(j)}(W))_{\s,k}^{\s',k'} :=
 \frac{1}{(2\pi)^{j}}
 \sum_{\substack{\s_i\in\{\pm\}, n_i\in\Z \\ 
 \sum_{i=1}^{p}\s_i n_i=\s k-\s'k' }}
 \big( ({\mathtt{q}}_{aux}^{(j)})_{n_1,\ldots,n_j}^{\s_1\cdots \s_{j}}\big)_{\s,k}^{\s',k'}
%  (\mathtt{r}_{2,\ep,\ep'})^{\s,\s'}_{n_1,n_2,k}
  w_{n_1}^{\s_1}\ldots w_{n_j}^{\s_{j}}  \, ,   \quad  
  k,k'\in \Z   \, , 
 \end{align}
 for some 
 $ \big( ({\mathtt{q}}_{aux}^{(j)})_{n_1,\ldots,n_j}^{\s_1\cdots \s_{j}}\big)_{\s,k}^{\s',k'}\in\CCC$.
For convenience we write
\[
\mathcal{Q}^{(j)}_{aux}(W)W=\mathcal{Q}^{(j)}_{aux}(\underbrace{W,\ldots,W}_{j-times})W\,.
\]
With this notation we have, for any $Y=\vect{y}{\bar{y}}$, 
\[
d_W\Big(\mathcal{Q}^{(j)}_{aux}(W)W\Big)[Y]=
\mathcal{Q}^{(j)}_{aux}(W,\ldots,W)Y
+\sum_{k=1}^{j}\mathcal{Q}^{(j)}_{aux}(W,\ldots,\underbrace{Y}_{k-th},\ldots,W)W\,.
\]
Therefore (recall \eqref{nonlinCommu}) the equation 
\eqref{salsa31} reads
\begin{equation}\label{salsa31tris}
\begin{aligned}
(\tilde{\mathcal{Q}}_{j}^{(j), \perp}(W))_{\s}^{\s'}w_{\s'}
&-\ii \s\Omega( \tilde{\mathcal{Q}}_{j}^{(j), \perp}(W))_{\s}^{\s'}w^{\s'}
+\ii\s' ( \tilde{\mathcal{Q}}_{j}^{(j), \perp}(W))_{\s}^{\s'}\Omega w^{\s'}\\
&+
\sum_{k=1}^{j}
(\mathcal{Q}^{(j)}_{aux}(W,\ldots,\ii E\Omega W,\ldots,W))_{\s}^{\s'}w^{\s'}=0\,.
\end{aligned}
\end{equation}
for any $\s,\s'\in\{\pm\}$.
Recall that $(\tilde{\mathcal{Q}}_{j}^{(j), \perp}(W))_{\s}^{\s'}$
has the form \eqref{R2epep'} with coefficients
 \begin{align*} %\label{salsa33}
 (\tilde{\mathcal{Q}}_{j}^{(j), \perp}(W))_{\s,k}^{\s',k'} :=
 \frac{1}{(2\pi)^{j}}
 \sum_{\substack{\s_i\in\{\pm\}, n_i\in\Z \\ 
 \sum_{i=1}^{p}\s_i n_i=\s k-\s'k' }}
 \big( (\tilde{\mathtt{q}}_{j}^{(j)})_{n_1,\ldots,n_j}^{\s_1\cdots \s_{j}}\big)_{\s,k}^{\s',k'}
%  (\mathtt{r}_{2,\ep,\ep'})^{\s,\s'}_{n_1,n_2,k}
  w_{n_1}^{\s_1}\ldots w_{n_j}^{\s_{j}}  \, ,   \quad  
  k,k'\in \Z   \, , 
 \end{align*}
for 
$\big( (\tilde{\mathtt{q}}_{j}^{(j)})_{n_1,\ldots,n_j}^{\s_1\cdots \s_{j}}\big)_{\s,k}^{\s',k'}\in\CCC$
and where the sum is restricted to indexes outside the 
set $\mathcal{S}_{j}$ (see \eqref{nonresOMEGA3}).
Passing to the Fourier coefficients
the equation \eqref{salsa31tris}
becomes (recall \eqref{salsa333})
\begin{equation*}%\label{salsa40}
 \big( (\tilde{\mathtt{q}}_{j}^{(j)})_{n_1,\ldots,n_j}^{\s_1\cdots \s_{j}}\big)_{\s,k}^{\s',k'}
 +
\ii \Big(\sum_{i=1}^{j}\s_i\omega_{n_i}-\s k+\s'k'\Big)
  \big( ({\mathtt{q}}_{aux}^{(j)})_{n_1,\ldots,n_j}^{\s_1\cdots \s_{j}}\big)_{\s,k}^{\s',k'}=0\,,
\end{equation*}
for any indexes satisfying
\begin{equation}\label{salsa41}
\sum_{i=1}^{j}\s_i n_i=\s k-\s'k'\,,\qquad 
(\s_1,\ldots,\s_{j},\s,\s', n_1,\ldots,k,k')\notin\mathcal{S}_{j}\,.
\end{equation}
Therefore we define the operator $\mathcal{Q}^{(j)}_{aux}(W)$ as in \eqref{salsa333}
with coefficients
\begin{equation}\label{salsa42}
  \big( ({\mathtt{q}}_{aux}^{(j)})_{n_1,\ldots,n_j}^{\s_1\cdots \s_{j}}\big)_{\s,k}^{\s',k'}=
  \frac{
  - \big( (\tilde{\mathtt{q}}_{j}^{(j)})_{n_1,\ldots,n_j}^{\s_1\cdots \s_{j}}\big)_{\s,k}^{\s',k'}}
  {\ii(\s_1n_1+\ldots+\s_{j}n_{j}+\s'k'-\s k)}
\end{equation}
for indexes satisfying \eqref{salsa41} and $0$ otherwise.
Thanks to \eqref{salsa42}, the bound \eqref{nonresOMEGA}
and reasoning as in Lemma $6.5$ in \cite{BFP}
one can check that 
$\mathcal{Q}^{(j)}_{aux}\in \tilde{\mathcal{R}}^{-\rho}_j\otimes\mathcal{M}_{2}({\CCC})$.
Finally the vector field 
$\mathcal{Q}^{(j)}_{aux}(W)W$ si Hamiltonian 
since it solves \eqref{salsa31} and the field
$\tilde{\mathcal{Q}}_{j}^{(j), \perp}(W)W$ is Hamiltonian (see \eqref{pantera}).
\end{proof}
We conclude the proof of Lemma \ref{prop:Pollinisalsa}.
Thanks to Lemma \ref{lem:salsa} we have that
the operator ${\mathcal{Q}}_{j}^{(j+1)}(W)$ in 
\eqref{salsa15} is equal to
$\bral \tilde{\mathcal{Q}}_{j}^{(j)}\brar(W)$.
Hence the expansion \eqref{salsa8}
follows by \eqref{salsa10}, \eqref{salsa14}
setting ${\mathcal{Q}}^{(j+1)}_{k}=
\tilde{\mathcal{Q}}^{(j)}_{k}$ for $2\leq k\leq j$.
Since, by Lemma \ref{lem:salsa}, $\mathcal{Q}^{(j)}_{aux}(W)W$ is Hamiltonian
then the flow in \eqref{salsa1} 
is symplectic.
The estimates \eqref{salsa9} follow by Theorem \ref{flussononlin}.
This concludes the proof.
\end{proof}

\begin{proof}[{\bf Proof of Corollary \ref{thm:energy}}]
Consider the system \eqref{YYYN}.
Since 
$\mathcal{Q}_N\in\Sigma\mathcal{R}^{-\rho}_2[r,N]
\otimes\mathcal{M}_2(\mathbb{C})$ and 
$\mathfrak{M}^{(N)}$ belongs to 
$\Sigma\Gamma^{m}_{2}[r,N]\otimes\mathcal{M}_{2}(\CCC)$ we write
\begin{equation}\label{sarosaro2}
\begin{aligned}
&\mathfrak{M}^{(N)}(W;\x)
=\sum_{j=2}^{N-1}\mathfrak{M}_{j}^{(N)}(W;\x)+\mathfrak{M}^{(N)}_{N}(W;\x)\,,
\quad
\mathfrak{M}^{(N)}_{j}\in \widetilde{\Gamma}^{m}_{j}\otimes\mathcal{M}_2(\CCC)\,,
\;\;\;
\mathfrak{M}^{(N)}_{N}\in {\Gamma}^{m}_{N}[r]\otimes\mathcal{M}_2(\CCC)
\\
&\mathcal{Q}_{N}(W)=
\sum_{j=2}^{N-1}\mathcal{Q}_{N,j}(W)+\mathcal{Q}_{N,N}(W)\,,
\quad \mathcal{Q}_{N,j}\in \widetilde{\mathcal{R}}^{-\rho}_{j}\otimes\mathcal{M}_2(\CCC)\,,
\;\;\;
\mathcal{Q}_{N,N}\in {\mathcal{R}}^{-\rho}_{N}[r]\otimes\mathcal{M}_2(\CCC)\,.
\end{aligned}
\end{equation}
In particular (see \eqref{sarosaro}) we have 
\begin{equation}\label{sarosaroApp}
\begin{aligned}
&\mathfrak{M}^{(N)}_{N}(W;\x):=\left(
\begin{matrix}
\mathfrak{m}^{(N)}_{N}(W;\x) & 0\\
0 & {\mathfrak{m}^{(N)}_{N}(W;-\x) }
\end{matrix}
\right)\,,
\qquad \mathfrak{m}^{(N)}_{N}(W;\x)=\ov{\mathfrak{m}^{(N)}_{N}(W;\x)}\,,
\end{aligned}
\end{equation}
with $\mathfrak{m}^{(N)}_{N}\in \Gamma_{N}^{m}[r]$ independent of $x\in \mathbb{T}$. 
We also recall that $\mathcal{Q}_{N,N}$ is a real-to-real matrix of operators, hence we can write (recall \eqref{smooth-terms2}, \eqref{opeBarrato}, \eqref{vinello})
\begin{equation}\label{saroMafia10}
\mathcal{Q}_{N,N}=\Big( (\mathcal{Q}_{N,N})_{\s}^{\s'}\Big)_{\s,\s'\in\{\pm\}}\,,
\qquad (\mathcal{Q}_{N,N})_{\s}^{\s'}=\ov{(\mathcal{Q}_{N,N})_{-\s}^{-\s'}}\,.
\end{equation}
By Theorem \ref{thm:mainBNF} we know that the vector field $\mathcal{Y}_{N}$ 
in \eqref{YYYN} is Hamiltonian. Hence, recalling the expansions 
\eqref{sarosaro2} and Definition \ref{def:resonant}, 
we have (see \eqref{X_H})
\begin{equation}\label{sarosaro3}
\begin{aligned}
&\mathcal{Y}_{N}(W)=\ii J \nabla H^{(N)}(W)\,,\quad
H^{(N)}(W)=\frac{1}{2}\int_{\mathbb{T}}\Omega W\cdot\ov{W}+H^{(N)}_{<N}(W)+H^{(N)}_{\geq N}(W)\,,\\
&H^{(N)}_{<N}(W)=\sum_{j=2}^{N-1} H^{(N)}_{j}(W)\,,\qquad
H^{(N)}_{j}(W):=\int_{\mathbb{T}} M_{j}(W)W\cdot\ov{W}\,,
\;\;\;
H^{(N)}_{\geq N}(W):=\int_{\mathbb{T}} M_{N}(W)W\cdot\ov{W}
\end{aligned}
\end{equation}
for some multilinear maps 
$M_j\in \widetilde{\mathcal{M}}_{j}\otimes\mathcal{M}_{2}(\CCC)$
and a non-homogeneous map $M_{N}\in \mathcal{M}_N[r]\otimes\mathcal{M}_2(\CCC)$.
We deduce that the term with
 highest homogeneity  in $\mathcal{Y}_{N}$
is
\begin{equation}\label{saroMafia}
\ii E\opbw\big(\mathfrak{M}^{(N)}_{N}(W;\x)\big)W+\mathcal{Q}_{N,N}(W)W
=\ii J\nabla H_{\geq N}^{(N)}(W)\,,
\end{equation}
while at lower orders we have
\begin{equation*}
\ii E\Omega W+
\sum_{j=2}^{N-1}\left(\ii E\opbw\big(\bral \mathfrak{M}^{(N)}_{j}\brar(W;\x)\big)W
+
\bral\mathcal{Q}_{N,j}\brar(W)W\right)=
\ii E\Omega W+\sum_{j=2}^{N-2}\ii J\nabla H_{j}^{(N)}(W)\,.
\end{equation*}
Moreover (recall Def. \ref{def:resonant})
these terms having lower degree of homogeneity  are \emph{resonant}.
Therefore we must have that
the Hamiltonians $H_j^{(N)}$ have the following form.
For $j=2p$ 
 (recall \eqref{nonresOMEGA3}) we have
\begin{equation*} %\label{resonantHamHam}
H_j^{(N)}(W)=
\sum_{
\substack{n_i\in \mathbb{Z},i=0,\ldots,j+1\\
\{|n_0|,\ldots,|n_{p}|\}=\{|n_{p+1}|,\ldots,|n_{j+1}|\}}
}
h_{n_0,\ldots,n_{j+1}}w_{n_0}\cdots w_{n_{p}}\ov{w_{n_{p+1}}} \cdots\ov{w_{n_{j+1}}}\,.
\end{equation*}
If $j$ is odd then $H_j^{(N)}\equiv0$. 
In particular the Hamiltonians $H_j^{(N)}$ are real valued and the coefficients
$h_{n_0,\ldots,n_{j+1}}\in \mathbb{C}$
 are invariant under permutations of the indexes $n_0,\ldots, n_{j+1}$.
Recalling  \eqref{Sobnorm2} we define
\begin{equation*} %\label{sobnormHamHam}
G(W):=\|w\|_{H^{s}}^{2}=\int_{\mathbb{T}}\langle D\rangle^{s}w\cdot
\langle D\rangle^{s}\ov{w}dx=
\sum_{j\in \mathbb{Z}} \langle j \rangle^{2s} |w_{j}|^{2} \,.
\end{equation*}
It is a straightforward computation to check that (recall \eqref{Poisson})
\begin{equation}\label{HamZERO}
\{G,H_{j}^{(N)}\}\equiv0\,.
\end{equation}
Then
\[
\begin{aligned}
\pa_{t}\|w\|_{H^{s}}^{2}&\stackrel{\eqref{Sobnorm2}}{=}
2{\rm Re}\big(\langle D\rangle^{s}\dot{w}, \langle D\rangle^{s}w\big)_{L^2}
\stackrel{\eqref{YYYN}, \eqref{sarosaro3}, \eqref{Poisson}}{=}
\{G,H^{(N)}\}\stackrel{\eqref{HamZERO}}{=}
\{G,H^{(N)}_{\geq N}\}\\
&\stackrel{\eqref{saroMafia}, \eqref{sarosaroApp}, \eqref{saroMafia10}}{=}
2{\rm Re}\big(\langle D\rangle^{s}
\ii\opbw(\mathfrak{m}^{(N)}_{ N}(W;\x) )w,\langle D\rangle^{s}w  \big)_{L^{2}}\\
&+2{\rm Re}\Big(\langle D\rangle^{s}\big(
(\mathcal{Q}_{N,N}(W))_{+}^{+}w+(\mathcal{Q}_{N,N}(W))_{+}^{-}\bar{w}
\big) ,\langle D\rangle^{s}w\Big)_{L^2}\,.
\end{aligned}
\]
Since the symbol $\mathfrak{m}^{(N)}_{ N}(W;\x)$
is real-valued and independent of $x\in \mathbb{T}$ we have
\[
{\rm Re}\big(\langle D\rangle^{s}
\ii\opbw(\mathfrak{m}^{(N)}_{ N}(W;\x) )w,\langle D\rangle^{s}w  \big)_{L^{2}}=0\,.
\]
By estimate \eqref{porto20} on the smoothing remainders 
and using the Cauchy-Schwarz inequality we get
\[
{\rm Re}\Big(\langle D\rangle^{s}\big( 
(\mathcal{Q}_{N,N}(W))_{+}^{+}w+(\mathcal{Q}_{N,N}(W))_{+}^{-}\bar{w}
\big) ,\langle D\rangle^{s}w\Big)_{L^2}\lesssim_{s}\|W\|_{H^{s}}^{N+2}\,.
\]
Therefore we have obtained 
$
\pa_{t}\|w\|_{H^{s}}^2\lesssim_{s}\|w\|_{H^{s}}^{N+2}\,,
$
by integrating in $t$ we get the \eqref{cor:energy}.
\end{proof}

%\bibliography{bibliografiaNLSnonlin.bib}

\begin{thebibliography}{10}

\bibitem{Baldi1}
P.~Baldi.
\newblock Periodic solutions of fully nonlinear autonomous equations of
  {B}enjamin-{O}no type.
\newblock {\em Annales de l'Institut Henri Poincar\'e (C) Analyse non
  lin\'eaire}, 30(1):33--77, 2013.

\bibitem{BBHM}
P.~Baldi, M.~Berti, E.~Haus, and R.~Montalto.
\newblock Time quasi-periodic gravity water waves in finite depth.
\newblock {\em Inventiones Mathematicae}, 214(2):739--911, Jul 2018.

\bibitem{BBM}
P.~Baldi, M.~Berti, and R.~Montalto.
\newblock {KAM} for quasi-linear and fully nonlinear forced perturbations of
  {A}iry equation.
\newblock {\em Math. Ann.}, 359, 2014.
%\newblock 10.1007/s00208-013-1001-7.

\bibitem{BBM1}
P.~Baldi, M.~Berti, and R.~Montalto.
\newblock {KAM} for autonomous quasilinear perturbations of {K}d{V}.
\newblock {\em Ann. I. H. Poincar\'e (C) Anal. Non Lin\'eaire}, 33, 2016.
%\newblock 10.1016/j.anihpc.2015.07.003.

\bibitem{BDGS}
D.~Bambusi, J.~M. Delort, B.~Gr\'ebert, and J.~Szeftel.
\newblock {A}lmost global existence for {H}amiltonian semi-linear
  {K}lein-{G}ordon equations with small {C}auchy data on {Z}oll manifolds.
\newblock {\em Comm. Pure Appl. Math.}, 60:1665--1690, 2007.
%\newblock 10.1002/cpa.20181.

\bibitem{BG}
D.~Bambusi and B.~Gr\'ebert.
\newblock {B}irkhoff normal form for partial differential equations with tame
  modulus.
\newblock {\em Duke Math. J.}, 135 n. 3:507--567, 2006.
%\newblock 10.1215/S0012-7094-06-13534-2.

\bibitem{BLM}
D. Bambusi, B. Langella, R. Montalto.
\newblock Reducibility of Non-Resonant Transport Equation on
with Unbounded Perturbations. 
\newblock {\em Ann. Henri Poincar\'e}, 20(3):1893--1929, 2019.
%https://doi.org/10.1007/s00023-019-00795-2


\bibitem{Bertibook} M. Berti. 
\newblock \emph{{N}onlinear {O}scillations of {H}amiltonian {PDE}s}.
\newblock Progress in Nonlinear Differential Equations and Its Applications. 
Birkh\"auser Verlag, 2006.

\bibitem{BD}
M.~Berti and J.M. Delort.
\newblock {\em {A}lmost global solutions of capillary-gravity water waves
  equations on the circle}.
\newblock UMI Lecture Notes, 2017.
\newblock (awarded UMI book prize 2017).

\bibitem{BFF1}
M.~Berti, R.~Feola, and L.~Franzoi.
\newblock Quadratic life span of periodic gravity-capillary water waves.
\newblock {\em preprint arXiv:1905.05424}, 2019.

\bibitem{BFP}
M.~Berti, R.~Feola, and F.~Pusateri.
\newblock Birkhoff normal form and long time existence for periodic gravity
  water waves.
\newblock {\em preprint arXiv:1810.11549}, 2018.

\bibitem{BM1}
M.~Berti and R.~Montalto.
\newblock {\em Quasi-periodic Standing Wave Solutions of Gravity-capillary
  Water Waves}.
\newblock Memoirs of the American Mathematical Society. American Mathematical
  Society, 2016.

\bibitem{bony}
J.~M. Bony.
\newblock {C}alcul symbolique et propagation des singularit\'es pour les
  \'equations aux d\'eriv\'ees partielle non lin\'eaire.
\newblock {\em Ann. Sci. \'Ecole Norm. Sup.}, 14:209--246, 1981.

\bibitem{Delort-2009}
J.~M. Delort.
\newblock {A} quasi-linear {B}irkhoff normal forms method. {A}pplication to the
  quasi-linear {K}lein-{G}ordon equation on $\mathds{S}^1$.
\newblock {\em Ast\'erisque}, 341, 2012.

\bibitem{Delort-Sphere}
J.~M. Delort.
\newblock {\em {Q}uasi-{L}inear {P}erturbations of {H}amiltonian
  {K}lein-{G}ordon {E}quations on {S}pheres}.
\newblock American Mathematical Society, 2015.
%\newblock 10.1090/memo/1103.

\bibitem{DelortSzeft1}
J.~M. Delort and J.~Szeftel.
\newblock {L}ong-time existence for small data nonlinear {K}lein--{G}ordon
  equations on tori and spheres.
\newblock {\em Internat. Math. Res. Notices}, 37, 2004.
%\newblock 10.1155/S1073792804133321.

\bibitem{DelortSzeft2}
J.~M. Delort and J.~Szeftel.
\newblock {L}ong-time existence for semi-linear {K}lein--{G}ordon equations
  with small cauchy data on {Z}oll manifolds.
\newblock {\em Amer. J. Math.}, 128, 2006.
%\newblock 10.1353/ajm.2006.0038.

\bibitem{EGK}
L.~H. Eliasson, B.~Gr\'ebert, and S.~B. Kuksin.
\newblock Kam for the nonlinear beam equation.
\newblock {\em Izv. RAN. Ser. Mat}, 16:1588--1715, 2016.

\bibitem{Faouplane}
E.~Faou, L.~Gauckler, and C.~Lubich.
\newblock {S}obolev stability of plane wave solutions to the cubic nonlinear
  {S}chr{\"o}dinger equation on a torus.
\newblock {\em Comm. Partial Differential Equations}, 38:1123--1140, 2013.

\bibitem{grefaou}
E.~Faou and B.~Gr\'ebert.
\newblock {Q}uasi invariant modified {S}obolev norms for semi linear reversible
  {P}{D}{E}s.
\newblock {\em Nonlinearity}, 23:429--443, 2010.
%\newblock 10.1088/0951-7715/23/2/011.

\bibitem{Feireisl}
E.~Feireisl.
\newblock Time-periodic solutions of a quasilinear beam equation via
  accelerated convergence methods.
\newblock {\em Aplikace matematiky}, 33(5):362--373, 1988.

\bibitem{FGP}
R.~Feola, F.~Giuliani, and M.~Procesi. 
\newblock Reducibility for a class of weakly dispersive linear operators arising from the Degasperis Procesi equation.
\newblock {\em Dynamics of Partial Differential Equations}, 16(1): 25-94, 2019.
%, 
%DOI: http://dx.doi.org/10.4310/DPDE.2019.v16.n1.a2.

\bibitem{FGP1}
R.~Feola, F.~Giuliani, and M.~Procesi.
\newblock Reducible Kam tori for degasperis-procesi equation.
\newblock {\em accepted on ``Comm. in Math. Phys'' arXiv:1812.08498}, 2019.

\bibitem{FGMP} R. Feola, F. Giuliani, R. Montalto, M. Procesi. 
\newblock Reducibility of first order linear operators on tori via Moser's theorem.
\newblock {\em Journal of Functional Analysis},
276(3) : 932-970, 2019. 
%DOI: https://doi.org/10.1016/j.jfa.2018.10.009.

\bibitem{FIloc}
R.~Feola and F.~Iandoli.
\newblock {L}ocal well-posedness for quasi-linear {N}{L}{S} with large {C}auchy
  data on the circle.
\newblock {\em Annales de l'Institut Henri Poincare (C) Analyse non
  lin\'eaire}, 36(1):119--164, 2018.

\bibitem{Feola-Iandoli-Long}
R.~Feola and F.~Iandoli.
\newblock {L}ong time existence for fully nonlinear {N}{L}{S} with small
  {C}auchy data on the circle.
\newblock {\em Annali della Scuola Normale Superiore di Pisa (Classe di
  Scienze)}, 2019.
\newblock to appear: 10.2422/2036-2145.201811-003.

\bibitem{FP}
R.~Feola and M.~Procesi.
\newblock Quasi-periodic solutions for fully nonlinear forced reversible
  {S}chr{\"o}dinger equations.
\newblock {\em Journal of Differential Equations}, 259(7): 3389--3447, (2015).
%\newblock 10.1016/j.jde.2015.04.025.

\bibitem{Filippo}
F.~Giuliani.
\newblock {Q}uasi-periodic solutions for quasi-linear generalized {K}d{V}
  equations.
\newblock 
{\em Journal of Differential Equations},   262(10) :\phantom{,}5052 -- 5132, 2017.

\bibitem{IPT}
G.~Iooss, P.I. Plotnikov, and J.F. Toland.
\newblock Standing waves on an infinitely deep perfect fluid under gravity.
\newblock {\em Arch. Ration. Mech. Anal.}, 177(3):367--478, 2005.
%\newblock 10.1007/s00205-005-0381-6.

\bibitem{kuksinbook}
S. B. Kuksin. 
\newblock \emph{Analysis of Hamiltonian PDEs}. 
\newblock Oxford University Press, 2000.


\bibitem{Met}
G.~M\'etivier.
\newblock {\em {P}ara-{D}ifferential {C}alculus and {A}pplications to the
  {C}auchy {P}roblem for {N}onlinear {S}ystems}, volume~5.
\newblock Edizioni della Normale, 2008.

\bibitem{P-T}
P.I. Plotnikov and J.F. Toland.
\newblock {N}ash-{M}oser theory for standing water waves.
\newblock {\em Arch. Ration. Mech. Anal.}, 159:1--83, 2001.

\bibitem{Ruda}
I.~A. Rudakov.
\newblock Periodic solutions of the quasilinear equation of forced beam
  vibrations with homogeneous boundary conditions.
\newblock {\em Izv. RAN. Ser. Mat}, 79:215--238, 2015.

\bibitem{zakhbook}
V. E. Zakharov, editor. 
\newblock \emph{What is integrability?} 
\newblock Springer, 1991.


\end{thebibliography}

\def\cprime{$'$}

\end{document}